\documentclass[a4paper,openright,twoside,11pt]{report}
\usepackage{geometry}
\usepackage[latin1]{inputenc}
\usepackage{fancyhdr}
\usepackage{indentfirst}
\usepackage{graphicx}
\usepackage{newlfont}
\usepackage{amssymb}
\usepackage{amsmath}
\usepackage{latexsym}
\usepackage{amsthm}
\usepackage{mathrsfs}
\usepackage{color}
\usepackage{mathtools}
\usepackage[numbered]{bookmark}
\usepackage{hyperref}
\usepackage{tikz}
\usepackage{tikz-cd}
\usepackage{tkz-euclide}
\usetkzobj{all}
\usepackage{array}

\theoremstyle{plain}
\newtheorem{theo}{Theorem}[section]
\newtheorem{prop}[theo]{Proposition}
\newtheorem{cor}[theo]{Corollary}
\newtheorem{lem}[theo]{Lemma}
\newtheorem{conj}[theo]{Conjecture}
\theoremstyle{definition}
\newtheorem{defn}[theo]{Definition}
\newtheorem{notat}[theo]{Notation}
\newtheorem{ex}[theo]{Example}
\theoremstyle{remark}
\newtheorem{rem}[theo]{Remark}

\theoremstyle{plain}
\newtheorem{theoC}{Theorem}[chapter]
\newtheorem{propC}[theoC]{Proposition}
\newtheorem{corC}[theoC]{Corollary}
\newtheorem{lemC}[theoC]{Lemma}
\newtheorem{conjC}[theoC]{Conjecture}
\theoremstyle{definition}
\newtheorem{defnC}[theoC]{Definition}
\newtheorem{notatC}[theoC]{Notation}
\newtheorem{exC}[theoC]{Example}
\theoremstyle{remark}
\newtheorem{remC}[theoC]{Remark}

\theoremstyle{plain}
\newtheorem{theoI}{Theorem}[]
\newtheorem{propI}[theoI]{Proposition}

\newtheorem{conjI}[theoI]{Conjecture}
\theoremstyle{definition}

\theoremstyle{remark}

\theoremstyle{plain}
\newtheorem{innercustomgeneric}{\customgenericname}
\providecommand{\customgenericname}{}
\newcommand{\newcustomtheorem}[2]{%
	\newenvironment{#1}[1]
	{%
		\renewcommand\customgenericname{#2}%
		\renewcommand\theinnercustomgeneric{##1}%
		\innercustomgeneric
	}
	{\endinnercustomgeneric}
}
\newcustomtheorem{lconj}{Conjecture}

\newcommand {\n}[1] {\mathbb{#1}}
\newcommand{\C}[1]{\mathcal#1}
\newcommand{\inn}{\!\in\!}
\newcommand{\sub}{\!\subset\!}
\newcommand{\imp}{\Rightarrow}
\DeclareMathOperator{\id}{id}
\DeclareMathOperator{\Ker}{Ker}
\DeclareMathOperator{\Imm}{Im}
\renewcommand{\Im}{\text{\upshape{Im}}}
\renewcommand{\Re}{\text{\upshape{Re}}}
\newcommand{\ds}{\displaystyle}

\DeclareMathOperator{\ord}{ord}
\DeclareMathOperator{\Div}{Div}
\renewcommand{\div}{\text{\upshape{div}}}
\newcommand{\ch}[1]{\overline{\n #1}}
\newcommand{\ov}{\overline}
\newcommand{\floor}[1]{\left\lfloor #1 \right\rfloor}
\newcommand{\mm}[4]{\ds\left(\!\!\!\begin{array}{cc}#1\!&\!#2\\#3\!&\!#4\end{array}\!\!\!\right)}
\renewcommand{\char}{\text{char }}
\newcommand{\be}{\begin{equation}}
\newcommand{\ee}{\end{equation}}
\DeclareMathOperator{\Pic}{Pic}
\renewcommand{\epsilon}{\varepsilon}
\newcommand{\pphi}{\varphi}
\newcommand{\cv}[1]{\frac{p_{#1}}{q_{#1}}}
\renewcommand{\S}[1]{\mathcal S_{\n #1}}
\newcommand{\PSL}[1]{\text{\upshape{PSL}}_2(\n#1)}
\newcommand{\vv}[2]{\ds\left(\!\!\!\begin{array}{c}#1\\#2\end{array}\!\!\!\right)}

\DeclareMathOperator{\Tr}{Tr}
\DeclareMathOperator{\tr}{tr}
\DeclareMathOperator{\SL}{SL}
\DeclareMathOperator{\GL}{GL}
\newcommand{\M}{\mathfrak m}

\author{Francesca Malagoli}
\title{Continued fractions in function fields:\newline	 Polynomial analogues of McMullen's and Zaremba's conjectures}
\date{\today}

\begin{document}\let\oldref=\ref \renewcommand{\ref}[1]{\textup{\oldref{#1}}}
\thispagestyle{empty}
\begin{center}
\vspace*{-32mm}{
\includegraphics[scale=0.32]{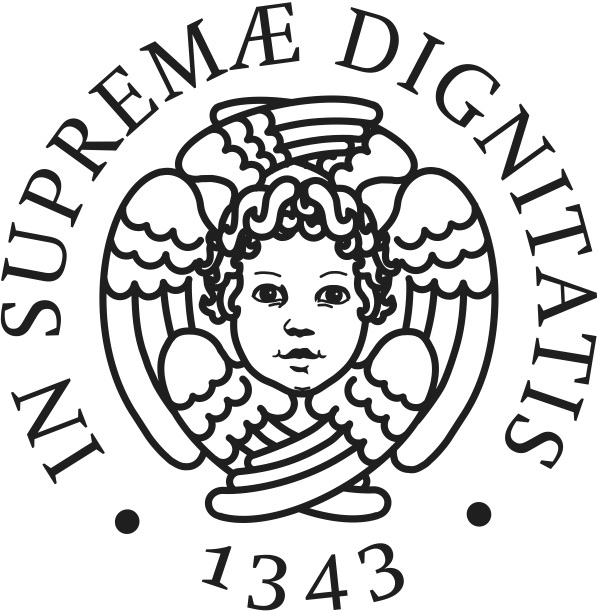}}
\end{center}

\vspace*{-0.18cm} \centerline{{\LARGE UNIVERSIT\`{A} DI PISA}}

{\centerline{DOTTORATO IN MATEMATICA} \centerline{XXVIII CICLO}

\centerline{\rule{16cm}{0.5mm}}\vspace*{2cm}

\centerline{\LARGE{Francesca Malagoli}} 
\vspace*{2.5cm}
\centerline{\Huge{\textbf{Continued fractions in function fields:}}} 
\vspace{0,3cm}
\centerline{ \Huge{ \textbf{polynomial analogues of}}}
\vspace{0,3cm}
\centerline{ \Huge{ \textbf{McMullen's and Zaremba's conjectures}}}
\vspace*{2cm}
\centerline{\rule{6cm}{0.5mm}}\medskip
\centerline{\large{TESI DI DOTTORATO}}
\centerline{\rule{6cm}{0.5mm}}
\vspace*{2.4cm}

\begin{flushright}
\makebox[7cm][l]
{\large{Relatore:}}\par
\makebox[7cm][l]
{\large{Umberto Zannier}}\par
\end{flushright}

\vspace*{1.8cm}

\centerline{\rule{16.0cm}{0.5mm}}\medskip
\centerline{\large{ANNO ACCADEMICO 2015/2016}}

\clearpage{\pagestyle{empty}\cleardoublepage}
\chapter*{Acknowledgements}
\thispagestyle{empty}
I would like to express my sincere gratitude to my advisor, Professor Umberto Zannier, for providing me with the opportunity to work on such an interesting subject and for the help offered by means of worthy remarks and insights.\par\medskip

I also would like to thank Professors Marmi and Panti for their help and attention, for their precious suggestions and for the time that they kindly devoted to it. I am grateful to the referees for their precise and valuable remarks. \par\medskip
 
I finally thank Professors Mirella Manaresi and Rita Pardini for their constant help and support.

\clearpage{\pagestyle{empty}\cleardoublepage}
\tableofcontents
\clearpage{\pagestyle{empty}\cleardoublepage}
\chapter*{Introduction}\addcontentsline{toc}{chapter}{Introduction}
The classical theory of real continued fractions, whose origins probably date back to the antiquity, was already laid down by Euler in the 18th century and then further developed, among others, by Lagrange and Galois.

As it is well known, there are close connections between the arithmetic behaviour of algebraic number fields and that of the algebraic function fields in one variable, this analogy being stronger in the case of function fields over finite fields. For instance, in the 19th century Abel \cite{Abel1826} and Chebyshev \cite{chebyshev1857integration} started to developed a theory of polynomial continued fractions. More precisely, the ring of polynomials $\n K[T]$ over a given field $\n K$ can play the role of the ring of integers $\n Z$ and thus its field of fractions $\n K(T)$ will correspond to the field of rational numbers $\n Q$. Then, the role of $\n R$ will be played by the completion of $\n K(T)$ with respect to the valuation $\ord$ associated to the degree, that is, by $\n L=\n K((T^{-1}))$, the field of formal Laurent series in $T^{-1}$.  Every formal Laurent series $\alpha=\sum\limits_{i\leq n}c_iT^i$ can be written (uniquely) as a regular simple continued fraction, that is, as $$\alpha=a_0+\cfrac{1}{a_1+\cfrac{1}{a_2+\cfrac{1}{\ddots}}}\ ,$$ where $a_0\inn\n K[T]$ and where the $a_i$ are polynomials of positive degree. As in the real case, the $a_i$ are called the \textit{partial quotients} of $\alpha$ and we will write $$\alpha=[a_0,a_1,\dots].$$ 

Almost all the classical notions and results can be translated in the function fields setting; in particular in this thesis we will consider the polynomial analogues of Zaremba's and McMullen's Conjectures on continued fractions with bounded partial quotients. We will focus especially on the second one, proving it when $\n K$ is an infinite algebraic extension of a finite field, when $\n K$ is uncountable, when $\n K=\ch Q$ and when $\n K$ is a number field; we will also see that, if $\n K$ is a finite field, then the polynomial analogue of Zaremba's Conjecture over $\n K$ would imply the polynomial analogue of McMullen's. Of course, according to the base field $\n K$, we will have to use different methods. \par\medskip

Real numbers with bounded partial quotients appear in many fields of mathematics and computer science, for instance in Diophantine approximation, fractal geometry, transcendental number theory, ergodic theory, numerical analysis, pseudo-random number generation, dynamical systems and formal language theory (for a survey, see \cite{shallit1992real}). 

While studying numerical integration, pseudo-random number generation and quasi-Monte Carlo methods, Zaremba in \cite{zaremba1972methode} (page 76) conjectured that for every positive integer $d\geq2$ there exists an integer $b<d$, relatively prime to $d$, such that all of the partial quotients in the continued fraction of $b/d$ are less than or equal to $5$. More generally, we can consider the following Conjecture:
\begin{conjI}[Zaremba]\label{0conj:Zarint}
	There exists an absolute constant $z$ such that for every integer $d\geq2$ there exists $b$, relatively prime to $d$, with the partial quotients of $b/d$ bounded by $z$.
\end{conjI}

In \cite{wilson1980limit}, Wilson proved that any real quadratic field $\n Q(\sqrt d)$ contains infinitely many purely periodic continued fractions whose partial quotients are bounded by a constant $m_d$ depending only on $d$ (for example, we can take $m_5=2$). McMullen \cite{mcmullen2009uniformly} explained these phenomena in terms of closed geodesics on the modular surface and conjectured the existence of an absolute constant $m$, independent from $d$:
\begin{conjI}[McMullen]\label{0conj:MMint}
	There exists an absolute constant $m$ such that for every positive squarefree integer $d$ the real quadratic field $\n Q(\sqrt d)$ contains infinitely many purely periodic continued fractions whose partial quotients are bounded by $m$.
\end{conjI}
Actually, McMullen conjectured that $m_d=2$ should be sufficient for every $d$.\par\medskip

Zaremba's and McMullen's Conjectures have been deeply studied by Bourgain and Kontorovich, who showed that they are special cases of a much more general local-global Conjecture and proved a density-one version of Zaremba's Conjecture (Theorem 1.2 in \cite{bourgain2014zaremba}). However, both Conjectures are still open.\par\medskip

In the polynomial setting, the property of having bounded partial quotients corresponds to the property of having partial quotients of bounded degree. When $\n K$ is a finite field, Laurent series whose continued fraction expansions have bounded partial quotients appear in stream cipher theory, as they are directly linked to the study of linear complexity properties of sequences and pseudorandom number generation. Indeed, Niederreiter (\cite{niederreiter1987sequences}, Theorem 2) proved that the linear complexity profile of a sequence is as close as possible to the expected behaviour of random sequences if and only if all the partial quotients of the corresponding Laurent series are linear. \medskip

We will study the following analogue of Zaremba's Conjecture for polynomial continued fractions over a field $\n K$:
\begin{lconj}{Z}[Polynomial analogue of Conjecture \ref{0conj:Zarint}]\label{0conj:Zarpol}
	There exists a constant $z_{\n K}$ such that for every non-constant polynomial $f\inn\n K[T]$ there exists $g\inn\n K[T]$, relatively prime to $f$, such that all the partial quotients of $f/g$ have degree at most $z_{\n K}$.
\end{lconj}

Actually, it is believed that it is enough to take $z_{\n K}=1$ for any field $\n K\neq\n F_2$ and $z_{\n F_2}=2$. This conjecture has already been studied many authors, including Blackburn \cite{blackburn1998orthogonal}, Friesen \cite{friesen2007rational}, Lauder \cite{lauder1999continued} and Niederreiter \cite{niederreiter1987rational}:

\begin{theoI}[Blackburn \cite{blackburn1998orthogonal}, Friesen \cite{friesen2007rational}]
	Conjecture \ref{0conj:Zarpol} holds with $z_{\n K}=1$ whenever $\n K$ is an infinite field.
	
	If $\n K=\n F_q$ is a finite field and $\deg f<q/2$, then there exists a polynomial $g$, relatively prime to $f$, such that all the partial quotients of $f/g$ are linear.
\end{theoI}

Analogously, we will study a polynomial version of McMullen's Conjecture. Let $\n K$ be a field with characteristic different from 2. Then for every polynomial $D\inn\n K[T]$ of even degree, which is not a square in $\n K[T]$ and whose leading coefficient is a square in $\n K$, the square root of $D$ is well defined as a formal Laurent series. We can then consider the following conjecture:
\begin{lconj}{M}[Polynomial analogue of Conjecture \ref{0conj:MMint}]\label{0conj:MMpol}
	Let $\n K$ be a field with $\char\n K\neq2$. Then there exists a constant $m_{\n K}$ such that for every polynomial $D\inn\n K[T]$ satisfying the previous conditions there exist infinitely many pairwise non-equivalent\footnote{Two formal Laurent series are said to be equivalent if their continued fraction expansions can be obtained one from the other by adding (or removing) finitely many partial quotients or by multiplication by a non-zero constant.} elements of $\n K(T,\sqrt D)\setminus\n K(T)$ whose partial quotients have degree at most $m_{\n K}$.
\end{lconj}

It is believed that for every admissible field $\n K$ it is enough to take $m_{\n K}=1$. We will see that this is true in all the cases where the Conjecture has been proved.

We will also consider the following strengthening of Conjecture \ref{0conj:MMpol} (with $m_{\n K}=1$): 

\begin{conjI}\label{0conj:multip}
	For every polynomial $D\inn\n K[T]$ satisfying the previous hypotheses there exists a polynomial $f\inn\n K[T]$ such that the partial quotients of $f\sqrt D$ (except possibly for finitely many of them) have degree 1.
\end{conjI}

\begin{theoI}
	Conjecture \ref{0conj:multip}, and thus Conjecture \ref{0conj:MMpol}, holds when $\n K$ is an uncountable field, when $\n K=\ch Q$ and when $\n K$ is a number field.
\end{theoI}

It is easy to see that this can never happen when $\n K$ is a finite field or, more generally, an algebraic extension of a finite field, because in this case the continued fraction of $f\sqrt D$ is always periodic, with infinitely many partial quotients of degree $\deg f+\frac12\deg D$. However, this also implies that the polynomial analogue of the Pell equation always has non-trivial polynomial solutions, which will allow us to follow other strategies. More precisely, we will prove the following:
\begin{theoI}
	If $\n K$ is an algebraic extension of a finite field, then Conjecture \ref{0conj:MMpol} is a consequence of Conjecture \ref{0conj:Zarpol}, with $m_{\n K}=z_{\n K}$. 
	
	In particular, if $\n K$ is an infinite algebraic extension of a finite field, then Conjecture \ref{0conj:MMpol} holds with $m_{\n K}=1$. 
\end{theoI}

\noindent{\large\textbf{Organization of the thesis}}

In the first two chapters, for the convenience of the reader, we will recall some known results on continued fractions in function fields, focusing on the tools that we will use to discuss the polynomial analogues of Zaremba's and McMullen's Conjectures. In particular, we will study the multiplication of a continued fraction by a polynomial; even if this is a very simple issue, we have not been able to find previous proofs of these results in the literature.\par\medskip

In the first Chapter, after some generic facts on continued fractions on normed fields, we will introduce formal Laurent series and their continued fraction expansion based on the polynomial part. The classical theory of real continued fractions can be transposed almost entirely in the function fields setting; in particular, the continued fraction expansion of a Laurent series will provide a sequence of rational functions, called convergents, which are its best approximations through rational functions. The only noteworthy differences between the real and the polynomial settings are due to the fact that now the absolute value is non-Archimedean; this will lead to uniqueness of continued fraction expansions and to more precise results on the approximation of a Laurent series by its convergents.

In section \ref{1sec:sec3} we will consider some simple algebraic operations on continued fractions and, aiming at the study of Conjectures \ref{0conj:Zarpol} and \ref{0conj:MMpol}, we will always highlight the results related to the partial quotients' degrees. In particular, in order to treat Conjecture \ref{0conj:multip}, in \ref{1subsec:mult} we will examine the connection between the degrees of the partial quotients of $\alpha$ and those of $(T-\lambda)\alpha$ or $\ds\frac\alpha{(T-\lambda)}$, where $\alpha$ is a formal Laurent series over $\n K$ and $\lambda$ is a constant:
\begin{propI}
	Let $\alpha=[a_0,a_1,\dots]\inn\n K((T^{-1}))$ be a formal Laurent series such that $1<\sup_n\deg a_n<\infty$, let $\cv n$ be its convergents and let $\lambda\inn\n K$. 
	
	Let $\beta=(T-\lambda)\alpha=[b_0,b_1,\dots]$. $$\text{If } q_n(\lambda)\neq0 \text{ for every } n, \text{ then } \sup_m \deg b_m=\max\{\sup_n\deg a_n-1,1\}.$$
	
	Let $\gamma=\ds\frac{\alpha}{T-\lambda}=[c_0,c_1,\dots]$. $$\text{If } p_n(\lambda)\neq0 \text{ for every } n, \text{ then } \sup_m \deg c_m=\max\{\sup_n\deg a_n-1,1\}.$$
\end{propI}

In the second Chapter (assuming the characteristic of $\n K$ to be different from 2) we will focus on quadratic irrationalities, for which, as in the real setting, more precise results can be proven. This case had already been treated by Abel, Chebyshev and, later, by Artin. We will see that when $\n K$ is an algebraic extension of a finite field all of the classical theory, including Lagrange's and Galois' well known Theorems, carries over to the polynomial setting. However, this is not true in general, because if $\n K$ is an infinite field then the set of polynomials over $\n K$ of degree bounded by a given constant is infinite. Moreover, as in the classical case, the continued fraction expansion of square roots is strictly connected with the existence of solutions to an analogue of Pell's equation. 

Finally, we will discuss a connection, already observed by Abel, between continued fractions in quadratic function fields and the theory of hyperelliptic curves.\par\medskip

In the third Chapter we will present the original Conjectures of Zaremba and McMullen on real continued fractions, giving a brief summary of the known results, mainly due to Bourgain and Kontorovich.\par\medskip

In Chapter 4 we will then consider Conjecture \ref{0conj:Zarpol}. As it was already known to Blackburn, when $\n K$ is infinite, Conjecture \ref{0conj:Zarpol} holds with $z_{\n K}=1$; we will prove again this result with a different method, based on the multiplication of a continued fraction by a polynomial, in Theorem \ref{4prop:Zarfinite}. On the other hand, when $\n K$ is a finite field there are only partial results towards Conjecture \ref{0conj:Zarpol}, concerning polynomials with small degree, such as Theorem \ref{4prop:Zarfinite} (Blackburn), Theorems \ref{4theo:Friesen} and \ref{4theo:Friesen2} (Friesen) or Corollary \ref{4cor:zarsplits}.\par\medskip

In Chapter 5 we will finally study Conjecture \ref{0conj:MMpol}, proving it, with $m_{\n K}=1$, in the following cases.
\begin{itemize} 
	\item Adapting to the polynomial setting a result of Mercat (Th\'{e}or\`{e}me 8.2 in \cite{mercat2013construction}) we will show that Conjecture \ref{0conj:MMpol} is a consequence of Conjecture \ref{0conj:Zarpol} as soon as the Pell equation has non-trivial solutions (Theorem \ref{5theo:Mercat}); in particular, this will imply that Conjecture \ref{0conj:MMpol} holds over every infinite algebraic extension of a finite field (Corollary \ref{5cor:MMFpbar}). 
	\item When $\n K$ is an uncountable field it is easy to show that Conjecture \ref{0conj:multip}, and thus Conjecture \ref{0conj:MMpol}, hold (Theorem \ref{5theo:MMC}). 
	\item The theory of reduction of a continued fraction modulo a prime will allow us to prove Conjecture \ref{0conj:multip} over the algebraic closure of $\n Q$ (Theorem \ref{5theo:MMbarQ} and Corollary \ref{5cor:MMchQ}). Actually, in this case we will also give two other direct proofs of Conjecture \ref{0conj:MMpol} (Proposition \ref{5prop:MMchq1}).
	\item Continued fractions of the form $f\sqrt D$ are linked to generalized Jacobians of the hyperelliptic curve $U^2=D(T)$ and, thanks to a Theorem of Zannier (\cite{zannier2016hyperelliptic}, Theorem 1.7), this leads to a proof of Conjecture \ref{0conj:multip} and of Conjecture \ref{0conj:MMpol} over every number field (Theorem \ref{66theo:MMQ}).
\end{itemize}

Finally, we will present a well-known connection between real continued fractions and geodesics in the hyperbolic plane and its link with McMullen's results and conjectures on continued fraction with bounded partial quotients.

\section*{Notation}
Throughout the thesis, we will use the following standard notations:\\
- for a field $\n K$ we will denote by $\n K^*=\n K\setminus\{0\}$ the group of its invertible elements and by $\ov{\n K}$ its algebraic closure;\\
- we will denote by $\n F_q$ the finite field with $q$ elements, where $q$ is a power of a prime $p$;\\
- for $\alpha\inn\n R$, we will denote respectively by $\floor\alpha, \{\alpha\}$ the integer and the fractional parts of $\alpha$: $\alpha=\floor{\alpha}+\{\alpha\}$ with $\floor{\alpha}\inn\n Z$ and $\{\alpha\}\inn[0,1)$; we will use the same symbols for their polynomial analogues but the distinction will be clear from the context;\\
- if $a,b$ are integers or polynomials, we will write $a|b$ for ``$a$ divides $b$'' and we will denote by $\gcd(a,b)=(a,b)$ their greatest common divisor;\\
- we will denote by $A^t$ the transposed of a matrix $A$.
\par\medskip
\vspace{0,2 cm}
\noindent All the examples presented in the thesis have been computed thanks to PARI/GP.
\clearpage{\pagestyle{empty}\cleardoublepage}
\chapter{Continued fractions of formal Laurent series}
It is easy to see that a continued fraction formalism can be introduced over any normed field. In particular, as it was already shown in works of Abel \cite{Abel1826}, Chebyshev \cite{chebyshev1857integration} or Artin \cite{artin1924quadratischeI}, the classical theory of real continued fractions has a nearly perfect analogue for function fields over finite fields, that is, when the roles of $\n Z, \n Q$ and $\n R$ are played, respectively, by the ring of polynomials $\n F_q[T]$ over a finite field $\n F_q$, by the field of rational functions $\n F_q(T)$ and by the field of formal Laurent series $\n F_q((T^{-1}))$.\par\medskip

After recalling some general properties of continued fractions over a normed field, we will introduce the field of formal Laurent series $\n L=\n K((T^{-1}))$, where $\n K$ is a generic field, to focus then on regular continued fractions over $\n L$. In particular, we will show how most of the classical results for real continued fractions have an analogue in this case and highlight the more important differences. Finally, we will study how to perform some simple operations with continued fractions, focusing on M\"{o}bius transformations of a continued fraction (we recall that, as in the classical case, it is not easy in general to add or multiply continued fractions).\par\medskip

As in the real case, more precise results can be given about quadratic irrationalities; however, there is a complete analogy between the real and the polynomial settings, especially regarding periodicity, only if the base field $\n K$ is (an algebraic extension of) a finite field. These issues will be discussed in Chapter 2.\par\medskip

Unless otherwise specified, the results presented in this Chapter and in the following one are well known and can be found, together with more details, in most of the works on polynomial continued fractions, such as \cite{lauder1999continued}, \cite{poorten2000quasi}, \cite{schmidt2000continued} or \cite{stein1999introduction}.

\section{General continued fraction formalism}
Let $\n F$ be a field and let $|\cdot|:\n F\to\n R^+$ be an absolute value over $\n F$.

\begin{defn}
	Let $a_0,a_1,\cdots, a_n\inn\n F$. When it has a sense, that is, when no division by zero occurs, we define the \textit{finite simple continued fraction} $[a_0,a_1,\dots,a_n]$ as \be[a_0,a_1,\dots a_n]=a_0+\cfrac1{a_1+\cfrac1{\ddots+\cfrac1{a_n}}}\in\n F.\ee
	Let $(a_n)_{n\geq0}$ be an infinite sequence of elements of $\n F$. When it exists, we define the \textit{infinite simple continued fraction} $[a_0,a_1,\dots]$ as \be[a_0,a_1,\dots]=\lim_{n\to\infty}[a_0,a_1,\dots,a_n]\inn\n F,\ee where the limit is of course taken with respect to the absolute value $|\cdot|$.
\end{defn}

\begin{rem}\label{1rem:nest}
	It is easy to see that continued fractions can be nested one at the end of the other, that is, $$\left[a_0,\dots,a_n,[a_{n+1},a_{n+2},\dots]\right]=[a_0,\dots,a_n,a_{n+1},\dots].$$
\end{rem}

\begin{notat}\label{1notat:C_n}
	Let $a_0,a_1,\dots$ be a sequence, finite or infinite, of elements of $\n F$. We will set \be C_{-1}()=0,\ C_{0}()=1\ee and, for $n\geq0$, $$C_{n+1}(a_0,\dots,a_n)=a_nC_n(a_0,\dots,a_{n-1})+C_{n-1}(a_0,\dots,a_{n-2}).$$
	
	Equivalently, $C_{n+1}(a_0,\dots,a_n)$ is the sum of all the possible products of $a_0,\dots,a_n$ in which $k$ disjoint pairs of consecutive terms are deleted, for $k=0,\dots,\floor{(n+1)/2}$ (where the empty product is set to be 1). That is, $C_{n+1}(a_0,\dots,a_n)=\ds\sum_{k=0}^{\mathclap{\floor{(n+1)/2}}}L_k(a_0,\dots,a_n)$, where $$\ds L_k(a_0,\dots,a_n)=\sum_{\mathclap{\substack{0\leq j_1,\dots, j_k\leq n-1\\ j_m\leq j_{m+1}-2\text{ for every }m}}}a_0\cdots \widehat{a_{j_1}a_{j_1+1}}\cdots\widehat{a_{j_k}a_{j_k+1}}\cdots a_{j_n}.$$ 
	
	In particular, we always have \be\label{1eq:symmetryC_n}C_{n+1}(a_0,\dots,a_n)=C_{n+1}(a_n,\dots,a_0).\ee
\end{notat}

\begin{lem}\label{1lem:multAC}
	Let $(a_n)_n\sub\n F$ and let $a\inn\n F^*$. Then for every $n\geq0$ $$C_{n+1}(a\,a_0,a^{-1}a_1,\dots,a^{(-1)^n}a_n)=\begin{cases}a\,C_{n+1}(a_0,\dots,a_n) &\text{ if } n\text{ is even}\\C_{n+1}(a_0,\dots,a_n) &\text{ if } n\text{ is odd}\end{cases}.$$
\end{lem}
\begin{proof}
	By induction.
\end{proof}

\begin{defn}\label{1def:continuants}
	Let $\alpha=[a_0,a_1,\dots]\inn\n F$; its \textit{continuants} $(p_n,q_n)$ are defined by \be p_n=C_{n+1}(a_0,\dots,a_n),\ q_n=C_n(a_1,\dots,a_n) \text{ for } n\geq-2,\ee where we set $q_{-2}=1$.
	
	When they exist, their quotients $\ds\cv n$ are called the \textit{convergents} of the continued fraction $[a_0,a_1,\dots]$.
\end{defn}

From now on, we will always assume that the continued fractions we consider are well defined, as well as their convergents.

\begin{lem}\label{1lem:continuants}
	Let $(p_i,q_i),(u_j,v_j)$ be, respectively, the continuants of $[a_0,\dots,a_n]$ and of $[a_n,\dots,a_0]$. Then $u_n=p_n,\ u_{n-1}=q_n,$ $v_n=p_{n-1}$ and $v_{n-1}=q_{n-1}$.
\end{lem}
\begin{proof}
	It follows immediately from \eqref{1eq:symmetryC_n}.
\end{proof}

\begin{lem}\label{1lem:continuants2}
	Let $\alpha=[a_0,a_1,\dots]$ and let $(\cv n)_n$ be its convergents. Then we have \be\cv n=[a_0,\dots,a_n] \text{ for every } n\geq0.\ee
\end{lem}

\begin{proof}		
	Of course, $\cv0=\frac{a_0}1=[a_0]$. If by inductive hypothesis $[b_0,\dots,b_m]=\frac{C_{m+1}(b_0,\dots b_m)}{C_m(b_1,\dots,b_m)}$ for $m< n$ and for every $b_0,\dots,\!b_m$, then, as $[a_0,\dots,a_n]\!=\!\left[a_0,\dots,a_{n-1}\!+\!\frac1{a_n}\right]$, we will have $[a_0,\dots,a_n]=\frac{\left(a_{n-1}+\frac1{a_n}\right)p_{n-2}+p_{n-3}}{\left(a_{n-1}+\frac1{a_n}\right)q_{n-2}+q_{n-3}}= \cv{n}$.
\end{proof}

\begin{rem}	\label{1rem:continuants}
	In particular, by the previous Lemma, \be\label{1eq:inv}[a_n,\dots,a_0]=\frac{p_n}{p_{n-1}}\text{ and } [a_n,\dots,a_1]=\frac{q_n}{q_{n-1}}.\ee
\end{rem}
	
\begin{lem}\label{1lem:pnqnprime}
	In the previous notations, we have \be\label{1eq:pnqnprime}q_np_{n-1}-q_{n-1}p_n=(-1)^n \text{ for } n\geq-1.\ee 
	
	More generally, setting $d_{i,n}=C_{i-1}(a_{n-i+2},\dots,a_n)$ for $i\geq0,\ n\geq i$, we have \be\label{1eq:d_in}q_np_{n-i}-q_{n-i}p_n=(-1)^{n+i-1}d_{i,n}.\ee
\end{lem}

\begin{proof}
	As in the real case, \eqref{1eq:pnqnprime} can be easily proved by induction.
	
	As for \eqref{1eq:d_in}, the case $i=0$ is trivially true and for $i=1$ we find again \eqref{1eq:pnqnprime}. Now, $q_np_{n-i}-q_{n-i}p_n=	a_n(q_{n-1}p_{n-i}-q_{n-i}p_{n-1})+(q_{n-2}p_{n-i}-q_{n-i}p_{n-2})$. Assuming by inductive hypothesis that \eqref{1eq:d_in} holds for $i-1$ and $i-2$ (for every $n$), we will have $q_np_{n-i}-q_{n-i}p_n=a_n(-1)^{n+i-1}d_{i-1,n-1}+(-1)^{n+i-1}d_{i-2,n-2}=(-1)^{n+i-1}d_{i,n}$. 
\end{proof}

\begin{lem}\label{1lem:formulaspnqn}
	Let $(a_n)$ be a sequence of elements of $\n F$ such that, for every $n\geq0$, the continuants $p_n,q_n$ and $\alpha_n=[a_n,a_{n+1},\dots]$ are well defined; let $\alpha=\alpha_0$. 
	\begin{enumerate}
		\item $\alpha=[a_0,\dots,a_n,\alpha_{n+1}]$ for every $n$, so \be\label{1eq:alphapnqn}\alpha=\frac{\alpha_{n+1}p_n+p_{n-1}}{\alpha_{n+1}q_n+q_{n-1}}\text{; equivalently, }\alpha_{n+1}=-\frac{q_{n-1}\alpha-p_{n-1}}{q_n\alpha-p_n}.\ee 
		\item $\ds\prod_{j=1}^n\alpha_j=\frac{(-1)^n}{p_{n-1}-\alpha\, q_{n-1}}$ for $n\geq1$. 
		\item If $\alpha=[a_0,\dots,a_n]=\cv n$, then $\prod\limits_{j=1}^n\alpha_j=q_n$. More generally, \be\alpha_k\cdots\alpha_n=C_{n-k+1}(a_k,\dots,a_n).\ee
		\item We have \be\label{1eq:alpha-pn/qn}\alpha-\cv n=\frac{(-1)^n}{q_n(\alpha_{n+1}q_n-q_{n-1})},\ee
		so $$|p_n-\alpha q_n|=|q_n\alpha_{n+1}-q_{n-1}|^{-1}.$$
	\end{enumerate} 
\end{lem}
\begin{proof}
	$1.$ follows directly from Remark \ref{1rem:nest} and from Lemma \ref{1lem:continuants2}, while $2.$ and $4.$ are an immediate consequence of $1.$ and \eqref{1eq:pnqnprime}. Finally, $3.$ follows from $2$, \eqref{1eq:pnqnprime} and \eqref{1eq:d_in}. 
\end{proof}

\begin{rem}\label{1rem:cauchy}
	In the previous hypotheses, setting $\alpha_{i,n}=[a_{i+1},\dots,a_n]$ we will have, similarly to \eqref{1eq:alphapnqn}, $\cv n=\frac{\alpha_{i,n}p_i+p_{i-1}}{\alpha_{i,n}q_i+q_{i-1}}$ for every $i< n$. Then \be\label{1eq:cauchy}\cv n-\cv i=\frac{(-1)^i}{q_i(\alpha_{i,n}q_i+q_{i-1})}.\ee
	
	Comparing \eqref{1eq:cauchy} and \eqref{1eq:d_in} we get that $C_{n-i-1}(a_{i+2},\dots,a_n)=\frac{q_n}{\alpha_{i,n}q_i+q_{i-1}}$, so \begin{multline}\label{1eq:C_n}C_n(a_1,\dots,a_n)=C_i(a_1,\dots,a_i)C_{n-i}(a_{i+1},\dots,a_n)+\\+C_{i-1}(a_1,\dots,a_{i-1})C_{n-i-1}(a_{i+2},\dots,a_n).\end{multline}
\end{rem}

It can also be useful to introduce a matrix formalism for continued fractions, generalizing the one used in the classical real case, which goes back at least to Frame \cite{frame1949continued}. This subject is examined also in \cite{borwein2014neverending}.

\begin{notat}\label{1notat:matrix}
	For $a\inn\n F$, let $M_a$ be the unimodular matrix $M_a=\mm a110\inn M_2(\n F)$ and, for $a_0,\dots,a_n\inn\n F$, let $$M_{(a_0,\dots,a_n)}=M_{a_0}\cdots M_{a_n}.$$ 
	
	$M_2\left(\n F\right)$ acts on $\n F$ by M\"obius transformations: we will write $$\mm ABCD\alpha=\frac{A\alpha+B}{C\alpha+D}.$$
	In particular, if $\alpha=[a_0,a_1,\dots]$ then $M_a\alpha=a+\frac1\alpha=[a,a_0,a_1,\dots]$ so, iterating, $M_{(a_0,\dots,a_n)}\alpha_{n+1}=\alpha$.
\end{notat}	

\begin{lem}\label{1lem:matrixaction}
	If $(p_n,q_n)_n$ are the continuants of $\alpha$, then for every $n\geq0$, \be\label{1eq:pnqnmatrix} M_{(a_0,\dots,a_n)}=\mm{p_n}{p_{n-1}}{q_n}{q_{n-1}}.\ee
\end{lem}
\begin{proof}
	By induction.
\end{proof}
	
\begin{rem}
	Some of the previous results can be found again immediately in this context. 
	
	For instance, as $p_nq_{n-1}-p_{n-1}q_n=\det M_{(a_0,\dots,a_n)}=(-1)^{n+1}$, we get again \eqref{1eq:pnqnprime}.
	
	Moreover, denoting by $\frac{u_j}{v_j}$ the convergents of $[a_n,\dots,a_0]$ as in Lemma \ref{1lem:continuants}, then $\mm {u_n}{v_n}{u_{n-1}}{v_{n-1}}=M_{(a_n,\dots,a_0)}=M_{(a_n,\dots,a_0)}^t=\mm{p_n}{p_{n-1}}{q_n}{q_{n-1}}$.
	
	Certainly, $M_{(a_0,\dots,a_n)}=M_{(a_0,\dots,a_i)}M_{(a_{i+1},\dots,a_n)}$ for every $i\leq n$, which implies \eqref{1eq:C_n}.  
\end{rem}

\begin{rem}\label{1rem:matrixaction2}
	We can consider the map from $M_2\left(\n F\right)$ to $\n F\cup\{\infty\}$ defined by $\ds\pphi:(M)=\frac\alpha\gamma$ for $\ds M=\mm \alpha\beta\gamma\delta$, where $\pphi(M)=\infty$ if and only if $\gamma=0$. In particular, we will have $\pphi(M_{(a_0,\dots,a_n)})=[a_0,\dots,a_n]$. This map is compatible with M\"obius transformations, that is, for every $m,M\inn M_2(\n F)$ we have $\pphi(Mm)= M\pphi(m)$. 
\end{rem}

\section{Continued fractions of formal Laurent series}
The theory of continued fractions of Laurent series, and especially of quadratic irrationalities, was firstly developed by Abel and Chebyshev as a tool to express hyperelliptic integrals in finite terms.\par\medskip

Abel \cite{Abel1826} was the first who considered systematically continued fractions in the hyperelliptic case, proving the following result:\medskip 

\textit{Let $D\inn\n Q[T]$ be a monic polynomial of even positive degree $2d$ which is not a perfect square. If there exists a non-trivial solution $(p,q)$ of the polynomial analogue of the Pell equation for $D$, that is, if there exist polynomials $p,q\inn\n Q[T]$ such that $p^2-Dq^2$ is a non-zero constant (with $q\neq0$), then, setting $f= p'/q$, we will have that $f$ is a polynomial of degree $d-1$ and that $\ds\int\frac{f(T)dT}{\sqrt{D(T)}}= \log\left(p(T) + q(T)\sqrt{D(T)}\right)$. Conversely, given such an indefinite integral, it follows that $f$ is a polynomial of degree $d-1$ and that $(p,q)$ is a solution of the Pell equation for $D$.}\medskip 

Indeed, if $p(T)^2-D(T)q(T)^2$ is a constant, then $p,q$ are relatively prime polynomials, so $2p'p-2q'qD-q^2D'=0$ implies that $q|p'$ and $(2q'D+qD')/2p=p'/q$. Thus $(p(T)+q(T)\sqrt{D(T)})'/(p(T)+q(T)\sqrt{D(T)})=f(T)/\sqrt{D(T)}$ with $f=p'/q$. Conversely, given the integral, adding it to its conjugate (the integral obtained replacing $\sqrt{D(T)}$ with $-\sqrt{D(T)}$ ) we get that $\log(p^2-Dq^2)$ must be a constant, that is, $p^2-Dq^2$ must be a non-zero constant.\par\medskip

As it was already well known, the solutions to the classical Pell equation for a positive integer $d$ can be found through the continued fraction expansion of $\sqrt d$. Thus, Abel was naturally led to define an analogous continued fraction expansion for a square root of $D$ when $D$ is a polynomial as above. He proved that the existence of a non-trivial solution to the polynomial Pell equation for $D$ is equivalent to the periodicity of the continued fraction expansion of $\sqrt D$, and that in this case the whole classical theory carries over to the polynomial setting. Moreover, he found a connection between the periodicity of the continued fraction  of $\sqrt D$ and the fact that the class of the divisor $(\infty_-)-(\infty_+)$ is a torsion point on the Jacobian of the hyperelliptic curve $\C H: U^2=D(T)$, where $\infty_-,\infty_+$ are the points at infinity on $\C H$. Chebyshev continued these studies, publishing a series of papers on this subject between 1853 and 1867; in particular, he considered extensively the case when $D$ has degree 4, that is, when the associated curve is an elliptic curve \cite{chebyshev1857integration}.\par\medskip 

A century later, Artin resumed this theory, in order to study the arithmetic of quadratic extensions of $\n F_p(T)$ in complete analogy to the classical theory of quadratic number fields \cite{artin1924quadratischeI}, \cite{artin1924quadratischeII}. In particular, this allowed him to find explicit formulas for class numbers of quadratic function fields.

\subsection{Formal Laurent series}
In the theory of continued fractions over function fields, the role of the ring of integers $\n Z$ in the classical case will be played by the polynomial ring $\n K[T]$, where $\n K$ is a field, and, consequently, its field of fractions $\n K(T)$, the field of rational functions over $\n K$, will correspond to $\n Q$.

$\n K(T)$ is a normed field with the natural valuation given by $$\ord\left(A/B\right)=-\deg A+\deg B \text{ if } A/B\neq0,\ A,B\inn\n K[T] \text{ and } \ord(0)=\infty.$$ We will denote by $|\cdot|$ the associated absolute value $$\left|A/B\right|=\mu^{\deg A-\deg B},\ |0|=0$$ with $\mu\inn\n R,\mu>1$ fixed (usually, $\mu=q$ if $\n K=\n F_q$ is a finite field and $\mu=e$ otherwise). Contrary to the usual absolute value on $\n Q$, the norm $|\cdot|$ is non-Archimedean, that is, $$|f+g|\leq\max\left\{|f|,|g|\right\}$$ for every $f,g\inn\n K(T)$ and, in particular, $|f+g|=\max\{|f|,|g|\}$ as soon as $|f|\neq|g|$. This, together with the fact that a set of polynomials of bounded norm is finite if and only if $\n K$ is a finite field, will lead to the principal differences between the classical theory of real continued fractions and the theory of continued fractions in function fields.\par\medskip

In place of $\n R$, we will then consider the completion of $\n K(T)$ with respect to the valuation $\ord$, that is, the field of formal Laurent series.

\begin{defn} 
	We will denote by $\n L_{\n K}$ or, when the base field $\n K$ is clear, simply by $\n L$, the set of \textit{formal Laurent series} over $\n K$, that is, $$\n L_{\n K}=\n K\left(\left(T^{-1}\right)\right)=\left\{\sum_{i=-\infty}^N c_iT^i,\ N\inn\n Z,\ c_i\inn\n K\ \forall i\right\}.$$
\end{defn}

Obviously $\n L$ is a ring with the natural operations of sum and product of formal series $$\sum_{\mathclap{i=-\infty}}^Nc_iT^i+\sum_{\mathclap{i=-\infty}}^Md_iT^i=\sum_{\mathclap{i=-\infty}}^{\mathclap{\max\{M,N\}}}(c_i+d_i)T^i,$$ $$\left(\sum_{i=-\infty}^Nc_iT^i\right)\left(\sum_{j=-\infty}^Md_jT^j\right)=\sum_{\mathclap{k=-\infty}}^{\mathclap{M+N}}\left(\sum_{i}c_id_{k-i}\right)T^k,$$ where we set $c_i=0$ for $i>N$, $d_j=0$ for $j>M$.\par\medskip

Actually, $\n L$ is a field: let $\alpha=\sum\limits_{i\leq N} c_iT^i\inn\n L$, $\alpha\neq0$, with $c_N\neq0$, then $\alpha$ is invertible in $\n L$ and its inverse is $\beta=\sum\limits_{\mathclap{j\leq -N}} d_jT^j$, where \be\label{1eq:Laurentinv} d_{-N}=c_N^{-1}\text{ and } d_{j-N}=-c_N^{-1}\sum_{\mathclap{i=N+j}}^{N-1} c_id_{j-i}\text{ for } j<0.\ee

As we can think of $\n K[T]$ as a subring of $\n L$ and any non-zero polynomial is invertible in $\n L$, we can identify $\n K(T)$ with a subfield of $\n L$. 

\begin{lem}\label{1lem:Laurseriesrat}
	Let $\alpha=\sum\limits_{i\leq N}c_iT^i\inn\n L$. Then $\alpha$ represents a rational function if and only if its coefficients $c_i$ eventually satisfy a linear recurrence relation, that is, if and only if there exist $m_0\leq N,\ M\geq1$ and $d_0,\dots, d_M\inn\n K$ such that $d_0c_j+\cdots+d_Mc_{j-M}=0$ for every $j\leq m_0$.
\end{lem} 
\begin{proof}	
	If $\sum\limits_{i=0}^M c_{j-i}d_i=0$ for $j\leq m_0$, setting $g=d_0+d_1T+\cdots+d_MT^M$, we have $g\alpha=\sum\limits_{m_0<k\leq N+M}\left(\sum\limits_{0\leq i\leq M}c_{k-i}d_i\right)T^k\inn\n K(T)$, so $\alpha$ is a rational function. 
	
	Conversely, if $f=a_0+\cdots+a_NT^N,\ g=b_0+\cdots+b_MT^M\inn\n K[T]$, then, by \eqref{1eq:Laurentinv}, the coefficients of $g^{-1}$ eventually satisfy a linear recurrence relation, so also the coefficients of $f/g$ eventually satisfy the same linear recurrence relation. More precisely, if $f/g=\sum\limits_{\mathclap{i\leq N-M}} c_iT^i$, then for $k$ small enough $\sum\limits_{i=0}^Mb_ic_{k+M-i}=0.$
\end{proof}

\begin{rem}
	If $\n K$ is a finite field it can be proved similarly that $\alpha$ is a rational function if and only if the sequence $(c_i)_i$ is eventually periodic. 
\end{rem}

\begin{ex} 
	Let $\alpha=\frac{T^3+T}{T^2+2T+1}\inn\n L_{\n Q}$. Then  for $k$ small enough the coefficients $d_k$ of $\alpha$ will satisfy $d_{k+2}+2d_{k+1}+d_k=0$. Indeed, the formal Laurent series that represents $\alpha$ is $T-2+4T^{-1}-6T^{-2}+8T^{-3}-10T^{-4}+12T^{-5}+\cdots$.
\end{ex}

\begin{notat} 
	The valuation $\ord$ and its associated norm $|\cdot|$ can be extended in a natural way to $\n L$: if $\alpha=\sum\limits_{\mathclap{i=-\infty}}^Nc_iT^i$ with $c_N\neq0$, we set $$\ord(\alpha)=-N,\ |\alpha|=\mu^N.$$
\end{notat}

\begin{lem} 
	$\n L$ is the completion of $\n K(T)$ with respect to the norm $|\cdot|$, that is, it is the smallest extension $\n K'$ of $\n K(T)$ such that any Cauchy sequence of elements of $\n K'$ has a limit in $\n K'$.
\end{lem}

\begin{proof}	
	Let $\alpha=\sum\limits_{i\leq N}c_iT^i\inn\n L$; for every $n\geq0$ let $\alpha_n=\sum\limits_{\mathclap{i=-n}}^N c_iT^i$. Then we have $\alpha_n\inn\n K(T)$ for every $n$ and $(\alpha_n)_{n\geq0}$ is a Cauchy sequence converging to $\alpha$. Thus, $\n L$ is contained in the completion of $\n K(T)$ with respect to $|\cdot|$.
	
	It is then enough to show that $\n L$ is complete. Let $(\alpha_n)_{n\geq1}\sub\n L$ be a Cauchy sequence, with $\alpha_n=\sum_i c_{i,n}T^i$ for every $n$. Then for every $\epsilon>0$ there exists $M\inn\n N$ such that $|\alpha_m-\alpha_n|<\epsilon$ for every $m,n>M$; in particular, for every $i\inn\n Z$ we have that there exists $M_i$ such that $|\alpha_m-\alpha_n|<\mu^i$ for every $m,n>M_i$. Equivalently, for $n>M_i$ the coefficients of the formal series $\alpha_n$ coincide at least up to degree $i$, so we can define $c_i$ as the common value of the $c_{i,n}$ for $n>M_i$. Let $\alpha=\sum_ic_iT^i$. Then we have $\alpha\inn\n L$ (the orders of the $\alpha_n$ have to be bounded, that is, $c_i=0$ for $i$ big enough) and $\alpha$ is, by construction, the limit of the sequence $(\alpha_n)$, that is, $|\alpha-\alpha_n|\to0$ as $n\to\infty$. 
	
	Then $\n L$ is the completion of $\n K(T)$.
\end{proof}

\begin{notat}
Let $\C O$ be the valuation ring associated to $\ord$, $$\C O=\!\{\alpha\inn\n L,\ \ord(\alpha)\geq0\}=\bigg\{\sum_{i\leq0}c_iT^i\inn\n L\bigg\}.$$ In particular, $\C O$ is a local ring with maximal ideal $\C M=(T^{-1})=\{\alpha\inn\n L,\ \ord(\alpha)>0\}$. Then $\C O/\C M$ can be identified with $\n K$; if $\alpha=\sum_{i\leq0}c_iT^i\inn\C O$, we will denote by $\ov \alpha=c_0$ its reduction in $\C O/\C M$.\par\medskip 

As $\n K$ is complete, Hensel's Lemma holds:
\end{notat}

\begin{prop}[Hensel's Lemma]
	Let $f\inn\C O[X]$, with $X$ transcendent over $\C O$, and let $\ov f\inn\n K[X]$ be its reduction modulo $\C M$. If there exists a root $c$ of $\ov f$ in $\n K$ with $\ov f\,'(c)\neq0$, then there exists a unique $\alpha\inn\C O$ which is a root of $f$ and such that $\ov\alpha=c$.
\end{prop}

\begin{lem}\label{1lem:rad} 
	Let $\n K$ be a field of characteristic different from 2, let $\alpha\inn\n L$ be a non-zero Laurent series of even order $\ord(\alpha)=-2N$ and whose leading coefficient is a square in $\n K$. Then, up to the choice of the sign, the square root of $\alpha$ is well defined in $\n L$.
\end{lem}	

\begin{proof}
	Let us consider $f(X)=X^2-T^{-2N}\alpha\inn\C O[X]$. In the previous notations, $\ov f$ has a root $c$ in $\n K$ and, as $\ov{T^{-2N}\alpha}\neq0$, $c\neq0$, so $\ov f'(c)=2c\neq0$. Then, by Hensel's Lemma, there exists $\beta\inn\C O$ such that $\beta^2=T^{-2N}\alpha$. Moreover, $\beta$ is unique up to the choice of a sign (that is, up to the choice of a square root in $\n K$ of the leading coefficient of $\alpha$); we will write $T^N\beta=\sqrt\alpha$.
\end{proof}

\begin{rem}
	More precisely, in the previous notations, if $\alpha=\sum\limits_{i\leq2N}c_iT^i$, with $c_{2N}$ a non-zero square, $c_{2N}=c^2$, then  $\alpha=c^2T^{2N}(1+\widetilde\alpha)$ with $\widetilde\alpha=\sum_{i<2N} c_i/c^2\,T^{i-2N}$. It then follows that $$\beta=c\,T^N\sum\limits_{k=0}^\infty\binom{1/2}k\widetilde\alpha^k.$$ The coefficients of $\beta$ can be also found through a recursive formula: $\beta=\sum_{j\leq N}b_jT^j$ with \be\label{1eq:Laurentrad}b_N=c \text{ and, for } j<N,\ b_j=\frac12c^{-1}\left(c_{j+N}-\sum_{\mathclap{k=j+1}}^{N-1}b_kb_{j+N-k}\right).\ee
\end{rem}

\begin{rem}
	On the other hand, if $\n K$ has characteristic 2, then the squaring map is the Frobenius endomorphism, so a formal Laurent series has a square root if and only if it is of the form $\sum_i c_i^2T^{2i}$, and in this case its (unique) square root is $\sum_ic_iT^i$.
\end{rem}

\begin{ex}\label{1ex:ex1}
	Let $\n K=\n Q$ and let $D=T^4+T^2+1$. Then it is easy to see that $$\sqrt D=T^2+\frac12+\frac 38 T^{-2}-\frac3{16} T^{-4}+\frac3{128} T^{-6}+\frac{15}{256}T^{-8}+\cdots\inn\n L_{\n Q}.$$
\end{ex}

\subsection{Regular continued fraction expansions of formal Laurent series}
In the notations of section 1.1, let us consider the theory of continued fractions over $\n F=\n L=\n L_{\n K}$ (for some field $\n K$), with the previously defined absolute value $|\cdot|$.

\begin{lem}\label{1lem:constants}
Let $(a_n)_{n\geq0}$ be a sequence of formal Laurent series such that $\ord a_n<0$, except possibly for finitely many indices $n$. Then the continued fraction $[a_0,a_1,\dots]$ is well defined in $\n L$.
\end{lem}
\begin{proof} 
As in Remark \ref{1rem:cauchy}, for every $i<n$ let $\alpha_{i,n}=[a_{i+1},\dots,a_n]$ and let $\cv n$ be the convergents of $[a_0,a_1,\dots]$ (for $i,n$ large enough, the $\alpha_{i,n}$ and the $\cv n$ are well defined). We have then seen that $|\cv n-\cv i|=|q_i|^{-1}|\alpha_{i,n}q_i+q_{i-1}|^{-1}$. For $i,n$ large enough $|\alpha_{i,n}|=|a_{i+1}|$ and $|\alpha_{i,n}q_i+q_{i-1}|=|q_{i+1}|$, so  $$\left|\cv n-\cv i\right|=|q_iq_{i+1}|^{-1}\xrightarrow[n,i\to\infty]{}0.$$ Thus $(\cv n)_n$ is a Cauchy sequence and, as $\n L$ is complete, it has a limit in $\n L$, that is, the infinite continued fraction $[a_0,a_1,\dots]$ is well defined in $\n L$. 
\end{proof}

\begin{rem}\label{1rem:constants}
Actually, the previous statement holds also if we allow the $a_n$ to have order 0 on a subsequence of non-consecutive indices $n$ or even if there are consecutive pairs $a_n,a_{n+1}$ of order 0 such that $a_na_{n+1}+1\neq0$.\par\medskip

In particular, an infinite continued fraction $[a_0,a_1,\dots]$ such that $a_n\inn\n K[T]$ for every $n$ and $\deg a_n\geq1$ for $n\geq1$ converges in $\n L$; in this case, $[a_0,a_1,\dots]$ is said to be a \textit{regular continued fraction}. 

Actually, any formal Laurent series $\alpha$ has a unique regular continued fraction expansion, which will be called \textit{the} continued fraction expansion of $\alpha$. As well as the classical continued fraction algorithm is based on the integer part, the polynomial continued fraction algorithm is based on the polynomial part.
\end{rem}

\begin{notat}
	Let $\alpha=\sum\limits_{i\leq N}c_iT^i\inn\n L$. The \textit{polynomial part} of $\alpha$ is $$\floor\alpha=\sum_{i=0}^N c_iT^i\inn\n K[T],$$ that is, $\floor\alpha$ is the unique polynomial $A$ such that $\ord(\alpha-A)>0$.

	Similarly to the real case, we will also write $$\{\alpha\}=\alpha-\floor\alpha$$ for the polynomial analogue of the fractional part of $\alpha$.
\end{notat}
If $\alpha,\beta\inn\n L$ and $k\inn\n K$, then $\floor{\alpha+\beta}=\floor\alpha+\floor\beta$ and $\floor{k\alpha}=k\floor\alpha$ (while $\floor{\alpha\beta}$ is not necessarily equal to $\floor\alpha\floor{\beta}$).

\begin{rem}
	In complete analogy with the real case we can then consider the following continued fraction algorithm.
	
	For $\alpha\inn\n L$, let $\alpha_0=\alpha,\ a_0=\floor\alpha$ and, for $n\geq1$, let $$\alpha_n=\frac1{\alpha_{n-1}-a_{n-1}},\ a_n=\floor{\alpha_n}$$ (where this procedure ends if and only if there exists $m$ such that $\alpha_m=a_m$, if and only if there exists $m$ such that $\alpha_m$ is a polynomial).

	Certainly $a_n\inn\n K[T]$ for every $n$ and $\deg a_n\geq1$ for $n\geq1$, so the continued fraction $[a_0,a_1,\dots]$ is regular and, in particular, it is well defined. Moreover, for every $n$, $\alpha=[a_0,\dots,a_{n-1},\alpha_n]$ and $\alpha_n=[a_n,a_{n+1},\dots]$, so \be\alpha=[a_0,a_1,\dots].\ee This is called \textit{the (regular) continued fraction expansion} of $\alpha$;  the $\alpha_n$ are called the \textit{complete quotients} of $\alpha$ and the $a_n$ are said to be its \textit{partial quotients}.\par\medskip
	
	Thus, the regular continued fraction of a formal Laurent series is obtained by iterating the polynomial analogue of the Gauss map, $$\begin{array}{llll}T:&\C M&\to&\C M\\ &\alpha&\mapsto&\begin{cases}\frac1\alpha-\floor{\frac1\alpha}&\text{if }\alpha\neq0\\0&\text{if }\alpha=0\end{cases}\end{array},$$ where, as before, $\C M$ is the set of formal Laurent series with strictly positive order. If $\alpha\inn\C M$, then $\alpha=\frac1{\alpha_1}$ and for every $n\geq1$ if $\alpha_n\neq0$, then $\frac1{\alpha_n}=T^{n-1}(\alpha)$.
\end{rem}

From now on, unless otherwise stated, we will always consider regular continued fraction expansions, and we will call convergents (respectively, continuants) of $\alpha\inn\n L$ the convergents (respectively, continuants) of its regular continued fraction expansion. 

\begin{ex}
	As in Example \ref{1ex:ex1}, let $D=T^4+T^2+1\inn\n Q[T]$, let $\alpha=\sqrt D$. Then we have\\
	$\ds\alpha_0=\alpha=\sqrt D,\ a_0=T^2+\frac12\\
	\alpha_1=\frac 1{\sqrt D-T^2-1/2}=\frac43(\sqrt D+T^2+1/2),\ a_1=\frac83T^2+\frac43\\
	\alpha_2=\frac3{4\sqrt D-4T^2-2}=\sqrt D+T^2+\frac12,\ a_2=2T^2+1\\
	\alpha_3=\frac 1{\sqrt D-T^2-1/2},\ a_3=\frac83T^2+\frac43\\
	\cdots$
 
	Then $\alpha=\left[T^2+\frac12,\frac83T^2+\frac43,2T^2+1,\frac83T^2+\frac43,\cdots\right]$; in particular, the sequence of the partial quotients of $\alpha$ is periodic. We will see in the next Chapter that this is not a coincidence and that it is related to the fact that $\alpha$ is quadratic over $\n Q[T]$.
\end{ex}

\begin{rem}
	If the continued fraction expansion of $\alpha$ is finite, obviously $\alpha$ is a rational function. The converse holds too; more precisely, if $\alpha=A/B\inn\n K(T)$ then the partial quotients of $\alpha$ are the successive quotients appearing in the Euclidean algorithm applied to $A,B$. In particular, the continued fraction expansion of $\alpha$ is finite.
\end{rem}

\begin{ex}
	For $\n K=\n Q$, let $\alpha=\ds\frac{T^5+1}{T^4+T^2+1}\inn\n Q(T)$. Then the regular continued fraction expansion of $\alpha$ is $\alpha=\left[T,-T,-T^2+T-2,1/3\,T+1/3\right]$.
\end{ex}

\begin{lem}
	Let $\alpha\inn\n L$, let $[a_0,a_1,a_2,\dots]$ be its regular continued fraction expansion and let $(p_n,q_n)_{n\geq0}$ be its continuants.
	\begin{enumerate}
	\item $p_n,q_n$ are relatively prime polynomials for every $n\geq1$.
	\item $p_nq_m-p_mq_n\inn\n K^*$ if and only if $m=n\pm1$.
	\item $|p_n|=|a_0a_1+1||a_2|\cdots|a_n| \text{ and } |q_n|=|a_1|\cdots|a_n|\geq\mu|q_{n-1}|>|q_{n-1}|$ for every $n$.
	\item $\deg a_{n+1}=\ord(p_n-\alpha q_n)-\deg q_n$; equivalently, \be\label{1eq:fondprop}|p_n-\alpha q_n|=|a_{n+1}q_n|^{-1}=|q_{n+1}|^{-1}.\ee
	\end{enumerate}
\end{lem}
\begin{proof}
	As the $p_n,q_n$ are polynomials, and, by \eqref{1eq:pnqnprime}, $q_np_{n-1}-p_nq_{n-1}=(-1)^n$, then $p_n,q_n$ are relatively prime for every $n$. 
	
	2. follows from \eqref{1eq:d_in}: $p_nq_m-p_mq_n\inn\n K^*$ if and only if $p_nq_m-p_mq_n=\pm1$, if and only if $m=n\pm1$.
	
	3. is a consequence of the fact that $|\alpha_n|=|a_n|>1$ for $n>0$.
	
	By \eqref{1eq:alpha-pn/qn}, $\ds\left|\alpha-\cv n\right|=|q_n(\alpha_{n+1}q_n-q_{n-1})|^{-1}=|q_n^2a_{n+1}|^{-1}=|q_nq_{n+1}|^{-1}$, that is, $\ds|p_n-\alpha q_n|=|a_{n+1}q_n|^{-1}=|q_{n+1}|^{-1}$.
\end{proof}

\begin{lem}\label{1lem:confrontocf}
	Let $\alpha=[a_0,a_1,\dots],\beta=[b_0,b_1,\dots]\inn\n L$, with $\alpha\neq\beta$, and let $i$ be the integer such that $a_n=b_n$ for $n=0,\dots,i-1$ and $a_i\neq b_i$. If $i=0$, then $|\alpha-\beta|=|a_0-b_0|$. Otherwise, let $d=\deg a_1+\cdots+\deg a_{i-1}$. Then we have $$|\alpha-\beta|=\frac{|a_i-b_i|}{\mu^{2d}|a_ib_i|}<\mu^{-2d}.$$
\end{lem}

\begin{proof}
	Let us assume $i\!>\!0$; let $(\cv n)_n$ be the convergents of $\alpha$. Then obviously $\cv0,\dots,\cv {i-1}$ are also convergents of $\beta$, so, by \eqref{1eq:alphapnqn}, $\ds\alpha=\frac{\alpha_ip_{i-1}+p_{i-2}}{\alpha_iq_{i-1}+q_{i-2}}$, $\ds \beta=\frac{\beta_ip_{i-1}+p_{i-2}}{\beta_iq_{i-1}+q_{i-2}}$ and $\alpha-\beta=\frac{(-1)^i(\alpha_i-\beta_i)}{(\alpha_iq_{i-1}+q_{i-2})(\beta_iq_{i-1}+q_{i-2})}$. Then,  $|\alpha-\beta|=|a_i-b_i||a_ib_i|^{-1}\mu^{-2d}$.
\end{proof}

Formally, the same result holds also if $\beta=[a_0,\dots,a_{i-1}]$, setting $|b_i|=\infty$; in this case we find again \eqref{1eq:fondprop}: $\Big|\alpha-\cv {i-1}\Big|=|q_{i-1}|^{-2}|a_i|^{-1}$.

\begin{rem}
	In particular we have that, differently from the real case where all rational numbers have two possible regular continued fraction expansions, the regular continued fraction expansion of a Laurent series is always unique. 
\end{rem}

\begin{rem}\label{1rem:confrontocf2}
	A converse holds too: let $\alpha=[a_0,a_1,\dots],\beta=[b_0,b_1,\dots]\inn\n L$. Then $a_0=b_0,\dots, a_i=b_i$ if and only if $|\alpha-\beta|<\mu^{-2d}$, where $d=\deg a_1+\cdots+\deg a_i$.
	Indeed, let us assume by contradiction that $|\alpha-\beta|<\mu^{-2d}$, $a_0=b_0,\dots,a_{k-1}=b_{k-1}$ but $a_k\neq b_k$, with $k<i$. Then, by the previous Lemma, $\mu^{-2d}>|\alpha-\beta|=\frac{|a_k-b_k|}{|q_{k-1}|^2|a_kb_k|}$, that is, $2(\deg a_k+\cdots+\deg a_i)<\deg a_k+\deg b_k-\deg(a_k-b_k)\leq2\deg a_k$, contradiction.
\end{rem}

In particular, in the case where $\beta=p/q$ is a rational function we immediately get the following first form of Best Approximation Theorem:
\begin{lem}\label{1lem:best2}
	Let $\alpha\inn\n L$, let $p,q$ be two relatively prime polynomials. Then the following are equivalent:
	\begin{enumerate}
		\item $\cv{}$ is a convergent of $\alpha$;
		\item $\ds\left|\alpha -\frac pq\right|<|q|^{-2}$, (equivalently, $\ord(p-\alpha q)>\deg q$).
	\end{enumerate}
\end{lem}
\begin{proof}
	Let $\alpha=[a_0,a_1,\dots]\inn\n L$. Then $\cv{}$ is a convergent of $\alpha$ if and only if $\cv{} =[a_0,\dots,a_i]$ for some $i$, if and only if, by the previous Remark, $\ds\left|\alpha-\cv{} \right|<\mu^{-2(\deg a_1+\dots+\deg a_i)}$, if and only if $\ds\left|\alpha-\cv{} \right|<|q|^{-2}$.
\end{proof}

\begin{rem}
	We can see here a first significant difference between the integer and the polynomial settings. Indeed, the corresponding result in the classical case (see for example \cite{olds1963continued}, Theorem 3.8 and Theorem 5.1) is:
	
	\textit{Let $\alpha\inn\n R$. If $p_n/q_n$ is a convergent of $\alpha$, then $|\alpha-p_n/q_n|<q_n^{-2}$. Conversely, if $p,q\inn\n Z$ and $|\alpha-p/q|<\frac12 q^{-2}$, then $p/q$ is a convergent of $\alpha$. Moreover, of any two consecutive convergents of $\alpha$, at least one satisfies the previous inequality.}
	
	The simpler form of the Theorem in the polynomial case is due to the fact that the absolute value is now non-Archimedean.
\end{rem}

\begin{rem}\label{1rem:bestappr}
	If $p,q\inn\n K[T]$ are polynomials such that $|p-\alpha q|<|q|^{-1}$, then $p/q$ is a convergent of $\alpha$ even if $p,q$ are not relatively prime. In this case, by \eqref{1eq:fondprop} the following partial quotient has degree $\ord(p-\alpha q)-\deg q+2\deg (p,q)$.
	
	The converse is not always true: let $(p_n,q_n)$ be a continuant of $\alpha$ and let $d\inn\n K[T]$. Then we have $|dp_n-\alpha dq_n|=|d||a_{n+1}q_n|^{-1}$, so $|dp_n-\alpha dq_n|<|dq_n|^{-1}$ if and only if $2\deg d<\deg a_{n+1}$.
\end{rem}

As in the real case, the convergents to a Laurent series provide its best possible approximations through rational functions.

\begin{defn} 
	Let $\alpha\inn\n L$, let $p,q$ be two relatively prime polynomials with $q\neq0$. We will say that $p/q$ is a \textit{best (rational) approximation} to $\alpha$ if for every $p',q'\inn\n K[T]$ with $\deg q'\leq\deg q$ and $\frac pq\neq\frac{p'}{q'}$ we have \be|p-\alpha q|<\left|p'-\alpha q'\right|.\ee
\end{defn}

\begin{lem}\label{1lem:badlyappr}
	Let $p,q$ be two relatively prime polynomials and let us assume that $\deg q\leq\deg q_n$ and that $\cv{}\neq\cv n$ (where, as before, $p_n,q_n$ are the continuants of $\alpha\inn\n L$). Then $$|p-\alpha q|\geq|p_{n-1}-\alpha q_{n-1}|=|q_n|^{-1}.$$
	Moreover, the equality $|p-\alpha q|=|q_n|^{-1}$ holds if and only if $\frac pq=[a_0,\dots,a_{n-1},a_n+k]$ with $k\inn\n K^*$.\medskip

\end{lem}	

\begin{proof}
	By \eqref{1eq:fondprop}, the previous inequality is certainly true if $\cv{}=\cv k$ is a convergent of $\alpha$ with $k<n$.
	
	Let us then assume that $\cv{}$ is not a convergent of $\alpha$; let $\alpha=[a_0,a_1,\dots]$ and let $\cv{}=[a_0,\dots,a_{i-1},b_i,\dots]$ with $b_i\neq a_i$. Of course, $i\leq n$. By Lemma \ref{1lem:confrontocf} then $|p-\alpha q|=|q|\frac{|a_i-b_i|}{|q_{i-1}|^2|a_ib_i|}\geq\frac{|a_i-b_i|}{|q_{i-1}a_i|}\geq|q_i|^{-1}\geq|q_n|^{-1}$.
\end{proof}

\begin{rem}
	In particular, as $|p_n-\alpha q_n|<|p_{n-1}-\alpha q_{n-1}|$ for every $n$, all the convergents of $\alpha$ are best approximations to $\alpha$.
\end{rem}

In the real setting, by the classical Best Approximation Theorem, the convergents of a real number $\alpha$ are exactly its best approximations, namely (see \cite{khinchin1997continued}, Theorems 16, 17):\par\medskip
		
\textit{Let $\alpha\inn\n R$, let $\cv n$ be a convergent of $\alpha$ with $n\geq2$. If $p,q$ are integers such that $0<q\leq q_n$ and such that $\frac pq\neq\cv n$, then $|q_n\alpha-p_n|<|q\alpha-p|$. Moreover, a reduced fraction $\frac {p'}{q'}$ with $q'\geq q_2$ that satisfies the previous property is a convergent.}\par\medskip

An analogue result holds also in the functions field case:  

\begin{theo}[polynomial Best Approximation Theorem]
	Let $\alpha\inn\n L$ and let $(\cv n)_n$ be its convergents.
	Let $p,q\inn\n K[T]$ with $q\neq0$ be two relatively prime polynomials. Then $$\frac pq \text{ is a best approximation to } \alpha \text{ if and only if it is a convergent of } \alpha.$$
\end{theo}

\begin{proof}
	It remains only to prove that every best approximation to $\alpha$ is also a convergent.
	 
	Let $\cv{}$ be a best approximation to $\alpha$ and let us assume that $\cv{}$ is not a convergent of $\alpha$. Let $n$ be the unique integer such that $\deg q_{n-1}<\deg q\leq\deg q_n$. Then $|p-\alpha q|<|p_{n-1}-\alpha q_{n-1}|$ while, by the previous Remark, $|p-\alpha q|\geq|p_{n-1}-\alpha q_{n-1}|$, contradiction.
\end{proof}

\section{Some operations with continued fractions}\label{1sec:sec3}
While there is no general algorithm to compute the sum or the product of continued fractions, it is possible to study some easier operations. In particular, we will focus on the multiplication (or division) of a continued fraction by a linear polynomial, which we will later apply to the study of the polynomial analogues of Zaremba's and McMullen's Conjectures, and on the polynomial analogue of Serret's Theorem on M\"{o}ebius transformations of continued fractions.

As we will be interested in the study of continued fractions with partial quotients of bounded degree, we will always highlight how the operation under consideration modifies the degrees of the partial quotients. In this perspective, we introduce the following notation:

\begin{notat}\label{1notat:K,ovK}
	For $\alpha=[a_0,a_1,\dots]\inn\n L$, we set \be K(\alpha)=\sup_{n\geq1}\deg a_n\inn\n N\cup\{\infty\}.\ee If $K(\alpha)<\infty$, we will say that $\alpha$ has \textit{bounded partial quotients}, or that $\alpha$ is \textit{badly approximable}.
	
	In fact, by Lemma \ref{1lem:badlyappr}, $|q||p-\alpha q|\!\!\geq\!\!|q_n||p_n-\alpha q_n|\!\!=\!\!|a_{n+1}|^{-1}$ for any pair of relatively prime polynomials $p,q$ with $\deg q_n\leq\deg q<\deg q_{n+1}$. Thus $\alpha$ has bounded partial quotients if and only if $$\liminf_{|q|\to\infty}|q||p-\alpha q|>0.$$
	
	If $\alpha\inn\n L\setminus\n K(T)$, that is, if the continued fraction of $\alpha$ is infinite, we will also set $$\ov K(\alpha)=\limsup_n\,\deg a_n\inn\n N\cup\{\infty\}.$$
\end{notat}

Denoting by $(\alpha_n)_n$ the complete quotients of $\alpha$, we have $K(\alpha_n)\leq K(\alpha)$, while $\ov K(\alpha_n)=\ov K(\alpha)$ for every $n$. In any case, there certainly exists $N$ such that $\ov K(\alpha)=K(\alpha_n)$ for every $n\geq N$.

\begin{rem}\label{1rem:matrcorr}
	Many of the following formulas can be easily proved using the the matrix formalism introduced in \ref{1notat:matrix}.
	
	Let us denote by $\mathbf M$ the multiplicative monoid generated by the identity and by the matrices $M_a$, with $a\inn\n K[T]\setminus\n K$.
	The inverse of $M_a$ in $\GL_2(\n K[T])$ is $M_{(0,-a,0)}$, so the only invertible element of $\mathbf M$ is the identity.
	
	The restriction of the map $\pphi:M_2(\n L)\to\n L$ defined in Remark \ref{1rem:matrixaction2} gives a natural bijection between $\n K(T)$ and the set of matrices of the form $M_a\cdot M'$ with $a\inn\n K[T]$ and $M'\inn\mathbf M$: if $\alpha=[a_0,\dots,a_n]\inn\n K(T)$, we will say that it corresponds to the matrix $M_\alpha=M_{(a_0,\dots,a_n)}=M_{a_0}\cdot M_{(a_1,\dots,a_n)}$. 
	
	With an abuse of notation, for $\alpha=[a_0,a_1,\dots]\inn\n L\setminus\n K(T)$, we will denote by $M_\alpha$ any matrix of the form $M_{(a_0,\dots,a_k,\alpha_k)}$ with $k\geq0$; we will also write $M_\alpha=M_{(a_0,a_1,\dots)}$. This gives a bijection between $\n L$ and the formal infinite products of matrices of the form $M_a$ with $a\inn\n K[T]$ (where $a$ can be a constant only in the first term). 
\end{rem}

\subsection{Continued fractions with some constant partial quotients} 
As we have mentioned in Remark \ref{1rem:constants}, continued fractions with isolated constant partial quotients converge, as well as continued fractions with pairs of consecutive constant partial quotients with $a_na_{n+1}+1\neq0$. Thus, it can be useful to allow also constant, and possibly zero, partial quotients. In particular, as it is shown by Van der Poorten (\cite{poorten1998formal}, Lemma 2 and Corollary 2), it is possible to transform any converging continued faction with some constant partial quotients in a regular continued fraction following a (possibly infinite) algorithm. 

\begin{lem}For every sequence $(a_n)_n$ of elements of $\n L$ we have
	\be\label{1eq:zero}[a_0,\dots,a_n,0,a_{n+1},a_{n+2},\dots]=[a_0,\dots,a_n+a_{n+1},a_{n+2},\dots].\ee
\end{lem}

\begin{proof}
	$M_aM_0M_b=M_{a+b}$ for every $a,b\inn\n L$, so $M_{(a_0,\dots,a_n,0,a_{n+1},\dots)}=M_{(a_0,\dots,a_n+a_{n+1},\dots)}$.
\end{proof}

\begin{lem}\label{1lem:addrempartq}
	Let $\alpha,\beta,\gamma\inn\n L$ with $\beta\neq0$. Then \be\label{1eq:VdPmanip1}[\alpha,\beta,\gamma]=[\alpha+\beta^{-1},-\beta^2\gamma-\beta]
	\ee and, conversely, \be\label{1eq:VdPmanip2}[\alpha+\beta,\gamma]=[\alpha,\beta^{-1},-\beta^2\gamma-\beta].\ee
\end{lem}

\begin{proof}
	$[\alpha+\beta^{-1},-\beta^2\gamma-\beta]=\alpha+\beta^{-1}+\frac1{-\beta^2\gamma-\beta}=\frac{\alpha\beta\gamma+\alpha+\gamma}{\beta\gamma+1}=[\alpha,\beta,\gamma]$, while \eqref{1eq:VdPmanip2} follows directly from \eqref{1eq:VdPmanip1} substituting $\beta$ with $\beta^{-1}$ and $\gamma$ with $-\beta^2\gamma-\beta$. 
\end{proof}

\begin{rem}\label{1rem:Kconstants}
	In particular, allowing constant, and possibly zero, partial quotients in a converging continued fraction, the degrees of the neighbouring non-constant partial quotients do not decrease when it is transformed in the regular expansion. Thus, if $[b_0,b_1,\dots]$ is a (non-regular) continued fraction expansion for $\alpha\inn\n L$, with $b_n\inn\n K[T]$, possibly constant, we will have $$K(\alpha)\leq\sup_{n\geq1}\deg b_n\, \text{ and } \,\ov K(\alpha)\leq\limsup\deg b_n.$$
\end{rem}

\begin{lem}
	Let $\alpha=[b_0,b_1,\dots]$ be a continued fraction expansion for $\alpha$, with the $b_n\inn\n K[T]$ possibly constant polynomials. Let $u_n,v_n$ be the continuants of such continued fraction and let $\cv m$ be the convergents of the regular expansion of $\alpha$. Then for every $m$ there exists $n_m$ such that $\ds\cv m=\frac{u_{n_m}}{v_{n_m}}$. 
\end{lem}
\begin{proof}
	This follows directly from the facts that $C_{n+3}(a_0,\dots,a_n,\!0,\!c)\!=\!C_{n+1}(a_0,\dots,\!a_n+c)$ and $C_{n+3}(a_0,\dots,a_n,\beta,\gamma)=-\beta^{-1}C_{n+2}(a_0,\dots,a_{n-1},a_n+\beta^{-1},-\beta^2\gamma-\beta)$.
\end{proof}

\subsection{Folding Lemma}
The following results can be found in \cite{niederreiter1987rational} (Lemma 2 and Lemma 3); see also \cite{lauder1999continued}, \cite{poorten1998formal} for different proofs.
\begin{notat}\label{1notat:manip}
	For any finite sequence of polynomials $\overrightarrow w=a_1,\dots,a_n$, we will denote by $\overleftarrow{w}$ the inverse sequence $\overleftarrow{w}=a_n,\dots,a_1$ and, for $k\inn\n K^*$, we will denote by $k\overrightarrow w$ the sequence $k\overrightarrow w=ka_1,k^{-1}a_2,\dots,k^{(-1)^{n-1}}a_n$. Moreover, for $a\inn\n K[T]$, we will write $\overrightarrow w+a=a_1,\dots,(a_n+a),\, \overleftarrow{w}+a=(a_n+a),a_{n-1},\dots,a_1$.
\end{notat}

\begin{lem}[Folding Lemma]\label{1lem:manip}
	Let $\alpha=[a_0,a_1,a_2,\dots]\inn\n L$, let $(\cv n)_n$ be its convergents and let $\overrightarrow{w_n}$ be the sequence $a_1,\dots,a_n$.
	
	For every $a\inn\n K[T]\setminus\{0\}$ we have \be\label{1eq:folding}[a_0,\overrightarrow{w_n},a,-\overleftarrow{w_n}]=\frac{a\,p_nq_n+(-1)^n}{a\,q_n^2}.\ee
	For every $e_1,e_2,c\inn\n K$ such that $e_1^2=e_2^2=1$ and $c^2=e_1e_2$,  \be\label{1eq:folding2}[e_1a_0,e_1\overrightarrow{w_n}+c,e_2\overleftarrow{w_n}-c^3]=\frac{e_1p_nq_n+c(-1)^n}{q_n^2}.\ee

\end{lem}

\begin{proof}
	These formulas can be easily proved using the matrix correspondence defined in Remark \ref{1rem:matrcorr}.
	
	Indeed, let $\gamma=[a_0,\overrightarrow{w_n},a,-\overleftarrow{w_n}]$. By definition, the matrix corresponding to $\gamma$ is $M_\gamma\!=\!M_{(a_0,\dots,a_n)} M_aM_{(-a_n,\dots,-a_1)}$. Now, $M_{(-a_n,\dots,-a_1)}\!\!=\!\!\mm{(-1)^{n-1}p_n}{(-1)^np_{n-1}}{(-1)^nq_n}{(-1)^{n-1}q_{n-1}}^t\!\!\!M_{-a_0}^{-1}$, thus 
	$M_\gamma=\mm{(-1)^na\,p_nq_n+1}{(-1)^{n-1}a\,p_n^2+(-1)^na\,a_0p_nq_n+a_0}{(-1)^na\,q_n^2}{(-1)^{n-1}a\,p_nq_n+(-1)^naa_0q_n^2+1}$ and $\gamma=\frac{ap_nq_n+(-1)^n}{aq_n^2}$.
	
	\eqref{1eq:folding2} can be proved similarly.
\end{proof}

\begin{lem}
	Let $\alpha=[a_0,a_1,\dots], \beta=[b_0,b_1,\dots]\inn\n L$ and let $(\cv n)_n, (\frac{u_m}{v_m})_m$ be their convergents. Then \be\label{1eq:juxtap}[a_n,\dots,a_0,b_0,\dots,b_m]=\frac{p_nu_m+q_nv_m}{p_{n-1}u_m+q_{n-1}v_m}.\ee
\end{lem}

\begin{proof}
	This is an immediate consequence of Remarks \ref{1rem:continuants} and \ref{1rem:matrcorr}.
\end{proof}

\subsection{Unary operations}
We will now consider some simple unary operators on continued fractions. For $f:\n L\to\n L$, $\alpha=[a_0,a_1,\dots]\inn\n L$, we will denote by $\beta=f(\alpha)=[b_0,b_1,\dots]$ its image under $f$. In some cases, and in particular when $f$ is the multiplication or the division by a linear polynomial, it is much easier to study the convergents of $\beta$ than its partial quotients; we will denote by $p_n/q_n$, respectively by ${u_m}/{v_m}$, the convergents of $\alpha$ and $\beta$.

\begin{lem}[Addition of a polynomial]\label{1lem:suma}
	Let $a$ be a polynomial in $\n K[T]$ and let $f(\alpha)=\alpha+a$, for $\alpha\inn\n L$. Then \be\label{1eq:suma}\beta=\alpha+a=[a_0+a,a_1,a_2,\dots],\ee so \be\label{1eq:sumaK}K(\beta)=K(\alpha),\ \ov K(\beta)=\ov K(\alpha)\ee and $$u_n=C_{n+1}(a_0+a,\dots,a_n)=p_n+aq_n,\ v_n=q_n \text{ for every } n.$$
\end{lem}

\begin{lem}[Multiplication by a constant]
	Let $f(\alpha)=c\,\alpha$, with $c\inn\n L^*$. Then formally \be\label{1eq:multAform}\beta=[c\,a_0,c^{-1}a_1,\dots,c^{(-1)^n}a_n,\dots].\ee
\end{lem}
\begin{proof}
	It follows easily from the definition of continued fraction.
\end{proof}

In general, to ensure that the previous expansion is regular, we must assume $c\inn\n K^*$. For example, we have $-\alpha=[-a_0,-a_1,\dots,-a_n,\dots]$. 

\begin{lem}\label{1lem:multA}
	If $\beta=c\,\alpha$ with $c\inn\n K^*$, then $$K(\beta)=K(\alpha) \text{ and } \ov K(\beta)=\ov K(\alpha).$$
	
	Moreover, $$\begin{cases}u_n=c\,p_n,\ v_n=q_n&\text{if } n\text{ is even}\\ u_n=p_n,\ v_n=c^{-1}q_n&\text{if } n\text{ is odd}\end{cases}.$$
\end{lem}
\begin{proof}
	The first equality follows directly from the previous Lemma and the second is a consequence of Lemma \ref{1lem:multAC},
\end{proof}

This relation has been expressed by Schmidt in \cite{schmidt2000continued} as $$c[c'a_0,c\,a_1,c'a_2,\dots]=c'[c\,a_0,c'a_1,c\,a_2,\dots],$$ for $c,c'$ non-zero polynomials in $\n K[T]$.

\begin{lem}[Inverse]
	Let $f(\alpha)=\frac1\alpha$. Then we have \be\label{1eq:1/alpha}\beta=\begin{cases}[0,a_0,a_1,\dots] & \text{if } a_0\notin\n K\\ [a_1,a_2\dots]& \text{if } a_0=0\\ [a_0^{-1},-a_0^2a_1-a_0,-a_0^{-2}a_2,-a_0^2a_3,\dots]& \text{if } a_0\inn\n K^*\end{cases}.\ee 
	
	If $a_0\notin\n K$ we have that the convergents $\frac{u_n}{v_n}$ of $\beta$ are given by $$u_n=C_{n+1}(0,a_0,\dots,a_{n-1})=q_{n-1}, v_n=C_n(a_0,\dots,a_{n-1})=p_{n-1}$$ and, of course, similar formulas hold also in the other cases.
	
	If $a_0\notin\n K$, then $K(\beta)=\max\{K(\alpha),\deg a_0\}$, while $K(\beta)\leq K(\alpha)$ if $a_0=0$ and $K(\beta)=K(\alpha)$ if $a_0\inn\n K^*$. In all cases, $$\ov K(\beta)=\ov K(\alpha).$$
\end{lem}
\begin{proof}
	$[0,a_0,a_1,\dots]$ is always a continued fraction expansion for $\beta$ but it is regular if and only if $\deg a_0>0$. Otherwise, it can be transformed in the regular continued fraction expansion of $\beta$ by \eqref{1eq:zero} if $\alpha=0$ or by \eqref{1eq:VdPmanip1} if $\alpha\inn\n K^*$.
\end{proof}

\begin{lem}[Frobenius endomorphism]
	Let $\n K$ be a field of characteristic $l$. Applying the Frobenius endomorphism $f(\alpha)=\alpha^l$ it is clear that $$\beta=\alpha^l=[a_0^l,a_1^l,\dots], \text{ with } u_n=p_n^l,\ v_n=q_n^l \text{ for every } n$$ and $$K(\beta)=l\,K(\alpha),\ \ov K(\beta)=l\,\ov K(\alpha).$$
\end{lem}

\begin{rem}
	In particular, if $m$ is a power of $l$ and $\alpha\!=\![a_0,\dots,a_n,a_0^m,\dots,a_n^m\!,a_0^{2m}\!,\dots]$, then $\alpha=[a_0,\dots,a_n,\alpha^m]$, so $\alpha$ is algebraic of degree at most $m+1$ over $\n K(T)$. Thus, unlike the real case, in positive characteristic it is easy to find explicit examples for the continued fraction expansions of algebraic elements of various degrees. An algebraic irrational element $\alpha$ of this type is said to be of class IA; such special elements have been studied, for instance, by Schmidt \cite{schmidt2000continued}.   
\end{rem}

\begin{lem}[Substitution]\label{1lem:subst}
	For a non-constant polynomial $P(X)\inn\!\n K[X]$ we can consider the composition with P map $f:\n L_T=\n K((T^{-1}))\to\n K((X^{-1}))\!=\!\n L_X$, that is, $f(\alpha(T))=\alpha(P(X))$. Of course, $$\beta=[a_0(P(X)),a_1(P(X)),\dots] \text{ and } u_n(X)=p_n(P(X)),\ v_n(X)\!=q_n(P(X)) \text{ for every } n.$$ Moreover, $$K(\beta)=(\deg P) K(\alpha) \text{ and } \ov K(\beta)=(\deg P)\ov K(\alpha).$$
\end{lem}

\begin{rem}
	In particular, it is easy to construct examples of continued fractions such that the degrees of their partial quotients are all multiples of a given integer. 
\end{rem}

\subsection{Multiplication of a continued fraction by a polynomial}\label{1subsec:mult}
As we will see in Theorem \ref{1theo:Grisel}, there exists a recursive algorithm to express the partial quotients of a (not necessarily regular) continued fraction expansion of $\frac AB\alpha$ in terms of the partial quotients of $\alpha$, where $\alpha$ is a Laurent series and $A/B$ a rational function. However, thanks to the best approximation Theorems, it is possible to find in an easier way some of the convergents of $\frac AB\alpha$, from which immediately follow bounds for the degrees of its partial quotients. To the best of our knowledge, the following simple results do not appear in the literature, even if similar methods to those of Theorem \ref{1theo:multgen} are used, in a slightly different context, in \cite{pinner1993some}, Theorem 2.

\begin{theo}\label{1theo:multgen}
	Let $\alpha=[a_0,a_1,a_2,\dots]\inn\n L$, let $A,B$ be relatively prime non-zero polynomials and let $\beta=\frac AB\alpha=[b_0,b_1,b_2,\dots]$. Let $(p_n,q_n)_n,(u_m,v_m)_m$ be, respectively, the continuants of $\alpha$ and $\beta$. For every $n$, let $A_n=\gcd(A,q_n)$ and let $B_n=\gcd(B,p_n)$. Provided that $$\deg a_{n+1}>\deg A+\deg B-2\deg A_n-2\deg B_n,$$ then there exists $m\inn\n N$ and there exists $k\inn\n K^*$ such that \be u_m=k\frac {A}{A_nB_n}p_n,\ v_m=k\frac{B}{A_nB_n}q_n.\ee In this case, we also have \be\label{1eq:degmult}\deg b_{m+1}=\deg a_{n+1}-\deg A-\deg B+2\deg A_n+2\deg B_n.\ee
\end{theo}

\begin{proof}
	Let $u=k\frac {Ap_n}{A_nB_n}, v=k\frac{Bq_n}{A_nB_n}$, with $k\inn\n K^*$. In particular, $u,v$ are relatively prime polynomials. Moreover, $$|u-v\beta|=|A||A_nB_n|^{-1}|p_n-q_n\alpha|=|A||A_nB_n|^{-1}|q_na_{n+1}|^{-1}.$$ Now, by Lemma \ref{1lem:best2}, $u/v$ is a convergent of $\beta$ if and only if $|u-v\beta|<|v|^{-1}$, that is, if and only if $\deg a_{n+1}>\deg A+\deg B-2\deg A_n-2\deg B_n$.	In that case, \eqref{1eq:degmult} follows directly from \eqref{1eq:fondprop}.
\end{proof}

\begin{cor}\label{1cor:degrees}
	In the previous notations we have that $\alpha$ is badly approximable if and only if $\beta=\frac AB\alpha$ is. Moreover, in this case \be\label{1eq:degrees}K(\alpha)-\deg AB\leq K(\beta)\leq K(\alpha)+\deg AB.\ee
\end{cor}

\begin{proof}
	If $K(\alpha)>\deg AB$, applying the previous Theorem to all the partial quotients with large enough degree we will have $K(\beta)\geq K(\alpha)-\deg AB$, which is obvious in the case $K(\alpha)\leq\deg AB$. Inverting the roles of $\alpha,\beta$ and of $A,B$ we will have $K(\alpha)\geq K(\beta)-\deg AB$.
\end{proof}

Refining Theorem \ref{1theo:multgen}, it is possible to give explicitly, up to multiplicative constants, all the continuants of the product (or the quotient) of a continued fraction by a linear polynomial:

\begin{prop}\label{1prop:prod}
	Let $\alpha=[a_0,a_1,\dots]$ and, for $\lambda\inn\n K$, let $\beta=(T-\lambda)\alpha=[b_0,b_1,\dots]$. Let $(p_n,q_n)_n,( u_m,v_m)_m$ be, respectively, the continuants of $\alpha$ and $\beta$.
	\begin{enumerate}
		\item If $q_n(\lambda)\neq0$ and $\deg a_{n+1}>1$, then there exist $m\inn\n N,\ k\inn\n K^*$ such that $$u_m(T)=k\,(T-\lambda)p_n(T),\ v_m(T)=k\,q_n(T)$$ and in that case $\deg b_{m+1}=\deg a_{n+1}-1$;
		\item If $q_n(\lambda)q_{n+1}(\lambda)\neq0$, then there exist $m\inn\n N,\ k\inn\n K^*$ such that $$u_m(T)=k\,(p_n(T)q_{n+1}(\lambda)-p_{n+1}(T)q_n(\lambda)),$$ $$v_m(T)=k\,\frac{q_n(T)q_{n+1}(\lambda)-q_{n+1}(T)q_n(\lambda)}{T-\lambda}$$ and in that case $\deg b_{m+1}=1$;
		\item If $q_n(\lambda)=0$, then there exist $m\inn\n N,\ k\inn\n K^*$ such that $$u_m(T)=k\,p_n(T),\ v_m(T)=k\,\frac{q_n(T)}{T-\lambda}$$ and in that case $\deg b_{m+1}=\deg a_{n+1}+1$.
	\end{enumerate}
	Moreover, the previous formulas give exactly (and only once) all the convergents of $\beta$.
\end{prop}

\begin{proof}
	1. and 3. follow directly from Theorem \ref{1theo:multgen}; since, in the previous notations, $A=T-\lambda$ and $B=1$, they are the only possible cases.
	
	In the hypothesis of $2.$, setting $u=q_{n+1}(\lambda)p_n-q_n(\lambda)p_{n+1}, v=\frac{q_{n+1}(\lambda)q_n-q_n(\lambda)q_{n+1}}{T-\lambda}$, we have $|u-v\beta|=|q_{n+1}(\lambda)(p_n-q_n\alpha)-q_n(\lambda)(p_{n+1}-q_{n+1}\alpha)|=|q_{n+1}|^{-1}$. As $|v|=|q_{n+1}||T|^{-1}$ and $u,v$ are relatively prime polynomials, then $(u,v)$ is, up to a multiplicative constant, a continuant of $\beta$. By \eqref{1eq:fondprop}, the degree of the following partial quotient of $\beta$ is necessarily 1.
	
	It follows easily from an examination of the degrees of the continuants $u_m,v_m$ and of the partial quotients $b_{m+1}$ that the convergents thus found are all distinct and that they must be all of the convergents of $\beta$. 
\end{proof}	

Of course, we can obtain similar formulas for the convergents of $\ds\frac{\alpha}{T-\lambda}$:
\begin{prop}
	In the previous notations, let $\ds\beta=\frac\alpha{T-\lambda}$. 
	\begin{itemize}
		\item If $p_n(\lambda)\neq0$ and $\deg a_{n+1}>1$, then there exist $m\inn\n N,\ k\inn\n K^*$ such that $$u_m=k\,p_n,\ v_m=k\,(T-\lambda)q_n$$ and in that case $\deg b_{m+1}=\deg a_{n+1}-1$; 
		\item If $p_n(\lambda)p_{n+1}(\lambda)\neq0$, then there exist $m\inn\n N,\ k\inn\n K^*$ such that $$u_m=k\,\frac{p_n(T)p_{n+1}(\lambda)-p_{n+1}(T)p_n(\lambda)}{T-\lambda},$$ $$\ v_m=k\,(q_n(T)p_{n+1}(\lambda)-q_{n+1}(T)p_n(\lambda))$$ and in that case $\deg b_{m+1}=1$;
		\item If $p_n(\lambda)=0$, then there exist $m\inn\n N,\ k\inn\n K^*$ such that $$u_m=k\,\frac{p_n}{T-\lambda},\ v_m=k\,q_n$$ and in that case $\deg b_{m+1}=\deg a_{n+1}-1$.
	\end{itemize}
	Moreover, the previous formulas give exactly (and only once) all the convergents of $\beta$.
\end{prop}

\begin{cor}\label{1cor:decreasedegree}
	Let $\alpha\inn\n L$, let $(\cv n)_n$ be its convergents and let $\lambda\inn\n K$. $$\text{If } q_n(\lambda)\neq0 \text{ for every } n, \text{ then } K((T-\lambda)\alpha)=\begin{cases}1&\text{if } K(\alpha)=1\\K(\alpha)-1&\text{otherwise}\end{cases}.$$ 
	
	$$\text{If } p_n(\lambda)\neq0 \text{ for every } n, \text{ then } \ds K\left(\frac\alpha{T-\lambda}\right)=\begin{cases}1&\text{if }K(\alpha)=1\\K(\alpha)-1&\text{otherwise}\end{cases}.$$
	
	Analogously, if $q_n(\lambda)\neq0$, respectively, $p_n(\lambda)\neq0$, for every large enough $n$, then, assuming $K(\alpha)>1$, $\ov K((T-\lambda)\alpha)=\ov K(\alpha)-1$, respectively $\ds\ov K\left(\frac\alpha{T-\lambda}\right)=\ov K(\alpha)-1$.
\end{cor}
 
\begin{lem}\label{1lem:(f+a)/g}
	Let $\alpha\inn\n L$ with $K(\alpha)>1$ and let $\lambda,a\inn\n K$. Then $$\text{if } p_n(\lambda)+aq_n(\lambda)\neq0 \text{ for every } n, \text{ then } K\left(\frac{\alpha+a}{T-\lambda}\right)=K(\alpha)-1.$$
	
	Similarly, if $p_n(\lambda)+aq_n(\lambda)\!\neq\!0$ for every large enough $n$, then $\ds\ov K\left(\frac{\alpha+a}{T-\lambda}\right)\!=\ov K(\alpha)-1$.
\end{lem}
\begin{proof}
It follows immediately from the previous Corollary and \eqref{1eq:suma},
\end{proof}

\begin{rem}
	Lemma \ref{1lem:addrempartq} allows us to recover the regular expansion of the multiplication of a continued fraction by any rational function.
	
	For example, if $\alpha=[a_0,a_1,\dots]\inn\n L$ and $A$ is a polynomial such that $A\nmid a_1$, then $A\alpha\!=\![Aa_0,a_1/A,\dots]\!=\!\!\left[Aa_0,\floor{a_1/A}\!,\left\{a_1/A\right\}^{-1}\!,-\!\left\{a_1/A\right\}^2[a_2,a_3,\dots]\!-\!\left\{a_1/A\right\}\right]=\dots$. By repeatedly applying (possibly infinitely many times) \eqref{1eq:VdPmanip2}, \eqref{1eq:VdPmanip1} or \eqref{1eq:zero} we can thus obtain the regular continued fraction of $A\alpha$. 
\end{rem}

Actually, adapting to the polynomial case a result due to Mend\`es France \cite{france1973fractions}, Grisel in \cite{grisel1996longueur}, Th\'{e}or\`{e}me 2, gave an algorithm for the continued fraction expansion of the product of a formal Laurent series by a rational function (where constant partial quotients are allowed).
\begin{notat}
	Let $\alpha=[a_0,\dots,a_n]$; we will denote by $[b_1,\dots,b_k,[\alpha],b_{k+1},\dots]$ the continued fraction $[b_0,\dots,b_k,a_0,\dots,a_n,b_{k+1},\dots]$.
\end{notat}

\begin{theo}[Grisel]\label{1theo:Grisel}
	Let $\alpha=[a_0,a_1,\dots]$ and let $A,B$ be relatively prime polynomials; for every $i\geq0$ let $a_i=ABa_i'+h_i$, with $a_i',h_i\inn\n K[T]$ and $\deg h_i<\deg AB$. Let us set $\delta_{-1}=1, u_{-1}=0, Q_{-1}=0, \delta_0=A$ and, for $i\geq0$, let us define recursively $$H_i=\frac{AB}{\delta_i},\ a_i''=\floor{\frac{(-1)^{u_{i-1}}\delta_ih_i-\delta_{i-1}Q_{i-1}}{H_i}},\ \frac{j_i}{H_i}=\left\{{\frac{(-1)^{u_{i-1}}\delta_ih_i-\delta_{i-1}Q_{i-1}}{H_i}}\right\},$$ $u_i=u_{i-1}+\ell\left(j_i/H_i\right)$, where $\ell([a_0,\dots,a_n])=n$ is the length of the continued fraction; let $Q_i,\widetilde{Q_i}$ be, respectively, the denominators of the second to last and of the last convergents of $j_i/H_i$ and let $\delta_{i+1}={H_i}/{\widetilde{Q_i}}$. Then \be\frac AB\alpha=\left[(-1)^{u_{-1}}\delta_0^2a_0'+a_0'',\left[\frac{H_0}{j_0}\right],\dots,\left[\frac{H_{i-1}}{j_{i-1}}\right],(-1)^{u_{i-1}}\delta_i^2a_i'+a_i'',\left[\frac{H_i}{j_i}\right],\dots\right],\ee
	where, if $j_i=0$, we set $[H_i/j_i]=[0,0],\ \ell(j_i/H_i)=0,\ Q_i=0$ and $\widetilde Q_i=1$. In general, this may not be the regular continued fraction expansion of $\frac AB\alpha$, since some of the $(-1)^{u_{i-1}}\delta_i^2a_i'+a_i''$ could be constants. In this case, the regular expansion can be recovered by applying (possibly infinitely many times) \eqref{1eq:VdPmanip2}.
\end{theo}

\begin{ex}
	For example, let $\alpha=[T,T+1,T-2,T^2]=\frac{T^5-T^4-T^2+T+1}{T^4-T^3-T^2+T+1}$. Then $T\alpha=[T^2,0,0,1,T,-1,T,T,-T]=[T^2+1,-T,T-1,T,-T]$, by \eqref{1eq:VdPmanip1} and \eqref{1eq:zero}.
\end{ex}

\begin{rem}
	In the previous notations, let $A=T-\lambda$ and $B=1$ with $\lambda\inn\n K$.
	
	Let us assume that $q_n(\lambda)\neq0$ for every $n$. It is easy to see by induction that for $n\geq1$ we have $\delta_n=(-1)^{\floor{\frac{n-1}2}}q_{n-1}(\lambda), H_n=(-1)^{\floor{\frac{n-1}2}}q_{n-1}(\lambda)^{-1}(T-\lambda)$, $j_n=(-1)^{\floor{\frac n2}}q_n(\lambda),$ $a_n''=0$, $u_n=n, \widetilde{Q_n}=(-1)^{n-1}(q_{n-1}q_n)(\lambda)^{-1}(T-\lambda)$, $Q_n=1$, where $h_n=a_n(\lambda)$ and $\deg a_n'=\deg a_n-1$. Then $$(T-\lambda)\alpha=[(T-\lambda)a_0,\dots,(-1)^{n-1}q_{n-1}(\lambda)^2a_n',(-1)^{n-1}(q_{n-1}q_n)(\lambda)^{-1}(T-\lambda),\dots],$$ which implies again the first result of Corollary \ref{1cor:decreasedegree}.
	
	On the other hand, if $\alpha=[a_0,\dots,a_N,\alpha_{N+1}]$ with $q_n(\lambda)\neq0$ for $n<N$ and $q_N(\lambda)=0$, then $(T-\lambda)\alpha=[(T-\lambda)a_0,\dots,q_{n-1}(\lambda)^2a_n',(-1)^{n-1}q_{n-1}(\lambda)^2(T-\lambda)\alpha_{n+1}]$, giving again that, if $q_n(\lambda)=0$, the degree of the following partial quotient is increased by 1.
	
	Inverting the roles of the $p_n,q_n$ similar results can be found for $\ds\frac{\alpha}{T-\lambda}$.
\end{rem}

\subsection{A polynomial analogue of Serret's Theorem}
The multiplication of a continued fraction by a rational function is a special case of a M\"{o}bius transformation in $M_2(\n K[T])$.
\begin{rem}\label{1rem:K((Aa+B)/(Ca+D))}
	Let $\beta\!=M\alpha$, with $M=\mm ABCD\inn M_2(\n K[T])$ such that $\det M\neq0$. Then, $\beta\!=\frac 1C\left(A-\frac{\det M}{C\alpha+D}\right)$, so $\alpha$ is badly approximable if and only if $\beta$ is and, assuming that $K(C\alpha+D)=K\left( 1/(C\alpha+D)\right)$, then  $$K(\alpha)-2\deg C-\deg(\det M)\leq K(\beta)\leq K(\alpha)+2\deg C+\deg(\det M).$$
	
	Of course, the same relation holds for $\ov K(\alpha),\ov K(\beta)$. Actually, we will see in Remark \ref{1rem:Serretgen} that $$\ov K(\alpha)-\deg(\det M)\leq\ov K(\beta)\leq\ov K(\alpha)+\deg(\det M).$$ 
\end{rem}

\begin{notat}\label{1notat:equiv}
	We will say that $\alpha,\beta\inn\n L$ are equivalent, and write $\alpha\sim\beta$, if they are $\GL_2(\n K[T])$-equivalent, that is, if there exists $M\inn\GL_2(\n K[T])$ such that $\alpha=M\beta$.
\end{notat}

Certainly, $\sim$ is an equivalence relation.
	
In particular, any two rational functions are equivalent and, by \eqref{1eq:pnqnmatrix}, $\alpha\sim\alpha_n$ for every $n$ (where the $\alpha_n$ are the complete quotients of $\alpha$).\par\medskip

Actually, Serret \cite{serret1849cours} proved that two real numbers are $\GL_2(\n Z)$-equivalent if and only if the tails of their continued fraction expansions coincide. An analogous result holds in the polynomial case:

\begin{theo}[polynomial analogue of Serret's Theorem]\label{1theo:serret}
	Let $\alpha,\beta$ be irrational Laurent series. Then the following are equivalent:
	\begin{enumerate}
		\item $\beta\sim\alpha$;
		\item there exist $m,n\inn\n N, l\inn\n K^*$ such that $l\alpha_n=\beta_m$;
		\item the continued fraction expansions of $\alpha,\beta$ eventually coincide up to multiplication by a constant, that is, there exist $n,m\inn\n N, l\inn\n K^*$ such that $\alpha=\![a_0,\dots,a_{n-1},c_1,c_2,\dots]$ and $\beta=[b_0,\dots,b_{m-1},lc_1,l^{-1}c_2,\dots]$.
	\end{enumerate}
\end{theo}

\begin{lem}\label{1lem:decompmatr}
	Let $M=\mm ABCD\inn M_2(\n K[T])$ be a unimodular matrix. Then we have $\deg A>\deg B>\deg D$ and $\deg A>\deg C>\deg D$ if and only if there exists a constant $c\inn\n K^*$ such that $M\mm c00{\pm c^{-1}}\inn\mathbf M$.
\end{lem}
\begin{proof} 
	Let $M_{(a_0,\dots,a_n)}\!=\!\!\mm ABCD\!\inn\mathbf M$, where $a_0,\dots,a_n$ are non-constant polynomials. Then it is easy to see that $\deg A\!>\deg B\!>\deg D$ and $\deg A\!>\deg C\!>\deg D$.
	
	Conversely, let $M\!=\!\mm ABCD$ be a unimodular matrix such that $\deg A\!>\!\deg B\!>\!\deg D$ and $\deg A>\deg C>\deg D$. Then, applying the Euclidean algorithm to the columns of $M$ it can be seen that there exists $c\inn\n K^*$ such that $M\mm c00{\pm c^{-1}}\inn\mathbf M$.

	More precisely, $M\!=M_{(a_0,\dots,a_n)}\mm c00{\pm c^{-1}}$, where $\frac AC\!=\![a_0,\dots,a_n],\ \frac BD\!=\![a_0,\dots,a_{n-1}]$ and $c\inn\  K^*$.
\end{proof}

\begin{rem}\label{1rem:decompmatr}
	Reasoning in the same way on the rows of $M$ we have that there exist $b_0,\dots,b_n\inn\n K[T]$, with $\deg b_i>0$ for every $i$, and there exists $d\inn\n K^*$ such that $M=\mm d00{\pm d^{-1}}M_{(b_0,\dots,b_n)}$.\par\medskip

	It can be proved similarly that if $M\inn\GL_2(\n K[T])$ is an invertible matrix with $\det M=k\inn\n K^*$, then $M=M_{(a_0,\dots,a_n)}\mm c00{\pm c^{-1}}\mm k001$ with $c\inn\  K^*$ and $a_0,\dots,a_n$ polynomials such that $\deg a_1,\dots,\deg a_{n-1}>0$.
\end{rem}

\begin{proof}[Proof of Theorem \ref{1theo:serret}]
	The equivalence of 2. and 3. follows directly from the definition of complete quotients and from \eqref{1eq:multAform}.
	
	If $\alpha_n=l\beta_m$, with $l\inn\n K^*$, then obviously $\alpha\sim\alpha_n\sim\beta_m\sim\beta$.
	
	Let us assume that $\beta=M\alpha$ with $M\inn\GL_2(\n K[T])$ and $l=\det M\inn\n K^*$. By the previous Remark, $M$ can be written as $M=M_{(d_0,\dots,d_n)}\mm c00{\pm c^{-1}}\mm l001$ with $c\inn\n K^*$ and $d_i\inn\n K[T],\ \deg d_i>0$ for $0<i<n$. 
	Then $\beta=[d_0,\dots,d_n,\pm c^2l\alpha]$. If $\deg d_n>0$ and $\ord(\alpha)<0$, then $\pm lc^2\alpha=\beta_{n+1}$ and we directly find the regular continued fraction for $\beta$. Otherwise, we can still recover it by applying a finite number of times \eqref{1eq:zero} or \eqref{1eq:VdPmanip1}. In any case, the regular continued fraction of $\beta$ will be of the desired form.
\end{proof}

\begin{lem}
	Let $\alpha,\beta$ be two equivalent formal Laurent series with $\alpha=[a_0,a_1,\dots]$ and $\beta=[b_0,b_1,\dots,b_M,a_0,a_1,\dots]$; let $p_n/q_n,\ {u_m}/{v_m}$ be, respectively, the convergents of $\alpha$ and $\beta$. Then, as in \eqref{1eq:juxtap}, $$u_m=u_Mp_{m-M-1}+u_{M-1}q_{m-M-1},\ v_m=v_Mp_{m-M-1}+v_{M-1}q_{m-M-1}$$ for every $m>M$.
\end{lem}

\begin{cor}\label{1cor:Serret}
	In particular, if $\alpha\sim\beta$, then \be\label{1eq:Serret}\ov K(\alpha)=\ov K(\beta).\ee
\end{cor}

\begin{rem}\label{1rem:Serretgen}
	Let $M=\mm ABCD\inn M_2(\n Z)$ with $\det M\neq0$. As in Remark \ref{1rem:decompmatr}, essentially by performing the Euclidean algorithm on its first column, we can write $M$ as $M=M_{(a_0,\dots,a_n)}M'$, with $M'$ of the form $M'=\mm{(A,C)}{B'}0{\pm\frac{\det M}{(A,C)}}$.
	
	Setting as before $\beta=M\alpha$, we will have $\beta\sim\frac{(A,C)\alpha+B'}{(\det M)/(A,C)}$, and then, if $r=(A,C)$, 
	$$\begin{array}{rcll}  &\hspace{-0,2cm}\ov K(\beta)\hspace{-0,4cm}&=\ov K\left(\frac{r\alpha+B'}{(\det M)/r}\right) &\text{ by \eqref{1eq:Serret}},\\\ov K(r\alpha+B')-\deg((\det M)/r)\leq&\hspace{-0,2cm}\ov K(\beta)\hspace{-0,4cm}&\leq\ov K(r\alpha+B')+\deg((\det M)/r) &\text{ by \eqref{1eq:degrees}},\\ \ov K(r\alpha)-\deg((\det M)/r)\leq&\hspace{-0,2cm}\ov K(\beta)\hspace{-0,4cm}&\leq\ov K(r\alpha)+\deg((\det M)/r) &\text{ by \eqref{1eq:sumaK}},\\\ov K(\alpha)-\deg(\det M)\leq&\hspace{-0,2cm}\ov K(\beta)\hspace{-0,4cm}&\leq\ov K(\alpha)+\deg(\det M) &\text{ by \eqref{1eq:degrees}}.\end{array}$$
	
	As well as Serret's Theorem, this Remark is the polynomial analogue of a well-known result for the real case. Indeed, as it is proved by Lagarias and Shallit \cite{lagarias1997linear}, if $\alpha$ is a real number with partial quotients bounded (eventually) by some constant $k$ and $\beta=\frac{a\alpha+b}{c\alpha+d}$, with $a,b,c,d\inn\n Z$ and $ad-bc\neq0$, then the partial quotients of $\beta$ are eventually bounded by $(ad-bc)(k+2)$. 
\end{rem}

\begin{lem}
	Let $\alpha=[a_0,\dots,a_{n-1},\alpha_n]$ and let $\beta=[b_0,\dots,b_{m-1},l\alpha_n]$ with $l\inn\n K^*$. 
	Then $\beta=M\alpha$ with $$M=M_{(b_0,\dots,b_{m-1}-l^{-1}a_{n-1},\dots,-l^{(-1)^{k-1}}a_k,\dots,-l^{(-1)^{n}}a_0)}\cdot\begin{cases}
	\mm0l10&\text{if } n \text{ is odd}\\\mm01l0& \text{if } n \text{ is even}
	\end{cases}.$$
\end{lem}
\begin{proof}	
	Of course, $M=M_{(b_0,\dots,b_{m-1})}\mm l001 M_{(a_0,\dots,a_{n-1})}^{-1}$; the lemma follows from the facts that $\mm l001 M_a=M_{la}\mm100l$ for every $a,l$ and that $M_{(a_0,\dots,a_n)}^{-1}=M_{(0,-a_n,\dots,-a_0,0)}$.
\end{proof}
\clearpage{\pagestyle{empty}\cleardoublepage}
\chapter{Quadratic irrationalities in function fields}\label{ch:quadr}
In the real case, more precise results can be given on the continued fraction expansions of real quadratic irrationalities. In particular, Lagrange proved that the continued fraction expansion of any real quadratic irrationality is eventually periodic and Galois showed that the continued fraction of a real quadratic irrationality $\alpha$ is purely periodic if and only if $\alpha$ is reduced, that is, if $-1<\alpha'<0<1<\alpha$, where $\alpha'$ is the conjugate of $\alpha$. In particular, if $d$ is a positive non-square integer, the continued fraction of $\sqrt d$ is periodic and its period starts from the second partial quotient. The periodicity of the continued fraction of $\sqrt d$ is directly linked to the fact that the Pell equation for $d$ has non-trivial solutions; actually, all its solutions descend from the convergents of $\sqrt d$.\par\medskip

We will see that, in the polynomial case, if the base field $\n K$ is an algebraic extension of a finite field, all such results have perfect analogues, which can be proved using the same techniques as in the real setting. In the other cases, however, most continued fractions of quadratic irrationalities are not periodic and for most polynomials the analogue of the Pell equation has no solutions (apart from the trivial ones).\par\medskip

Having briefly recalled some basic definitions and results of algebraic geometry, we will present a classical connection, already known to Abel, between continued fraction expansions of quadratic irrationals and hyperelliptic curves. However, this requires the polynomial $D$ to be squarefree; as we will see in section \ref{66Sec}, similar results can be obtained also in the general case by substituting the usual Jacobian variety with an appropriate generalized Jacobian. 

\begin{notatC}
	From now on, $\n K$ will be a field with characteristic different from $2$. We will denote by $\S K$ the set of polynomials D of positive even degree in $\n K[T]$, which are not a perfect square in $\n K[T]$ and whose leading coefficient is a square in $\n K$. As we have seen in Lemma \ref{1lem:rad}, if $D\inn\S K$ then $D$ has a square root in $\n L\setminus\n K(T)$, which is unique up to the choice of a sign, that is, up to the choice of a square root in $\n K$ of the leading coefficient of $D$. In fact, a polynomial $D$ of $\n K[T]$ has an irrational square root in $\n L$ if and only if $D\inn\S K$. In this case, we will write $\sqrt D$, assuming the sign to have been properly chosen.

	We will always denote the degree of such a polynomial $D$ by $\deg D=2d$.\par\medskip

	We will want to study the continued fraction expansions of formal Laurent series of the form $\ds\alpha=\frac{A+B\sqrt D}C$, with $A,B,C\inn\n K[T],\ C\neq0$ and $D\inn\S K$. If $B$ is different from zero, that is, if $\alpha$ is not a rational function, then $\alpha$ is said to be a \textit{quadratic irrationality}.\par\medskip

	Let $D\inn\S K$. Then $\n K(T,\sqrt D)$ is a quadratic extension of $\n K(T)$ with Galois group $G=\{\id,\sigma\}$, with $\sigma(\sqrt D)=-\sqrt D$. If $\alpha\in\S K,\ \alpha=\frac{A+B\sqrt D}C$, its norm  and trace are given, respectively, by $N(\alpha)=\alpha\alpha'=\frac{A^2-B^2 D}{C^2},\ \Tr(\alpha)=\alpha+\alpha'=2\frac AC$, where we denote by $\alpha'=\sigma(\alpha)$ the Galois conjugate of $\alpha$.
\end{notatC}

Let $\alpha=[a_0,a_1,a_2,\dots]$. Then, by \eqref{1eq:suma}, \eqref{1eq:multAform}, $\alpha'=[\Tr(\alpha)-a_0,-a_1,-a_2,\dots]$. However, in general $\Tr(\alpha)$ is not a polynomial, so this is not the regular continued fraction expansion of $\alpha'$. We could recover the regular continued fraction expansion of $\alpha'$ using the results of Lemma \ref{1lem:addrempartq}; we will see in Remark \ref{2rem:contfracalpha'} that, in some cases, the continued fraction of $\alpha'$ can be obtained directly from that of $\alpha$.

\begin{remC}\label{2rem:rsD'}
	If $\alpha=\frac{A+B\sqrt D}C$ is a quadratic irrationality, there exist polynomials $r,s,D'$, with $s\neq0$, such that \be\label{2eq:defrs}\alpha=\frac{r+\sqrt{D'}}s\text{, with } s|(r^2-D').\ee
	Indeed, if $m=\frac C{(C,A^2-B^2D)}$, it is enough to take $r=m\,A, s=m\,C, D'=m^2\,B^2D$ (where, as before, we are assuming that we have chosen the correct sign for the square root). Such a triple of polynomials $r,s,D'$ is certainly not unique but, assuming $A,B,C$ to be relatively prime, the previous choice provides a triple of minimal degree.
\end{remC}

From now on, writing $\alpha=\frac{r+\sqrt D}s$ we will always assume that $D\inn\S K$, that $r,s\inn\n K[T]$ with $s\neq0$ and that $s|(r^2-D)$.

\begin{lemC}\label{2lem:tn,delta}
	Let $\alpha=\frac{r+\sqrt D}s=[a_0,a_1,\dots]$ and let $\delta=\floor{\sqrt D}$. Let us set $s_0=s,\ r_0=r$ and $$r_n=a_{n-1}s_{n-1}-r_{n-1},\ s_n=\frac{D-r_n^2}{s_{n-1}} \text{ for } n\geq1.$$ Then $r_n,s_n$ are polynomials in $\n K[T]$ for every $n$ and it is easy to see by induction that \be\alpha_n=\frac{r_n+\sqrt D}{s_n}\text{, so } a_n=\floor{\frac{r_n+\delta}{s_n}},\ee where we denote by $\alpha_n$ the complete quotients of $\alpha$. 
	
	Let $t_n$ be the remainder in the Euclidean division of $r_n+\delta$ by $s_n$. Then we have $a_n=\frac{r_n+\delta-t_n}{s_n}$, so $\delta-t_n=r_{n+1}$.
\end{lemC}

\begin{defnC}
	A quadratic irrationality $\alpha\inn\n L$ is said to be \textit{reduced} if $$\ord(\alpha)<0 \text{ and } \ord(\alpha')>0.$$
\end{defnC}

\begin{propC}\label{2prop:red}
	If $\alpha=\frac{r+\sqrt D}s$ is a reduced quadratic irrationality, then $\alpha_n$ is reduced for every $n$. 
	
	Moreover, for every quadratic irrationality $\alpha$ there exists an index $N$ such that $\alpha_N$ is reduced (that is, such that $\alpha_n$ is reduced for every $n\geq N$). Then, for $n\geq N$ we have
	$$\deg r_n=-\ord(\sqrt D)=\deg(a_ns_n)=d$$ and the coefficients of $r_n$ and $\sqrt D$ of degrees $d,d-1$ coincide. In particular, \be0<\deg a_n\leq d, \text{ and } 0\leq\deg s_n<d \text{ for every } n\geq N.\ee Moreover, $$a_n=\floor{\frac{2\sqrt D}{s_n}} \text{ for every } n\geq N.$$
\end{propC}

\begin{corC}
	Any quadratic irrationality is badly approximable and, in the previous hypotheses and notations, $$\ov K\left(\frac{r+\sqrt D}s\right)\leq d=\frac 12\deg D.$$
\end{corC}

\begin{proof}[Proof of Proposition \ref{2prop:red}]
	It is easy to see that $\alpha_n$ is reduced if and only if $\deg s_{n-1}<d$ and that if $\alpha_n$ is reduced, then $\deg(a_ns_n)\leq d$, which implies that if $\alpha_N$ is reduced, then $\alpha_n$ is for every $n\geq N$. Moreover, it can be seen that if $\alpha_n$ is not reduced, then $\deg s_n<\deg s_{n-1}$, so there exists $N$ such that $\alpha_N$ is reduced.	
	
	If $\alpha_n$ is reduced, then $\deg t_{n-1}<\deg s_n$, so $a_n=\floor{\frac{r_n+\delta}{s_n}}=\floor{\frac{2\delta-t_{n-1}}{s_n}}=\floor{\frac{2\sqrt D}{s_n}}$ (where the $t_n$ and $\delta$ are defined as in Lemma \ref{2lem:tn,delta}). 
\end{proof}

\begin{remC}\label{2rem:sqrtDred}
	Let $\alpha=\sqrt D$; of course, $\alpha$ is not reduced but $\ds\alpha_1=\frac{\delta+\sqrt D}{D^2-\delta}$ is, so $\alpha_n$ is reduced for every $n\geq1$. In particular, $\deg a_0=d$ and $\deg a_n\leq d$ for every $n\geq1$, that is, in the notations of \ref{1notat:K,ovK}, $$K(\sqrt D)\leq d.$$
\end{remC}

In the real case, continued fractions of real quadratic irrationalities are directly linked with solutions of the Pell equation. Actually, the polynomial analogue of Pell's equation arises naturally in the study of quadratic irrationalities of $\n L$.
\begin{lemC}
	Let $\alpha$ be a quadratic irrationality, let $(\cv n)_n$ be its convergents. For every $n$, let $\pphi_n=p_n-\alpha q_n$. Then we have $$\pphi_n\pphi_n'=(-1)^{n+1}\frac {s_{n+1}}{s_0}.$$
	In particular, for $\alpha=\sqrt D$ we have \be\label{2eq:snpell}s_{n+1}=(-1)^{n+1}\left(p_n^2-Dq_n^2\right).\ee
\end{lemC}

\begin{proof}
	By Lemma \ref{1lem:formulaspnqn}, $\ds\pphi_n\pphi_n'=\prod\limits_{i=1}^{n+1}\frac1{\alpha_i\alpha_i'}=\prod\limits_{i=1}^{n+1}\frac{s_i^2}{r_i^2-D}=\prod\limits_{i=1}^{n+1}-\frac{s_i}{s_{i-1}}=(-1)^{n+1}\frac{s_{n+1}}{s_0}$.
\end{proof}

If $\alpha=\sqrt D$, the Best Approximation Theorem (Lemma \ref{1lem:best2}) can be written in the following way:
\begin{propC}\label{2prop:bestquadr}
	Let $D\inn\S K$ with $\deg D=2d$, let $\alpha=\sqrt D$, let $p,q\inn\n K[T]$ be two relatively prime polynomials. Then $\cv{}$ is, up to the sign, a convergent of $\alpha$ if and only if \be\label{2eq:bestquadr}\deg(p^2-Dq^2)\leq d-1.\ee
\end{propC}

\begin{proof}
	Let $n\geq0$. As $\alpha_{n+1}$ is reduced, by Proposition \ref{2prop:red} $\deg s_{n+1}<d$, that is, by the previous Lemma, $\deg(p_n^2-Dq_n^2)\leq d-1$.

	Conversely, let $p,q$ be two relatively prime polynomials satisfying \eqref{2eq:bestquadr}. Assuming $|p+q\sqrt D|>|p-q\sqrt D|$, we must have $|p|=|q\sqrt D|=\mu^{d+\deg q}$, so $|p-q\sqrt D|<|q|^{-1}$. By Lemma \ref{1lem:best2} then $p/q$ is a convergent of $\sqrt D$.
\end{proof}

\begin{defnC}
	Let $D\inn\S K$. The polynomial analogue of the Pell equation for $D$ is \be\label{2eq:pell}X^2-DY^2\inn\n K^*,\ee where we look for solutions $X,Y\inn\n K[T]$.
	
	\eqref{2eq:pell} always has the trivial solutions $(x,y)=(k,0)$ with $k\inn\n K^*$. A solution $(x,y)$ of \eqref{2eq:pell} is said to be \textit{non-trivial} if $y\neq0$ (equivalently, if $x$ is not constant).
	If the Pell equation for $D$ has non-trivial solutions, $D$ is said to be a \textit{Pellian} polynomial.
\end{defnC}

\begin{lemC}\label{2lem:unit}
	If $D$ is Pellian, then \eqref{2eq:pell} has infinitely many non-trivial solutions. More precisely, those solutions form a group where, identifying a solution $(x,y)$ with the quadratic irrationality $x+y\sqrt D$, the operation is given by multiplication in $\n K(T,\sqrt D)$.	
	
	Let $(x_1,y_1)$ be a non-trivial solution with $\deg x_1,\deg y_1$ minimal. Then the set of solutions to \eqref{2eq:pell} is exactly \be\n K^*\times\left\langle x_1+y_1\sqrt D\right\rangle.\ee  
\end{lemC}	

\begin{proof}
	Pell's equation is equivalent to $N(X+Y\sqrt D)\inn\n K^*$, that is, $(x,y)$ is a solution of Pell's equation if and only if $x+y\sqrt D$ is a unit in $\n K[T,\sqrt D]$.
\end{proof}

\begin{remC}\label{2rem:degan-pell}
	By Proposition \ref{2prop:bestquadr}, if the Pell equation for $D$ has solutions, they must be, up to a multiplicative constant, continuants of $\alpha=\sqrt D$. More precisely, $D$ is Pellian if and only if there exists $n$ such that $s_{n+1}\inn\n K^*$, if and only if there exists $n$ such that $\deg a_{n+1}=d$.
	
	By the previous Lemma then either there are infinitely many partial quotients of degree $d$ or the only partial quotient of degree $d$ is $a_0$. 
\end{remC}

In \cite{zannier2016hyperelliptic}, Theorem 1.3, Zannier improved the previous classical results:
\begin{theoC}[Zannier]\label{2theo:Z2}
	Let $D\inn\S K$ be a polynomial of degree $2d$. Then we have $$\ov K(\sqrt D)\leq d/2,$$ unless $D$ is of the form $D(T)=r(T)^2\widetilde D(T)$ with $\widetilde D$ a Pellian polynomial of degree at least $\frac 32d$.
\end{theoC}

\section{Periodicity and quasi-periodicity}
\begin{defn}
	Let $\alpha$ be an irrational Laurent series. The continued fraction of $\alpha$ is said to be \textit{quasi-periodic} if there exist $n\geq0,m\geq1$ and a constant $c\inn\n K^*$ such that $$\alpha_{n+m}=c\,\alpha_n,$$ where the $\alpha_n$ are the complete quotients of $\alpha$. The smallest positive integer $m$ for which the previous equality holds is called the \textit{quasi-period} of $\alpha$.

	If there exist $N\geq0,M\geq1$ such that $$\alpha_{N+M}=\alpha_N,$$ then the continued fraction of $\alpha$ is said to be \textit{periodic} and the smallest integers $M,N$ satisfying the previous equality are called, respectively, its \textit{period} and its \textit{pre-period}. If $N=0$, $\alpha$ is said to be \textit{purely periodic}.
\end{defn}

If the continued fraction of $\alpha$ is periodic with $\alpha_{N+M}=\alpha_N$, then it is easy to see by induction that $\alpha_{N+M+k}=\alpha_{N+k}$ for every $k\geq0$. We will write $$\alpha=[a_0,\dots,a_{N-1},\ov{a_N,\dots,a_{N+M-1}}].$$ 

If the continued fraction of $\alpha$ is quasi-periodic with $\alpha_{n+m}=c\,\alpha_n$ we will write $$\alpha=\left[a_0,\dots,a_{n-1},\ov{a_n,\dots,a_{n+m-1}}^{\ c}\right].$$

\begin{lem}
	Let $\alpha$ be quasi-periodic, $\alpha=\left[a_0,\dots,a_{n-1},\ov{a_n,\dots,a_{n+m-1}}^{\ c}\right]$. Then we have \begin{equation*}\begin{split}\alpha & =\left[a_0,\dots,a_{n-1},a_n,\ov{a_{n+1},\dots,a_{n+m-1},c\,a_n}^{1/c}\right]=\\ & =\left[a_0,\dots,a_{n+m-1},\ov{c\,a_n,\dots,c^{(-1)^{m-1}}a_{n+m-1}}^{c^{(-1)^m}}\right]=\end{split}\end{equation*} In particular, $\alpha_{n+2m}=c^{(-1)^m}\alpha_{n+m}=c^{1+(-1)^m}\alpha_n$, so $$\alpha= \left[a_0,\dots,a_{n-1},\ov{a_n,\dots,a_{n+2m-1}}^{c^{1+(-1)^m}}\right].$$
\end{lem}

\begin{lem}\label{2lem:quasiper-per}
	If the continued fraction expansion of $\alpha$ is periodic with period $M$, then it is also quasi-periodic and its quasi-period $m$ divides $M$. Moreover, period and quasi-period start at the same index.
	
	Let us assume that the continued fraction of $\alpha$ is quasi-periodic; let $m,n,c$ be as above. 
\begin{enumerate}	
	\item If $m$ is odd then the continued fraction expansion of $\alpha$ is periodic and its period is $\ds M=\begin{cases}
	m &\text{ if } c=1\\ 2m &\text{otherwise}\end{cases}$. 
	
	\item If $m$ is even and $c$ is a primitive $j$-th root of unity, then the continued fraction of $\alpha$ is periodic with period $M=j\,m$.
	
	\item If $m$ is even and $c$ is not a root of unity, then the continued fraction of $\alpha$ is not periodic.
\end{enumerate}
	In particular, if $\n K$ is a finite field (or an algebraic extension of a finite field), a continued fraction is periodic if and only if it is quasi-periodic.
\end{lem}
\begin{proof}
	It follows from the previous Lemma. For instance, if $m$ is odd, then we have $\alpha_{n+2m}=\alpha_{n}$, which implies 1.
\end{proof}

Periodic and quasi-periodic continued fractions can be studied using the matrix formalism of \ref{1notat:matrix}.

\begin{prop}\label{2prop:permatr}
	Let $\alpha=[a_0,a_1,\dots]$ be an irrational Laurent series. Then the following are equivalent:
	\begin{enumerate}
		\item the continued fraction expansion of $\alpha$ is quasi-periodic;
		\item there exists $M\inn\GL_2(\n K[T])$, not a multiple of the identity, such that $\alpha=M\alpha$;
		\item $\alpha$ is a quadratic irrationality satisfying $P\alpha^2+Q\alpha+R=0$ with $P,Q,R$ relatively prime polynomials such that $D=Q^2-4PR$ is Pellian. 
	\end{enumerate}
\end{prop}

\begin{proof}
	$1.\imp2.$  Let $\alpha_{n+m}=c\,\alpha_n$ with $m>0$ and $c\inn\n K^*$. For every $l$, let $M_l\!=\!M_{(a_0,\dots,a_{l-1})}$. Then in particular we will have $\alpha=M_{n+m}\alpha_{n+m}=M_{n+m}\mm c001M_n^{-1}\alpha$, where of course $M_{n+m}\mm c001M_n^{-1}\inn\GL_2(\n K[T])$ and is not a multiple of the identity matrix.

	$2.\imp3.$ Let $\alpha=M\alpha$ with $M=\mm p{p'}q{q'}\inn\GL_2(\n K[T])$ not a multiple of the identity. Then $\alpha$ is a quadratic irrationality satisfying $P\alpha^2+Q\alpha+R=0$, where $q=Py,\ q'-p=Qy,$ $-p'=Ry$, with $y=\gcd(q,q'-p,p')$. Then $y^2D=(q'+p)^2-4l$, where $l=\det M=pq'-p'q\inn\n K^*$, which gives a non-trivial solution to the Pell equation for $D$.

	$3.\imp2.$ We have $\Tr(\alpha)=-Q/P$ and $\ N(\alpha)=R/P$. Let $(x,y)$ be a non-trivial solution of Pell's equation for $D$ and let $M=\mm p{p'}q{q'}$ with $p\!=x-Qy,\ q\!=2Py, \ p'\!\!=-N(\alpha)q,$ $q'\!=p-\Tr(\alpha) q\inn\n K[T]$. Then $\det M=x^2-Q^2y^2+4PRy^2=x^2-Dy^2\inn\n K^*$, $M$ is not a multiple of the identity and $\ds M\alpha=\frac{p\alpha-N(\alpha) q}{q\alpha+p-\Tr(\alpha) q}=\alpha$.

	$2.\imp1.$ Let $\alpha=M\alpha$ with $M\inn\GL_2(\n K[T])$ not a multiple of the identity matrix. By Serret's Theorem \ref{1theo:serret} there exist $m,n\inn\n N, c\inn\n K^*$ such that $\alpha_{m+n}=c\,\alpha_n$. As $M$ is not a multiple of the identity, $m$ cannot be zero, so the continued fraction of $\alpha$ is quasi-periodic.
\end{proof}

\begin{rem}
	Let $D\inn\S K$ be a squarefree polynomial. If there exists an element $\alpha$ of $\n K(T,\sqrt D)\setminus\n K(T)$ whose continued fraction is quasi-periodic, then $D$ is Pellian, that is, also the continued fraction of $\sqrt D$ is quasi-periodic.
	
	Conversely, let us assume that $D$ is Pellian. Let $\alpha=\frac{A+B\sqrt D}C$ with $A,B,C$ relatively prime polynomials and $B,C\neq0$. Then the continued fraction of $\alpha$ is quasi-periodic if and only if the Pell equation for $D$ has non-trivial solutions $(x,y)$ such that $\frac{BC}{(C,A^2-B^2-D)}$ divides $y$.
\end{rem}

\begin{lem}\label{2lem:eigen}
	If $\alpha=M\alpha$, then $\alpha\inn\n K(T,\sqrt D)$ where $D$ is the discriminant of $M$, $$D=(\tr\, M)^2-4\det M.$$ Similarly, if $\alpha=[a_0,\dots,a_{n-1},\ov{a_n,\dots,a_{n+m}}^c]$, then $\alpha\inn\n K(T,\sqrt D)$ where $D$ is the discriminant of $M_{(a_n,\dots,a_{n+m})}$, that is, $$D=\left(C_{m+1}(a_n,\dots,a_{n+m})+C_{m-1}(a_{n+1},\dots,a_{n+m-1})\right)^2+4(-1)^m.$$
	
	Moreover, $\alpha=M\alpha$ if and only if $\vv{\alpha}{1}$ is an eigenvector of $M$.
\end{lem}
\begin{proof}
	The first statement follows from the proof of the previous Proposition and the second one from the definition of the action of $\GL_2(\n K[T])$ over $\n L$
\end{proof}

\begin{rem}\label{2rem:eigen}
	If the continued fraction expansion of $\alpha$ is purely periodic, $\alpha\!=\![\ov{a_0,\dots,a_n}]$, then $\alpha=M\alpha$ for  $M=M_{(a_0,\dots,a_n)}$; we will say that $M$ is the matrix \textit{canonically associated} to $\alpha$. 
	
	Conversely, let $\alpha=M\alpha$ with $M\inn\GL_2(\n K[T])$ not a multiple of the identity. By Remark \ref{1rem:decompmatr}, there exist polynomials $a_0,\dots,a_n$, with $\deg a_1,\dots,\deg a_{n-1}>0$, such that $M=M_{(a_0,\dots,a_n)}\mm {l_1}00{l_2}$ with $l_1l_2=\det M\inn\n K$. Then a (possibly non regular) continued fraction expansion for either $\alpha$ or $\alpha'$ is $[\ov{a_0,\dots,a_n}^{l_1/l_2}]$.
\end{rem}

	Now, when they are quasi-periodic, the continued fractions of $\alpha$ and $\alpha'$ are strictly related. The following Lemma is the polynomial version of a result proved by Galois \cite{galois1828analyse} for real quadratic irrational numbers having a purely periodic continued fraction expansion. 
\begin{lem}
	If $\alpha$ is a quadratic irrationality with purely quasi-periodic continued fraction, $\alpha=[\ov{a_0,\dots,a_n}^c]$, then \be\label{2eq:contfracalpha'} \alpha'=\left[0,\ov{-c\,a_n,-c^{-1}a_{n-1},\dots,-c^{(-1)^n}a_0}^{c^{(-1)^{n+1}}}\right].\ee
\end{lem}

\begin{proof}
	Certainly $\alpha=M\mm c001\alpha$, where $M\!=\!M_{(a_0,\dots,a_n)}$, that is, $\alpha=\mm {c^{-1}}001M^{-1}\alpha=$ $=M_{(0,-ca_n,\dots,-c^{(-1)^n}a_0,0)}c^{(-1)^n}\alpha$, so the continued fraction expansion of either $\alpha$ or $\alpha'$ is $\left[\ov{0,-c\,a_n,\dots,-c^{(-1)^n}a_0,0}^{c^{(-1)^n}}\right]=\left[0,\ov{-c\,a_n,\dots,-c^{(-1)^n}a_0}^{c^{(-1)^{n+1}}}\right]$. As this is obviously different from $\alpha$, it must be the continued fraction expansion of $\alpha'$.
\end{proof}

\begin{rem}\label{2rem:contfracalpha'}
	Let $\alpha$ be a quadratic irrationality with quasi-periodic continued fraction, $\alpha=[a_0,\dots,a_{n-1},\ov{a_n,\dots a_{n+m}}^c]$. Then $\alpha'=[a_0,\dots,a_{n-1},(\alpha_n)']$, so, by the previous Lemma, $$\alpha'=\left[a_0,\dots,a_{n-1}-c\,a_{n+m},\ov{-c^{-1}a_{n+m-1},\dots,-c^\epsilon a_n,-c\,c^{-\epsilon}a_{n+m}}^{c^\epsilon}\right],$$ where $\epsilon=(-1)^m$.
	
	Let us assume $n$ to be minimal, that is, let $c^{-1}a_{n-1}\neq a_{n+m}$.
	\begin{itemize}
		\item If $n=0$ (that is, if $\alpha$ is purely quasi-periodic) then $\ord(\alpha)<0$ and $\ord(\alpha')>0$ (in particular, $\ord(\alpha-\alpha')=\ord(\alpha)<0$).
		\item If $n=1$, then $\alpha'=\left[a_0-c\,a_{m+1},\ov{-c^{-1}a_m,\dots,-c^\epsilon a_1,-c\,c^{-\epsilon}a_{m+1}}^{c^\epsilon}\right]$, so $\ord(\alpha')\leq\!0$ and $\ord(\alpha-\alpha')=-\deg a_{m+1}<0$.
		\item If $n=2$, then $\alpha'=\left[a_0,a_1-c\,a_{m+2},\ov{-c^{-1}a_{m+1},\dots,-c^\epsilon a_2,-c\,c^\epsilon a_{m+2}}^{c^\epsilon}\right]$. Then $\ord(\alpha)=\ord(\alpha')$ and $\ord(\alpha-\alpha')\geq0$; more precisely, $\ord(\alpha-\alpha')=0$ if and only if $\deg(a_1-c\,a_{m+2})=0$.
		\item If $n>2$, then $\floor\alpha=\floor{\alpha'}$, so $\ord(\alpha-\alpha')>0$.
	\end{itemize}
	This also implies the following Lemma:
\end{rem}

\begin{lem}\label{2lem:galois}
	Let $\alpha$ be a quadratic irrationality with quasi-periodic continued fraction $\alpha=\left[a_0,\dots,a_{n-1},\ov{a_n,\dots,a_{n+m}}^c\right]$, where we assume $n$ minimal. Then:
	\begin{enumerate}
		\item $\ord(\alpha-\alpha')>0$ if and only if $n\geq3$ or $n=2$ and $\deg(a_1-c\,a_{m+2})>0$;
		\item $\ord(\alpha-\alpha')=0$ if and only if $n=2$ and $\deg(a_1-c\,a_{m+2})=0$;
		\item $\ord(\alpha-\alpha')<0$ and $\ord(\alpha')\leq0$ if and only if $n=1$;
		\item $\ord(\alpha-\alpha')<0$ and $\ord(\alpha')>0$ if and only if $n=0$.
	\end{enumerate}
\end{lem}

In particular, \textit{4.} is the polynomial analogue of the classical Galois' Theorem, that is, assuming the continued fraction of $\alpha$ to be quasi-periodic, then $\alpha$ is reduced if and only if its continued fraction is purely quasi-periodic.

\begin{lem}
	In the notations of \ref{1notat:K,ovK}, we also have $$K(\alpha)=K(\alpha')$$ for every quadratic irrationality $\alpha$ with quasi-periodic continued fraction expansion and, even more so, $\ov K(\alpha)=\ov K(\alpha')$. 
\end{lem}

\begin{rem}
	Let $\alpha\inn\n L$ be quasi-periodic, $\alpha=\left[a_0,\dots,a_n,\ov{b_1,\dots,b_m}^c\right]$, and let us assume that $\Tr(\alpha)\inn\n K[T]$, that is, $\{\alpha\}=-\{\alpha'\}$.

	Possibly doubling the quasi-period, we can assume $m$ to be even, then by Remark \ref{2rem:contfracalpha'}, $\alpha'=\left[a_0,\dots,a_n-c\,b_m,\ov{-c^{-1}b_{m-1},\dots,-c^{-1}b_1,-c^2b_m}^{c^{-1}}\right]$. Now, $\{\alpha+\alpha'\}=0$ implies that either $\alpha$ is purely quasi-periodic or $n=0$ or $n=1$ and $2a_1=c\,b_m$. Moreover, we must have $c=c^{-1}$, so $c=\pm1$ and, again, possibly doubling the period, we can assume $c=1$. Thus, the continued fraction of $\alpha$ is in fact periodic.\par\medskip

	If $\alpha$ is purely periodic, then $\left[\,\ov{b_2,\dots,b_m,b_1}\,\right]=-\left[\,\ov{-b_m,\dots,-b_1}\,\right]$, that is, $b_i=b_{m-i+2}$ for $i=2,\dots,m$; in this case, $\Tr(\alpha)=b_1$.

	If $n=0$ then $\left[\,\ov{b_1,\dots,b_m}\,\right]=-\left[\,\ov{-b_{m-1},\dots,-b_1,-b_m}\,\right]$, that is, $b_i=b_{m-i}$ for $i=1,\dots,m-1$; in this case, $\Tr(\alpha)=2a_0-b_m$.

	If $n=1$ then $\left[a_1,\ov{b_1,\dots,b_m}\,\right]=-\left[a_1-b_m,\ov{-b_{m-1},\dots,-b_1,-b_m}\,\right]$, that is, $b_m=2a_1$ and $b_i=b_{m-i}$ for $i=1,\dots,m-1$; in this case, $\Tr(\alpha)=2a_0$.
	
	Certainly, the converse holds too, that is, $\alpha$ has polynomial trace if and only if its continued fraction expansion satisfies one of the previous conditions.
	
	In particular, $\Tr(\alpha)=0$, that is, $\alpha=\frac BC\sqrt D$, if and only if the continued fraction of either $\alpha$ or $1/\alpha$ is of the form $\left[a_0,\ov{b_1,b_2,\dots,b_2,b_1,2a_0}\,\right]$.  
\end{rem}

\begin{notat}
	Let $a_1,\dots,a_{m-1}\inn\n K[T]$ and let $c$ be a non-zero constant. The sequence $a_1,\dots,a_{m-1}$ is said to be \textit{skew-symmetric} of skew $c$ if $a_{m-i}=c^{(-1)^i}a_i$ for every $i$.
\end{notat}
In particular, if $m$ is even, then $a_1,\dots,a_{m-1}$ is skew-symmetric if and only if it is symmetric.

\begin{theo}\label{2theo:sqrtD}
	Let $D\inn\S K$ and let $\alpha=\frac BC\sqrt D$, with $B,C$ non-zero polynomials such that $\ord\left(\alpha\right)<0$. If the continued fraction expansion of $\alpha$ is quasi-periodic, $\alpha=\left[a_0,\ov{a_1,\dots,a_{m-1},2c\,a_0}^{c^{-1}}\,\right]$, then $a_1,\dots,a_{m-1}$ is skew-symmetric of skew $c$, so the continued fraction of $\alpha$ is periodic and it has the form $$\frac BC\sqrt D=\left[a_0,\ov{a_1,\dots,a_{m-1},2c\,a_0,a_{m-1},\dots,a_1,2a_0}\,\right].$$ Moreover, $\alpha_i=-\left(c^{(-1)^i}\alpha_{m-i+1}'\right)^{-1},\ r_i=r_{m-i+1}$ and $s_i=c^{(-1)^i}s_{m-i}$  for $i=1,\dots,m$.
	
	In particular, if $\alpha=\sqrt D$, then $s_{j\, m}=\begin{cases}1 &\text{ if } j \text{ is even}\\c^{-1}&\text{ if } j \text{ is odd}\end{cases}$. 
\end{theo}

\begin{proof}[Sketch of proof]
	The first part of the Theorem follows immediately from the case $n=0$ of the previous Remark, while the second part can be easily proved using \eqref{2eq:contfracalpha'}, \eqref{1eq:multAform}, \eqref{1eq:1/alpha}.
\end{proof}

Actually, Friesen (Theorem 1.1 in \cite{friesen1992continued}$\!$) showed that for any choice of a skew-symmetric sequence of polynomials $a_1,\dots,a_{m-1}$ there exist infinitely many polynomials $D\inn\S K$ such that the continued fraction of $\sqrt D$ has the previous form\footnote{Friesen proved this result for monic polynomials over finite fields but his method works in any case.}.\par\medskip

As in the classical case, the periodicity of the continued fraction of $\sqrt D$ is immediately connected with the existence of non-trivial solutions of the Pell equation for $D$: 
\begin{prop}\label{2prop:eqper}
	Let $\sqrt D=[a_0,a_1,\dots]$ with $D\inn\S K$. Then, the following are equivalent: 
	\begin{enumerate}
		\item the continued fraction of $\sqrt D$ is quasi-periodic;
		\item the continued fraction of $\sqrt D$ is periodic;
		\item there exists $i>0$ such that $\deg a_i=d$ (equivalently, such that $s_i\inn\n K^*$);
		\item $\deg a_i=d$ (equivalently, $s_i\inn\n K^*$) for infinitely many $i$;
		\item $D$ is Pellian; 
		\item $\n K[T,\sqrt D]$ has non-constant units;
		\item the continued fraction of $\frac{r+\sqrt D}s$ is quasi-periodic for every $r,s\inn\n K[T]$ such that $s$ divides $r^2-D$.
	\end{enumerate}
	
	In that case, if $m$ is the quasi-period of $\sqrt D$, then $$\deg a_i=d \text{ if and only if } s_i\inn\n K^*,\text{ if and only if } m|i$$
	and all the solutions to the Pell equation for $D$ are given by $$(x,y)=k(p_{j\,m-1},q_{j\,m-1}),\text{ with } k\inn\n K^*, j\inn\n N.$$ 
\end{prop}

\begin{proof}
	We have already seen in Theorem \ref{2theo:sqrtD} that if the continued fraction of $\sqrt D$ is quasi-periodic with quasi-period $m$, then it is periodic and $s_{j\, m}\inn\n K^*$ for every $j$.

	On the other hand, if $s_i\inn\n K^*$ with $i>0$, then $\alpha_i=s_i^{-1}\sqrt D+s_i^{-1}r_i$, so $\alpha_{i+1}=s_i\alpha_1$. Then $\sqrt D$ is quasi-periodic and $i$ is a multiple of the quasi-period $m$.
	
	The connection with the existence of solutions to the Pell equation for $D$ and the equivalence of $5.$, $6.$ and $7.$ follow directly from Lemma \ref{2lem:unit}, Remark \ref{2rem:degan-pell} and Proposition \ref{2prop:permatr}.
\end{proof}

\begin{lem}
	Let $\alpha=\sqrt D$ have a quasi-periodic continued fraction expansion, let $m$ be its quasi-period, with $a_m=2c\,a_0$, and let $n$ be its period. Then the solutions to the Pell equation for $D$ are given, up to multiplicative constants, by the following: 
	\begin{enumerate}
		\item if $n=m$ is even, then $c=1$ and $p_{j\,m-1}^2-Dq_{j\,m-1}^2=1$ for every $j\inn\n N$;
		\item if $n=m$ is odd, then $c=1$ and $p_{j\,m-1}^2-Dq_{j\,m-1}^2=(-1)^j$ for every $j\inn\n N$;
		\item if $n=2m$ with $m$ odd, then $p_{j\,m-1}^2-Dq_{j\,m-1}^2=\begin{cases}1&\text { if } j \text{ is even}\\-c^{-1}&\text{ if } j \text{ is odd}\end{cases}$.
	\end{enumerate}
\end{lem}
\begin{proof}
	It follows from \eqref{2eq:snpell} and from the last formula of Proposition \ref{2prop:eqper}.
\end{proof}	

\begin{rem} 
	Let $D\inn\S K$ and let $\widetilde{\n K}$ be an extension of $\n K$. As the continued fraction expansion does not change for extensions of the base field, the sets of the \textit{monic} solutions to the Pell equation for $D$ in $\n K[T]$ and in $\widetilde{\n K}[T]$ coincide.
\end{rem}

When $\n K$ is a finite field, the theory of polynomial quadratic continued fractions is completely analogue to the classical one; in particular, the equivalent conditions of Theorem \ref{2theo:sqrtD} are always satisfied and the following analogue of Lagrange's Theorem holds:
\begin{theo}[polynomial analogue of Lagrange's Theorem]\label{2theo:lagrange}
	Let $\alpha\inn\n K((T^{-1}))$ with $\n K=\n F_q$ a \text{finite} field. Then the following statements are equivalent:
	\begin{enumerate}
		\item the continued fraction expansion of $\alpha$ is periodic
		\item the continued fraction expansion of $\alpha$ is quasi-periodic
		\item $\alpha$ is a quadratic irrationality
	\end{enumerate}
\end{theo}

\begin{proof}
	We have already seen in Lemma \ref{2lem:quasiper-per} that if $\n K$ is a finite field, $1.$ is equivalent to $2.$ and, by Proposition \ref{2prop:permatr}, $2.$ always implies $3.$ Let us prove that $3.$ implies $1.$
	
	If $\alpha$ is a quadratic irrationality, $\alpha=\frac{r+\sqrt D}s$ with $\deg D=d$, by Proposition \ref{2prop:red} there exists $N$ such that $\alpha_n$ is reduced for every $n\geq N$. Then $\deg r_n=d$ and $\deg s_n<d$ for every $n\geq N$ (and the coefficients of the $r_n$ of degrees $d,d-1$ coincide with those of $\sqrt D$). As $\n K$ is finite, there are only finitely many possibilities for the $r_n,s_n$, so there exist $n\geq N, m>0$ such that $r_n=r_{n+m}$ and $s_n=s_{n+m}$, that is, such that $\alpha_n=\alpha_{n+m}$.
\end{proof}

\begin{cor}
	If $\alpha=\frac{r+\sqrt D}s\inn\n F_q$, with $\deg D=2d$, then the length of the period of $\alpha$ is at most $q^{2d-1}$. 
\end{cor}
\begin{proof}
	In the notations of the proof of the previous Theorem we have that, for $n$ large enough, there are at most $q^{d-1}$ possible choices for the coefficients of $r_n$ and at most $q^d$ possible choices for the coefficients of $s_n$. Thus, the period of $\alpha$ has length at most $q^{2d-1}$. 
\end{proof}

\begin{lem}
	Theorem \ref{2theo:lagrange} holds also over any algebraic extension of a finite field. 
\end{lem}
\begin{proof}	
	If $\n F'$ is an algebraic extension of a finite field $\n F_q$ and if $\alpha$ is a quadratic irrationality over $\n F'$, there exists a finite subfield $\n F$ of $\ch F_q$ such that $\alpha$ is a quadratic irrationality over $\n F$: if $\alpha=\frac{r+\sqrt D}s$, it is enough to consider the field generated over $\n F_q$ by the coefficients of $r,s,D$ and by a square root of the leading coefficient of $D$. As the continued fraction of a Laurent series does not change when extending the base field, by Lagrange's Theorem the continued fraction expansion of $\alpha$ is periodic.
\end{proof}

\begin{cor}\label{2cor:lagrange}
	Let $\n K$ be an algebraic extension of a finite field. Then any polynomial $D\inn\S K$ is Pellian. Equivalently, for every $D\inn\S K$, $$K(\sqrt D)=\ov K(\sqrt D)=d.$$
	
	In particular, for every non-zero polynomial $P$ there exist infinitely many multiples $j\,m$ of the quasi-period $m$ of $\sqrt D$ such that $P$ divides $q_{j\,m-1}$.
\end{cor}
\begin{proof}
	The first statement follows immediately from Theorem \ref{2theo:lagrange}, Proposition \ref{2prop:eqper} and Remark \ref{2rem:sqrtDred}.
	
	Let $P\neq0\inn\n K[T]$. Then $P^2D$ is a Pellian polynomial, so there exist infinitely many pairs $(x,y)$ such that $x^2-y^2P^2D\inn\n K^*$. Then, the pairs $(x,Py)$ are solutions to the Pell equation for $D$ of the desired form, which, by the last statement of Proposition \ref{2prop:eqper}, implies the second claim.
\end{proof}

\begin{rem}\label{2rem:Z1}
	In the notations of Theorem \ref{2theo:lagrange}, it is always true that $1.$ implies $2.$ and $2.$ implies $3.$ However, if $\n K$ is not an algebraic extension of a finite field, there may exist quadratic irrationalities whose continued fraction is not quasi-periodic and there may exist quasi-periodic continued fractions that are not periodic.
	
	Indeed, as soon as there exists $c\inn\n K^*$ which is not a root of unity, by Lemma \ref{2lem:quasiper-per} any quasi-periodic continued fraction of the form $\alpha=\left[\ov{a_1,\dots,a_m}^c\right]$ with $m$ even will not be periodic. 
	
	We will see in Corollary \ref{6cor:yureductioncondition} an algebro-geometric condition for Pellianity, based on the reduction theory of abelian varieties; this will imply that there exist non-Pellian polynomials $D\inn\S C$ (or $\S Q$), so the continued fraction of $\sqrt D$ will not even be quasi-periodic (for example, $D=T^4+T+1$ is non-Pellian, see Example \ref{6ex:nonPell}). 
\end{rem}

If $\deg D=2$ then the continued fraction of $\sqrt D$ is always periodic, of the form $\sqrt D=\left[a_0,\ov{2c\,a_0,2a_0}\,\right]$, $D$ is always Pellian and $K(\sqrt D)=1$.

If $\deg D=4$ then either $D$ is Pellian (and the continued fraction of $\sqrt D$ is periodic) or all the partial quotients of $\sqrt D$, apart from $a_0$, are linear, that is, $K(\sqrt D)=1$.

It is possible to prove that if $\deg D=6$, then either $D$ is Pellian or only finitely many of the partial quotients of $\sqrt D$ can have degree 2, that is, either $\ov K(\sqrt D)=3$ or $\ov K(\sqrt D)=1$.

If $\deg D\geq8$ different cases can occur; for example, $D$ can be Pellian or $D$ could be of the form $D=F(T^2)$ with $\deg F=4$ and in this case, by Lemma \ref{1lem:subst}, either $D$ is Pellian or all the partial quotients of $\sqrt D$ apart from the first one have degree 2. If $D$ is non-Pellian, $\sqrt D$ can have at most finitely many partial quotients of degree $3$ but in \cite{contfracexamples} we show that there exist non-Pellian polynomials $D\inn\n C[T]$ of degree $8$ with infinitely many partial quotients of degree $1$ and infinitely many partial quotients of degree $2$.\par\medskip

A generic polynomial $D\inn\S C$ will not be Pellian and the continued fraction of $\sqrt D$ will not be periodic. However, Zannier (\cite{zannier2016hyperelliptic}, Theorem 1.1) proved that for every polynomial $D\inn\S C$, at least the sequence of the \textit{degrees} of the partial quotients of $\sqrt{D(T)}$ is eventually periodic. 

\section{Algebro-geometric approach}
As it was already known to Abel and Chebychev, the continued fraction expansion of quadratic irrationalities of the form $\frac{A+B\sqrt D}C$, where $D$ is a squarefree polynomial of even degree, is linked to algebro-geometric properties of the hyperelliptic curve $\C H$ of affine model $U^2=D(T)$. In particular, $D$ is Pellian if and only if $[(\infty_-)-(\infty_+)]$ is a torsion point on the Jacobian of $\C H$, where $\infty_\pm$ are the points at infinity on $\C H$.

Afterwards,this subject has been  greatly developed, among others, by Adams and Razar \cite{adams1980multiples}, Berry \cite{berry1990periodicity}, \cite{berry1998construction}, Platonov \cite{platonov2014number}, Van der Poorten and Tran \cite{poorten2000quasi}, \cite{van2002periodic}, Zannier \cite{zannier2014unlikely}, \cite{zannier2016hyperelliptic}.\par\medskip

Before showing this connection, we will recall some basic results on algebraic curves; more details and proofs can be found in most books introducing algebraic geometry, such as \cite{lang1972introduction}, \cite{serre1988algebraic} or \cite{silverman2009arithmetic}.

\subsection{Reminder of basic results on curves}
Let $\n K$ be an algebraically closed field and let $\C C$ be a curve defined over $\n K$. We will denote by $\n K[\C C]$ the \textit{ring of coordinates} of $\C C$ and by $\n K(\C C)$ its \textit{function field}. For any subfield $\n K_1$ of $\n K$ we will denote by $\C C(\n K_1)$ the set of points of $\C C$ defined over $\n K_1$; we will also write $\C C(\n K)=\C C$.\par\medskip

Let $P$ be a point of $\C C$; we will denote by $O_P$ the \textit{local ring} of $\C C$ at $P$, that is, the ring of rational functions regular at $P$: $$O_P=\left\{f/g\inn\n K(\C C),\ g(P)\neq0\right\}.$$ $O_P$ is a local ring with maximal ideal $M_P=\left\{f/g\inn O_P,\ f(P)=0\right\}$. $P$ is said to be \textit{smooth} if $ \dim_{\n K} M_P/M_P^2=1$ and $\C C$ is said to be a \textit{smooth} curve if all its points are smooth.

If $P$ is a smooth point of $\C C$, then $O_P$ is a discrete valuation ring and $M_P$ is principal; a generator $t$ of $M_P$ is called a \textit{uniformizer} for $\C C$ at $P$. The natural valuation of $O_P$ is denoted by $\ord_P$: $\ord_P(f)=n$ if and only if $f=t^ng$ with $g\inn O_P\setminus M_P$. Of course, this valuation can be extended to $\n K(\C C)$ and $P$ is said to be a zero (respectively, a pole) of a rational function $f$ if $\ord_P(f)>0$ (respectively, $\ord_P(f)<0$). Any non-zero rational function $f$ has finitely many zeros and finitely many poles and the numbers of its zeros and poles (counted with multiplicity) coincide; this number is called the \textit{degree} of $f$ and it is denoted by $\deg f$. From now on, if $f\inn\n K[T]$, we will denote by $\deg_T f$ its degree as a polynomial. 

The local rings $O_P$ form a subsheaf $O$ of the constant sheaf $\n K(\C C)$.\par\medskip

From now on we will assume, unless stated otherwise, that $\C C$ is a smooth curve.

We will denote by $\Div(\C C)$ the \textit{group of divisors} of $\C C$, that is, the free abelian group generated by the points of $\C C$. Let $A=\sum_{P\in\C C}a_P(P)\inn\Div(\C C)$. Then $A$ is said to be \textit{positive}, or \textit{effective} ($A\geq0$), if $a_P\geq0$ for every $P$ and $A$ is said to be \textit{prime} to a point $Q$ if $a_Q=0$. The \textit{degree} of $A$ is defined as $\deg A=\sum_P a_P$; in particular, the divisors of degree 0 form a subgroup of $\Div(\C C)$, which will be denoted by $\Div^0(\C C)$.\par\medskip

Let $f$ be a non-zero rational function. Its divisor is defined as $$\div(f)=\sum_{P\in\C C}\ord_P(f)(P)$$ and such a divisor is said to be \textit{principal}. Principal divisors form a subgroup of $\Div(\C C)$, that will be denoted by $\C P(\C C)$. The quotient group $$\Pic(\C C)=\Div(\C C)/\C P(\C C),$$ is called the \textit{Picard group}, or \textit{divisor class group}, of $\C C$; we will denote by $[A]$ the class in $\Pic(\C C)$ of a divisor $A$. Two divisors $A_1,A_2$ are said to be \textit{linearly equivalent} if $[A_1]=[A_2]$, that is, if there exists a rational function $f$ such that $A_1=A_2+\div(f)$; in this case we will write $A_1\sim A_2$.\par\medskip 

Given a point $P\inn\C C$ and a divisor $A$, let $L(A)_P$ be the set of rational functions $f$ such that $\ord_P(f)+\ord_P(A)\geq0$; the $L(A)_P$ form a subsheaf $L(A)$ of the constant sheaf $\n K(\C C)$. Let $\C L(A)=H^0(\C C,L(A))$, that is, $$\C L(A)=\left\{f\inn\n K(\C C)^*,\ \div(f)+A\geq0\right\}\cup\{0\}.$$ It can be seen that $\C L(A)$ is a finite-dimensional $\n K$-vector space, let $\ell(A)=\dim_{\n K}\C L(A)$.

Let $\C I(A)=H^1(\C C,L(A))$; $\C I(A)$ is still a finite-dimensional $\n K$-vector space, let $i(A)$ be its dimension. 

\begin{theo}[Riemann-Roch]
	Let $\C C$ be a smooth curve. Then there exists an integer $g$, called the \textup{genus} of $\C C$, such that for every divisor $A\inn\Div(\C C)$ \be	\ell(A)-i(A)=\deg (A)-g+1.\ee
\end{theo}

\begin{cor}\label{6cor:RRcor}
	Let $P_\infty$ be a point of $\C C$. Then for every divisor $A$ of degree 0 there exists a divisor $B$ of degree $g$ such that \be\label{6eq:RRcor}A\sim B-g(P_\infty).\ee
	
	Moreover, there exists $n\!\leq\! g$ such that there exists a unique $B\!\geq\!0$ with $A\sim B-n(P_\infty)$. 
\end{cor}

As $\deg(\div(f))=0$ for every rational function $f$, the group of principal divisors $\C P(\C C)$ is actually a subgroup of $\Div^0(\C C)$, so we can consider $\Pic^0(\C C)=\Div^0(\C C)/\C P(\C C)$. 
$\Pic^0(\C C)$ can be given a structure of abelian variety over $\n K$, that is, a structure of projective variety over $\n K$ such that the multiplication and the inverse maps are morphisms of varieties. As an abelian variety, $\Pic^0(\C C)$ has dimension $g$; it is called the \textit{Jacobian variety} of $\C C$ and denoted by $\C J$. 

\begin{notat}
	For every fixed point $P_\infty$ we can consider the map from the curve $\C C$ into its Jacobian $\C J$ given by 
	$$j_{P_\infty}(P)=[(P)-(P_\infty)];$$ this map is an embedding as soon as $g\geq1$.	
	
	For $i\geq1$, let $$\C W_i^{P_\infty}=\{j_{P_\infty}(P_1)+\cdots+j_{P_\infty}(P_i),\ P_1,\cdots,P_i\inn\C C\},$$ that is, $\C W_i^{P_\infty}=\{[(P_1)+\cdots+(P_i)-i(P_\infty)],\ P_1,\dots,P_i\inn\C C\}.$ Now, $$\C C\simeq\C W_1^{P_\infty}\sub\C W_2^{P_\infty}\sub\cdots\sub\C W_g^{P_\infty}=\C J$$ for any choice of the point $P_\infty$. The $\C W_i^{P_\infty}$ are closed subvarieties of $\C J$ (even if in general they are not algebraic groups) and $$\dim\C W_i^{P_\infty}=i \text{ for } i=1,\dots,g.$$
\end{notat} 

\subsection{Quadratic continued fractions and hyperelliptic curves}\label{6sec:hyper}
Let $\n K$ be a field of characteristic different from 2. Let $D\inn\n K[T]$ be a  polynomial of even degree $\deg_T D=2d\geq4$ whose leading coefficient is a square in $\n K$ and which is not a square in $\n K[T]$, that is, in the previous notations, let $D\inn\S K$. 

Assuming additionally that $D$ is \textit{squarefree}, we can consider the hyperelliptic curve $\C H_D$ of affine model $$\C H_D: U^2=D(T)$$ (the projective closure of this affine curve could be singular, we will always consider a desingularization). It can be seen that $\C H_D$ has genus $$g=d-1.$$

Its function field is $$\n K(\C H_D)=\n K(T,U)=\left\{\frac{a(T)+b(T)\sqrt{D(T)}}{c(T)},\ a,b,c\inn\n K[T], c\neq0\right\}.$$ 

$\C H_D$ has two points at infinity, defined over $\n K$. We will denote them by $\infty_+,\ \infty_-$, choosing their signs so that $$\ord_{\infty_+}\left(\sqrt D-\floor{\sqrt D}\right)>0 \text{ and } \ord_{\infty_-}\left(\sqrt D-\floor{\sqrt D}\right)<0.$$

Embedding $\C H_D$ in its Jacobian $\C J_D$ through the map $j:P\mapsto[(P)-(\infty_+)]$, by Corollary \ref{6cor:RRcor} any point of $\C J_D$ can be written as the sum of at most $d-1$ points in $j(\C H_D)$. We will see that the convergents of $\sqrt D$ and the degrees of its partial quotients can be recovered from the writing of the multiples of $\delta=[(\infty_-)-(\infty_+)]\inn\C J_D$ as minimal sums of points in the image of $\C H_D$. In particular, the continued fraction of $\sqrt D$ is periodic if and only if $\delta$ is a torsion point of $\C J_D$.

\begin{notat}
	$\n K(\C H_D)$ is a quadratic extension of $\n K(T)$; we will denote its non-trivial involution by $f\mapsto f'$: if $f=\frac{a+b\sqrt D}c$, then $f'=\frac{a-b\sqrt D}c$. If $P=(t,u)$ is an affine point of $\C H_D$ we will denote by $P'=(t,-u)$ its conjugate point (where $\infty_\pm'=\infty_\mp$).
\end{notat}

\begin{lem}
	Let $f=\frac {a+b\sqrt D}c$ be a rational function on $\C H_D$, with $a,b,c$ relatively prime polynomials. If $b=0$, that is, if $f=\frac ac\inn\n K(T)$, then $$\deg \frac ac=2\,|\deg_T a-\deg_T c|.$$ 
	
	On the other hand, if $b\neq0$, then $$\deg f\geq d+\deg_T b.$$ 
\end{lem}
\begin{proof}
	The first statement is obvious.
	
	Let then $b\neq0$. It follows from the fact that $a,b,c$ are relatively prime that the zeros of $c$ give rise to at least $\deg_T c$ poles of $f$ (counted with multiplicity). Now, if $\deg_T c\geq d+\deg_T b$ we immediately have $\deg f\geq \deg_T c\geq d+\deg_T b$. On the other hand, assuming that $\deg_T c<d+\deg_T b$, at least one between $\infty_+,\infty_-$ is a pole of $f$ of multiplicity at least $\deg_T b+d-\deg_T c$, giving again that $\deg f\geq d+\deg_T b$.
\end{proof}

\begin{notat}
	We will denote by $\delta$ the degree-zero divisor $$\delta=(\infty_-)-(\infty_+).$$ We will still denote by $\delta$ its class in the Jacobian: $\delta=[\delta]\inn\C J_D$.
\end{notat}

\begin{rem}
	We can think of a rational function $\alpha=\ds\frac{a+b\sqrt D}c\inn\n K(\C H_D)$ as a formal Laurent series in $\n K(T,\sqrt D)\sub\n K((T^{-1}))$; let $\alpha=\sum\limits_{\mathclap{i=-\infty}}^m c_nT^n$ be its expansion, with $m=-\ord(\alpha)$. By our choice of the signs of $\infty_\pm$ we have $$\ord_{\infty_+}(\alpha)=\ord(\alpha) \text{ and } \ord_{\infty_-}(\alpha)=\ord(\alpha').$$
\end{rem}

\begin{rem}
	Let $\alpha\inn\n K(\C H_D)\setminus\n K(T)$ and let $[a_0,a_1,\dots]$ be its continued fraction expansion; let $\cv n$ be its convergents and let us denote by $l_n$ the degree of $a_n$. Let $$\pphi_n=p_n-q_n\alpha.$$ Then, by \eqref{1eq:fondprop}, $$\div(\pphi_n)=A+r(\infty_+)+s(\infty_-), \text{ with } r=\deg_T q_{n+1},\ s=-\deg_T p_n$$ and where $A$ is an affine divisor of degree $\deg A=\deg_T p_n-\deg_T q_{n+1}=l_0-l_{n+1}$. 
\end{rem}	

\begin{lem}\label{6lem:multdelta}
	For every $n$ there exist $P_1,\dots,P_{d-l_{n+1}}$, affine points of $\C H_D$, such that $$(\deg_Tp_n)\delta\sim (P_1)+\cdots+(P_{d-l_{n+1}})-(d-l_{n+1})(\infty_+).$$ 
	More generally, for every $m$ we have \be\label{6eq:mdelta} m\delta\sim (P_{s+1})+\cdots+(P_{k})+s(\infty_-)-(d-l_{n+1}+s)(\infty_+),\ee where $k=d-l_{n+1}+s,\ s=m-\deg_T p_n$ and where $n$ is the only integer such that $$\deg_T p_n\leq m<\deg_T p_{n+1}.$$
\end{lem}
\begin{proof}
	In the notations of the previous Remark, let $\alpha=\sqrt D$. Then $\alpha$ has no affine poles, so $A$ is effective. Moreover, $\deg A=d-l_{n+1}\leq g$ and \be\div(p_n-q_n\sqrt D)=A+(\deg_Tq_{n+1})(\infty_+)-(\deg_Tp_n)(\infty_-).\ee 
	On the other hand, $\deg A=\deg_T(\pphi_n\pphi_n')$, so we obtain again that $$\deg (p_n^2-Dq_n^2)=d-l_{n+1}.$$ Thus, there exist $d-l_{n+1}$ affine points of $\C H_D$, $P_1,\dots,P_{d-l_{n+1}}$, such that $$(\deg_Tp_n)\delta\sim (P_1)+\cdots+(P_{d-l_{n+1}})-(d-l_{n+1})(\infty_+).$$ 
	
	Now, let $m\delta\sim(P_1)+\cdots+(P_k)-k(\infty_+)$ with $k$ minimal. By \eqref{6eq:RRcor}, we have $k\leq g$ and in particular $P_i\neq\infty_+$ for every $i$ and $P_i\neq P_j'$ for every $i\neq j$. Let us assume $P_1,\dots,P_s=\infty_-$ and $P_{s+1}=\cdots=P_k\neq\infty_-\,$. Then there exists $f\inn\n K(\C H_D)$ such that $\div(f)=(P_{s+1})+\cdots+(P_k)+(m-k)(\infty_+)-(m-s)(\infty_-)$. As $f$ has no affine poles, it must be of the form $f=a+b\sqrt D$, with $a,b$ polynomials. Actually, as the $P_i$ are pairwise non-conjugated, $a,b$ must be relatively prime. Now, $k-s=\deg_T(a^2-Db^2)<d$, so by \eqref{2eq:bestquadr} $a/b$ is, up to the sign, a convergent of $\sqrt D$, that is (up to a multiplicative constant) either $f=\pphi_n$ or $f=\pphi_n'$. By minimality of $k$, certainly $m\geq k$, so $f=\pphi_n$ and we get \ref{6eq:mdelta}.
\end{proof}

\begin{lem}\label{6lem:zerosphi}
	As before, let $P=(t,u)$ be an affine zero of $\pphi_n$. Then $$[\n K(P):\n K]\leq 2(d-1).$$
\end{lem}
\begin{proof}
	As $t$ is a root of $p_n^2-Dq_n^2$, we have $[\n K(t):\n K]\leq d-1$, so $[\n K(P):\n K]\leq 2(d-1).$
\end{proof}	

\begin{lem}
	Let us consider the embedding $j=j_{\infty_+}$ of $\C H_D$ in its Jacobian $\C J_D$, $$j(P)=[(P)-(\infty_+)]\inn\C J_D.$$
	
	For $i=1,\dots,g$, let $\C W_i=\C W_i^{\infty_+}$. Then for every $m$. $$m\delta\inn\C W_{d-l_{n+1}+s}\setminus\C W_{d-l_{n+1}+s-1}, \text{ where } \deg_T p_n\leq m<\deg_T p_{n+1} \text{ and } s=m-\deg_T p_n.$$
\end{lem}
\begin{proof}
	It follows from Lemma \ref{6lem:multdelta}
\end{proof}

\begin{ex}
	Let $D=T^8-8T^7+18T^6-4T^5-13T^4-4T^3+18T^2-8T+1$; it can be seen that $D$ is a squarefree, non-Pellian polynomial.	In the previous notations, the minimal representations of the first multiples of $\delta$ as sums of points in $j(\C H_D)$ are:
	\indent$\hspace{-0,4cm} \delta=[(\infty_-)-(\infty_+)]\inn\C W_1\\
	2\delta=[2(\infty_-)-2(\infty_+)]\inn\C W_2\\
	3\delta=[3(\infty_-)-3(\infty_+)]\inn\C W_3\\
	4\delta=[(0,1)+(1,1)-2(\infty_+)]\inn\C W_2\ (\textrm{as } p_0^2-Dq_0^2=-12T^2+12T)\\
	5\delta=[(0,1)+(1,1)+(\infty_-)-3(\infty_+)]\inn\C W_3\\
	6\delta=[(A_1)+(A_2)-2(\infty_+)]\inn\C W_2,\ \text{with } A_i=(t_i,u_i)$, where the $t_i$ are roots of 
	
	\indent\hspace{1cm} $p_1^2-Dq_1^2=\frac13T^2-T+\frac13\\
	7\delta=[(A_1)+(A_2)+(\infty_-)-3(\infty_+)]\inn\C W_3\\
	8\delta=[(B_1)+(B_2)+(B_3)-3(\infty_+)]\inn\C W_3,\ \text{with } B_i=(t_i',u_i')$, where the $t_i'$ are roots of \indent\hspace{1cm} $p_2^2-Dq_2^2=72T^3-108T^2-180T+72\\
	9\delta=[(C_1)+(C_2)+(C_3)-3(\infty_+)]\inn\C W_3,\ \text{with } C_i=(t_i'',u_i'') \text{ and } t_i''\inn\left\{1,\frac{9\pm\sqrt{65}}4\right\}\\
	\cdots$
\end{ex}

\begin{theo}[Abel, 1826]\label{6theo:equiv}
	Let $D\inn\n K[T]$ be a squarefree polynomial of degree $2d$ whose leading coefficient is a square in $\n K$. Then the following statements are equivalent:
	\begin{enumerate}
		\item the continued fraction expansion of $\sqrt D$ is periodic;
		\item $\sqrt D$ has infinitely many partial quotients of degree $d$;
		\item the continued fraction of $\frac{r+\sqrt D}s$ is quasi-periodic for every $r,s\inn\n K[T]$ such that $s$ divides $r^2-D$;
		\item $D$ is Pellian, that is, there exist $p,q\inn\n K[T]$ with $q\neq0$ such that $p^2-Dq^2\inn\n K^*$;
		\item $\n K[T,\sqrt D]$ has non-constant units;
		\item $\delta=[(\infty_-)-(\infty_+)]$ is a torsion point of the Jacobian $\C J_D$ of $\C H:U^2=D(T)$.
	\end{enumerate}
	In this case, if $\sqrt D=\left[a_0,\ov{a_1,\dots,a_N}^k\right]$ with $N$ minimal quasi-period, $\delta$ has torsion order $\deg_T q_N$ and $(p_{N-1},q_{N-1})$ is the minimal non-trivial solution to the Pell equation for $D$.
\end{theo}
\begin{proof}
	We have already seen in Proposition \ref{2prop:eqper} the equivalence of $1.\ 2.\ 3.\ 4.$ and $5.$\par\medskip
	
	Let us then prove that $6.$ is equivalent to $2.$
	
	$\delta$ is a torsion point of $\C J_D$ if and only if there exists $m\neq0$ such that $m\delta=0\inn\C J_D$ and, by \eqref{6eq:mdelta}, this is equivalent to the existence of (infinitely many) partial quotients of degree $d$.
	
	Moreover, by \eqref{6eq:mdelta}, $m\delta\sim0$ if and only if $\deg a_{n+1}=d$, if and only if $(p_n,q_n)$ is a solution to the Pell equation for $D$ and in this case, $m=\deg_T p_n=\deg_T q_{n+1}$. \par\medskip
	
	Let us prove again that $\delta$ is a torsion point if and only if $\n K[T,\sqrt D]$ has non-trivial units. 
	
	$m\delta=0\inn\C J_D$ with $m\neq0$ if and only if there exists a rational function $f\inn\n K(T,\sqrt D)$ such that $\div(f)=m\delta$. In this case of course $\div(f')=-m\delta$ and $f$ must be of the form $f=a+b\sqrt D$ with $a,b\inn\n K[T]$ and $b\neq0$, so $f\inn\n K[T,\sqrt D]$. Then, as $\div(ff')=0$, $f$ must be a unit of $\n K[T,\sqrt D]$.
	
	Conversely, if $f=a+b\sqrt D$ is a unit of $\n K[T,\sqrt D]$, then $f,f'$ have no affine poles and $\div(ff')=0$, so we must have $\div(f)=m\delta$ for some $m\neq0$, that is, $\delta$ will be a torsion point of $\C J_D$.
\end{proof}

\begin{rem}\label{6rem:multdeltanonsf}
	If $D$ is non-squarefree, the affine curve $\C H_D:U^2=D(T)$ is no longer smooth; in particular, it is not the affine part of an hyperelliptic curve. Thus, in this case, we will have to slightly modify the previous methods in order to obtain similar results.
	
	Let $D=b^2\widetilde D$ with $b\inn\n K[T]$ non constant and with $\widetilde D$ a squarefree polynomial of degree $\deg_T\widetilde D\!=\!2\widetilde d$. Then $p/q$ is a convergent of $\sqrt D$ (with $p,q$ relatively prime polynomials) if and only if, as formal Laurent series, $\ord(p-q\sqrt D)>\deg_T q$, if and only if $\ord_{\infty_+}\left(p-qb\sqrt{\widetilde D}\,\right)\!>-\!\ord_{\infty_+}q$ as rational functions in $\n K(\C H_{\widetilde D})$, if and only if $$\div\left(p-qb\sqrt{\widetilde D}\,\right)=(P_1)+\cdots+(P_{d-l_{n+1}})+(\deg_T q+l_{n+1})(\infty_+)-\deg_T p(\infty_-),$$ with $l_{n+1}>0$.
	
	Equivalently, $(\deg_T p)\delta=(P_1)+\cdots+(P_{d-l_{n+1}})-(d-l_{n+1})(\infty_+)-\div\left(p-qb\sqrt{\widetilde D}\,\right)$, where $\delta=(\infty_-)-(\infty_+)\inn\Div(\C H_{\widetilde D})$.\par\medskip
	
	On the other hand, let $m\delta=(P_1)+\cdots+(P_k)-k(\infty_+)-\div(\pphi)$, with $\pphi\inn\n K(\C H_{\widetilde D})$ of the form $\pphi=p-qb\sqrt{\widetilde D}$ and with $k\leq m$ minimal. As before, $\infty_-$ is the only pole of $\pphi$, so we must have $\deg_T p=\deg_T q+\deg_T b+\widetilde d=-\ord_{\infty_-}\pphi\leq m$, while $\ord_{\infty_+}\pphi\geq m-k$. If $k<\deg_T b+\widetilde d$ then $\ord_{\infty_+}\pphi>\deg_T q$, so $p/q$ is a convergent of $b\sqrt{\widetilde D}$.
	
	Then, as in Lemma \ref{6lem:multdelta}, the convergents of $\sqrt D=b\sqrt{\widetilde D}$ are related to the ways of writing the multiples of $\delta$ as sums of at most $d=\deg_T b+\widetilde d$ divisors of the form $(P_i)-(\infty_+)$, with $P_i$ an affine point of $\C H_{\widetilde D}$, plus the divisor of a function of the form $p+qb\sqrt{\widetilde D}$. As we will see in section \ref{66Sec}, this corresponds to substituting the usual Jacobian $\C J_{\widetilde D}$ of $\C H_{\widetilde D}$ with a generalized Jacobian.
\end{rem}	

\subsubsection{Reduction of abelian varieties and Pellianity}
A condition to determine if a polynomial $D$ defined over a number field is Pellian or not follows from the theory of reduction of abelian varieties, developed by Serre and Tate, \cite{serre1968good}. Here, we will just recall the main results and show their connection with pellianity; proofs, comments and details can be found for instance in \cite{silverman2000diophantine}, part C, while this application of the reduction theory to continued fractions appears also in \cite{yu1999arithmetic} (page 5). 

\begin{lem}
	Let $\C A$ be an abelian variety defined over a number field $\n K$. Then, The multiplication by $m$ map $[m]:\C A(\ch Q)\to\C A(\ch Q)$ is surjective and its kernel $\C A_m$ is finite.
\end{lem}

\begin{rem}
	Assuming that all the $m$-torsion points are defined over $\n K$, that is, $\C A_m\sub\C A(\n K)$, and that the group $\mu_m$ of the $m$-th roots of unity is contained in $\n K$ the following holds:
\end{rem}

\begin{theo}
	Let $\C A$ be an abelian variety defined over a number field $\n K$. Let $\nu$ be a non-Archimedean valuation of $\n K$, let $\widetilde{\n K}$ be the residue field of $\nu$ and let $p=\char\widetilde{\n K}$ be the residue characteristic of $\nu$. Assuming that $\C A$ has good reduction modulo $\nu$, let $\widetilde{\C A}$ be its reduction. Then for every $m\geq 1$ such that $p$ does not divide $m$ the reduction map $\C A_m(\n K)\to\widetilde{\C A}(\widetilde{\n K})$ is injective. Actually, it induces an isomorphism $$\C A_m(\n K)\simeq\widetilde{\C A}_m(\widetilde{\n K}).$$
\end{theo}

\begin{lem}
	As before, let $\C A$ be an abelian variety over $\n K$ and let $P$ be a point of $\C A$. If there exist valuations $\nu_1,\nu_2$ of good reduction for $\C A$ and of different, positive, characteristics $p_1,p_2$ such that the reductions of $P$ modulo $\nu_i$ have orders $m_i$ and, for every $k_1,k_2\inn\n N$, $$m_1p_1^{k_1}\neq m_2p_2^{k_2},$$ then $P$ cannot be a torsion point of $\C A$.	
	
	On the other hand, if there exist integers $k_1,k_2$ such that $m_1p_1^{k_1}=m_2p_2^{k_2}$, they are unique, so if $P$ is a torsion point of $\C A$, its order must be their common value.
\end{lem}
\begin{proof}
	Let $\nu$ be a valuation of good reduction for $\C A$ and of residue characteristic $p>0$ and, for every point $P$ of $\C A$, let us denote by $\widetilde P$ its image in the reduced variety $\widetilde{\C A}$.
	
	By the previous Theorem, if $P$ is a torsion point of $\C A$ of order $n$, then $\widetilde P$ is a torsion point of $\widetilde{\C A}$ of order $m$, with $n=mp^k$ for some $k\inn\n N$. It is then enough to apply this result at the same time to the two valuations $\nu_1,\nu_2$.
\end{proof}

In particular, we can apply the previous results to $\C A=\C J_D$ the Jacobian of an hyperelliptic curve $\C H_D$ and to $P=\delta=[(\infty_-)-(\infty_+)]$. In this case, $\C J_D$ has good reduction modulo a valuation $\nu$ if and only if the residue characteristic of $\nu$ is different from $2$, $D$ can be reduced modulo $\nu$ and its reduction is still a squarefree polynomial of degree $2d$. The previous Lemma and Theorem \ref{6theo:equiv} then lead to the following characterization of Pellian polynomials:

\begin{cor}\label{6cor:yureductioncondition}
	Let $D\inn\n K[T]$ be a squarefree polynomial of even degree $2d$ and whose leading coefficient is a square in $\n K$; let $\C H_D$ be the hyperelliptic curve $U^2=D(T)$, let $\C J_D$ be its Jacobian and let $\delta=[(\infty_-)-(\infty_+)]\inn\C J_D$. If there exist two valuations $\nu_1,\nu_2$ of $\n K$, of good reduction for $\C J_D$ and of different, positive, characteristics $p_1,p_2\neq2$, such that, denoting by $m_i$ the order of $\widetilde\delta_{\nu_i}$ in $\widetilde{\C J}_{\nu_i}$, $$m_1p_1^{k_1}\neq m_2p_2^{k_2} \text{ for every } k_1,k_2,$$ then $D$ is non-Pellian and the continued fraction of $\sqrt D$ is non-periodic.\par\medskip
	
	On the other hand, if $m_1p_1^{k_1}=m_2p_2^{k_2}=m$, then the length of the quasi-period of $\sqrt D$ is at most $m-g$, so, computing finitely many of its partial quotients, it can be checked effectively whether $D$ is Pellian or not.
\end{cor}

\begin{ex}\label{6ex:nonPell}
	Let $D=T^4+T+1\inn\n Q[T]$. Then $D$ can be reduced modulo any prime and has always squarefree reduction. Let $m_3,m_5$ be, respectively, the orders of the reductions of $\delta$ modulo $3,5$ in $\widetilde{\C J}_3,\widetilde{\C J}_5$. It can be seen that $m_3=7$ and $m_5=9$, so $D$ is non-Pellian. 
	
	However, for some choice of the primes $p_1,p_2$, the orders of the reductions of $\delta$ could be compatible, for example we have $m_{17}=m_{19}=25$.  
\end{ex}	
\clearpage{\pagestyle{empty}\cleardoublepage}
\chapter{Zaremba's and McMullen's conjectures in the real case}
In the following Chapters we will study the polynomial analogues of Zaremba's and McMullen's Conjectures on continued fractions with bounded partial quotients. 

Here, we will briefly review the original statements, for the real case, and some of the known results, focusing on the works of Bourgain and Kontorovich. Discussions of these and related results can be found in \cite{bourgain2013some}, \cite{kontorovichapplications}, \cite{kontorovich2013apollonius}.

\begin{notatC} We will write $f(x)\ll g(x)$ for $x\to\infty$ if there exists a constant $c>0$ such that $|f(x)|\leq c\, |g(x)|$ for every sufficiently large $x$. If $f\ll g\ll f$, we will write $f\asymp g$.
\end{notatC}

\section{Zaremba's Conjecture}
\begin{notat}
	As in \ref{1def:cf}, we will set $[a_0,a_1,a_2,\dots]=a_0+\cfrac1{a_1+\cfrac1{a_2+\cfrac1{\ddots}}}\in\n R$, where $a_0,a_1,a_2,\dots$ is a (finite or infinite) sequence of integers, with $a_i>0$ for $i\geq1$.\par\medskip
	
	Contrary to the polynomial case, rational numbers have two different regular continued fraction expansions: indeed,  $[a_0,\dots,a_n]=[a_0,\dots,a_n-1,1]$ with $a_0,\dots,a_n\inn\n N$, $a_1,\dots,a_{n-1}\geq 1$ and $a_n>1$; in this section, we will always consider the second kind of expansion, having 1 as the last partial quotient. 
\end{notat}

For $\alpha=[a_0,\dots,a_k,1]\inn\n Q$, similarly to the notations introduced in \ref{1notat:K,ovK} in the polynomial case, we will set $$K(\alpha)=\max\{a_1,\dots,a_k,1\}.$$ More generally, if $\alpha=[a_0,a_1,\dots]\inn\n R$, we define $$K(\alpha)=\sup_{n\geq1}a_n\inn\n N\cup\{\infty\}.$$ For every positive integer $z$ we can consider the sets
$$\mathfrak R_z=\left\{b/d\inn\n Q,\ K(b/d)\leq z\right\},\ \mathfrak D_z=\left\{d\inn\n N,\ \exists\, b \text{ such that } (b,d)=1 \text{ and } b/d\inn\mathfrak R_z\right\}.$$

For example, $\mathfrak D_1$ is the set of Fibonacci numbers and $\mathfrak R_1$ is the set of quotients of two consecutive Fibonacci numbers.\par\medskip

The study of numerical integration by the quasi-Monte Carlo method and of pseudo-random numbers led Zaremba to express in \cite{zaremba1972methode} a conjecture about the existence of rational numbers having only ``small" partial quotients in their continued fraction expansions.\par\medskip

Let $f:\n R^s\to\n R$ be a (sufficiently regular) function, with $s>1$. The \textit{quasi-Monte Carlo method} of numerical integration consists in the approximation of the integral $\ds\int_{[0,1]^s}f(t)dt$ with a finite sum of the form $\ds\frac 1d\sum_{i=1}^df(t_i)$, where $\C T=(t_i)_{i=1}^d$ is a suitable sequence of points in $[0,1]^s$. If the \textit{total variation} $V(f)\!=\!\!\!\max\limits_{\alpha\,\sub\{1,\dots,s\}}\|\partial^{(\alpha)}f\|_{L^1([0,1]^s)}$ of $f$ is finite, then, by Koksma - Hlawka inequality, $$\ds\left|\int_{[0,1]^s}f(t)dt-\frac1d\sum_{i=1}^df(t_i)\right|\leq c\,V(f)D(\C T),$$ where $c$ is an absolute constant and $D(\C T)\!=\!\!\!\sup\limits_{\substack{I\,\sub\, [0,1]^s\\I\text{ box}}}\left||I|-\frac1d(\#\{i\leq d,\ t_i\inn I\})\right|$ is the \textit{discrepancy} of the sequence $\C T$.

Zaremba was thus interested in the explicit construction of sequences of points that have the smallest possible discrepancy.  
In the case $s=2$, this problem is directly linked with continued fraction expansions. Indeed, Zaremba had already proved in \cite{zaremba1966good}, Proposition 4.3, that if $\C T=\left\{t_i=\left(\frac id,\left\{\frac{bi}d\right\}\right),\ 1\leq i\leq d\right\}$, where $\{\cdot\}$ denotes the fractional part of a real number, with $b,d$ relatively prime positive integers such that $b/d\inn\mathfrak R_z$, then $$D(\C T)<\left(\frac{4z}{\log(z+1)}+\frac{4z+1}{\log d}\right)\frac{\log d}d.$$ 

W. Schmidt proved that for every sequence $\C T\sub[0,1]^2$ of $d$ points the discrepancy is at least $D(\C T)>c\frac{\log d}d$, where $c$ is an absolute constant. Thus, Zaremba's model realizes the optimal discrepancy, on condition that we have a control over $z$, that is, if there exists a constant $z$ such that for every $d$ there exists $b$ with $b/d\inn\mathfrak R_z$.

Relying on numerical evidence, in \cite{zaremba1972methode} (page 76) Zaremba conjectured that this is the case for $z\geq5$: 

\begin{conj}[Zaremba]
	For any positive integer $d$ there exists $b\inn\n N$, relatively prime to $d$, such that $K(b/d)\leq 5$, that is, $$\mathfrak D_5=\n N.$$
\end{conj}

More generally, we can consider the following Conjecture:

\begin{conj}\label{3conj:Zar2}
	There exists $z\inn\n N$ such that $\mathfrak D_z=\n N$.
\end{conj}

Zaremba could only prove that for every positive integer $d$ there exists $b$, relatively prime to $d$, such that \be\label{3eq:Zartheo}K\left(b/d\right)\leq c\log d,\ee where $c$ is an absolute constant.\par\medskip

It has been shown that $[1,10^6]\sub\mathfrak D_5$ and that $[1,10^6]\setminus\mathfrak D_4=\{54,150\}$. Moreover, Niederreiter \cite{niederreiter1986dyadic} (Theorems 1 and 2) proved that $\mathfrak D_3$ contains all the powers of 2 and 3 and that all the powers of 5 are in $\mathfrak D_4$; actually, he gave an explicit method, based on the folding algorithm\footnote{The same method, the polynomial analogue of the folding algorithm, \eqref{1eq:folding}, will allow us to prove the polynomial analogue of Zaremba's Conjecture for powers of linear polynomials, see Lemma \ref{4lem:folded}.} to construct the corresponding elements of $\mathfrak R_3$ or $\mathfrak R_4$. This led him to conjecture that $\mathfrak D_3$ contains every large enough integer. 

Relying on a larger numerical evidence, Hensley conjectured that the same result should hold for $\mathfrak D_2$:

\begin{conj}[Hensley, \cite{hensley1996polynomial}]\label{3conj:Hensley1}
	$\mathfrak D_2=\n N\setminus F_2$, where $F_2$ is a finite set.
\end{conj}

Actually, Hensley conjectured something much more general.

For a fixed finite alphabet $\C A\sub\n N$, we can consider the limit set $$\mathfrak C_{\C A}=\Big\{[0,a_1,a_2,\dots],\ a_i\inn\C A\ \forall i\Big\}.$$ The elements of $\mathfrak C_{\C A}$ are said to be \textit{uniformly badly approximable} or \textit{absolutely Diophantine} of height $z=\max\C A$. If $\C A=\{1,\dots,z\}$, similarly to the previous notation we will write $\mathfrak C_{\C A}=\mathfrak C_z$. For example, $\mathfrak C_{\{1\}}=\left\{\frac{1+\sqrt5}2\right\}$, while $\mathfrak C_{\C A}$ is an infinite set as soon as $\# \C A\geq2$. Actually, in this case $\mathfrak C_{\C A}$ is an uncountable set of Lebesgue measure 0 and, more precisely, it is a Cantor-like set. Let $\delta_{\C A}$ be its Hausdorff dimension; certainly $\delta_{\C A}\to1$ when $\C A$ tends to the whole $\n N$. Hensley (Theorem 1 in \cite{HENSLEY1992336}) proved that for $\C A=\{1,\dots,z\}$ we have $$\delta_z=1-\frac 6{\pi^2z}+o(1/z) \text{ for } z\to\infty;$$ in particular for every $\delta<1$ there exists a finite alphabet $\C A$ such that $\delta_{\C A}>\delta$. 

Similarly to the previous notations, let $\mathfrak R_{\C A}$ be the set of the convergents of elements in $\mathfrak C_{\C A}$ and let $\mathfrak D_{\C A}$ be the set of the denominators of elements in $\mathfrak R_{\C A}$.

Let $\mathfrak R_{\C A}(N)=\left\{b/d\inn\mathfrak R_{\C A},\ 1\leq b<d<N\right\}$. Hensely showed that \be\label{3eq:HensleyRN}\#\mathfrak R_{\C A}(N)\asymp N^{2\delta_{\C A}}\text{ for } N\to\infty,\ee where the implied constant can depend on $\C A$.  

In particular, if $\delta_{\C A}<\frac12$ then $\n N\setminus \mathfrak D_{\C A}$ is an infinite set. Indeed, let $\mathfrak D_{\C A }(N)=\mathfrak D_{\C A}\cap[1,N]$. Then we have $\#\mathfrak D_{\C A}(N)\!\leq\#\mathfrak R_{\C A}(N)\!\ll N^{2\delta_{\C A}}$. So in order to have $\#\mathfrak D_{\C A}(N)=N+O(1)$, a necessary condition is $\delta_{\C A}\geq\frac12$. 

For example, as $\delta_{\{1,3\}}<\frac12$, for every $N\inn\n N$ there exists $d\geq N$ not in $\mathfrak D_{\{1,3\}}$, that is, such that $b/d\notin\mathfrak R_{\{1,3\}}$ for every integer $b$ relatively prime to $d$, that is, in the continued fraction expansion of $b/d$ there are always partial quotients different from $1$ and $3$.

On the other hand, $\#\mathfrak D_{\C A}(N)\geq\frac1N\,\#\mathfrak R_{\C A}(N)\gg N^{2\delta_{\C A}-1}$, that is,  $\#\mathfrak D_{\C A}(N)$ grows at least as a power of $N$ as soon as $\delta_{\C A}>\frac12$. This led Hensley to conjecture that the previous necessary condition is also sufficient, that is:

\begin{conj}[Hensley, Conjecture 3 in \cite{hensley1996polynomial}]\label{3conj:Hensleyfalse}
	$\!$For any finite alphabet $\C A$, $\mathfrak D_{\C A}\!\!=\!\n N\setminus F_{\C A}$, where $F_{\C A}$ is a finite set, if and only if $\delta_{\C A}>\frac12$.
\end{conj}

As $\delta_{\{1,2\}}>\frac12$ (as it is shown, for example, by Jenkinson and Pollicott), this would imply Conjecture \ref{3conj:Hensley1} and, as a consequence, Conjecture \ref{3conj:Zar2}.\par\medskip 		

However, Bourgain and Kontorovich proved that Hensley's general Conjecture does not hold: for $\C A=\{2,4,6,8,10\}$ we have $\delta_{\C A}\sim 0,517$ but arbitrarily large numbers do not belong to $\mathfrak D_{\C A}$; indeed, they showed that if $d\inn\mathfrak D_{\C A}$ then $d\not\equiv3\pmod 4$. They thus proposed an alternative version of this conjecture, Conjecture \ref{3conj:BKforZar}, taking into account such congruence obstructions.

\section{McMullen's Conjecture}
It can be proved that any real quadratic field $\n Q(\sqrt d)$ contains infinitely many different purely periodic continued fractions with uniformly bounded partial quotients:
\begin{theo}\label{3theo:MM}
	For any real quadratic field $\n Q(\sqrt d)$ there exists a constant $m_d\inn\n N$, depending only on $d$, such that $\n Q(\sqrt d)$ contains infinitely many purely periodic continued fractions which are absolutely Diophantine of height $m_d$, that is, $$\{\alpha\inn\n Q(\sqrt d)|\ \alpha \text{ purely periodic}\}\cap\mathfrak C_{m_d} \text{ is an infinite set.}$$ 
\end{theo}

In particular McMullen (Theorem 1.1 in \cite{mcmullen2009uniformly}) proved this result using the connection between geodesics of the modular surface and continued fractions; we will discuss such connection and a proof of this Theorem in Appendix \ref{ch:hyper}.
\par\medskip

The previous Theorem can also be proved with algebraic methods, giving an explicit infinite sequence of absolutely Diophantine purely periodic continued fractions contained in a given real quadratic field. For example, in the same paper McMullen noticed that $\alpha_n=\Big[\,\ov{(1,s)^n,1,s+1,s-1,(1,s)^n,1,s+1,s+3}\,\Big]\inn\mathfrak C_{s+3}$ (where $(1,s)^n$ means that the sequence $1,s$ is repeated $n$ times) is in $\n Q\Big(\sqrt{s^2+4s}\,\Big)$ for every $n$, which gives a proof of Theorem \ref{3theo:MM} choosing $s=2x-2$,  where $(x,y)$ is a non-trivial solution of the Pell equation for $d$, that is, $x^2-dy^2=1$. Indeed, in this case for every $n$ we will have $\alpha_n\inn\n Q\Big(\sqrt{4x^2-4}\,\Big)=\n Q\Big(\sqrt{dy^2}\,\Big)=\n Q\Big(\sqrt d\,\Big)$.

Other constructions of special sequences of uniformly absolutely Diophantine purely periodic continued fractions with similar patterns of partial quotients and that lie in some prescribed real quadratic field are also given in \cite{mercat2013construction} (Th\'{e}or\`{e}me 1.2)\footnote{We will see a polynomial analogue of some of these results in Chapter \ref{ch:MM}.}, \cite{wilson1980limit}, \cite{woods1978markoff}.\par\medskip

This led McMullen to ask (\cite{mcmullen2009uniformly}, page 22) if the constants $m_d$ could be replaced by $2$ for every $d$. More generally, we can ask if the $m_d$ can be replaced by some absolute constant $m$:
\begin{conj}[McMullen]\label{3conj:MM}
	There exists $m\inn\n N$ such that $$\{\alpha\inn\n Q(\sqrt d)|\ \alpha \text{ purely periodic}\}\cap\mathfrak C_m \text{ is infinite for any positive squarefree integer }d.$$
\end{conj}

However, even the much weaker question of whether or not there exists a constant $m$ such that every real quadratic field contains at least an irrational number absolutely Diophantine of height $m$ seems to be still open.\par\medskip

In \cite{mercat2013construction}, Th\'{e}or\`{e}me 8.2, Mercat proved that this Conjecture is weaker then Zaremba's Conjecture\footnote{We will see the polynomial analogue of this result in Theorem \ref{5theo:Mercat}.}: 
\begin{theo}[Mercat]\label{3theo:Mercat}
	If Zaremba's Conjecture \ref{3conj:Zar2} holds for some constant $z$, then McMullen's Conjecture \ref{3conj:MM} holds for height $m=z+1$.
\end{theo}

Going further, in \cite{mcmullendynamics} McMullen proposed an even stronger Conjecture:
\begin{conj}[Arithmetic Chaos Conjecture]\label{3conj:arithmchaos}
	There exists $m\inn\n N$ such that for every real quadratic field $\n Q(\sqrt d)$, the cardinality of the set $$C_l=\left\{[\ov{a_0,\dots,a_l}\,]\inn\n Q(\sqrt d)\cap\mathfrak C_m\right\}$$ grows exponentially as $l\to\infty$.
\end{conj}
Even in this case, McMullen originally formulated the Conjecture with $m=2$.

Nevertheless, there are even no known examples of quadratic fields $\n Q(\sqrt d)$ such that the cardinality of $\{[\ov{a_0,\dots,a_l}\,]\inn\n Q(\sqrt d)\cap\mathfrak C_{m_d}\}$ grows exponentially with respect to $l$ for some constant $m_d$ depending on $d$.
  
\section{The work of Bourgain and Kontorovich}
Bourgain and Kontorovich reformulated both Zaremba's and McMullen's Conjectures in terms of thin orbits and showed that the two of them would follow from a much more general Local-Global Conjecture. These methods also allowed them to prove a density-one version of Zaremba's Conjecture.\par\medskip

As in Lemma\ref{1lem:decompmatr} for the polynomial case, the map $[a_0,\dots,a_n]\mapsto M_{(a_0,\dots,a_n)}$, where $M_{(a_0,\dots,a_n)}=\mm{a_0}110\dots\mm{a_n}110$, gives a canonical correspondence between finite continued fractions and matrices $\mm abcd\inn\GL_2(\n Z)$ such that $a>b>d\geq0$ and $a>c>d\geq0$. In particular, we still have that if $[a_0,\dots,a_n]=\frac pq$, then $M_{(a_0,\dots,a_n)}=\mm p*q*$. Moreover, if $\alpha$ is a real quadratic irrationality, then the continued fraction of $\alpha$ is always eventually periodic and if $\alpha=[a_0,\dots,a_{n-1},\ov{a_n,\dots,a_m}\,]$, then $\alpha\inn\n Q(\sqrt d)$ where $d=\left(\tr\, M_{(a_n,\dots,a_m)}\right)^2+4(-1)^{m-n}$ is the discriminant of $M_{(a_n,\dots,a_m)}$.\par\medskip 

For a fixed finite subset $\C A$ of $\n N$, let $\C G_{\C A}$ be the semi-group generated by the matrices $\mm a110$ with $a\inn\C A$. Then, denoting by $\{e_1,e_2\}$ the canonical basis of $\n Z^2$, we will have that $\mathfrak R_{\C A}$ is in bijection with the orbit of $e_1$ under the action of $\C G_{\C A}$, while $\mathfrak D_{\C A}=\left\langle e_2,\C G_{\C A}\cdot e_1\right\rangle$. 

Let $F_z:\GL_2(\n Z)\to\n Z$ be the linear map $F_z(M)\!=\langle e_2,M\cdot e_1\rangle$, that is, $F_z:\mm abcd\mapsto c$. Then $F_z(\C G_{\C A})=\mathfrak D_{\C A}$ and Zaremba's Conjecture \ref{3conj:Zar2} is equivalent to the existence of a finite alphabet $\C A$ such that $$F_z(\C G_{\C A})=\n N.$$ 

Analogously, McMullen's Conjecture is linked to the linear map $F_m(M)=\tr M$: if there exists $\C A$ such that $F_m(\C G_{\C A})=\n N$, then any real quadratic field contains an absolutely Diophantine element of height $m=\max\C A$.\par\medskip

Actually, it is more convenient to work in $\SL_2(\n Z)$ then in $\GL_2(\n Z)$, so Bourgain and Kontorovich considered the sub-semigroups $\Gamma_{\C A}=\C G_{\C A}\cap\SL_2(\n Z)$, that is, the semigroups generated by the matrix products $\mm a110\mm b110$ with $a,b\inn\C A$. Now, the orbit $\C G_{\C A}\cdot e_1$ is a union of orbits of $\Gamma_{\C A}\cdot e_1$, so it is enough to study the second one: $\mathfrak R_{\C A}\simeq\C G_{\C A}\cdot e_1=\Gamma_{\C A}\cdot e_1\cup\bigcup_{a\in\C A}\mm a110\Gamma_{\C A}\cdot e_1$. In particular, the Hausdorff dimension does not change when considering $\Gamma_{\C A}$ instead of $\C G_{\C A}$ and the limit set of infinite continued fractions $[0,a_1,a_2,\dots]$ such that $[0,a_1,\dots,a_{2n}]\inn\mathfrak R_{\C A}$ for every $n$ is of course still $\mathfrak C_{\C A}$.

Moreover, it follows from \eqref{3eq:HensleyRN} that $\#(\Gamma_{\C A}\cap B_N)\asymp N^{2\delta_{\C A}}$ for $N\to\infty$, where $B(N)\sub\SL_2(\n R)$ is the ball of size $N$ about the origin, with respect to the norm $\left\Arrowvert\mm abcd\right\Arrowvert=\sqrt{a^2+b^2+c^2+d^2}$. 

Now, as soon as $\#\C A\geq2$, the Zariski closure of $\Gamma_{\C A}$ is the whole $\SL_2(\n R)$; the set of integer points of $\SL_2(\n R)$ is $\SL_2(\n Z)$ and $\#(\SL_2(\n Z)\cap B_N)\asymp N^2$. As $\delta_{\C A}<1$ for every finite alphabet $\C A$, we have that $\Gamma_{\C A}$ has Archimedean zero density in the integer points of its Zariski closure; $\Gamma_{\C A}$ is thus said to be a \textit{thin integer set}. We are interested in cases where $\Gamma_{\C A}$ is thin but its image under a linear map $F:\Gamma_{\C A}\to\n Z$ is not, in the sense that it has at least positive density in $\n Z$.\par\medskip

Let  $F:\SL_2(\n Z)\to\n Z$ be a surjective linear map; we will consider $F(\Gamma_{\C A})$ for a given alphabet $\C A$. The \textit{multiplicity} of an integer $d$ is defined by $$\text{mult}(d)=\#\{M\inn\Gamma_{\C A},\ F(M)=d\};$$ in particular, $d$ is said to be \textit{represented} if $\text{mult}(d)>0$. As for some integers $d$ this multiplicity might be infinite, we will consider $\text{mult}_N(d)=\#\{M\inn\Gamma_{\C A}\cap B_N,\ F(M)=d\}$, for $N\inn\n N$. Naively, if $d$ is of order $N$, one may expect $\text{mult}_N(d)$ to be of order $\frac1N\#(\Gamma_{\C A}\cap B_N)\asymp N^{2\delta_{\C A}-1}$. However, as in the case of Hensley's Conjecture, there may be congruence obstructions that make this prediction false.\par\medskip

An integer $d$ is said to be \textit{admissible} for an alphabet $\C A$ (and a map $F$) if it passes all congruence obstructions, that is, if $$d\inn F(\Gamma_{\C A})\!\!\pmod q\ \text{ for every } q>1.$$ Actually, it can be proved using the theory of Strong Approximation that there exists an integer $q(\C A)$ such that $d$ is admissible if and only if $d\inn F(\Gamma_{\C A})\pmod{q(\C A)}$. Let $\mathfrak U_{\C A}$ be the set of admissible integers for $\C A$.
\par\medskip

Bourgain and Kontorovich conjectured that the previous naive prediction holds for all admissible integers (Conjecture 1.3.1 in \cite{kontorovichapplications}):
\begin{conj}[Local-Global Conjecture]\label{3conj:local-global}
	Let $F:\SL_2(\n Z)\to\n Z$ be a surjective linear map and let $\C A$ be a finite alphabet with $\#\C A\geq2$. For every admissible $d$ then $$\text{mult}_N(d)=N^{2\delta_{\C A}-1-o(1)} \text{ for } N\to\infty \text{ and } d\asymp N.$$ 
	In particular, if $\delta_{\C A}>1/2$ then for every admissible large enough $d$ there exists $N$ such that $\text{mult}_N(d)\geq1$, that is, all sufficiently large admissible integers are represented.
\end{conj}

In the case of Zaremba's Conjecture, that is for $F=F_z$, this leads to the following revised version of Hensley's Conjecture \ref{3conj:Hensleyfalse}: 
\begin{conj}[Bourgain, Kontorovich, Conjecture 1.7 in \cite{bourgain2014zaremba}]\label{3conj:BKforZar}
	If $\delta_{\C A}>1/2$, then $\mathfrak D_{\C A}$ contains every sufficiently large admissible integer.
\end{conj}

As the alphabet $\C A=\{1,2\}$ has no congruence obstructions and $\delta_{\C A}>1/2$, the Local-Global Conjecture would still imply \ref{3conj:Hensley1} and of course Zaremba's Conjecture \ref{3conj:Zar2}.\par\medskip

Bourgain and Kontorovich made a major step towards Zaremba's Conjecture by proving a density-one version, namely, they showed that, for $z$ large enough, $\mathfrak D_z$ contains almost every natural number:

\begin{theo}[Bourgain, Kontorovich, Theorem 1.8 in \cite{bourgain2014zaremba}]\label{3theo:bourgain-kontorovich}
	There exists an explicit constant $\delta_0<1$ such that $\mathfrak D_{\C A}$ contains almost every admissible integer for every finite alphabet $\C A$ with $\delta_{\C A}>\delta_0$. More precisely, if $\delta_{\C A}>\delta_0$ there exists an effectively computable constant $c$ depending only on $\C A$ such that $$\frac{\#(\mathfrak D_{\C A}\cap[N/2,N])}{\#(\mathfrak U_{\C A}\cap[N/2,N])}=1+O\left(e^{-c/\sqrt{\log N}}\right) \text{ for } N\to\infty,$$ where the implied constant depends only on $\C A$. 
\end{theo}

Bourgain and Kontorovich have shown that it is enough to take $\delta_0=\frac{307}{312}$.  

For $z\geq2$, the alphabet $\C A=\{1,\dots,z\}$ has no congruence obstructions, that is, $\mathfrak U_{\C A}=\n N$. Moreover, $\delta_{50}>\delta_0$. It follows that:

\begin{cor}
	If $z\geq50$, then $\mathfrak D_z$ contains almost every large enough integer. Actually, there exists an effectively computable constant $c$ such that $$\#\mathfrak D_z(N)=N+O(Ne^{-c\sqrt{\log N}}) \text{ for } N\to\infty,$$ where as before $\mathfrak D_z(N)=\mathfrak D_z\cap[1,N]$. 
\end{cor}

Bourgain and Kontorovich proved Theorem \ref{3theo:bourgain-kontorovich} using a local-global principle for thin orbits.  They adapted to this case techniques developed for the study of sequences of integers produced by orbits of subgroups of $\SL_2(\n Z)$, with the difference that in this case they only have a semigroup (so for example they cannot use automorphic tools and they have to employ the thermodynamic formalism of Ruelle's transfer operators). They used the Hardy - Littlewood circle method, analysing exponential sums on major arcs and minor arcs. 
\par\medskip

Refining their methods, the previous results can be slightly improved; in particular, \hbox{Frolenkov} and Kan (\cite{frolenkov2013reinforcement}, Theorem 2.1) proved positive density statements: they showed that $$\text{if } \delta_{\C A}>5/6 \text{ then } \#\mathfrak D_{\C A}(N)\gg N.$$

Combining the methods used by the previous authors, in \cite{huang2015improvement}, Theorem 1.6, Huang proved that $$\text{if }\delta_{\C A}>5/6, \text{  then } \ds\frac{\#(\mathfrak D_{\C A}\cap[N/2,N])}{\#(\mathfrak U_{\C A}\cap[N/2,N])}=1+O\left(e^{-c\sqrt{\log N}}\right)\text{ as }N\to\infty.$$ 

In the case of McMullen's Conjecture, that is with $F=F_m=\tr$, the Local-Global Conjecture \ref{3conj:local-global} would imply
\begin{conj}[Bourgain, Kontorovich, Conjecture 1.13 in \cite{bourgain2013beyond}]
	If $\delta_{\C A}>\frac12$, then for any sufficiently large admissible integer $d$ there exists $M\inn\Gamma_{\C A}$ such that $\tr\,M=d$. Moreover, the multiplicity of an admissible $d\inn[N,2N)$ is $$\text{mult}_N(d)=\#\left\{M\inn\Gamma_{\C A},\ \tr\, M=d\,,\ \|M\|\leq N\right\}>N^{2\delta_{\C A}-1-o(1)}.$$
\end{conj}

\begin{prop}[Bourgain, Kontorovich, Lemma 1.16 in \cite{bourgain2013beyond}]
	The Local-Global Conjecture \ref{3conj:local-global} implies McMullen's Arithmetic Chaos Conjecture \ref{3conj:arithmchaos} with $m=2$.
\end{prop}

\begin{proof}[Sketch of Proof]
	Let $\C A=\{1,2\}$ and let $\n K=\n Q(\sqrt d)$ be a real quadratic field. If $M\inn\Gamma_{\C A}$, that is, if $M=M_{(a_0,\dots,a_l)}$, with $a_0,\dots,a_l\inn\{1,2\}$ and $l$ odd, then $\log\|M\|\asymp l$. Let $N$ be a large enough parameter and let $x\asymp N$ be a solution to the Pell equation for $d$, $x^2-d\, y^2=4$. Now, $x$ is admissible for $\C A$, so there exist at least $\text{mult}_N(x)>N^{c_1}>c_2^l$ matrices $M\inn\Gamma_{\C A}$ with trace $x$ (where $c_1,c_2$ are appropriate positive constant).
\end{proof}

In the case of McMullen's Conjecture, Bourgain and Kontorovich could not give density-one or positive-proportion results like those on Zaremba's Conjecture; the main differences are due to the fact that $F_m$ is no longer a bilinear form. 

Thanks to the connection between Zaremba's and McMullen's Conjectures, Mercat proved that if $\mathfrak M$ is the set of the positive squarefree integers $d$ such that $\n Q(\sqrt d)$ contains a reduced quadratic irrationality which is absolutely Diophantine of height 51, then $$\#(\mathfrak M\cap[1,N])\gg\sqrt N \text{ for } N\to\infty.$$
Moreover, he proved that the sets of integers $d$ such that $\n Q(\sqrt{d^2-1})$, respectively, $\n Q(\sqrt{d^2+1})$, contains a reduced quadratic irrationality which is absolutely Diophantine of height 51 have density 1.
\clearpage{\pagestyle{empty}\cleardoublepage}
\chapter[Polynomial analogue of Zaremba's Conjecture]{\texorpdfstring{Polynomial analogue of Zaremba's conjecture:\\ \LARGE{some known results}}{Polynomial analogue of Zaremba's conjecture}}
In this Chapter we will consider a polynomial version of Zaremba's Conjecture \ref{3conj:Zar2}, where rational functions replace rational numbers, while we will deal with the polynomial analogue of McMullen's Conjecture \ref{3conj:MM}, concerning the quadratic irrationalities defined in Chapter \ref{ch:quadr}, in the following Chapter. As in \ref{1notat:K,ovK}, for $\alpha=[a_0,a_1,\dots]\inn\n L=\n K((T^{-1}))$ we will set $$K(\alpha)=\sup_{i\geq1}\deg a_i,\  \ov K(\alpha)=\limsup_i\deg a_i$$ and we will say that $\alpha$ is \textit{badly approximable} if $K(\alpha)<\infty$ (equivalently, if $\ov K(\alpha)<\infty$).\par\medskip 

\begin{lconj}{Z}[Polynomial analogue of Zaremba's Conjecture]\label{4conj:zaremba}
	There exists a constant $z_{\n K}$ (possibly depending on the base field $\n K$) such that for every non-constant polynomial $f\inn\n K[T]$ there exists a polynomial $g$ relatively prime to $f$ such that $$\ds K\left(g/f\right)\leq z_{\n K}.$$
\end{lconj}

By analogy with Conjecture \ref{3conj:Zar2}, we will call this statement \textit{Zaremba's Conjecture over $\n K$}. This problem has been studied, among others, by Blackburn \cite{blackburn1998orthogonal}, Niederreiter \cite{niederreiter1987rational}, Friesen \cite{friesen1992continued}, Lauder \cite{lauder1999continued}, Mesirov and Sweet \cite{mesirov1987continued} (in the case $\n K=\n F_2$).

Actually, it is believed that it is enough to take \be\label{4eq:Zarstrong}z_{\n K}=\begin{cases}1&\text{if }\n K\neq\n F_2\\2&\text{if }\n K=\n F_2\end{cases};\ee we will call \eqref{4eq:Zarstrong} \textit{Zaremba's strong Conjecture}. It is easy to see that Zaremba's Conjecture with $z_{\n F_2}=1$ does not hold (Lemma \ref{4lem:notZarF2}); in this case, Mesirov and Sweet conjectured that \ref{4conj:zaremba} holds over $\n F_2$ with $z_{\n F_2}=2$ (\cite{mesirov1987continued}). In the following, we will nearly always assume the base field to be different from $\n F_2$; we will present some known results about this special case at the end of this Chapter.

Firstly, we will show why Zaremba's strong Conjecture seems to be plausible over fields of cardinality greater than 2 and we will present some very general constructions of rational fractions with small partial quotients. Blackburn \cite{blackburn1998orthogonal} proved that Zaremba's strong Conjecture holds over any infinite field (Corollary \ref{4cor:Zarinf}); we will discuss, besides his original method, another proof of this result. On the other hand, if $\n K=\n F_q$ is a finite field different from $\n F_2$, so far it has been proved only that there exist polynomials $g$ such that $K(g/f)=1$ when the degree of $f$ is small with respect to $q$; we will present, and slightly improve, some results in this direction, mostly due to the work of Friesen \cite{friesen2007rational}.

\section{Likelihood of Zaremba's Conjecture}	
\begin{rem}
	When looking for a polynomial $g$ such that $K(g/f)=1$, without loss of generality we can always assume $\deg g<\deg f$. Indeed, if $g_1\equiv g_2\pmod f$, then $g_1$ is relatively prime to $f$ if and only if $g_2$ is and the continued fractions of ${g_1}/f,\ {g_2}/f$ differ only in their first partial quotient, so in particular $K({g_1}/f)=K({g_2}/f)$.
\end{rem}

\begin{defn}
	We will say that a formal Laurent series $\alpha\inn\n L$ is \textit{normal} if $$K(\alpha)=1.$$
	
	If $f,g$ are polynomials in $\n K[T]$, by writing \textit{``$g/f$ is normal''} we will mean that $f,g$ are relatively prime and $K(g/f)=1$.
\end{defn}

\begin{lem}\label{4lem:Zariffconv}
	Let $\alpha\inn\n L$ and let $\cv n$ be its convergents. Then $\alpha$ is normal if and only if $$\deg q_j=j \text{ for every } j.$$ 
\end{lem}

\begin{rem}
	Let us assume that $\alpha$ has positive order, $\alpha=\sum_{i<0}c_iT^i$. As we will see better in Lemma \ref{4lem:hankel}, there exists a sequence of matrices $(H_j)_{j>0}$, with $H_j\inn M_j(\n K)$, called the Hankel matrices, such that $\deg q_n=j$ for some $n$ if and only if $\det H_j\neq0$. 
	
	Thus, $\alpha$ is normal if and only if all the Hankel determinants are different from zero and, more generally, $\ov K(\alpha)=1$ if and only if $\det H_j\neq0$ for every large enough $j$.\par\medskip
	
	Then, at least when $\n K$ is an infinite field, one should expect a generic formal power series $\alpha$ to be eventually normal. In particular, polynomial analogues of Zaremba's and McMullen's Conjectures are likely to hold. In fact, it is well known that \eqref{4eq:Zarstrong} holds over every infinite field and we will show in the following Chapter that the polynomial version of McMullen's Conjecture holds over uncountable fields, over infinite algebraic extensions of finite fields and over $\ch Q$ and $\n Q$.
\end{rem}

Actually, Zaremba's and McMullen's Conjectures are expected to hold even over finite fields (actually, we will see in Theorem  \ref{5theo:Mercat} that in this case Zaremba's Conjecture would imply McMullen's Conjecture).

\begin{defn} 
	Let $\n K$ be a finite field and let $f\inn\n K[T]$ be a non-constant polynomial. We define the \textit{orthogonal multiplicity}
	\footnote{The word ``orthogonal'' is justified by the connection with the classical \textit{orthogonal sequences of polynomials}, that is, sequences  of polynomials $f_0,f_1,\dots$ such that $\deg f_i=i$ for every $i$ and which are pairwise orthogonal with respect to some symmetric bilinear form $\pphi:\n K[T]\times\n K[T]\to\n K$, non-degenerate over $\langle1,\dots,T^n\rangle$ for every $n$ and such that $\pphi(Tf,g)=\pphi(f,Tg)$ for every $f,g\inn\n K[T]$. It can be seen that a monic polynomial $f\inn\n K[T]\setminus\n K$ has positive orthogonal multiplicity if and only if it occurs in some orthogonal sequence of polynomials. Actually, the orthogonal multiplicity of a monic polynomial $f$ is exactly the number of the orthogonal sequences of monic polynomials in which $f$ occurs.} of $f$ as $$m(f)=\#\left\{g\inn\n K[T],\ \deg g<\deg f \text{ and } g/f \text{ is normal}\right\}.$$ 
	For every field $\n K$, we will say that a non-constant polynomial $f$ has \textit{positive orthogonal multiplicity} if there exists $g\inn\n K[T]$ such that $g/f$ is normal.
\end{defn}
	
Thus Zaremba's strong Conjecture \eqref{4eq:Zarstrong} for fields $\n K\neq\n F_2$ can be reformulated as: \begin{center}
		\textit{for every non-constant polynomial $f$ there exists $g\inn\n K[T]$ such that $g/f$ is normal.}
	\end{center}
Equivalently,
	\begin{center}
		\textit{any non-constant polynomial $f$ has positive orthogonal multiplicity.}
	\end{center}

\begin{lem}\label{4lem:heuristicsZar}
	Let $\n K=\n F_q$ be a finite field. Then the average value for the orthogonal multiplicity of a monic polynomial of degree $d$ over $\n F_q$ is $(q-1)^d$. 
\end{lem}
\begin{proof}
	Let $f\inn\n F_q[T]$ be a non-constant polynomial and let $d=\deg f$. Certainly, $f$ has positive orthogonal multiplicity if and only if there exist $a_1,\dots,a_d\inn\n F_q^*, b_1,\dots,b_d\inn\n F_q$ such that, in the notations of \ref{1def:continuants}, $f=k\,C_d(a_1T+b_1,\dots,a_dT+b_d)$ for some constant $k$. We can then assume $f$ to be monic. 
	
	Now, the number of monic polynomials in $\n F_q[T]$ of degree exactly $d$ is $q^d$, while the number of possible choices for the $a_n,b_n$ is $(q-1)^dq^d$, so the average value for the orthogonal multiplicity of a monic polynomial $f$ of degree $d$ is $(q-1)^d$. 
\end{proof}

\begin{rem}
	If $q=2$, the previous average value is $1$, therefore, for every degree $d$, either all polynomials have orthogonal multiplicity exactly equal to $1$ or there exist polynomials with zero orthogonal multiplicity. Then Zaremba's Conjecture over $\n F_2[T]$ with $z_{\n F_2}=1$ is very unlikely to hold; we will see that in fact it does not (Lemma \ref{4lem:notZarF2}). 
	
	On the other hand, if $q>2$ the average value for the orthogonal multiplicity grows exponentially with the degree of $f$, so, unless there are very large deviations from this average, every polynomial with large degree is likely to have positive orthogonal multiplicity, that is, one should expect Zaremba's Conjecture to hold with $z_{\n F_q}=1$, as long as $q>2$. However, as we have already mentioned, up to now it has only been proved that $m(f)>0$ for every polynomial $f\inn\n F_q[T]$ with degree \textit{small} enough (with respect to $q$).
\end{rem}

\subsection{General remarks on the ``orthogonal multiplicity''}
\begin{lem}\label{4lem:quadrZar}$\ $
	\begin{enumerate}
	\item Any linear polynomial $f\inn\n K[T]$ has positive orthogonal multiplicity and, if $\n K=\n F_q$ is a finite field, then $m(f)=q-1$.
	\item Let $f\inn\n K[T]$ be a quadratic polynomial. Then $f$ has positive orthogonal multiplicity if and only if there exists a linear polynomial $g$ relatively prime to $f$.
	\end{enumerate}
\end{lem}

	The second condition is certainly satisfied as soon as the cardinality of $\n K$ is greater then $2$. On the other hand, for $q=2$ we have the following Lemma:

\begin{lem}\label{4lem:notZarF2}
	Zaremba's Conjecture with $z=1$ does not hold over $\n F_2$.
\end{lem} 
\begin{proof}
	 Let $f=T^2+T\inn\n F_2[T]$. Any linear polynomial $g$ is a factor of $f$, so, by the previous Lemma, $m(f)=0$.
\end{proof}

As for real continued fractions, we can find (very sparse) special sets of polynomials with positive orthogonal multiplicity and now, in the polynomial case, for every given polynomial with positive orthogonal multiplicity, infinitely many others can be built.

\begin{rem}
	As the degrees of the partial quotients are invariant for multiplication by an invertible constant, $K(g/f)=K(c\,g/f)$ for every $c\inn\n K^*$, so $m(f)=m(c\,f)$ for every non-constant polynomial $f$ and for every constant $c\inn\n K^*$. In particular, Zaremba's Conjecture holds on $\n K[T]$ if and only if it holds on the set of monic polynomials $f\inn\n K[T]\setminus\n K$.
\end{rem}	
	
\begin{lem}
	A polynomial $f(T)$ has positive orthogonal multiplicity if and only if the same is true for every polynomial of the form $f_1(T)=f(aT+b)$ with $a\inn\n K^*,b\inn\n K$. Moreover, in this case $m(f_1)=m(f)$.
\end{lem}
\begin{proof}
	By Lemma \ref{1lem:subst}, the substitution of $T$ with a linear polynomial does not modify the degrees of the partial quotients. 
\end{proof}

\begin{notat}
	Let $f\inn\n K[T]$ be a non-constant polynomial. We will say that a polynomial $F$ is \textit{$f$-folded} if $$F=f \text{ or } F=ah^2,$$ where $h$ is an $f$-folded polynomial and $\deg a\leq1$. If $f$ is a linear polynomial , we will simply say that $F$ is folded.
	
	It is easy to see that $F$ is $f$-folded if and only if $F=\widetilde F f^{2^n}$, where $\widetilde F$ is folded and $\deg\widetilde F<2^n$. 
\end{notat}

\begin{rem}
	It follows directly from Lemma \ref{1lem:manip} that $$m(F)\geq m(f)$$ for every $f$-folded polynomial $F$; in particular, if $f$ has positive orthogonal multiplicity then all the $f$-folded polynomials have positive orthogonal multiplicity. 
	
	More precisely, if $g/f$ is normal and $F$ is $f$-folded, then repeated applications of \eqref{1eq:folding} or \eqref{1eq:folding2} will provide a polynomial $G$ such that $G/F$ is normal, as well as the continued fraction expansion of $G/F$. For example, applying \eqref{1eq:folding2} with $e_1=e_2=1, c=(-1)^d$, where $d=\deg f$, or, if $\deg a=1$, applying \eqref{1eq:folding}, we will have that \be\label{4eq:foldK}\frac{fg+1}{f^2},\ \frac{agf+1}{af^2} \text{ are normal}.\ee 
\end{rem}

In particular, we have the following Lemma:
\begin{lem}\label{4lem:folded}
	All the powers of linear polynomials have positive orthogonal multiplicity: if $\deg P=1$, then $$K(g_d/P^d)=1, \text{ where } g_d=\sum\limits_{j=0}^rP^{d-\floor{d/{2^j}}} \text{ and } 2^r\leq d<2^{r+1}.$$ 
\end{lem}
\begin{proof}
	As in Lemma \ref{4lem:quadrZar}, $K(1/P)=1$. Let $l=\floor{d/2}$ and let us assume by inductive hypothesis that $K({g_l}/{P^l})=1$. Now, $g_d=\begin{cases}g_lP^l+1&\text{if } d\text{ is even}\\g_lP^{l+1}+1&\text{if } d \text{ is odd}\end{cases}$, so, by \eqref{4eq:foldK}, $K(g_d/P^d)=1$.
\end{proof}
	
\begin{lem}
	Let $\n K=\n F_q$ be a finite field, let $F$ be an $f$-folded polynomial, $F=\widetilde Ff^{2^n}$ with $\widetilde F$ folded. Then the orthogonal multiplicity of $F$ is at least $$m(F)\geq(q-1)^nm(f).$$ 
\end{lem}
\begin{proof}
	It is enough to show that if $F=h^2$ or $F=ah^2$ with $h$ $f$-folded and $\deg a=1$, then $m(F)\geq(q-1)m(h)$. This is trivially true if $h$ has zero orthogonal multiplicity, so we will assume $m(h)>0$. Let $a_1,\dots,a_d$ be linear polynomials such that, setting, $w=a_1,\dots,a_n$, in the notations of Lemma \ref{1lem:manip}, $[0,\overrightarrow w]=g/h$, with $K(g/h)=1$ and $g,h$ relatively prime. Then $[0,\overrightarrow{w}+1,\overleftarrow w-1]=\frac{hg+(-1)^n}{F}$ so, for every $k\inn\n K^*$, $K\left(k\frac{gh+(-1)^n}{F}\right)=1$. On varying $w,k$, the continued fractions that we obtain are all distinct, so $m(F)\geq (q-1)m(h)$. Analogously, if $F=ah^2$, then for every $k\inn\n K^*$ we have $[0,\overrightarrow{w},ka,-\overleftarrow{w}]=\frac{agh+(-1)^mk^{-1}}{F}$, $K\left(\frac{agh+(-1)^mk^{-1}}{F}\right)=1$ and, again, when $w,k$ vary the corresponding continued fractions are all distinct, so $m(F)\geq (q-1)m(h)$.
\end{proof}

\section{Results on Zaremba's Conjecture}
Obviously for every polynomial $f$ of degree $d$ and for every polynomial $g$ we have $K(g/f)\leq d$. When $\n K$ is a finite field different from $\n F_2$, Niederreiter improved that naive remark, proving a polynomial analogue of Zaremba's result \eqref{3eq:Zartheo}.

\begin{theo}[Niederreiter, Theorem 4 in \cite{niederreiter1987rational}]
	Let $\n K=\n F_q$ be a finite field different from $\n F_2$ and let $f\inn\n F_q[T]$ be a non-constant polynomial of degree $d$ . Then there exists an irreducible polynomial $g\inn\n F_q[T]$ relatively prime to $f$ and such that $$K(g/f)<2+2\log_qd.$$
\end{theo}

\begin{proof}[Sketch of Proof]
	If $2+2\log_qd>d$, it is sufficient to take $g$ any irreducible polynomial relatively prime to $f$.
	
	Let us then consider the case $2+2\log_qd\leq d$ and let us assume $f$ to be monic. It can be shown that there exists at least a polynomial $g\inn\n F_q[T]$ which is monic, irreducible, of degree $d$, relatively prime to $f$ and such that for every $\alpha,\beta\inn\n F_q[T]\setminus\{0\}$ with $g|(\alpha-f\beta)$ we must have $\deg\alpha+\deg\beta>d-2-2\log_qd$. 
	
	Let $g/f=[1,a_1,\dots,a_n]$. Then, $\deg a_i=\deg p_i-\deg p_{i-1}$ for every $i$, where $p_i/q_i$ are the convergents of $f/g$. For $i=1,\dots,n$, let $\alpha_i=fp_{n-i}-gq_{n-i}$ and let $\beta_i=p_{n-i}$. Then $\alpha_i,\beta_i$ are non-zero polynomials and $g|(\alpha_i-f\beta_i)$ for every $i$, thus $\deg\alpha_i+\deg\beta_i>d-2-2\log_qd$. Now, in the notations of Lemma \ref{1lem:pnqnprime} we will have $\deg(fp_{n-i}-gq_{n-i})=\deg d_{i,n}=\deg C_{i-1}(a_{n-i+2},\dots,a_n)\!=$ $=d-\deg p_{n-i+1}$. Thus, $d-2-2\log_qd<d-\deg p_{n-i+1}+\deg p_{n-i}=d-\deg a_{n-i+1}$, that is, $\deg a_{n-i+1}\!<\!2+2\log_qd$ for every $i$.
\end{proof}

It has been known since the work of Blackburn that Zaremba's strong Conjecture holds over any infinite field and that, over finite fields, any polynomial with small enough degree has positive orthogonal multiplicity. Actually, for any polynomial $f$ defined over a finite field $\n F_q$, Blackburn's method allows to construct, whenever they exist, all the polynomials $g$ such that $g/f$ is normal (see \cite{blackburn1998orthogonal}, Theorem 2). Yet, their existence is guaranteed only if $2q\geq\deg f(\deg f+1)$. 
	
After briefly reporting Blackburn's proof, we will present a second, slightly simpler method that allows to construct explicitly, under the same hypothesis on $\deg f$, a completely reducible polynomial $g$ such that $g/f$ is normal. However, in general our method does not allow to construct all the solutions $g$ and there also exist polynomials $f$ with positive orthogonal multiplicity for which this algorithm cannot find any suitable $g$.

\begin{theo}\label{4prop:Zarfinite}
	Let $f\inn\n K[T]$ be a non-constant polynomial, let $d=\deg f$. If $$|\n K|\geq\frac{d(d+1)}2$$ then: \begin{enumerate}
		\item $m(f)>0$, that is, there exists a polynomial $g$ such that $g/f$ is normal 
		\item the polynomial $g$ can be chosen to be completely reducible.\end{enumerate}
\end{theo}

\begin{cor}\label{4cor:Zarinf}
	Zaremba's strong Conjecture \eqref{4eq:Zarstrong} holds over every infinite field.
\end{cor}

Actually, from the proofs of Theorem \ref{4prop:Zarfinite}, it will also follow that if $\n K$ is an infinite field, then for every non-constant polynomial $f\inn\n K[T]$ there exist infinitely many monic (completely reducible) polynomials $g$ with $\deg g<\deg f$ and such that $g/f$ is normal.\par\medskip

Blackburn's proof of Theorem \ref{4prop:Zarfinite} is based on the connection between continued fractions and Hankel matrices, at which we already hinted:

\begin{lem}\label{4lem:hankel}
	Let $\alpha=\sum_{i<0} c_iT^i\inn\n L$. Then there exists a convergent $p_n/q_n$ of $\alpha$ such that $\deg q_n=j$ if and only if $\det H_j(\alpha)\neq0$, where $H_j(\alpha)$ is the $j$-th Hankel matrix: $$H_j(\alpha)=\left(\begin{array}{ccc} c_{-1}&\cdots&c_{-j}\\c_{-2}&\cdots&c_{-j-1}\\&\cdots&\\c_{-j}&\cdots&c_{-2j+1}\end{array}\right).$$
\end{lem}

\begin{proof}
	By Lemma \ref{1lem:best2}, $\alpha$ has a convergent $\cv n$ such that $\deg q_n=j$ if and only if there exist polynomials $p_n,q_n$, with $\deg q_n=j$, such that $\ord(p_n-\alpha q_n)>j$, if and only if there exists a polynomial $q_n$ of degree $j$ such that $\ord\{q_n\alpha\}>j$, where $\{\cdot\}$ denotes the polynomial analogue of the fractional part.
	
	We may assume $q_n$ to be monic; let $q_n=T^j+b_1T^{j-1}+\cdots+b_j$. Then the previous condition gives a system of $j$ linear equations in the variables $b_1,\dots,b_j$, whose associated matrix is $\left(\begin{array}{ccc|c} c_{-1}&\cdots&c_{-j}&-c_{-j-1}\\c_{-2}&\cdots&c_{-j-1}&-c_{-j-2}\\&\cdots&&\cdots\\c_{-j}&\cdots&c_{-2j+1}&-c_{-2j}\end{array}\right)$, which has rank $j$, unless $\alpha=g/f$ is a rational function with $\deg f<j$ (by Lemma \ref{1lem:Laurseriesrat}). Thus, the system under consideration has a (unique) solution if and only if $\det H_j\neq0$. 
\end{proof}

\begin{proof}[Blackburn's Proof of part 1. of \ref{4prop:Zarfinite}]
	Let $f\inn\n K[T]$ be a polynomial, let $d=\deg f$. As in Lemma \ref{4lem:Zariffconv} then the orthogonal multiplicity of $f$ is $m(f)=\#(S\setminus\bigcup_{j=1}^dV_j)$, where $S$ is the set of all the rational functions $g/f$ with $\deg g<d$ (and $g$ not necessarily prime to $f$) and $V_j$ is the set of the rational functions $g/f$, with $\deg g<\deg f$, that do not have a convergent whose denominator has degree $j$, that is, by the previous Lemma, $V_j=\{g/f,\ \deg g<\deg f \text{ and } \det H_j(g/f)=0\}$.
	
	Let $f=a_0+\cdots+a_dT^d$. By Lemma \ref{1lem:Laurseriesrat}, a Laurent series of positive order $\sum\limits_{i<0}c_iT^i$ represents a rational function of the form $g/f$ if and only if the $c_i$ satisfy the linear recurrence relation $\sum\limits_{k=0}^da_kc_{-k+i}=0$ for $i\leq-1$. That is, $S$, which is obviously a $\n K$-vector space of dimension $d$, can be identified with the set of sequences $(c_i)_{i<0}$ satisfying the previous linear recurrence relation. 
	
	Now, for every $j=1,\dots,d$ there exists a (non-zero) polynomial $h_j\inn\n K[X_1,\dots,X_d]$ of degree at most $j$ such that $\det H_j(\alpha)=h_j(c_{-1},\dots,c_{-d})$ for every $\alpha=\sum\limits_{i<0} c_iT^i\inn S$. Thus, $V_j$ is the affine variety corresponding to the ideal $(h_j)\sub\n K[X_1,\dots,X_d]$. Then $V=\bigcup_{j=1}^dV_j$ is the affine variety corresponding to the ideal $(h)$, with $h=h_1\cdots h_d$, where $\deg h\leq\frac{d(d+1)}2$. 
	
	Thus, $f$ has zero orthogonal multiplicity if and only if $V=S$, if and only if $h(t_1,\dots,t_d)=0$ for every $t_1,\dots,t_d\inn\n K$. 
	
	If $\n K$ is an infinite field this certainly cannot happen.
	
	If $\n K=\n F_q$, then the polynomial $h\inn\n F_q[X_1,\dots,X_d]$ has at most $q^{d-1}\deg h$ zeros in $\n F_q^d$, so $V$ is different from $S$ as soon as $\ds q^d>q^{d-1}\frac{d(d+1)}2$.
	
	We can have $q=\frac12d(d+1)$ if and only if $d=2$ and $q=3$ and we have already seen that any polynomial of degree 2 in $\n F_3$ has positive orthogonal multiplicity.
\end{proof}

\begin{ex}
	Let us consider $f=T^3+\ov2\inn\n F_7[T]$. Then for every polynomial $g$ with $\deg g<3$ the Laurent series representing $g/f$ is of the form $\sum\limits_{i<0}c_iT^i$, with $c_{i-3}+\ov2c_i\!=\!0$ for every $i$. 
	
	Then the polynomials $g$ of degree 2 such that the continued fraction of $g/f$ is normal are in bijective correspondence with the triples $c=(c_{-1},c_{-2},c_{-3})\inn\n F_7^3$ such that $\det H_j(c)\neq0$ for $j=1,2,3$. Now, $\det H_1(c)=c_{-1}, \det H_2(c)=c_{-1}c_{-3}-c_{-2}^2$ and $\det H_3(c)\!=\!c_{-1}c_{-3}c_{-5}+2c_{-2}c_{-3}c_{-4}-c_{-3}^3-c_{-2}^2c_{-5}-c_{-1}c_{-4}^2\!\!=\!3c_{-1}^3+2c_{-2}^3-c_{-3}^3+c_{-1}c_{-2}c_{-3}$. It can be seen that the only solution of $\det H_3(c)=0$ is $c=0$, so there are exactly $6^2\cdot7$ polynomials $g$ such that $g/f$ is normal. For example, for $c=(1,1,0)$, we will have $g=T^2+T$ and $g/f=\left[0,T-\ov1,T-\ov1,\ov4T+\ov1\,\right]$.
\end{ex}

We will now give a second proof of Theorem \ref{4prop:Zarfinite}, based on successive multiplications of a continued fraction by a linear polynomial; this will lead us to construct completely reducible polynomials $g$. Indeed, as we have seen in Corollary \ref{1cor:decreasedegree}, if $q_n(\lambda)\neq 0$ for every $n$, where the $p_n/q_n$ are the convergents of $\alpha\inn\n L$, then $K((T-\lambda)\alpha)=K(\alpha)-1$ (unless $\alpha$ is already normal).
\begin{proof}[Second Proof of \ref{4prop:Zarfinite}]
	Let $f\inn\n K[T]$ be a polynomial of degree $d$. Trivially, $K(1/f)=d$ and the convergents of $1/f=[0,f]$ are simply $\frac01,\frac1f$.
	
	If $|\n K|>d$, there exists $\lambda_1\inn\n K$ that is not a root of $f$. Then, by Proposition \ref{1prop:prod} and Corollary \ref{1cor:decreasedegree}, $K\left(\frac{T-\lambda_1}f\right)=d-1$ and the continued fraction of $\pphi_1=\frac{T-\lambda_1}f$ will be of the form $[0,a_1,k(T-\lambda_1)]$, with $\deg a_1=d-1$. 
	
	For $n<d-1$, let us assume that we have constructed $\lambda_1,\dots,\lambda_n$ such that the continued fraction of $\pphi_n=\frac{(T-\lambda_1)\cdots(T-\lambda_n)}{f}$ is of the form $[0,a_n,b_1,\dots,b_n]$, with $b_1,\dots,b_n\inn\n K[T]$ linear polynomials and with $a_n\inn\n K[T]$ a polynomial of degree exactly $d-n$; let $\cv i$ be the convergents of $f_n$. The $q_i$ have at most $l_n=(d-n)+(d-n+1)+\cdots+d$ distinct roots. Then, as soon as $|\n K|>l_n$, there exists $\lambda_{n+1}\inn\n K$ such that $q_i(\lambda_{n+1})\neq0$ for every $i$. Thus, setting $\pphi_{n+1}=\pphi_n(T-\lambda_{n+1})$, by Proposition \ref{1prop:prod} and Corollary \ref{1cor:decreasedegree} we will have $K(\pphi_{n+1})=d-n-1$ and the continued fraction of $\pphi_{n+1}$ will be of the form $[0,a_{n+1},b'_1,\dots,b'_{n+1}]$ with $\deg a_{n+1}=d-n-1$ and with the $b'_i$ linear polynomials. 
	
	Thus, if $|\n K|\geq\frac{d(d+1)}2$, there exist $\lambda_1,\dots,\lambda_{d-1}\inn\n K$ that are not zeros of $f$ and such that $$K\left(\frac{(T-\lambda_1)\cdots(T-\lambda_{d-1})}{f}\right)=1.$$
\end{proof}

\begin{ex}
	As in the previous Example, let $f=T^3+\ov2\inn\n F_7[T]$. 
	
	$f$ is irreducible over $\n F_7$, so, in the notations of the previous proof, we can take $\lambda_1=0$; then $\pphi_1=\frac{T}{T^3+\ov2}=\left[0,T^2,\ov4T\right]$, so $q_0=1, q_1=T^2, q_2=\ov4T^3+\ov1$. 
	
	Then, we will have to take $\lambda_2\neq0$. In particular, for $\lambda_2=-1$, we will get again that $K\left(\frac{T^2+T}{T^3+\ov2}\right)=1$. 
	
	However, using this method, we can find only 19 of the 42 monic polynomials $g$ of degree $2$ such that $g/f$ is normal, the others being irreducible over $\n F_7$.
\end{ex}

\begin{rem}
	The previous condition on $\deg f$ is sufficient but not necessary. 
	
	For example, let $\n K=\n F_5$ and let $f=T^5-T$. Then we have $|\n K|\not\geq\frac{d(d+1)}2$, however $f$ has positive orthogonal multiplicity (actually, $m(T^5-T)=400$), for example $\frac{T^4+T^3+\ov2 T^2-T-\ov2}{T^5-T}=\left[0,T-\ov1,-T+\ov1,T-\ov1,\ov2T-\ov2,\ov2T-\ov2\right]$.
	
	In this case Blackburn's method would still provide all the desired polynomials $g$, while our method would fail, as there are no linear polynomials in $\n F_5[T]$ relatively prime to $f$.
\end{rem}

\subsection{Finite fields}
Let $\n K=\n F_q$ be a finite field. We have already seen that every polynomial $f\inn\n K[T]$ of degree $d$ with $d(d+1)\leq2q$ has positive orthogonal multiplicity. Friesen greatly improved this result; however, his method, differently from the previous ones, does not lead to the explicit construction of a polynomial $g$ such that $g/f$ is normal.

\begin{theo}[Friesen, Theorem 1 in \cite{friesen2007rational}]\label{4theo:Friesen}
	Let $f\inn\n F_q[T]$ be a non-constant polynomial of degree $d$. If $$d\leq q/2$$ then $m(f)>0$, that is, there exists $g\inn\n F_q[T]$ such that $g/f$ is normal.
\end{theo}

\begin{lem}[Friesen]\label{4lem:zar}
	Let $a_0,\dots,a_i\inn\n F_q[T]$ be non-constant polynomials and let $A=\sum_i\deg a_i$. If $f$ is a polynomial of degree $d\geq2A$, there exist exactly $q^{d-2A}$ polynomials $g\inn\n F_q[T]$ with $\deg g<d$ such that $f/g=[a_0,\dots,a_i,\dots]$. More precisely, there exists a polynomial $g\inn\n F_q[T]$ with $\deg g<d$ such that $$\frac f{g+P}=[a_0,\dots,a_i,\dots]$$ for every $P\inn\n F_q[T]$ with $\deg P<d-2A$.
\end{lem}

\begin{proof}[Proof of Lemma \ref{4lem:zar}]
	 By \eqref{1eq:juxtap}, there exists $g$ such that the first partial quotients in the continued fraction expansion of $f/g$ are $a_0,\dots,a_i$. Indeed, let $u/v=[a_i,\dots,a_0]$. Certainly $u,v$ are relatively prime polynomials, so there exist $r,s\inn\n F_q[T]$ such that $ur+vs=f$. Choosing $s$ of minimal degree, $\deg s<A$, necessarily we must have $\deg r\geq A$, so in particular $\deg r>\deg s$. Then, setting $\frac rs=[b_0,\dots,b_j]$ by \eqref{1eq:juxtap} we will have $[a_0,\dots,a_i,b_0,\dots,b_j]=f/g$, where $g=C_i(a_0,\dots,a_{i-1})r+C_{i-1}(a_1,\dots,a_{i-1})s$. 
	
	As $a_0$ is non-constant, we have $\deg g<\deg f$, so $g/f=[0,a_0,\dots,a_i,\dots]$. Now, by Remark \ref{1rem:confrontocf2}, for every $g_1\inn\n F_q[T]$ we have that the first partial quotients of ${g_1}/f$ are $0,a_0,\dots,a_i$ if and only if $\ord(g/f-{g_1}/f)>-2A$, if and only if $\deg(g-g_1)<d-2A$.
\end{proof}

\begin{rem} 
	In the previous Lemma, $f$ and $g$ are not necessarily coprime.
	
	For example, let $\n F_q=\n F_3$, and let us consider $i=1$, $a_0=T,a_1=T+1$, $f=T^5$. Then ($A=2, d=5$) there are exactly 3 polynomials $g\inn\n F_3[T]$ such that the first two partial quotients of $f/g$ are $a_0,a_1$: $\frac{T^5}{T^4-T^2+T}=\left[T,T+\ov1,T^2-T\right]$ ($f,g$ are not relatively prime), $\frac{T^5}{T^4-T^2+T-\ov1}=\left[T,T+\ov1,-T+\ov1,-T^2-\ov1\right],\ \frac{T^5}{T^4-T^2+T+\ov1}=\left[T,T+\ov1,T,T-\ov1,T\right]$.
\end{rem}

\begin{cor}[Friesen]
	If $f\inn\n F_q[T]$ and $\deg f=d\geq2A$, there exist exactly $q^{d-2A}$ polynomials $g\inn\n F_q[T]$ with $\deg g<d$ such that $f/g=[\dots,a_0,\dots,a_i]$ (and the difference between any two such polynomials has degree at most $d-2A-1$).
\end{cor}
\begin{proof}
	It is enough to apply \eqref{1eq:inv} to the previous Lemma.
\end{proof}

\begin{lem}[Lauder, Proposition 3.9 in \cite{lauder1999continued}]
	If $f\inn\n F_q[T]$ is a polynomial of degree $d$, then $$m(f)\leq(q-1)^{\lceil{d/2}\rceil}q^{\floor{d/2}},$$ where $\floor\cdot$, $\lceil\cdot\rceil$ are, respectively, the floor and ceiling functions.
\end{lem}
\begin{proof}
	If $f$ has even degree $d=2i$, then, by the previous Lemma, for any choice of $i$ linear polynomials $a_0,\dots,a_{i-1}$ there exists polynomial $g$ such that $f/g=[a_0,\dots,a_{i-1},\dots]$ and $g$ is unique. As there are $(q-1)^iq^i$ possible choices for $a_0,\dots,a_{i-1}$, then $m(f)\leq(q-1)^iq^i$.
	
	Analogously, if $f$ has odd degree $d=2i+1$ then for any choice of $a_0,\dots,a_{i-1}$ linear there exist exactly $q$ polynomials $g$ such that $f/g=[a_0,\dots,a_{i-1},\dots]$. However, in the notations of the previous prof, for $s$ of minimal degree, as $\deg s<i$ the continued fraction of $f/g$ will have length at most $2i$, so either $f,g$ are not relatively prime or $K(f/g)>1$. Thus for any choice of $a_0,\dots,a_{i-1}$ we will have to consider only $q-1$ possible polynomials $g$. As before, there are $(q-1)^iq^i$ possible choices for $a_0,\dots,a_{i-1}$, so $m(f)\leq(q-1)^{i+1}q^i$.
\end{proof}

\begin{notat}
	We can consider the polynomial analogue of the Euler's $\phi$ function in $\n F_q[T]$: if $f\inn\n F_q[T]$ is a non-constant polynomial, we can define $\phi(f)$ as the number of polynomials $g$ relatively prime to $f$ and such that $\deg g<\deg f$.
	
	If $f$ is irreducible, then $\phi(f)=q^{\deg f}-1$ and $\phi(f^n)=(q^{\deg f}-1)q^{(n-1)\deg f}$. Moreover, $\phi$ is a multiplicative function, so for every polynomial $f$ we have $\phi(f)\!\geq\!(q-1)^{\deg\! f}\!$.
\end{notat}

\begin{proof}[Proof of Theorem \ref{4theo:Friesen}]
	As before, if $\deg f=d$, it is enough to consider polynomials $g$ with $\deg g<d$; then $g/f$ is normal if and only if all the partial quotients of $f/g$ (included the first one) have degree 1.
	
	Let us assume that $f$ has even degree $d=2i$. Then there are exactly $q^{2i}$ rational functions $f/g$ with $\deg g<\deg f$. 
	
	By Lemma \ref{4lem:zar}, for any $a_0,\dots,a_{i-1}\inn\n F_q[T]$ linear polynomials there exists a unique polynomial $g$ with $\deg g<d$, not necessarily relatively prime to $f$, such that $f/g=[a_0,\dots,a_{i-1},\dots]$. As there are $q^i(q-1)^i$ possible choices for the $a_0,\dots,a_{i-1}$, there are exactly $q^{2i}-q^i(q-1)^i$ polynomials $g$ with $\deg g<d$ such that the continued fraction of $f/g$ does not begin with $i$ linear partial quotients. Analogously, there are exactly $q^{2i}-q^i(q-1)^i$ polynomials $g$ with $\deg g<d$ such that the continued fraction of $f/g$ does not end with $i$ linear partial quotients. Moreover, there are exactly $q^{2i}-\phi(f)$ polynomials $g$ with $\deg g<2i$ not relatively prime to $f$. Thus, there are at least $$q^{2i}-2(q^{2i}-q^i(q-1)^i)-(q^{2i}-\phi(f))$$ polynomials $g$, with $\deg g<d$, relatively prime to $f$ and with $g/f$ normal. By the previous remarks, that quantity is positive if $2(q-1)^iq^i-2q^{2i}+(q-1)^{2i}>0$, which is verified for $q\geq2d$.\par\medskip
	
	Now, let us assume that $f$ has odd degree $d=2i+1$. Let $g$ be a polynomial relatively prime to $f$ and with $\deg g<d$. If the first $i$ and the last $i$ partial quotients in the continued fraction expansion of $f/g$ are linear, then $K(f/g)=1$. We can repeat the previous reasoning, but now for any choice of $i$ linear polynomials $a_0,\dots,a_i$ there will exist exactly $q$ polynomials $g\inn\n F_q[T]$ such that the continued fraction expansion of $f/g$ begins (respectively, ends) with $a_0,\dots,a_i$. Thus, we will have that there exists $g\inn\n F_q[t]$ such that $g/f$ is normal if $$q^{2i+1}-2(q^{2i+1}-q^{i+1}(q-1)^i)-(q^{2i+1}-\phi(f))>0,$$ which is certainly true if $2q^{i+1}(q-1)^i-2q^{2i+1}+(q-1)^{2i+1}>0$, verified for $q\geq2d$. 
\end{proof}

Refining this proof, Friesen could give an even better result in the case where $f$ is an irreducible polynomial:
\begin{theo}[Friesen, Theorem 2 in \cite{friesen2007rational}]\label{4theo:Friesen2}
	Let $f\inn\n F_q[T]$ be an irreducible polynomial of degree $d$. If $$d\leq q,$$ then $m(f)>0$, that is, there exists $g\inn\n F_q[T]$ such that $g/f$ is normal.
\end{theo}

\begin{proof}
	As $f$ is irreducible, we have $\phi(f)=q^d-1$.
	
	If $d=2i$ is even, following the proof of the previous Theorem we will have that there exists $g$ such that $g/f$ is normal if $q^{2i}-2(q^{2i}-q^i(q-1)^i)-1>0$, while, if $d=2i+1$ is odd, as before, it is sufficient to have $-q^{2i+1}+2q^{i+1}(q-1)^i-1>0$. It can be seen that both inequalities are verified for $q\geq d$. 
\end{proof}

We can give a similar result for the polynomials $f$ that split completely over $\n F_q$, based on Lemma \ref{1lem:(f+a)/g}, where we have proved that if $-a\neq\cv n(\lambda)$ for every $n$, where $\cv n$ are the convergents of $\alpha\inn\n L$, then $K(\frac{\alpha+a}{T-\lambda})=K(\alpha)-1$ (unless $\alpha$ is already normal). Actually, in this case we will give an explicit method to construct a polynomial $g$ such that $g/f$ is normal.
\begin{prop}\label{4prop:zarsplits}
	Let $f\inn\n F_q[T]$ be a polynomial of degree $d$, let $\n F'$ be an algebraic extension of $\n F_q$ such that $f$ splits completely over $\n F'$ and such that $|\n F'|>d$. Then $f$ has positive orthogonal multiplicity over $\n F'$, that is, there exists $g\inn\n F'[T]$ such that $g/f$ is normal.
\end{prop}

\begin{proof}
	Let $\n F$ be a splitting field of $f$, let $f=k(T-\lambda_1)\cdots(T-\lambda_d)$ with $\lambda_1,\dots,\lambda_d\inn\n F$. Let $\frac{g_1}{f_1}=\frac{b_1}{T-\lambda_1}$ with $b_1\inn(\n F)^*$. Then we have $K({g_1}/{f_1})=1$. Let $\frac{g_2}{f_2}=(\frac{b_1}{T-\lambda_1}+b_2)\frac1{T-\lambda_2}$; by Lemma \ref{1lem:(f+a)/g}, as soon as $b_2\neq0,-b_1(\lambda_2-\lambda_1)^{-1}$, we will have $K({g_2}/{f_2})=1$, where $b_1+b_2(T-\lambda_1)$, $(T-\lambda_1)(T-\lambda_2)$ are relatively prime. 
	
	By inductive hypothesis let us assume that, for some $n<d$, there exist $b_1,\dots,b_n\inn\n F$ such that $K({g_i}/{f_i})=1$ for $i=1,\dots,n$, where $g_i=g_{i-1}+b_i(T-\lambda_1)\cdots(T-\lambda_{i-1})$ and $f_i=(T-\lambda_1)\cdots(T-\lambda_i)$, that is, $\frac{g_i}{f_i}=\left(\frac{g_{i-1}}{f_{i-1}}+b_i\right)\frac1{T-\lambda_i}$; let us also assume that $g_i,f_i$ are relatively prime for every $i$. Again, by Lemma \ref{1lem:(f+a)/g}, as soon as $|\n F|>n+1$ there exists $b_{n+1}$ such that $f_{n+1},g_{n+1}$ are relatively prime and $K\Big(\frac{g_{n+1}}{f_{n+1}}\Big)=1$. Indeed, it is enough to choose $b_{n+1}$ different from $-\cv i(\lambda_{n+1})$ for every $i$, where the $\cv i$ are the convergents of $g_n/f_n$. 
	
	Then, possibly extending $\n F$ to a field $\n F'$ such that $|\n F'|>d$, there will exist $b_1,\dots,b_d$ such that ${g_d}/{f_d}$ is normal, where $kf_d=f$ and $$g_d=b_1+b_2(T-\lambda_1)+\cdots+b_d(T-\lambda_1)\cdots(T-\lambda_{d-1})\inn\n F'[T].$$ 	
\end{proof}

\begin{cor}\label{4cor:zarsplits}
	If $f\inn\n F_q[T]$ is a polynomial of degree $d$ that splits completely and $$d<q,$$ then $m(f)>0$, that is, there exists $g\inn\n F_q[T]$ such that $g/f$ is normal.
\end{cor}

\begin{rem}
	More generally, in the hypotheses of Proposition \ref{4prop:zarsplits}, $f$ certainly has positive orthogonal multiplicity over $\n F_q$ if there exist $b_1,\dots,b_d$ as before satisfying also $\begin{cases}b_d\inn\n F_q\\b_d(-\lambda_1-\cdots-\lambda_{d-1})+b_{d-1}\inn\n F_q\\\cdots\\(-1)^{d-1}\lambda_1\cdots\lambda_{d-1}b_d+\cdots-\lambda_1b_2+b_1\inn\n F_q \end{cases}$, where $\lambda_1,\dots,\lambda_d$ are the roots of $f$.
\end{rem}

\begin{ex}
	Let us consider $f=T^4-1=\prod_{i=1}^4(T-\ov i)\inn\n F_5[T]$. In the notations of the previous proof, let us choose $\lambda_i=\ov i$ and $b_1=\ov1$.
	
	The convergents of $1/(T-1)$ are simply $\ov0,\frac {\ov1}{T-\ov1}$; substituting $\lambda_2=\ov2$ we get that we will have to take $b_2\neq\ov0,-\ov1$; let us choose $b_2=\ov1$.
	
	Now, $\frac{g_2}{f_2}=\frac T{T^2+\ov2T+\ov2}=[0,T+\ov2,\ov3T]$. Evaluating its convergents in $\lambda_3=\ov3$ we find that we will have to take $b_3\neq\ov0,\ov1$; let us choose $b_3=-\ov1$.
	
	Then, $\frac{g_3}{f_3}=\frac {-T^2-T+\ov3}{T^3-T^2+T-\ov1}=[0,-T+\ov2,-T+\ov2,\ov3T-\ov1\,]$. As before, we will have to take $b_4\neq\ov0,\ov2,\ov3$; choosing $b_4=\ov1$ we will get $\frac{T^3+\ov3T^2+\ov2}{T^4-\ov1}=[0,T+\ov2,-T-\ov1,\ov3T-\ov1,-T+\ov1\,]$.
\end{ex}

\begin{rem}
	For a given polynomial $f$, the polynomials $g$ such that $g/f$ is normal and that can be found following the previous algorithm depend on the order chosen on the roots of $f$. In particular, there exist polynomials $f$ with positive orthogonal multiplicity such that this method provides solutions only for some choices of the order of the $\lambda_i$. For example, for $f=T^5-T^3\inn\n F_3[T]$ with $\lambda_1=\lambda_2=\lambda_3=\ov0,\lambda_4=\ov1$ and $\lambda_5=-\ov1$ the previous algorithm does not provide any solution, while choosing the order $\lambda_1=-\ov1,\lambda_2=\ov0,\lambda_3=\ov1,\lambda_4=\lambda_5=\ov0$ we can get, for example, that $g/f$ is normal for $g=T^4-T^3+T^2+T-\ov1$.
	
	Moreover, in general this method will not allow to find all the polynomials $g$ such that $g/f$ is normal. For example, it can be seen that $f=T^3\inn\n F_3[T]$ has orthogonal multiplicity $8$ but following this algorithm we may find only $4$ suitable polynomials $g$.
\end{rem}

\begin{rem}
	The previous conditions on $d=\deg f$, that guarantee that $m(f)>0$, are definitely not necessary. In fact, it can be seen computationally that every polynomial of degree $d$ has positive orthogonal multiplicity at least in the following cases:
	$$\begin{array}{ccc}\n K=\n F_3&\text{and}&d\leq11\\\n K=\n F_4&&d\leq7\\\n K=\n F_5&&d\leq8\\\n K=\n F_7&&d\leq8\end{array}$$
\end{rem}

\subsection{\texorpdfstring{$\n K=\n F_2$}{K=F2}}
The case $\n K=\n F_2$ should be analysed separately from all the others. Indeed, as we have already seen, when the base field is $\n F_2$ many of the previous results do not hold or are trivial. On the other hand, working over this field it is possible to give more precise characterizations of normal Laurent series. This problem has been studied by different authors, such as Baum and Sweet, Lauder, Blackburn or Mesirov.

\begin{theo}[Baum, Sweet,\cite{baum1977badly}, page 577]
	Let $\alpha=\sum\limits_{i<0}c_iT^i\inn\n F_2((T^{-1}))$. Then $\alpha$ is normal if and only if $c_{-1}=1$ and $c_{-i}+c_{-2i}+c_{-2i-1}=0$ for every $i\geq1$.
\end{theo} 

\begin{prop}[Blackburn, Proposition 2 in \cite{blackburn1998orthogonal}]
	Let $f\inn\n F_2[T]$ be a non-constant polynomial. Then the orthogonal multiplicity of $f$ is either $0$ or $2^k$, where $k$ is the number of distinct non-linear irreducible factors of $f$.
\end{prop}

By Lemma \ref{4lem:heuristicsZar}, this implies that for every positive integer $d$ there exist polynomials $f\inn\n F_2[T]$ of degree $d$ such that $m(f)=0$.

As the orthogonal multiplicity of a polynomial in $\n F_q[T]$ is a multiple of $q-1$, this is the only case where there may exist polynomials with orthogonal multiplicity $1$. Moreover, by the previous Proposition, a polynomial may have orthogonal multiplicity 1 only if it splits completely in linear factors. This problem has been completely solved by Blackburn:
\begin{prop}[Blackburn, Theorem 1 in \cite{blackburn1998orthogonal}]
	Let $f\inn\n F_2[T]$. Then $f$ has orthogonal multiplicity $1$ if and only if $f=T^{m_1}(T+1)^{m_2}$ with $\binom{m_1+m_2}{m_1}$ even.
\end{prop}

Lauder, resuming the work of Mesirov and Sweet, showed that any irreducible polynomial has positive orthogonal multiplicity:
\begin{prop}[Lauder, Proposition 3.17 in \cite{lauder1999continued}]
	If $f\inn\n F_2[T]$ is a power of an irreducible non-linear polynomial then $m(f)=2$. 
\end{prop}

Moreover, he proved the following result towards Zaremba's strong Conjecture over $\n F_2$ (that is, with $z_{\n F_2}=2$):
\begin{prop}[Lauder, proposition 4.9 in \cite{lauder1999continued}]
	If $f\inn\n F_2[T]$ splits completely into linear factors, then there exists $g\inn\n F_2[T]$ relatively prime to $f$ and such that $K(g/f)\leq2$.
\end{prop}
\clearpage{\pagestyle{empty}\cleardoublepage}
\chapter{Polynomial analogue of McMullen's conjecture}\label{ch:MM}
As in Chapter \ref{ch:quadr}, let $\n K$ be a field with characteristic different from $2$ and let $\S K$ be the set of polynomials $D\inn\n K[T]$ of even degree, non squares in $\n K[T]$ and whose leading coefficient is a square in $\n K$; 
we will usually denote the degree of such a polynomial by $2d$. We have seen (Lemma \ref{1lem:rad}) that every quadratic irrationality of the form $\alpha=\frac{A+B\sqrt D}C$, with $D\inn\S K$ and with $A,B,C$ polynomials, $B,C\neq0$, is well defined in $\n L=\n K((T^{-1}))$ and that (Remark \ref{2rem:sqrtDred}) the degrees of its partial quotients are eventually bounded in terms of the degrees of $B,C$ and $D$: $\ov K(\alpha)\leq(d+\deg BC)$. Moreover, if $\n K$ is an algebraic extension of a finite field, then the continued fraction expansion of any quadratic irrationality is periodic and $K(\sqrt D)=\ov K(\sqrt D)=d$. 

As in \ref{1notat:equiv}, we will say that $\alpha,\beta\inn\n L$ are \textit{equivalent}, and we will write $\alpha\sim\beta$, if they are $\GL_2(\n K[T])$-equivalent, that is, if there exists $M\inn\GL_2(\n K[T])$ such that $\alpha=M\beta$. By the polynomial analogue of Serret's Theorem \ref{1theo:serret}, $\alpha,\beta$ are equivalent if and only if the continued fractions of $\alpha$ and $k\beta$ eventually coincide for some non zero constant $k$.\par\medskip

An analogue of McMullen's Theorem \ref{3theo:MM} holds also in the polynomial case, so it will have a sense to consider an analogue of his Conjecture \ref{3conj:MM}.

\begin{theoC}\label{5theo:MM2pol}
	For every polynomial $D\inn\S K$ with $\deg D=2d$ there exist infinitely many pairwise non-equivalent elements $\alpha$ of $\n K(T,\sqrt{D(T)})$ such that $$K(\alpha)\leq d.$$
\end{theoC}

In \cite{mercat2013construction}, Th\'{e}or\`{e}me 1.2, Mercat gave an alternative proof of Theorem \ref{3theo:MM}, considering purely periodic, quasi-palindromic continued fractions, that is, real continued fractions of the form $\alpha=[\ov{a_0,a_1,a_2,\dots,a_2,a_1}]\inn\n Q(\sqrt d)$. He proved that for every such $\alpha\inn\n Q(\sqrt d)$ we can construct infinitely many pairwise non-equivalent elements of $\n Q(\sqrt d)$ whose continued fraction is purely periodic with partial quotients bounded by a constant depending only on the $a_i$. His methods can be easily adapted to the polynomial setting:

\begin{propC}\label{5prop:Mercat}
	Let $\alpha=[\ov{a_0,a_1,\dots,a_1}]\inn\n K(T,\sqrt D)$ be a purely periodic, quasi-palindromic continued fraction. Then there exist $b_1,\dots,b_i,c_1,\dots,c_j\inn\n K[T]$, polynomials (possibly constant) of degree at most $K(\alpha)$, such that $$\alpha_n=\left[\ov{b_1,\dots,b_i,(a_0,\dots,a_1)^n,c_1,\dots,c_j,(a_1,\dots,a_0)^n}\right]\inn\n K(T,\sqrt D) \text{ for every } n$$ (where $[\cdots(m_1,\dots,m_k)^n\cdots]$ means that the sequence of partial quotients $m_1,\dots,m_k$ is repeated $n$ times). 
	
	Moreover, the $\alpha_n$ are pairwise non-equivalent.
\end{propC}

\begin{remC}
	Let $\alpha$ be a quadratic irrationality with purely periodic continued fraction expansion, $\alpha=[\ov{a_0,a_1,\dots,a_n}]$. Then $\alpha=A\alpha$ for $A=M_{(a_0,a_1,\dots,a_n)}$; so, as in lemma \ref{2lem:eigen}, $\alpha\inn\n K(T,\sqrt D)$, where $D$ is the discriminant of the characteristic polynomial of $A$, that is, $D=\text{discr}(A)=\tr(A)^2-4\det(A)$.
	
	We can also allow the $b_k,c_l$ be constant for some $k,l$ (actually, this will be the case in the construction provided in the following proof), as long as the corresponding infinite continued fraction converges in $\n L$. However, in this case it is possible that the regular continued fraction expansions of the $\alpha_n$ are no longer periodic. 
	
	Setting $A=M_{(a_0,a_1,\dots,a_1)}$, the previous Proposition is then equivalent to the existence in the monoid generated by the identity and by the matrices $M_a$ with $a$ a polynomial of degree at most $K(\alpha)$ of matrices $B,C$ such that for every $n$ the continued fraction associated to $BA^nC(A^t)^n$ converges in $\n L$ and such that, for every $n$, $\text{discr}(BA^nC(A^t)^n)=P_n^2\,\text{discr}(A)$ for some polynomial $P_n$. 
	
	As we have already seen in Remark \ref{1rem:constants}, the continued fraction converges if all the $b_k,c_l$ have positive degree, if there are non-consecutive non-zero constants or even if there are only pairs of consecutive constants whose product is different from -1.  
\end{remC}

\begin{proof}[Sketch of the Proof of Proposition \ref{5prop:Mercat}]
	It can be proved that there exist matrices $B,C$ in the monoid generated by the $M_a$ with $\deg a\leq K(\alpha)$ and there exists a matrix $H\inn M_2(\n K[T])$ of rank 1 such that $\tr(BMCM^t)=(\tr(HM))^2-2\det M$ for every matrix $M\inn M_2(\n K[T])$ and such that $\tr(HA^n)=\tr(A^{n+2})$ for every $n\geq0$. Indeed, setting $N=M_{(a_2,\dots,a_2)}$, we can take 
	\begin{multline}\nonumber B=M_{a_1}NM_{(a_1,1,1,a_1-1)}NM_{a_1},\ C=AM_{(a_0/2,1,1,a_0/2-1)}A^t,\ H=A\mm0{a_0}02A \\\text{ if } \det N=-1;\end{multline}\vspace{-1cm}
	\begin{multline}\nonumber B=M_{a_1}NM_{a_1}M_{(a_0/2,1,1,a_0/2-1)}M_{a_1}NM_{a_1},\ C=M_{(a_0,a_1)}NM_{(a_1,1,1,a_1-1)}NM_{(a_1,a_0)},\\ H=A\mm2{-a_0}00A \text{ if } \det N=1.\end{multline}
	In particular then $\tr(BA^nC(A^t)^n)=(\tr(HA^n))^2-2\det A^n$ for every $n\geq0$ and thus  $\text{discr}(BA^nC(A^t)^n)=(\tr(HA^n))^2((\tr(HA^n))^2-4\det A^n)=(\tr(A^{n+2}))^2\text{discr}(A^{n+2})$, which coincides, up to multiplication by a square, to the discriminant of $A$. 
	
	By construction, the continued fractions corresponding to the previous matrices converge to quadratic irrationalities, pairwise non-equivalent, and the degrees of their partial quotients are bounded by $\max_i\deg a_i=K(\alpha)$ (we have seen in Remark  \ref{1rem:Kconstants} that constant partial quotients can be removed without increasing the degrees of the other partial quotients). 
	
	It is easy to adapt the previous proof also to the cases $A=M_{a_0}$ (as in the following Example) or $A=M_{(a_0,a_1)}$.
\end{proof}

\begin{exC}\label{5ex:mercat}
	Let $\alpha=\left[\,\ov{T}\,\right]=3\left(T+\sqrt{T^2+\ov4}\right)\inn\n F_5((T^{-1}))$. Setting, in the previous notations (with $N=M_{a_1}=\text{Id}$), $B=M_{\left(\ov3T,\ov1,\ov1,\ov3T-\ov1\right)}$ and $C=M_{\left(T,\ov1,\ov1,T-\ov1\right)}$, we will have  $\alpha_n\!=\!\left[\ov{\ov3T,\ov1,\ov1,\ov3T-\ov1,(T)^n,T,\ov1,\ov1,T-\ov1,(T)^n}\right]\inn\n F_5\!\left(T,\sqrt{T^2+\ov4}\right)$ for every $n\geq0$. \begin{align}\nonumber \text{ For example, }& \alpha_0=\frac{\left(T+\ov1\right)\left(T^3-T^2+T+\ov1\right)+T\left(T^2+\ov2\right)\sqrt{T^2+\ov4}}{-T^3+T^2-\ov2T+1},\\\nonumber &\alpha_1=\frac{\left(T-\ov1\right)\left(T+\ov1\right)^2\left(T+\ov2\right)^2+\left(T+\ov2\right)\left(T+\ov3\right)\left(T^2+\ov3\right)\sqrt{T^2+\ov4}}{-T^4+T^3+T^2+\ov2T+\ov2},\\\nonumber &\alpha_2\!=\!\frac{\left(T^8\!-\!T^6\!-\!T^5\!+\!\ov2T^3\!+\!\ov2T^2\!+\!T\!+\!\ov1\right)\!+\!T\left(T^2\!+\!\ov2\right)\left(T^4\!-\!T^2\!+\!\ov2\right)\sqrt{T^4\!+\!\ov4}}{-T^7+T^6-T^5-T^4-T^2+T+\ov1},\\\nonumber &\dots\end{align}
\end{exC}

\begin{proof}[Proof of Theorem \ref{5theo:MM2pol}]
	Let us assume that the polynomial Pell equation for $D$ has no non-trivial solutions. Then $\n K$ is an infinite field, $K(\sqrt D)<d$ and, by Corollary \ref{1cor:degrees}, for every $\lambda\inn\n K$ we must have $K((T-\lambda)\sqrt D)\leq d$, where the $\alpha_\lambda=(T-\lambda)\sqrt D$ are pairwise non-equivalent.
	
	On the other hand, if the Pell equation for $D$ has non-trivial solutions, by Theorem \ref{2theo:sqrtD}, $K(\sqrt D)=d$ and the continued fraction of $\sqrt D+\floor{\sqrt D}$ is purely periodic and quasi-palindromic, so by Proposition \ref{5prop:Mercat} we can construct infinitely many, pairwise non-equivalent elements $\alpha_n\inn\n K(T,\sqrt D)$ such that $K(\alpha_n)=d$.
\end{proof}

We can then ask if, in the statement of Theorem \ref{5theo:MM2pol}, $d$ can be replaced by a constant independent of $D$, that is, we can consider the following analogue of McMullen's Conjecture:
\begin{lconj}{M}[Polynomial analogue of McMullen's Conjecture]\label{5conj:MMinf}
	There exists a constant $m_{\n K}$ (possibly depending on the base field $\n K$) such that for every polynomial $D\inn\S K$ there exist infinitely many pairwise non-equivalent quadratic irrationalities $\alpha\inn\n K(T,\sqrt D)$ such that $$K(\alpha)\leq m_{\n K}.$$
\end{lconj}

By analogy with \ref{3conj:MM}, we will call this statement \textit{McMullen's Conjecture over $\n K$}.

	If $\alpha$ is a quadratic irrationality in $\n K(T,\sqrt D)$, so are its complete quotients $\alpha_n$ and we have already seen that $K(\alpha_n)=\ov K(\alpha)$ for every large enough $n$. Certainly, all the complete quotients $\alpha_n$ are equivalent to $\alpha$. Thus, McMullen's Conjecture over $\n K$ is equivalent to:\par\medskip

	\textit{There exists a constant $m_{\n K}$ (possibly depending on the base field $\n K$) such that for every polynomial $D\inn\S K$ there exist infinitely many pairwise non-equivalent quadratic irrationalities $\alpha\inn\n K(T,\sqrt D)$ such that $$\ov K(\alpha)\leq m_{\n K}.$$}

It is believed that it should be enough to take \be\label{5eq:MMstrong}m_{\n K}=1\ee for every field $\n K$, that is, it is thought that for every polynomial $D\inn\S K$, the field $\n K(T,\sqrt D)$ has infinitely many pairwise non-equivalent normal elements. We will call \eqref{5eq:MMstrong} \textit{McMullen's strong Conjecture over $\n K$}.
\par\medskip

A first step in the direction of Conjecture \ref{5conj:MMinf} could be the proof of the following weaker statement: 
\begin{conjC}\label{5conj:MMfin}
	There exists a constant $m'_{\n K}$ (possibly depending on the base field $\n K$) such that for every polynomial $D\inn\S K$ there exists $\alpha_D\inn\n K(T,\sqrt D)$ such that $K(\alpha_D)\!\leq\! m'_{\n K}$.
\end{conjC}

If $\n K$ is an infinite field and the previous Conjecture holds for some constant $m'_{\n K}$ then, reasoning as in the proof of Theorem \ref{5theo:MM2pol}, we get that McMullen's Conjecture \ref{5conj:MMinf} holds with $m_{\n K}=m'_{\n K}+1$. Thus, over infinite fields, Conjectures \ref{5conj:MMinf} and \ref{5conj:MMfin} are equivalent.\par\medskip

We have seen that when $\n K$ is not an algebraic extension of a finite field, a generic polynomial $D$ should not be Pellian, so it has a sense to consider the following stronger Conjecture:
\begin{conjC}\label{5conj:mult}
	Let $\n K$ be a field which is not an algebraic extension of a finite field. Then for every polynomial $D\inn\S K$ there exist infinitely many monic polynomials $P$ such that $$\ov K(P\sqrt D)=1.$$
\end{conjC} 
As $P\sqrt D, Q\sqrt D$ are equivalent if and only if $P=kQ$ with $k\inn\n K^*$, the previous Conjecture immediately implies McMullen's strong Conjecture \eqref{5eq:MMstrong}.\par\medskip

Applying the results on the multiplication of a continued fraction by a linear polynomial seen in Proposition \ref{1prop:prod} and Corollary  \ref{1cor:decreasedegree}, it will be easy to see that Conjecture \ref{5conj:mult}, and thus McMullen's strong Conjecture \eqref{5eq:MMstrong}, hold over every uncountable field (Theorem \ref{5theo:MMC}).

In section \ref{5sec:5.3} we will study the reduction of formal Laurent series modulo a prime. This will allow us to show that Conjecture \ref{5conj:mult} holds also over $\ch Q$ (Corollary \ref{5cor:MMchQ}). Actually, these techniques will also allow us to prove directly McMullen's strong Conjecture \eqref{5eq:MMstrong} over $\ch Q$ (Proposition \ref{5prop:MMchq1}). We will then give another proof of \eqref{5eq:MMstrong} over $\ch Q$, based on Lemma \ref{1lem:(f+a)/g}.

It will follow directly (Theorem \ref{66theo:MMQ}) from a theorem of Zannier \cite{zannier2016hyperelliptic} that Conjecture \ref{5conj:mult} holds over every number field, that is, over every finite extension of $\n Q$. To prove this result, he applied an analogue for algebraic groups of Skolem-Mahler-Lech's Theorem to a suitable generalized Jacobian of an Hyperelliptic curve; this is why we will review in section \ref{66sec:gen} the theory of generalized Jacobians associated to a modulus  
\par\medskip

On the other hand, over algebraic extensions of finite fields we cannot look for normal quadratic irrationalities of the form $P\sqrt D$, as their continued fractions will always be periodic, with infinitely many partial quotients of degree $d+\deg P$.

Adapting to the polynomial setting a result of Mercat (Theorem \ref{5theo:Mercat}), we will prove a connection between the polynomial analogues of Zaremba's and McMullen's Conjectures (Conjectures \ref{4conj:zaremba}, \ref{5conj:MMinf}), which will allow us to see that McMullen's strong Conjecture \eqref{5eq:MMstrong} holds over every infinite algebraic extension of a finite field (Corollaries \ref{5cor:ZarimpMM}, \ref{5cor:MMFpbar}). On the other hand, over finite fields we will only have that McMullen's Conjecture follows from Zaremba's; however, in this case, even the existence of normal elements in every quadratic extension is still an open problem.

\section{\texorpdfstring{$\n K=\ch F_p$: a connection between Zaremba's and McMullen's Conjectures}{Algebraic closure of Fp: a connection between Zaremba's and McMullen's Conjectures}}
As we have already mentioned (Theorem \ref{3theo:Mercat}), in \cite{mercat2013construction} Mercat proved that in the real case Zaremba's Conjecture implies McMullen's Conjecture. Actually, assuming Zaremba's Conjecture with a constant $z$, his method allows to explicitly construct, for every positive, non square $d$, infinitely many purely periodic elements of $\n Q(\sqrt d)$ whose partial quotients are bounded by $z+1$.

His proof can be easily adapted to the polynomial case but it requires the existence of non-trivial solutions to the Pell equation, which always occur in the real setting and over algebraic extensions of finite fields but is unlikely in the other cases. Combining this with the fact that Zaremba's (strong) Conjecture holds over every infinite field, we will have that McMullen's (strong) Conjecture holds over every infinite algebraic extension of a finite field.

\begin{theo}[polynomial analogue of Theorem \ref{3theo:Mercat}]\label{5theo:Mercat}
	Let $D\inn\S K$ be a Pellian polynomial; let $(X,Y)$ be a non-trivial solution to the Pell equation for $D$, with $$X^2-DY^2\!=t\inn\{\pm1\}.$$ Let us assume that we have chosen the sign of $Y$ so that $\ord(X+Y\sqrt D)<0$. Let $Z$ be a polynomial relatively prime to $X$ and with $\deg Z<\deg X$; let $Z/X=[0,a_1,\dots,a_n]$ and let $k\inn\n K^*$ be the constant such that, in the notations of \ref{1notat:C_n}, $\begin{cases} C_{n+1}(0,a_1,\dots,a_n)\!=\!kZ\\C_n(a_1,\dots,a_n)=kX\end{cases}\!\!\!\!$. Then $$\frac{X-kZ+Y\sqrt D}{kX}=\left[\ov{2/k,-a_1,\dots,-a_n,(-1)^{n+1}2t/k,a_n,\dots,a_1}\right].$$ In particular, $$K\left(\frac{X-kZ+Y\sqrt D}{kX}\right)\leq K\left(\frac ZX\right).$$ 
\end{theo}

\begin{proof}
	Let $\alpha=\frac{X-kZ+Y\sqrt D}{kX}$; it is enough to show that $\vv\alpha1$ is an eigenvector of the matrix $M$ canonically associated to $\pphi=\left[\ov{2/k,-a_1,\dots,-a_n,(-1)^{n+1}2t/k,a_n,\dots,a_1}\right]$. Indeed, by Remark \ref{2rem:eigen}, in this case $\pphi$ must be equal either to $\alpha$ or to its conjugate $\alpha'$; as the polynomial part of $\pphi$ is $2/k$, it follows from our choice of the sign of $Y$ that $\pphi=\alpha$. 
	
	Let $\epsilon=(-1)^n$ and let us denote by $p_i/q_i$, for $i=1,\dots,n$ the convergents of $X/Z\!=\!\![a_1,\dots,a_n]$, where in this case, contrary to the usual notation, $p_i\!=\!C_i(a_1,\dots,a_i)$ and $q_i=C_{i-1}(a_2,\dots,a_i)$. 
	
	Then, by Lemma \ref{1lem:multA}, $M=M_{2/k}\mm{\epsilon\, p_n}{-\epsilon\, p_{n-1}}{-\epsilon\, q_n}{\epsilon\, q_{n-1}}M_{-\epsilon2t/k}\mm{p_n}{q_n}{p_{n-1}}{q_{n-1}}$ and it is easy to see that $M\vv\alpha1=(1-2 tX^2-2 tXY\sqrt D)\vv\alpha1$.
	
	Of course, the given continued fraction expansion for $\alpha$ is not regular; however, as we have already seen in Remark \ref{1rem:Kconstants}, the degrees of the partial quotients of the regular continued fraction expansion of $\alpha$ will still be bounded by $K(Z/X)$. 
\end{proof}

We have seen that if the Pell equation for a polynomial $D\inn\S K$ has non-trivial solutions, then it has infinitely many essentially different solutions. Thus, if Zaremba's Conjecture holds over $\n K$, applying the previous result to different solutions of the Pell equation we will find infinitely many pairwise non-equivalent quadratic irrationalities in $\n K(T,\sqrt D)$ whose partial quotients' degrees are bounded by an absolute constant that depends only on $\n K$:

\begin{cor}
	If Zaremba's Conjecture \ref{4conj:zaremba} holds over $\n K$ for some constant $z_{\n K}$ and $D\inn\S K$ is Pellian, then there exist infinitely many, pairwise non-equivalent quadratic irrationalities $\alpha\inn\n K(T,\sqrt D)$ such that $K(\alpha)\leq z_{\n K}$. 
\end{cor}

As any polynomial over a finite field is Pellian, we have
\begin{cor}\label{5cor:ZarimpMM}
	Over every finite field of characteristic different from $2$  McMullen's Conjecture is a consequence of Zaremba's Conjecture (with $m_{\n K}=z_{\n K}$). In particular, McMullen's strong Conjecture would follow from Zaremba's strong Conjecture.
\end{cor}

We have already seen that if $\n K$ is an infinite algebraic extension of a finite field, then Zaremba's strong Conjecture holds over $\n K$  (Corollary \ref{4cor:Zarinf}) and that any polynomial in $\S K$ is Pellian (Corollary \ref{2cor:lagrange}), so:
\begin{cor}\label{5cor:MMFpbar}
	McMullen's strong Conjecture \eqref{5eq:MMstrong} holds over every infinite algebraic extension of a finite field.
\end{cor}

\begin{ex}
	Let $D=T^8+T^4\inn\n Q[T]$. It is easy to see that $D$ is Pellian with $\sqrt D=\left[T^4+\frac12,\ov{-8T^4-4,2T^4+1}\right]$, so in particular the minimal solution to the Pell equation for $D$ is $(X,Y)=(2T^4+1,2)$. As Zaremba's strong Conjecture holds over $\n Q$, there exist (infinitely many) polynomials $Z\inn\n Q[T]$ relatively prime to $X$, with $\deg Z<4$ and such that $K(Z/X)=1$; for example, $\frac X{T^3+T}=\left[2T,-\frac12T,-\frac43T,\frac32T\right].$ Then, by the previous Theorem ($t=1,k=1,\epsilon=1$) $$\left[\ov{2,-2T,-\frac12T,-\frac43T,\frac32T,-2,\frac32T,-\frac43T,-\frac12T,2T}\right]=1+\frac{-T^3-T+2\sqrt D}{2T^4+1}.$$  
	
	Starting from successive solutions to the Pell equation we will get that, for example, 
	$-1+\frac{T^7+T^5+(-8T^4-4)\sqrt D}{8T^8+8T^4+1}$,\\
	$-1+\frac{T^{11}+T^9+(-32T^8-32T^4-6)\sqrt D}{32T^{12}+48T^8+18T^4+1}$,\\
	$1+\frac{-T^{15}-T^{13}+8(2T^4+1)(8T^8+8T^4+1)\sqrt D}{128T^{16}+256T^{12}+160T^8+32T^4+1}$\\ 
	are normal elements of $\n Q(T,\sqrt{D(T)})$.
\end{ex}

\section{\texorpdfstring{$\n K$ an uncountable field}{K an uncountable field}}
Let $D\inn\S K$ be a polynomial of degree $2d$ and let $p_n/q_n$ be the convergents of $\sqrt D$. We have already seen in Remark \ref{2rem:degan-pell} that $K(\sqrt D)\leq d$, where equality holds if and only if $D$ is Pellian. 

\begin{lem}\label{5lem:MMrootsqn}
	If for every polynomial $D\inn\S K$ (not necessarily squarefree) there exists $\lambda\inn\n K$ that is not a root of any of the denominators of the convergents of $\sqrt D$, then Conjecture \ref{5conj:MMfin} holds over $\n K$ with $m'_{\n K}=1$.
	
	More precisely, for every $D\inn\S K$ of degree $2d$ there will exist $\lambda_1,\dots,\lambda_{d-1}\inn\n K$ such that $$K\left((T-\lambda_1)\cdots(T-\lambda_{d-1})\sqrt D\right)=1.$$
	
	If for every polynomial $D$ we can find infinitely many constants $\lambda$ that satisfy the previous condition we will have that Conjecture \ref{5conj:mult}, and thus McMullen's strong Conjecture, hold over $\n K$.
\end{lem}
\begin{proof}
	Let $\alpha=\sqrt D$, with $D\inn\S K$ and $\deg D=2d$. 
	
	By Corollary \ref{1cor:decreasedegree}, if there exists $\lambda_1\inn\n K$ such that $q_n(\lambda_1)\neq0$ for every $n$, then, setting $\alpha_1=(T-\lambda_1)\alpha$, we would have $$K(\alpha_1)=\max\{K(\alpha)-1,1\}\leq d-1$$ (in particular, $(T-\lambda_1)^2D$ is non-Pellian). 
	
	Again, if there exists $\lambda_2$ that is not a root of any denominator of the convergents of $\alpha_1$, then $$K(\alpha_2)\leq d-2, \text{ where } \alpha_2=(T-\lambda_1)(T-\lambda_2)\sqrt D.$$ Going on in this way, if at every step we can find a constant $\lambda_i$ that is not a root of the denominators of the convergents of $\alpha_{i-1}=(T-\lambda_1)\cdots(T-\lambda_{i-1})\sqrt D$, we will have $$K(\alpha_{d-1})=1$$ (actually, already $K(\alpha_{k-1})=1$, with $k=K(\alpha)$).
\end{proof}
	
\begin{theo}\label{5theo:MMC}
	Conjecture \ref{5conj:mult} and McMullen's strong Conjecture \eqref{5eq:MMstrong} hold over every uncountable field (of characteristic different from $2$).
	
	More precisely, if $\n K$ is uncountable, for every polynomial $D\inn\S K$ there exist infinitely many different $(d-1)$-uples $\lambda_1,\dots,\lambda_{d-1}$, where $2d=\deg D$, such that $$K\left((T-\lambda_1)\cdots(T-\lambda_{d-1})\sqrt D\right)=1.$$
\end{theo}
\begin{proof}
	In the proof of the previous Lemma, at every step we have to exclude at most the countably many zeros of the denominators of the convergents. Then if $\n K$ is an uncountable field at every step we can choose uncountably many different $\lambda_i$.
\end{proof}

\begin{lem}\label{5lem:MMrootsqn2}
	If at every step, there exists a constant $\lambda$ that is a zero of at most finitely many denominators of the convergents, then there exist $\lambda_1,\dots,\lambda_{d-1}\inn\n K$ such that $$\ov K\left((T-\lambda_1)\cdots(T-\lambda_{d-1})\sqrt D\right)=1,$$ while finitely many initial partial quotients can have larger degrees.
\end{lem}

The existence of constants that are not roots of any denominator of the convergents of $\sqrt D$, or even that are roots of only finitely many of them, is not at all obvious if $\n K$ is a countable field (and is of course false if $\n K$ is an algebraic extension of a finite field). We will show in section \ref{5sec:5.4} that this is true for $\n K=\ch Q$ and we will see in Theorem \ref{66theo:Zrootsqn} that Zannier \cite{zannier2016hyperelliptic} proved this over every number field.

\section{\texorpdfstring{$\n K=\ch Q$: reduction of Laurent series}{Algebraic closure of Q: reduction of Laurent series}}\label{5sec:5.3}
Of course when $\n K=\ch Q$ we cannot apply directly the previous approach, as a priori it may happen that, for some polynomial $D$, every algebraic number is a root of infinitely many denominators of the convergents to $\sqrt D$. However, we will show that this is not the case: on the contrary, Conjecture \ref{5conj:mult} holds. In order to prove this statement, we will introduce the theory of the reduction of Laurent series modulo a prime, which will allow us to use the results on continued fractions of quadratic irrationalities over finite fields presented in Chapter 2.

\subsection{Reduction of a formal Laurent series modulo a prime}
Let $\nu$ be a discrete valuation of $\n K$, that is, let $\nu:\n K^*\to\n Z$ be a surjective group homomorphism such that $$\nu(a+b)\geq\min\{\nu(a),\nu(b)\} \text{ for every } a,b\inn\n K^*.$$ Let $\C O_{\nu}$ be its valuation ring, $\C O_{\nu}=\{a\inn\n K^*,\ \nu(a)\geq0\}\cup\{0\}$; in particular, $\C O_\nu$ is local with maximal ideal $\C M_{\nu}\!=\!\{a\inn\C O_{\nu},\ \nu(a)>0\}\cup\{0\}$ and $\C O_{\nu}^*\!=\!\C O_{\nu}\setminus\C M_{\nu}\!=\!\{a\inn\C O_{\nu},\ \nu(a)\!=\!0\}$. $\C M_{\nu}$ is principal; a generator $\pi$ of $\C M_\nu$ is called a \textit{uniformizing parameter} for $\nu$. Now, $\n K$ is the field of fractions of $\C O_{\nu}$ and any $a\inn\n K^*$ can be written uniquely as $a=x\,\pi^l$, with $x\inn\C O_{\nu}^*$ and $l\inn\n Z$; actually, $l=\nu(\alpha)$. 

We will denote by $k_{\nu}$ the residue field $k_{\nu}=\C O_{\nu}/\C M_{\nu}$ and we will denote by $\ov a$ the class in $k_{\nu}$ of an element $a$ of $\C O_{\nu}$. 

\begin{defn}
	We will say that a formal Laurent series $\alpha=\sum\limits_{\mathclap{n=-\infty}}^N c_nT^n\inn\n L$ \textit{can be reduced} modulo $\nu$ if $c_n\inn\C O_\nu$ for every $n$. We will denote by $\C L_\nu$ the ring of the formal Laurent series reducible modulo $\nu$, that is, $$\C L_\nu=(\C O_\nu[[T^{-1}]])[T].$$
	
	Let $\n L_\nu$ be the field of formal Laurent series over $k_\nu$, $$\n L_\nu= k_\nu((T^{-1})).$$ Then the reduction homomorphism from $\C O_\nu$ to $k_\nu$ extends naturally to a homomorphism from $\C L_\nu$ to $\n L_\nu$. If $\alpha\inn\C L_\nu$, we will denote by $\ov\alpha$ its image under this homomorphism, $$\ov\alpha=\sum\limits_{\mathclap{n=-\infty}}^N\ov{c_n}\,T^n\inn\n L_{\nu};$$ we will say that $\ov\alpha$ is the \textit{reduction modulo $\nu$} of $\alpha$.
\end{defn}

We will denote by $\ord_{\n K},\ord_{k_\nu}$, respectively, the natural valuations of $\n L$ and of $\n L_\nu$. If $\alpha$ can be reduced modulo $\nu$, then $\ord_{\n K}(\alpha)\leq\ord_{k_\nu}(\ov\alpha)$, and $\ord_{\n K}(\alpha)=\ord_{k_\nu}(\ov\alpha)$ if and only if $\nu(c_N)=0$, with $N=-\ord_{\n K}(\alpha)$.\par\medskip

Then, if $\alpha\inn\C L_\nu$, we can compare the continued fraction expansion of $\alpha$ over $\n K$ with the continued fraction expansion of $\ov\alpha$ over $k_{\nu}$. This problem has already been studied by Van der Poorten in \cite{poorten1998formal}, \cite{poorten1999reduction} or \cite{poorten2005specialisation}; more details can be found in Merkert's thesis \cite{Olaf}.

\begin{ex}
	Let $\n K=\n Q$ and let $\nu_l$ be the $l$-adic valuation, where $l$ is a prime, that is, $\nu_l(l^k\frac ab)=k$ if $l$ is relatively prime to both $a$ and $b$. Then $\C O_l=\left\{a/b,\ l\nmid b\right\}$, $\C M_l=(l)$ and the residue field is $k_l=\n F_l$.
	
	In this case, a formal Laurent series $\alpha=\sum\limits_{\mathclap{n=-\infty}}^Nc_nT^n$ is reducible modulo $l$ if and only if all the denominators of the $c_n$ are relatively prime to $l$.
\end{ex}

\begin{rem}
	Of course $\C L_\nu$ is a ring, that is, if $\alpha,\beta\inn\n L$ can be reduced modulo $\nu$, then $\alpha\pm\beta$ and $\alpha\beta$ are still reducible modulo $\nu$. However, $\C L_\nu$ is not a field: if $\alpha=\sum\limits_{n=0}^Nc_nT^n$ can be reduced modulo $\nu$, with $N=-\ord(\alpha)$, then, by \eqref{1eq:Laurentinv}, $\alpha^{-1}$ can also be reduced modulo $\nu$ if and only if $\nu(c_N)=0$.
	
	Moreover, if $\alpha\inn\C L_\nu$ and $\sqrt\alpha$ is well defined as a Laurent series (that is, if $N$ is even and $c_N$ is a square in $\n K$), by \eqref{1eq:Laurentrad} $\sqrt\alpha$ can be reduced modulo $\nu$ if $\nu(2)\leq0$ and $\nu(c_N)=0$.
\end{rem}

\begin{lem}\label{5lem:conditionsforgr}	
	Let $\n K=\n Q$. Then any rational function and any quadratic irrationality $\frac{A+B\sqrt D}C$, with $A,B,C,D\inn\n Q[T],$ $C\neq0$ and $D\inn\S K$ can be reduced modulo all but finitely many primes. 
\end{lem}
\begin{proof}
	$\frac{A+B\sqrt D}C$ is reducible modulo a prime $l$ as soon as $l$ is different from $2$ and it does not divide neither the (finitely many) denominators of the coefficients of $A,B,C,D$ nor the numerators of the leading coefficients of $C$ and $D$. 
\end{proof}

	Obviously, the previous reasoning remains true when $\n K$ is a finite extension of $\n Q$ and $\nu$ is an extension to $\n K$ of the $l$-adic valuation.

\begin{rem} 
	Let $\alpha=[a_0,a_1,\dots]\inn\n L$. If $\alpha$ and all its partial quotients $a_n$ can be reduced modulo $\nu$, then $\ov\alpha=[\ov{a_0},\ov{a_1},\dots]$. A priori, this may not be the regular continued fraction expansion of $\ov\alpha$, as some of the $\ov{a_n}$ could be constants. Actually, it can be seen that in this case $\deg\ov a_n=\deg a_n$ for every $n\geq1$. 
	
	Roughly, if for some $n\geq1$ we had that $\deg\ov a_n<\deg a_n$, then the leading coefficient of $a_n$ would have positive valuation, which would imply that some coefficients of $\alpha$ as a formal Laurent series have negative valuation. We will give a more precise proof of the fact that reducing the degrees of the partial quotients cannot decrease in Lemma \ref{5lem:reddegrees}. 
	
	Thus, if $\alpha$ and all the $a_n$ are reducible modulo $\nu$, then $$\ov\alpha=[\ov{a_0},\ov{a_1},\dots], \text{ with } \deg\ov{a_n}=\deg a_n \text{ for every } n\geq1 \text{ and } \deg\ov{a_0}\leq\deg a_0.$$\par\medskip
	
	It may happen that all the $a_n$ can be reduced modulo $\nu$ but $\alpha$ can not. For example, $\alpha=[0,5T+1]=\frac 1{5T+1}=\frac15T^{-1}-\frac1{25}T^{-2}+\cdots\inn\n Q((T^{-1}))$ is not reducible modulo 5 but all its partial quotients are. In such a case, contrary to what we said above, we would have that the degree of the partial quotients decreases after the reduction (formally, $\ov\alpha=[\,\ov0,\ov1\,]=\ov1$ ).\par\medskip
 
	On the other hand, it may happen that $\alpha$ can be reduced modulo $\nu$ but not all of the $a_n$ can. In this case it is harder to find the partial quotients of $\ov\alpha$, but Van der Poorten proved that its convergents can still be easily recovered from those of $\alpha$.
\end{rem}

\begin{lem}\label{5lem:ordpnqn}
	Let $\alpha\inn\C L_\nu$ and let $(\cv n)_n$ be its convergents; of course, the $p_n,q_n$ may not be reducible modulo $\nu$ or may reduce to $\ov 0$. 
	
	For every $n\geq0$, let $$h_n=\pi^{i_n}$$ be the unique power of $\pi$ such that $h_nq_n$ can be reduced modulo $\nu$ and $\ov{h_nq_n}\neq\ov0$  (that is, let $-i_n$ be the minimal valuation of the coefficients of $q_n$). Then $h_np_n$ can be reduced modulo $\nu$ too. 	
	
	Moreover, if $\ord\ov\alpha\leq0$ we must also have $\ov{h_np_n}\neq\ov0$. So, in this case, for every $n$ there exists a unique $h_n=\pi^{i_n}$ such that $$h_np_n,h_nq_n \text{ can be reduced modulo } \nu \text{ and } \ov{h_np_n},\ov{h_nq_n}\neq\ov0.$$ 
\end{lem}	
\begin{proof}
	If, by contradiction, $h_np_n$ was not reducible, there would exist $j_n>i_n$ such that $\pi^{j_n}p_n$ can be reduced modulo $\nu$ and $\ov{\pi^{j_n}p_n}\neq\ov0$ (as before, $-j_n$ would be the minimum of the valuations of the coefficients of $p_n$). But in that cas we would have $0\!\leq\!\deg\left(\ov{\pi^{j_n}p_n}\right)\!\!=\!-\ord_{k_\nu}\left(\ov{\pi^{j_n}p_n-\alpha\,\pi^{j_n}q_n}\right)\!\leq\!-\ord_{\n K}(p_n-\alpha q_n)\!<\!0$, contradiction.
	
	Let $\ord\ov\alpha\leq0$ and let us assume by contradiction that $\ov{h_np_n}=\ov0$. Then we would have  $0\geq\ord_{k_\nu}\ov\alpha\geq\ord_{k_\nu}\ov{\alpha\, h_nq_n}=\ord_{k_\nu}\left(\,\ov{h_np_n-\alpha\, h_nq_n}\,\right)>0$, contradiction.
\end{proof}

\begin{lem}
	In the notations of the previous Lemma, let us set $$x_n=h_np_n,\ y_n=h_nq_n.$$  Then the rational functions $\ov{x_n}/\ov{y_n}\inn k_\nu(T)$ are convergents of $\ov\alpha$.
\end{lem}
\begin{proof}
	We have $\ord_{k_\nu}\left(\,\ov{x_n-\alpha y_n}\,\right)\geq\ord_{\n K}(p_n-\alpha q_n)>\deg q_n\geq\deg\ov{y_n}$ so, by Remark \ref{1rem:bestappr}, $\ov{x_n}/\ov{y_n}$ are convergents of $\ov\alpha$.
\end{proof}
	
Van der Poorten proved that the converse holds too, that is, all the convergents of $\ov\alpha$ can be found as reductions of convergents of $\alpha$. However, it is possible that, for some $n$, $\ov{x_n},\ov{y_n}$ are no longer relatively prime and there might exist $m\neq n$ such that $\ov{x_n}/\ov{y_n}=\ov{x_m}/\ov{y_m}$. We will see that if $\ov{x_{n-1}}/\ov{y_{n-1}}\neq\ov{x_n}/\ov{y_n}$, then $\deg \ov{y_n}=\deg q_n$ and $\ov{x_n},\ov{y_n}$ are relatively prime in $k_\nu[T]$, so they give, up to a multiplicative constant, a continuant of $\ov\alpha$.
	
The following results can be found, with different proofs and in slightly different contexts, in \cite{cantor1994continued},\cite{poorten1999reduction},\cite{poorten2005specialisation}. We will show here a slightly simplified version of the proof given by Van der Poorten in \cite{poorten1999reduction} (Theorem 2.1); another version of this proof can be found in Merkert's Ph.D. thesis (Theorem 7.2 in \cite{Olaf}).

\begin{theo}\label{5theo:red}
	Let $\alpha\inn\C L_\nu$ be a formal Laurent series that can be reduced modulo $\nu$ and let $(\cv n)_n$ be its convergents. for every $n$, let $h_n\inn\n K^*$ and let $x_n,y_n\inn\n K[T]$ be defined as in the previous lemma. Then $$\{\ov{x_n}/\ov{y_n}, n\geq0\}$$ is exactly the set of the convergents of $\ov\alpha$.
\end{theo}

\begin{proof}
	Let $(\frac{u_m}{v_m})_m$ be the convergents of $\ov\alpha$. By the previous lemma, for every $n$ there exists a (unique) integer $\rho(n)$ such that $\ov{x_n}/\ov{y_n}=u_{\rho(n)}/v_{\rho(n)}$. Then $\rho$ is a well defined function, $\rho:[0,N]\to[0,\ov N\,]$, where $N,\ov N\inn\n N\cup\{\infty\}$ are, respectively, the lengths of the continued fractions of $\alpha, \ov\alpha$. 
	
	$\rho$ is not necessarily injective but it is always surjective. Indeed, $\rho(0)=0$, because $p_0=\floor\alpha$ must be reducible and $\ov{p_0}=\floor{\ov\alpha}=u_0$, while trivially $\ov q_0=\ov1=v_0$. Moreover, for every $n\geq0$,	
	\be\begin{split}\label{5eq:red}\deg v_{\rho(n+1)}\leq\deg \ov{y_{n+1}}\leq\deg q_{n+1}=\ord_{\n K}(p_n-\alpha q_n)\leq\\\leq\ord_{k_\nu}\left(\,\ov{x_n-\alpha y_n}\,\right)\leq\ord_{k_\nu}(u_{\rho(n)}-\ov\alpha\, v_{\rho(n)})=\deg v_{\rho(n)+1},\end{split}\ee thus $\rho(n+1)\leq\rho(n)+1$. Then there exists $N_0\leq\ov N$ such that $\Im(\rho)=\{0,\dots,N_0\}$.
	
	On the other hand, $\ov{x_n}/\ov{y_n}\to\ov\alpha$ for $n\to N$, so $\rho$ must be surjective, that is, $N_0=\ov N$ and all the convergents of $\ov\alpha$ can be found as reductions of convergents of $\alpha$.
\end{proof}

\begin{lem}
	If $\rho(n+1)=\rho(n)+1$, then $\ov{x_{n+1}},\ov{y_{n+1}}$ must be relatively prime and $$\deg v_{\rho(n+1)}=\deg q_{n+1}.$$
\end{lem}
\begin{proof}
	It follows immediately from the fact that all the inequalities in \eqref{5eq:red} must be equalities.
\end{proof}

\begin{lem}
	In the previous notations, we also have that $$\rho \text{ is non-decreasing.}$$
\end{lem}
\begin{proof}
	For every $n_1<n_2$ we will have $\deg v_{\rho(n_1)}\leq\deg\ov{y_{n_1}}\leq\deg q_{n_1}<\deg q_{n_2}<$ $<\ord_{\n K}(p_{n_2}-\alpha q_{n_2})\leq\deg v_{\rho(n_2)+1}$ as in \eqref{5eq:red}. Thus $\rho(n_1)<\rho(n_2)+1$ and $\rho$ is non-decreasing.
\end{proof}

\begin{rem}
	By \eqref{1eq:pnqnprime} for every $n$ we have $x_ny_{n+1}-x_{n+1}y_n=\pm h_nh_{n+1}$. As the $x_n,y_n$ can be reduced modulo $\nu$, we must have $$\nu(h_nh_{n+1})\geq0.$$
	
	If $\nu(h_n)=-\nu(h_{n+1})$, then $\ov{x_ny_{n+1}-x_{n+1}y_n}\inn k_\nu^*$, so we must have $\rho(n+1)=\rho(n)\pm1$. Actually, as we have already shown that $\rho$ is non-decreasing, $\rho(n+1)=\rho(n)+1$. 
	
	Otherwise, $\ov{x_ny_{n+1}-x_{n+1}y_n}=\ov0$, that is, $\rho(n)=\rho(n+1)$.\par\medskip
	
	We could also have used this remark as the firs step to prove the previous Theorem; actually, this is the strategy followed by Van der Poorten in \cite{poorten2005specialisation}.
\end{rem}

\begin{ex}
	Let $\alpha\!=\!\frac{3T^2-5T}{T^3+5}\!=\!\!\left[0,\frac13T\!+\!\frac59,\frac{27}{25}T\!-\!\frac{468}{125},\frac{625}{4212}T\!+\!\frac{125}{468}\right]$, its continuants are\\ $p_0=0,\ q_0=1\\p_1=1,\ q_1=1/3\,T+5/9\\ p_2={27}/{25}\,T-{468}/{125},\ q_2=9/{25}\,T^2-{81}/{125}\,T-{27}/{25}\\ p_3={25}/{156}\,T^2-{125}/{468}\,T,\ q_3={25}/{468}\,T^3+{125}/{468}.$
	
	$\alpha$ can be reduced modulo every prime and all its partial quotients can be reduced modulo every prime but $2,3,5,13$; let us consider its reductions modulo those primes.
	$$l=2\ \begin{array}{|c|c|c|c|c|c|}
	\hline i_n&\ov{x_n}&\ov{y_n}&\rho(n)&u_{\rho(n)}&v_{\rho(n)}\\\hline0&\ov0&\ov1&0&\ov0&\ov1\\\hline0&\ov1&T+\ov1&1&\ov1&T+\ov1\\\hline0& T&T^2+T+\ov1&2&T&T^2+T+\ov1\\\hline 2&T^2+T&T^3+\ov1&2&T&T^2+T+\ov1\\\hline
	\end{array}$$
	Indeed, modulo 2, $\ov\alpha=\frac{T}{T^2+T+\ov1}=\left[\,\ov0,T+\ov1,T\right]$.
	$$l=3\ \begin{array}{|c|c|c|c|c|c|}
	\hline i_n&\ov{x_n}&\ov{y_n}&\rho(n)&u_{\rho(n)}&v_{\rho(n)}\\\hline0&\ov0&\ov1&0&\ov0&\ov1\\\hline2&\ov0&-\ov1&0&\ov0&\ov1\\\hline-2&\ov1&T^2&1&\ov1&T^2\\\hline 2&T&T^3-\ov1&2&-T&-T^3+\ov1\\\hline
	\end{array}$$
	Indeed, modulo 3, $\ov\alpha=\frac{T}{T^3-\ov1}=\left[\,\ov0,T^2,-T\right]$
	$$l=5\ \begin{array}{|c|c|c|c|c|c|}
	\hline i_n&\ov{x_n}&\ov{y_n}&\rho(n)&u_{\rho(n)}&v_{\rho(n)}\\\hline0&\ov0&\ov1&0&\ov0&\ov1\\\hline0&\ov1&\ov2T&1&\ov1&\ov2T\\\hline3&\ov2&-T&1&\ov1&\ov2T\\\hline -2&T^2&\ov2T^3&1&\ov1&\ov2T\\\hline
	\end{array}$$
	Indeed, modulo 5, $\ov\alpha=\frac{\ov3}{T}=\left[\,\ov0,\ov2T\right]$.
	$$l=13\ \begin{array}{|c|c|c|c|c|c|}
	\hline i_n&\ov{x_n}&\ov{y_n}&\rho(n)&u_{\rho(n)}&v_{\rho(n)}\\\hline0&\ov0&\ov1&0&\ov0&\ov1\\\hline0&\ov1&\ov9T+\ov2&1&\ov1&\ov9T+\ov2\\\hline0&-T&\ov4T^2-\ov2T+\ov1&2&-T&\ov4T^2-\ov2T+\ov1\\\hline 1&T^2+\ov7T&\ov9T^3+\ov6&2&-T&\ov4T^2-\ov2T+\ov1\\\hline
	\end{array}$$
	Indeed, modulo 13, $\ov\alpha=\frac{T}{\ov9T^2+\ov2T-\ov1}=\left[\,\ov0,\ov9T+\ov2,-T\right]$
\end{ex}

\begin{lem}\label{5lem:reddegrees}
	If $\alpha\inn\C L_\nu$, then reducing modulo $\nu$ the degree of the partial quotients (apart from the first one) can only grow, so in particular $$K(\ov\alpha)\geq K(\alpha).$$ 
\end{lem}
\begin{proof}
	As $\rho$ is non-decreasing, for every $m$ there must exist integers $n_m,N_m$ such that $\{n\inn\n N,\ \rho(n)=m\}=[n_m,N_M]$, where $\rho(n_m-1)=m-1,\ \rho(N_m+1)=m+1$. Then $\deg v_{m}=\deg q_{n_m}$ and $\deg v_{m+1}=\deg q_{N_m+1}$. Thus, setting $\alpha=[a_0,a_1,\dots]$ and $\ov\alpha=[b_0,b_1,\dots]$, we will have $$\deg b_{m+1}=\deg a_{N_m+1}+\cdots+\deg a_{n_m+1}.$$ 
\end{proof}

\begin{cor}\label{5cor:reddegrees}
	If $\alpha\inn\C L_\nu$ and $\ov\alpha$ is badly approximable, respectively normal, in $\n L_\nu$, then $\alpha+\pi\beta$ is badly approximable, respectively normal, in $\n L$ for every $\beta\inn\C L_\nu$.
\end{cor}

\begin{ex}
	Let $\ds\alpha=\frac{\sqrt{T^6-1}}{T^2-1}=\Big[T,\ov{T,-T,-2T,-T,T,2T}\,\Big]\inn\n Q[T]$. Certainly $\alpha$ and all its partial quotients can be reduced modulo every prime $l$ different from $2$ and we have $\ov\alpha^l=\left[T,\ov{T,-T,-\ov2^l T,-T,T,\ov2^l T}\right]$.
	
	Then, for $l\neq2$, any Laurent series of the form $\frac{\sqrt{T^6-1+l A}}{T^2-1+l B}$ is normal, provided that $A,B$ are polynomials reducible modulo $l$ and with $\deg A<6,\ \deg B<2$.    
\end{ex}

\begin{ex}
	Let $\alpha=\frac{T^3+T^2+T+2+\sqrt{T^6-T^5+T^4-T^3-T^2-1}}{T^2+T+1}\inn\n Q((T^{-1}))$. Its continued fraction is $\alpha=\left[2T-\frac32,\frac87T-\frac{148}{49},\frac{686}{8983}T+\frac{23477517}{80694289},\dots\right]$, which does not look periodic but a priori it might have some partial quotients with degree larger than 1. 
	
	Actually, $\alpha$ can be reduced modulo $3$ and the continued fraction expansion of its reduction modulo 3 is $\ov\alpha=\left[\ov{-T,-T-\ov1,-T}\right]$, so $\alpha$ must already be normal.
\end{ex}

More generally, we will have
\begin{cor}
	McMullen's (strong) Conjecture over $\n Q$ would be a consequence of McMullen's (strong) Conjecture over finite fields and, by Corollary \ref{5cor:ZarimpMM}, it would also follow from Zaremba's (strong) Conjecture over finite fields.
\end{cor} 

Using a different method, we will see in Theorem \ref{66theo:MMQ} that McMullen's strong Conjecture over $\n Q$ actually holds.

Analogously, and thanks to the fact that McMullen's and Zaremba's Conjectures hold over the algebraic closure of finite fields, we can use reduction methods to prove McMullen's Conjecture over $\ch Q$.

\subsection{\texorpdfstring{Proof of McMullen's strong Conjecture for $\n K=\ch Q$}{McMullen's strong Conjecture over the algebraic closure of Q}}\label{5sec:5.4} 

Let $D\inn\C S_{\ch Q}$; we will denote by $\n K_D$ the smallest finite extension of $\n Q$ over which $\sqrt D$ is well defined as a formal Laurent series, that is, $\n K_D$ will be the finite extension of $\n Q$ generated by the coefficients of $D$ and by the square root of its leading coefficient.

The continued fraction expansion of $\sqrt D$ does not depend on the chosen base field, as long as the formal Laurent series under consideration is well defined, so in fact its continuants and its partial quotients are polynomials defined over $\n K_D$.

\begin{lem}
	Let $D\inn\C S_{\ch Q}$ and let $\n K=\n K_D$. Then, there exist infinitely many primes of $\n K$ with respect to whom $\sqrt D$ is reducible and its reduction is an irrational Laurent series.
\end{lem}
\begin{proof}
	Let $D=a_0^2+\delta$ with $a_0=\floor{\sqrt D}$. Then $D,a_0,\delta$ are reducible modulo infinitely many primes and the reduction of $D$ is a square if and only if the reduction of $\delta$ is zero, which can happen only for finitely many primes.
\end{proof}
	
\begin{prop}\label{5prop:MMchq1}
	McMullen's strong Conjecture\eqref{5eq:MMstrong} holds over $\ch Q$.
\end{prop}

\begin{proof}
	As before, let $D\inn\C S_{\ch Q}$ and let $\n K=\n K_D$. Let $l$ be a prime of $\n K_D$ such that the reduction of $\sqrt D$ modulo $l$ exists and is irrational; let $\nu=\nu_l$ be the corresponding valuation and let $k=k_{\nu_l}$ be the residue field. If $\alpha\inn\ch Q((T^{-1}))$ is reducible modulo $l$, as before we will denote by $\ov\alpha$ its reduction in $k((T^{-1}))$. As $k$ is a finite field, there exists a non-trivial solution $x,y\inn k[T]$ of the Pell equation for $\ov{\sqrt D}$. Possibly replacing $\n K$ with a finite extension $\n K'$ such that $l$ is still a prime of $\n K'$, we can assume that $k$ is large enough so that it exists a polynomial $z\inn k[T]$, relatively prime to $x$, such that $z/x$ is normal (by Theorem \ref{4theo:Friesen} it is enough to assume $\#k\geq 2\deg x$). Then, by Theorem \ref{5theo:Mercat}, $\frac{x-bz+y\ov{\sqrt D}}{bx}\inn k((T^{-1}))$ is normal, where $b\inn k$ is a suitable constant. 
	
	Then, by Corollary \ref{5cor:reddegrees}, for every $X,Y,Z\inn\n K[T]$ reducible modulo $l$, with reductions $\ov X=x,$ $\ov Y=y,$ $\ov Z=z$ and with $\deg X=\deg x$ and for every constant $B\inn\n K$ with $\ov B=b$ we will have that $$\frac{X-BZ+Y\sqrt D}{BX}\inn\ch Q((T^{-1})) \text{ is normal }.$$
	Obviously, we can choose the lifts $X,Y,Z,B$ in infinitely many different ways, obtaining infinitely many pairwise non-equivalent normal elements of $\ch Q(T,\sqrt D)$.
\end{proof}

\begin{ex}\label{5ex:MMbarq}
	Let $D=T^8-T^7-\frac34T^6+\frac72T^5-\frac{21}4T^4+\frac72T^3-\frac34T^2-T+1\inn\n Q((T^{-1}))$. It can be proved (see \cite{contfracexamples}) that $\ov K(\sqrt D)=2$, so the question of the existence of normal elements in $\n K(T,\sqrt{D(T)})$ is not a trivial problem.  It can be seen, for example with the methods of Corollary \ref{6cor:yureductioncondition}, that $D$ is a  non-Pellian polynomial, so we cannot apply directly Mercat's Theorem \ref{5theo:Mercat}.
	
	Certainly $\sqrt D$ can be reduced modulo every prime different from $2$ and its reduction is never rational. 
	
	Let us apply Mercat's Theorem to the reduction $\ov D=T^8-T^7-T^5-T^3-T+\ov1$ of $D$ modulo 3. It is easy to compute the minimal solution of the Pell equation for $\ov D$ over $\n F_3$: $(x,y)=(-T^{10}-T^9+T^6-T^3-T^2+\ov1,-T^6+T^4-T^3-T^2-T)$. Now, $z=T^6(T^3+T^2+T-\ov1)$ is relatively prime to $x$ over $\n F_3$ and $z/x$ is normal. Then $\frac{-T^{10}-T^9+T^6-T^3-T^2+\ov1+T^6(T^3+T^2+T-\ov1)+(-T^6+T^4-T^3-T^2-T)\ov{\sqrt D}}{T^{10}+T^9-T^6+T^3+T^2-\ov1}$ is normal in $\n F_3((T^{-1}))$ (in the previous notations, $b=\ov2\,$), so $$\frac{-T^{10}+T^8+T^7-T^3-T^2+1+3A_1-(T^6-T^4+T^3+T^2+T+3A_2)\sqrt D}{T^{10}+T^9-T^6+T^3+T^2-1+3A_3}$$ is normal in $\n Q((T^{-1}))$ for every $A_1,A_2,A_3\inn\n Q[T]$, polynomials reducible modulo 3 and with $\deg A_3<10.$
\end{ex}

Using techniques of reduction of a formal Laurent series modulo a prime we can also prove that Conjecture \ref{5conj:mult} holds over $\ch Q$, which of course will imply again McMullen's strong Conjecture \eqref{5eq:MMstrong}. Let us then prove that the hypothesis of Lemma \ref{5lem:MMrootsqn} is verified for $\n K=\ch Q$.

\begin{theo}\label{5theo:MMbarQ}
	Let $D\inn\C S_{\ch Q}$ be a polynomial of degree $2d$, let $(\frac{p_n}{q_n})_n$ be the convergents of $\sqrt D$. Then, for every $\lambda\inn\ch Q$ of the form $$\lambda=\zeta_r\sqrt[r]{1/\pi},$$ where $\pi$ is a large enough prime of $\n K_D$, $r\geq d$ and $\zeta_r$ is an $r$-th root of unity, we have $$q_n(\lambda)\neq0 \text{ for every } n\geq0.$$
\end{theo}

We will give two proof of this result: a first algebraic proof, relying on the previously discussed theory of reduction modulo a prime, 
and a second proof using the connection with algebro-geometric properties of hyperelliptic curves seen in section \ref{6sec:hyper}.

\begin{proof}[First proof]	
	As before, let $\n K=\n K_D$ be the smallest extension of $\n Q$ over which the formal Laurent series of $\sqrt D$ is well defined.
	
	Let $l>2$ be a prime number, let $\nu$ be an extension to $\n K$ of the $l$-adic valuation of $\n Q$. As before, let $\C O=\{x\inn\n K^*,\ \nu(x)\geq0\}\cup\{0\}$ be the valuation ring of $\nu$ and let $\pi$ be a uniformizing parameter for $\nu$, that is, a generator of the unique maximal ideal $\C M=\{x\inn\n K^*,\ \nu(x)>0\}$ of $\C O$. We will denote by $k$ the residue field $k=\C O/\C M$ and, if $\alpha\inn\C L=(\C O[[T^{-1}]])[T]$, we will denote by $\ov\alpha\inn\ k((T^{-1}))$ its reduction modulo $\nu$.\par\medskip
	
	As we have seen in Lemma \ref{5lem:conditionsforgr}, by choosing $\pi$ large enough we can assume that the Laurent series representing $\sqrt D$ can be reduced modulo $\nu$, that is $\sqrt D\inn\C L$, and we can assume that the reduction of $D$ modulo $\nu$ is not a perfect square. Moreover, we can choose $\pi$ so that the leading coefficient of $\sqrt D$ has valuation 0. Then $\ord\bigg(\,\ov{\sqrt D}\,\bigg)=\ord(\sqrt D)=-d$.
	
	Let $p/q=p_n/q_n$ be a convergent of $\sqrt D$. By Lemma \ref{5lem:ordpnqn}, as $\ord\ov{\sqrt D}<0$, there exists a (unique) integer $i$ such that $x=\pi^ip$ and $y=\pi^iq$ can be reduced modulo $\nu$ and both $\ov{x},\ov{y}$ are different from zero ($-i$ will be minimum of the valuations of the coefficients of $q$). Let us denote by ${u_j}/{v_j}$ the convergents of $\ov{\sqrt D}$ and by $b_j$ its partial quotients. By Theorem \ref{5theo:red}, ${\ov x}/{\ov y}$ is a convergent of $\ov{\sqrt D}$, that is, there exists $j$ such that ${\ov x}/{\ov y}={u_j}/{v_j}$.\par\medskip 
	
	Now, $d\geq\deg b_{j+1}=\deg v_{j+1}-\deg v_j$. If $\deg q>\deg v_j$, we have seen in lemma \ref{5lem:reddegrees} that there exists $N>n$ such that $\deg q_N=\deg v_{j+1}$ and $\deg v_j\leq\deg\ov y=\deg\ov{\pi^iq}$. Then, \be\label{5eq:r} \deg q-\deg\ov{\pi^iq}\leq d-1,\ee and this inequality holds trivially even if $\deg q=\deg\ov y=\deg v_j$. Thus, the degree of (a normalization of) $q$ cannot decrease too much when reducing modulo $\nu$.\par\medskip
	
	Let $\lambda\inn\ch Q$ be a root of $\pi T^r-1$, irreducible over $\n K$. Now, $\lambda$ is a root of $q$ if and only if it is a root of $y$, if and only if $(\pi T^r-1)$ divides $y$. As $\pi T^r-1$ is primitive, by Gauss's Lemma this is equivalent to the existence of a polynomial $F\inn\C O[T]$ such that $y=(lT^r-1)F$. Then, reducing modulo $\nu$, $\ov y=-\ov F$ and $\deg\ov y\leq\deg q-r$. For $r> d-1$ this would contradict \eqref{5eq:r}, so $\lambda$ cannot be a root of $q$.\par\medskip
	
	The choices of $\pi,r$ do not depend on $p,q$ (that is, they do not depend on $n$): we only require that $D$ is reducible modulo $\pi$, a prime over $l>2$, that its reduction is not a perfect square and that $r\geq d$. Thus, we have actually proved that for every large enough prime $\pi$ and for every integer $r\geq d$ a root $\lambda$ of the polynomial $\pi T^r-1$ cannot be a zero of $q_n$ for every $n$.
\end{proof}

\begin{proof}[Second proof]
	Let $D=b^2\widetilde D$, with $\widetilde D\inn\n K[T]$ squarefree and let $\deg \widetilde D=2\widetilde d$; let $\C H$ be the hyperelliptic curve of affine model $U^2=\widetilde D(T)$.
	
	Let $\pphi=p-qbU$, where $\frac pq=\cv n$ is a convergent of $\sqrt D$; let $\lambda\inn\ch Q$ not a root of $D$ and let $P=(\lambda,U(\lambda))\inn\C H_{\widetilde D}$. Then $q(\lambda)=0$ if and only if $\pphi(P)=\pphi'(P)$. 
	
	Let $l$ be a prime and let $\nu$ be an extension to $\n K$ of the $l$-adic valuation; let $\pi\inn\n K$ be a uniformizing parameter for $\nu$.\par\medskip
	
	We have seen in Remark \ref{6rem:multdeltanonsf} that $\pphi,\pphi'$ have, respectively, a zero and a pole at $\infty_+$. If $P$ is a point near enough to $\infty_+$ and far enough from the zeros of $\pphi'$ (with respect to $\nu$), then $\pphi,\pphi'$ cannot assume the same value in $P$. As formal Laurent series, let $\pphi=f_{-i}T^{-i}+f_{-i-1}T^{-i-1}+\cdots,\ \pphi'=g_jT^j+g_{j-1}T^{j-1}+\cdots$, where $j=\deg p$ and $i=\deg q_{n+1}>\deg q$.\par\medskip
	
	Let us assume that $\C H_{\widetilde D}$ has good reduction modulo $\nu$, that is, let us assume that $\widetilde D$ can be reduced modulo $\pi$ and that its reduction is still a squarefree polynomial of degree $2\widetilde d$. Let us also assume that $D$ can be reduced modulo $\nu$, with $\deg\ov D=2d$. Then the Laurent series representing $\sqrt D$ can be reduced modulo $\nu$, so, up to multiplication of $p,q$ (that is, of $\pphi$) by a suitable power of $\pi$, we can assume that also $\pphi,\pphi'$ have good reduction modulo $\nu$ and that their reductions are not $0$ (that is, $\nu(f_m),\nu(g_n)\geq0$ for every $m,n$ and there exist $M,N$ such that $\nu(f_M)=\nu(g_N)=0$) .\par\medskip
	
	Let $\lambda=\pi^{-1/r}$; let us still denote by $\nu$ its extension to $\n K(\lambda)$ (with $\nu(\lambda)=-\frac1r$). By Lemma \ref{6lem:zerosphi}, if we assume $m\geq d$ then $P$ cannot be too near to the zeros of $\pphi'$, as they have degree at most $2d-2$ over $\n K$. More precisely, as before, reducing modulo $\nu$ the degree of $\pphi'$ can decrease at most of $d-1$, so there must exist an element of $\{g_j,\dots,g_{j-d+1}\}$ that has absolute valuation 0. As $\nu(g_n)\inn\n N$ for every $n$, for $r\geq d$ then there cannot be cancellations in the $g_n\lambda^n$ and $\nu(\pphi'(\lambda))<0$.\par\medskip
	
	On the other hand, $\nu(\pphi(\lambda))=\nu(f_{-i}\pi^{i/r}+f_{-i-1}\pi^{(i+1)/r}+\cdots)\geq\nu(\pi^{i/r})=\frac1r>0$.
	Then $\pphi(\lambda)\neq\pphi'(\lambda)$, so $\lambda$ is not a root of $q=q_n$.
	
	As the choice of $\lambda$, that is, the choices of $\pi,r$, do not depend on $n$, we have that $\lambda$ is not a root of any continuant $q_n$.	
\end{proof}

\begin{cor}\label{5cor:MMchQ}
	 Conjecture \ref{5conj:mult} and McMullen's strong Conjecture hold over $\ch Q$: for every polynomial $D\inn\C S_{\ch Q}$ there exist infinitely many $(d-1)$-uples $\lambda_1,\dots,\lambda_{d-1}\inn\ch Q$, where $2d=\deg D$, such that $$K\left((T-\lambda_1)\cdots(T-\lambda_{d-1})\sqrt D\right)=1.$$
\end{cor}

Actually, we already have $K\left((T-\lambda_1)\cdots(T-\lambda_{k-1})\sqrt D\right)\!=\!1$, for $k\!=\!K(\sqrt D)$ and, analogously, $\ov K\left((T-\lambda_1)\cdots(T-\lambda_{\ov k-1})\sqrt D\right)=1$ for $\ov k=\ov K(\sqrt D)$.
\begin{proof}
	It follows directly from Lemma \ref{5lem:MMrootsqn}.
\end{proof}

\begin{lem} 
	For every $D\inn\C S_{\ch Q}$ there exist $\lambda_1,\dots,\lambda_{d-1}$ with $$[\n K_D(\lambda_1,\dots,\lambda_{d-1}):\n K_D]=\frac{(2d-2)!}{(d-1)!}.$$ such that $\n K_D(\lambda_1,\dots,\lambda_{d-1})(T,\sqrt{D(T)})$ has normal elements. 
\end{lem}
\begin{proof}
We have seen that we can choose $\lambda_i$ as a root of $\pi_iT^{r_i}-1$, with $\pi_i$ a prime of $\n K_D$ large enough to guarantee that $\sqrt D$ is reducible modulo $\nu$, where $\nu$ is the $\pi_i$-adic valuation of $\n K_D$ and that the reduction of $\sqrt D$ has still order $-d$ and where $r_i= d+i-1$. 
\end{proof}

\begin{ex}
	As in Example \ref{5ex:MMbarq}, let us consider the non-Pellian polynomial over $\n Q$
	$D=T^8-T^7-\frac34T^6+\frac72T^5-\frac{21}4T^4+\frac72T^3-\frac34T^2-T+1$. As it is proved in \cite{contfracexamples}, $\ov K(\sqrt D)=2$. $D$ is reducible modulo every prime different from $2$ and its reduction is never a square. Then, by the previous Theorem, $$\left(T-\sqrt[4]{1/l}\right)\sqrt{D(T)}$$ is (eventually) normal for every prime $l>2$.
\end{ex}

We can also give an alternative, direct proof of McMullen's strong Conjecture over $\ch Q$, based on Lemma \ref{1lem:(f+a)/g}. 
\begin{proof}[Second proof of Proposition \ref{5prop:MMchq1}]	
	Let $D\inn\C S_{\ch Q}$ be a polynomial of degree $2d$ and let $(\cv n)_n$ be the convergents of $\sqrt D$. Certainly $K(\sqrt D)\leq d$; let us assume that $K(\sqrt D)\neq1$. We have seen that if $a_1\neq-\frac{p_n(\lambda_1)}{q_n(\lambda_1)}$ for every $n$, then $K\left(\frac{\sqrt D+a_1}{T-\lambda_1}\right)=K(\sqrt D)-1\leq d-1$. 
	
	As before, let $\n K=\n K_D$. Then for every $\lambda_1\inn\n K$ we have $\frac{p_n(\lambda_1)}{q_n(\lambda_1)}\inn\n K\cup\{\infty\}$ for every $n$, so it is enough to take $a_1\inn\ch Q\setminus\n K$. In particular, it is enough to take $a_1$ in a quadratic extension of $\n K$. 
	
	Going on in this way, we can find $a_1,\dots,a_{d-1},\lambda_1,\dots,\lambda_{d-1}$ such that, setting $\alpha_0\!=\!\sqrt D$ and, for $i=1,\dots,d-1$, $$\alpha_i=\frac{\alpha_{i-1}+a_i}{T-\lambda_i},$$ we have $$K(\alpha_i)=\max\left\{1,K(\sqrt D)-i\right\}.$$ In particular $\alpha_{d-1}$ is normal.
	
	We can choose the $\lambda_i$ and the $a_i$ in infinitely many different ways, obtaining infinitely many pairwise non-equivalent normal elements of $\ch Q(T,\sqrt{D(T)})$, so McMullen's strong Conjecture \eqref{5eq:MMstrong} holds over $\ch Q$.
\end{proof}	

\begin{lem}
	For any $D\inn\C S_{\ch Q}$ there exist $\alpha\inn\n K'(T,\sqrt D)$ normal, with $$[\n K':\n K_D]\leq2^{d-1}.$$
\end{lem}
\begin{proof}
	In the notations of the previous proof $\alpha_{d-1}$ is defined over $\n K'=\n K_D(a_1,\dots,a_{d-1})$, where we can choose the $a_i$ so that $[\n K':\n K_D]\leq2^{d-1}.$
\end{proof} 
\begin{ex}
	Let $D\inn\S Q$ be a polynomial of degree $2d$ and let us denote by $l_n$ the $n$-th prime number. Then $$\left(\left(\left(\left(\sqrt D+\sqrt2\right)\frac1T+\sqrt3\right)\frac1T+\sqrt5\right)\frac1T+\cdots+\sqrt{l_{d-1}}\right)\frac1T$$ is normal (in the previous notations, $\lambda_i=0$ for every $i$ and $a_i=\sqrt{l_i}$).
	
	In particular, for $D$ as in Example \ref{5ex:MMbarq} we will have that $$\frac{\sqrt D+\sqrt l}{T}$$ is normal for every prime $l$.
\end{ex}

\begin{rem}\label{5rem:QimpchQ}
	We will see in Theorem \ref{66theo:MMQ} that Conjecture \ref{5conj:mult} holds also over every number field. Of course, this will imply again Conjecture \ref{5conj:mult} and McMullen's strong Conjecture over $\ch Q$. Indeed, we will see that for every $D\inn\C S_{\ch Q}$ there exist $\lambda_1,\dots,\lambda_{d-1}\inn\n K_D$ such that $$\ov K\left((T-\lambda_1)\cdots(T-\lambda_{d-1})\sqrt D\right)=1.$$
\end{rem}
\section{\texorpdfstring{$\n K=\n Q$: generalized Jacobians}{K=Q: generalized Jacobians}}\label{66Sec}
In order to study Conjecture \ref{5conj:mult} and McMullen's Conjecture over $\n Q$ we will consider generalized Jacobians of the hyperelliptic curves $\C H_D$. Indeed, as we have seen in Remark \ref{6rem:multdeltanonsf}, to treat the case of non-squarefree polynomials we have to consider a stricter equivalence relation on the group of divisors. We will show in section \ref{66subsec:pullpack} that this corresponds to working on a pullback of generalized Jacobians associated to moduli. Thus, we will firstly recall the construction of generalized Jacobians, and in particular of the generalized Jacobian associated to a modulus, to highlight the connection between the continued fraction expansion of $\sqrt D$, with $D$ non necessarily squarefree, and the minimal writing of the multiples of $\delta=[(\infty_-)-(\infty_+)]$ on a suitable generalized Jacobian.

Conjecture \ref{5conj:mult} will then follow from a Theorem of Zannier, proved applying to a generalized Jacobian a version for algebraic groups of Skolem-Mahler-Lech's Theorem.

\subsection{Construction of generalized Jacobians}\label{66sec:gen}
The classical construction of the Jacobian variety of a smooth algebraic curve $\C C$ can be generalized, essentially by modifying the equivalence relation on the group of degree-zero divisors $\Div^0(\C C)$ (or on one of its subgroups) and giving a structure of algebraic group to the quotient thus obtained.\par\medskip
 
Maxwell Rosenlicht discussed the construction and the properties of generalized Jacobians in the most general setting \cite{rosenlicht1952equivalence}, \cite{rosenlicht1954generalized}. Here we will follow the exposition of Serre (\cite{serre1988algebraic}, Chapter V), focused on generalized Jacobians associated to a modulus. We will then consider the case linked to the study of quadratic irrationalities of the form $b\sqrt{\widetilde D}$, which correspond to a pullback of generalized Jacobians associated to moduli.\par\medskip

Let $\n K$ be an algebraically closed field.
Let $\C C$ be a complete, projective, irreducible smooth curve defined over $\n K$. By ``glueing together'' some points of $\C C$  we can build a singular curve whose normalization is $\C C$.

\begin{notat}\label{66notat:sing}
	Let $S\sub\C C$ be a finite set of points and let $\mathscr R$ be an equivalence relation over $S$. Let $S'=S/\mathscr R$ and let us consider $$\C C'=(\C C\setminus S)\cup S', \text{ with } \pi:\C C\to\C C'$$ the canonical projection.

	For $Q\inn\C C'$, let $$O_Q=\bigcap_{\mathclap{\pi(P)=Q}}O_P,$$ where the $O_P$ are the local rings of $\C C$. Then $O_Q$ is a semi-local ring, let $r_Q$ be its Jacobson radical, that is, $\ds r_Q=\bigcap_{\mathclap{\pi(P)=Q}}M_P$. If $Q\inn\C C'\setminus S'$, let $O'_Q=O_Q$ and for $Q\inn S'$, let $O'_Q$ be a proper subring of $O_Q$ such that $$\n K+r_Q^{n_Q}\subseteq O'_Q\subseteq \n K+ r_Q$$ for some integer $n_Q$. For every point $Q$ of $\C C'$, let us set $$g_Q=\dim O_Q/O'_Q.$$ 
\end{notat}	

\begin{theo}\label{66theo:C'singular}
	$\C C'$ with the sheaf $O'$ formed by the rings $O'_Q$ is a non-smooth algebraic curve, its set of singular points is $S'$ and $\C C$ is its normalization.
\end{theo} 
	
We will denote by $\Div_S(\C C)$ the group of the divisors of $\C C$ prime to $S$ and we will denote by $\Div_S^0(\C C)$ its subgroup of degree-zero divisors. 

\begin{notat}\label{66notat:notat L}
	Let $g$ be the genus of $\C C$ and let $$g'=g+\sum_{Q\in S'}g_Q.$$ 

	Let $A\inn\Div_S(\C C)$ be a divisor prime to $S$. We can then consider the subsheaf $L'(A)$ of $\n K(\C C')=\n K(\C C)$ defined by $L'(A)_Q=\left\{\begin{array}{ll}O'_Q&\text{if } Q\inn S'\\ L(A)_Q&\text{otherwise}\end{array}\right.$. 

	Let $\C L'(A)=H^0(\C C',L'(A))$ and $\C I'(A)=H^1(\C C',L'(A))$. It can be proved that $\C L'(A),\C I'(A)$ are finite-dimensional $\n K-$vector spaces; let $\ell'(A),i'(A)$ be, respectively, their dimensions. 
\end{notat}

\begin{theo}[Generalized Riemann-Roch Theorem]\label{66theo:genRR}
	In the previous hypothesis and notations, for every divisor $A$ of $\C C$ prime to $S$ we have \be\ell'(A)-i'(A)=\deg (A)-g'+1.\ee
\end{theo}

\begin{cor}\label{66cor:RRgen}
	Let $P_\infty$ be a point of $\C C$ not in $S$. Then for every divisor $A\in\Div_S^0(\C C)$ there exists an effective divisor $B$ of degree $g'$ (not necessarily prime to $S$) such that $$A=B-g'(P_\infty)-\div(f)$$ with $\ds f\inn\bigcap_{\mathclap{Q\in S'}} O'_Q$.
\end{cor}

Instead of the classical equivalence relation $\sim$ between divisors, we can consider the strong equivalence relation $$A\approx B \text{ if and ony if } A=B+\div(f) \text{ with } f\inn\bigcap_{Q\in S'} (O'_Q)^*.$$ Rosenlicht has then studied in all generality the quotient group $\Div_S^0(\C C)/\approx$.\par\medskip

We will be particularly interested in the case where the equivalence relation $\C R$ over $\C C$, and by consequence the equivalence relation $\approx$ on $\Div^0(\C C)$, is defined by a modulus.
\subsubsection{Generalized Jacobian associated to a modulus}

\begin{defn}
	Let $S\sub \C C$ be a finite set of points. A \textit{modulus} $\M$ supported on $S$ is an effective divisor of the form $\sum\limits_{P\in S} n_P(P)$ with $n_P>0$ for every $P\inn S$. We can thus define the degree of $\M$ as $\deg(\M)=\sum_P n_P$.
\end{defn}

\begin{defn}
	Let $\M$ be a modulus on $\C C$ supported on $S$; for $f\inn\n K(\C C)$ and $k\inn\n K$ we will write $$f\equiv k\pmod{\M}$$ if $$\ord_P(f-k)\geq n_P \text{ for every } P\inn S.$$
\end{defn}

In particular if $f\equiv k\pmod\M$, then $f(P)=k$ for every $P\inn S$, so if $k\inn\n K^*$ we will have that $\div(f)$ is prime to $S$. 

\begin{notat}
	From now on, we will always consider a modulus $\M=\sum_{P\in S} n_P (P)$ on $\C C$ with support $S$. We will always assume $\M$ non trivial, that is $\deg\M\geq2$.
	
	The singular curve $\C C_\M$ associated to the modulus $\M$ is defined as in \ref{66notat:sing} by identifying all the points of $S$. Then $S'$ is reduced to a point, $S'=\{Q\}$ with $Q\notin\C C$, and $$\C C_\M=(\C C\setminus S)\cup\{Q\}.$$ 
	
	In the previous notations, we will have $O_Q\!=\!\{f\inn\n K(\C C),\ \ord_P(f)\geq0 \text{ for every } P\inn S\}$ and $r_Q=\{f\inn O_Q,\ \ord_P(f)>0 \text{ for every } P\inn S\}$.
	We can then take $$O'_Q=\{f\inn O_Q,\ f\equiv k\!\!\!\!\!\pmod\M \text{ with } k\inn\n K\}.$$ Indeed, as we assumed $\M$ non trivial, $O'_Q\neq O_Q$ and for $n\geq\max_{P\in S} n_P$ we have $\n K+r_Q^n\sub O'_Q\sub\n K+ r_Q$. Then $$g_Q=\deg\M-1, \text{ so } g'=g+\deg\M-1.$$
	
	By Theorem \ref{66theo:C'singular}, $\C C_\M$ with the sheaf $O'$ is singular and its only singular point is $Q$. 
\end{notat}

\begin{defn} 
	As $(O'_Q)^*=\{f\inn\n K(\C C),\ f\equiv k\pmod\M\text{ with } k\inn\n K^*\}$, two divisors $A,A'$ relatively prime to $S$ are equivalent with respect to the strong equivalence relation $\approx$, which will now be denoted by $\sim_\M$, if and only if $$A=A'+\div(f) \text{ with } f\equiv1\pmod\M.$$ In this case $A,A'$ are said to be $\M$-\textit{equivalent}.

	We will denote by $\Pic_\M(\C C)$ the quotient of $\Div_S(\C C)$ with respect to $\sim_\M$ and by $\Pic^0_\M(\C C)$ its subgroup formed by the classes of divisors of degree 0. We will denote by $[A]_\M$ the class in $\Pic_\M(\C C)$ of a divisor $A\inn\Div_S(\C C)$.
\end{defn}

$\Pic^0_\M(\C C)$ can be given a structure of algebraic group; it is then denoted by $\C J_\M$ and called the \textit{generalized Jacobian} of $\C C$ relative to the modulus $\M$.

\begin{lem}
	$\C J_\M$ is an extension of the usual Jacobian $\C J$ by $\Ker\Phi_\M$, where $\Phi_\M$ is the group morphism from the generalized Jacobian $\C J_\M$ to the classical Jacobian $\C J$ of $\C C$ defined by $$\Phi_\M([A]_\M)=[A].$$
\end{lem}
\begin{proof}	 
	$\Phi_\M$ is clearly well defined; actually, it is also an algebraic morphism. It is easy to see that $\Phi_\M$ is surjective, and that we have the short exact sequence $$0\to \Ker \Phi_\M\hookrightarrow\C J_\M\xrightarrow{\Phi_\M}\C J\to0.$$ 
\end{proof}	

\begin{prop}\label{66prop:ext}
	If $\char\n K=0$, then as an algebraic group, $\C J_\M$ is an extension of $\C J$ by the linear group $$\n G_m^{\# S-1}\times\prod_{P\in S}\n G_a^{n_P-1}.$$
\end{prop}
\begin{proof}[Sketch of proof]
Of course, $\Ker\Phi_\M\!=\!\Big\{[\div(f)]_\M,\ f\inn\n K(\C C), \ord_P(f)=0 \text{ for every } P\inn S\Big\}.$ 

Now, $[\div(f)]_\M=0\inn\C J_\M$ if and only if $f\inn(O'_Q)^*$, if and only if there exists $k\inn\n K^*$ such that $\ord_P(f-k)\geq n_P$ for every $P\inn S$. Let $U_P$ be the multiplicative group of the rational functions $f$ such that $\ord_P(f)=0$ and let $U_P^{(n_P)}$ be its subgroup formed by the functions $f$ such that $\ord_P(f-1)\geq n_P$. Thus, $[\div(f)]_\M=0$ if and only if there exists $k\inn\n K^*$ such that $kf\inn\bigcap\limits_{P\in S} U_P^{(n_P)}$ and $\Ker\Phi_\M\simeq(\prod_{P\in S}U_P/U_P^{(n_P)})/\n G_\M$ (where $\n G_m$ is the multiplicative group). If $\char\n K=0$, then it can be seen that as algebraic groups, $U/U^{(n_P)}\simeq\n G_m\times\n G_a^{n_P-1}$ (where $\n G_a$ is the additive group).
\end{proof}

\begin{ex}\label{66ex:hyperell1}
	As before, let $\C H_D$ denote the hyperelliptic curve of affine model $U^2=D(T)$, where $D$ is a squarefree polynomial of degree $2d$. Let $P=(t,u)$ be an affine point of $\C H_D$ and let $\M$ be the modulus $\M=(P)+(P')$, where $P'$ is the conjugate of $P$. Let us denote by $(\C H_D)_\M=\C H_{D,P}$ the singular curve obtained as before and by $Q$ its only singular point. Then we have $g_Q=1$.\par\medskip
	
	Let us firstly assume that $P\neq P'$, that is, $t$ is not a root of $D$. 
	
	Then $O_Q=O_P\cap O_{P'}=\{(a+b\sqrt D)/c\inn\n K(\C H_D),\ a,b,c\inn\n K[T] \text{ and }c(t)\neq0\}$ and, for $f\inn\n K(\C H_D)$ and $k\inn\n K$, we will have $f\equiv k\pmod\M$ if and only if $f(P)=f(P')=k$.\par\medskip
	
	Let now $P=P'=(t,0)$, that is, let $t$ be a zero of $D$. Then we have $\M=2(P)$. Thus, as before, $O_Q=O_P=\{(a+b\sqrt D)/c\inn\n K(\C H_D),\ a,b,c\inn\n K[T] \text{ and }c(t)\neq0\}$ and $f\!\equiv\! k\!\!\pmod\M$ if and only if $\ord_P(f-k)\geq2$.\par\medskip
	
	In both cases then $f\equiv k\pmod\M$ if and only if $f=\frac{a+b(T-t)\sqrt D}c$ with $c(t)\neq0$ and with $\frac{a(t)}{c(t)}\!=\!k$, so	$O'_Q=\{(a+b\sqrt D)/c\inn\n K(\C H_D),\ c(t)\!\neq0,\ b(t)=0\}$.\par\medskip

	Let $\C J_P$ be the corresponding generalized Jacobian. By Proposition \ref{66prop:ext}, if $P\neq P'$, then $\C J_\M$ is an extension of $\C J$ by $\n G_m$, while if $P=P'$, then $\C J_\M$ is an extension of $\C J$ by $\n G_a$.
\end{ex}

\subsubsection{Pullback of generalized Jacobians}\label{66subsec:pullpack}
\begin{rem}
	Let $\widetilde D\inn\n K[T]$ be a squarefree polynomial and let $D(T)=(T-\lambda)^2\widetilde D(T)$. As we have seen in Remark \ref{6rem:multdeltanonsf}, to link the convergents of $\sqrt D$ with the study of the multiples of $\delta=(\infty_-)-(\infty_+)$ we must restrict the usual linear equivalence to the functions of the form $\pphi=a+b(T-\lambda)\sqrt{\widetilde D}$, that is, to the elements of $O'_Q$ for $\C C=\C H_{\widetilde D}$, $\M=(P)+(P')$ and $P=\left(\lambda,\sqrt{\widetilde D(\lambda)}\right)$. This gives a first connection between generalized Jacobians and convergents of the square root of non-squarefree polynomials.\par\medskip
	
	However, if $D(T)=(T-\lambda_1)^2\dots(T-\lambda_n)^2\widetilde D(T)$, where the $\lambda_i$ are pairwise distinct constants, this connection becomes less straightforward. Indeed, to study as before the convergents of $\sqrt D$ we have to consider linear equivalence with respect to rational functions of the form $a+b(T-\lambda_1)\cdots(T-\lambda_n)\sqrt{\widetilde D}$, where $a$ is any polynomial. On the other hand, taking $\M=(P_1)+(P_1')+\cdots+(P_n)+(P_n')$ with $P_i=\left(\lambda_i,\sqrt{\widetilde D(\lambda_i)}\right)$, we have that $O'_Q$ is the set of functions of the form $a+b(T-\lambda_1)\cdots(T-\lambda_n)\sqrt{\widetilde D}$ where $a(\lambda_1)=\cdots=a(\lambda_n)\inn\n K$. 
	
	So we would be interested in restricting the usual linear equivalence to a set of functions bigger than $O'_Q$, this is why we will repeat the previous reasoning with a less strict equivalence relation.
\end{rem}

\begin{notat}\label{66lem:genJac2}
	As before, let $D\inn\n K[T]$ be a squarefree polynomial of even degree $2d$ and let $\C H=\C H_D$ be the hyperelliptic curve $U^2\!=\!D(T)$. Let $P_1\!=(t_1,u_1),\dots,P_n\!=(t_n,u_n)$ be distinct affine points of $\C H$ such that $P_i\neq P_j'$ for every $i,j$ (in particular, the $t_i$ are not roots of $D$). For $i=1,\dots,n$ let $\M_i$ be the modulus $\M_i=(P_i)+(P_i')$. Let $S=\{P_1,P_1',\dots,P_n,P_n'\}$; we can consider the equivalence relation on $S$ that identifies the $P_i$ with their conjugates: $$P\mathscr R \widetilde P \text{ if and only if } \widetilde P=P \text{ or } \widetilde P=P'.$$ Let $S'=\{Q_1,\dots,Q_n\}$, with $Q_i=[P_i]_{\mathscr R}$ and let $\C H_S=(\C H\setminus S)\cup S'$.
\end{notat}

\begin{rem}
	Let $f=\frac{a+b\sqrt D}c\inn\n K(\C H)$ with $a,b,c$ relatively prime polynomials. Then, as in Example \ref{66ex:hyperell1}, $f\inn O_{Q_i}=O_{P_i}\cap O_{P_i'}$ if and only if $c(t_i)\neq0$ and  $f\inn r_{Q_i}$ if and only if $c(t_i)\neq0$ and $a(t_i)=b(t_i)=0$, if and only if $f\equiv0\pmod{\M_i}$. We can then take $O'_{Q_i}=\n K+r_{Q_i}$, that is, $f\inn O'_{Q_i}$ if and only if $f\equiv k_i\pmod{\M_i}$ for some $k_i\inn\n K$, if and only if $c(t_i)\neq0$ and $b(t_i)=0$.
	
	Writing $O'_{Q_1}\cap\dots\cap O'_{Q_n}=O'_S$ and $(T-t_1)\cdots(T-t_n)=R_S$ we will have $$f\inn O'_S \text{ if and only if } f=\frac{a+bR_S\sqrt D}{c} \text{ with } c(t_i)\neq 0 \text{ for every }i,$$ if and only if for every $i$ there exists $k_i\inn\n K$ such that $f\equiv k_i\pmod{\M_i}$. 
	Moreover, $f\inn(O'_S)^*$ if and only if $k_1,\dots,k_n\inn\n K^*$. 
\end{rem}	

Certainly $g_{Q_i}=1$ for every $i$, so $$g'=g+n.$$
	
As in Notation \ref{66notat:notat L}, for every divisor $A$ prime to $S$ we can consider the $\n K$-vector space $\C L_S(A)=\{f\inn O'_S,\ A+\div(f)\geq0\}$. The generalized Riemann-Roch Theorem \ref{66theo:genRR} then becomes: $$\ell_S(A)-i_S(A)=\deg (A)-g-n+1.$$ In particular, by Corollary \ref{66cor:RRgen}, choosing $P_\infty=\infty_+$ we have that for every divisor $A$ prime to $S$ of degree 0 there exists a unique effective divisor $B$ of degree $l\leq g'$, with $l$ minimal, such that $A=B-l(\infty_+)-\div (f)$ with $f\inn O'_S$. It  is possible that $f\equiv0\pmod{\M_i}$ for some $i$, so it is possible that $B$ is no longer prime to $S$.

\begin{notat}	
	Let $A,A'$ be divisors of $\C C$ prime to $S$. We will write $A\approx A'$ if and only if $A=A'+\div(f)$ with $f\inn(O'_S)^*$; in this case we will say that $A,A'$ are equivalent with respect to $S$ and we will write $A\sim_S A'$. 

	We will denote by $\Pic_S(\C H)$ the group of the classes of divisors prime to $S$ modulo $\sim_S$ and by $\Pic^0_S(\C H)$ its subgroup formed by the classes of divisors of degree 0.
\end{notat}
	
	The results on generalized Jacobians associated to a modulus can be easily adapted to this case. In particular, $\Pic^0_S(\C H)$ has a structure of algebraic group, denoted by $\C J_S$; we will say that $\C J_S$ is the \textit{generalized Jacobian of $\C H$ associated to $S$}.

\begin{lem} 
	$\C J_S$ is an extension of $\C J$ by $\n G_m^n.$
\end{lem}
\begin{proof}
	As before, we can consider the surjective map $\Phi_S: \C J_S\to\C J$; then $\C J_S$ is an extension of $\C J$ by $\Ker(\Phi_S)$, which is isomorphic to $\n G_m^n$ as algebraic groups. 
\end{proof}

\begin{rem}\label{66rem:genjac3}
	More generally, if $\C C$ is a complete, irreducible, smooth curve we can consider a finite family of pairwise disjoint moduli  $M=\{\M_1,\dots,\M_n\}$, that is, moduli $\M_i=\sum_{P\in S_i}n_P(P)$ such that their supports $S_i$ are pairwise disjoint.
	
	Let $S=\cup_i S_i$, we can consider on $S$ the equivalence relation $$P\mathscr RQ\iff \text{ there exists }i\text{ such that }P,Q\inn S_i.$$ Let $S'=S/\mathscr R$ and, for every $i$, let $Q_i\!=\![P]_{\mathscr R}$ for $P\inn S_i$; let $\C C'\!=\!(\C C\setminus S)\cup S'$. In the notations of \ref{66notat:sing} we will have $O_i\!=\!O_{Q_i}\!=\!\!\bigcap\limits_{P\in S_i}\!\!O_P$, $r_i\!=\!r_{Q_i}\!=\!\{f\inn O_{Q_i},\ \ord_P(f)\!>\!0\text{ for every } P\inn S_i\}$ and we can take $O'_i=O'_{Q_i}=\{f\inn O_i,\ f\equiv k_i\pmod{\M_i}\text{ for some } k_i\inn\n K\}$. Then $O'_i$ is a proper subring of $O_i$ and $\n K+r_i^{n_i}\sub O'_i\sub\n K+r_i$, where $n_i=\max_{P\in S_i}n_P$. Moreover, $f\inn(O'_i)^*$ if and only if $f\equiv k_i\pmod{\M_i}$ with $k_i\inn\n K^*$. 
	As before, $\C C'$ is an algebraic curve, $S'$ is the set of its singular points and $\C C$ is its normalization.\par\medskip
	
	We will say that two divisors $A,A'$ are equivalent with respect to $M$, and we will write $A\sim_M A'$, if $A=A'+\div (f)$ with $f\inn (O'_i)^*$ for every $i$. As before, we will consider $\Pic_M(\C C)=\Div_S(\C C)/\sim_M,\ \Pic_M^0(\C C)=\Div^0_S(\C C)/\sim_M$. Again, $\Pic_M^0$ has a structure of algebraic group, we will denote it by $\C J_M$ and we will call it the generalized Jacobian of $\C C$ relative to $M$\footnote{$\C J_M$ is, in Roselincht's notations \cite{rosenlicht1954generalized}, the generalized Jacobian relative to the semi-local ring $\mathfrak o=\cap_iO_i=\{f\inn\n K(\C C),\ \text{for every } i\text{ there exists } k_i\inn\n K \text{ such that } f\equiv k_i\pmod{\M_i}\}$}. Exactly as before, $\C J_M$ is an extension of $\C J$ by the Kernel of the surjective standard map $$\Phi_M:\C J_M\to\C J.$$ We could study $\Ker\Phi_M$ with the previous methods, finding it to be a product of copies of $\n G_m$ and $\n G_a$. However, we can also see the Jacobian $\C J_M$ as the pullback of the Jacobians $\C J_{\M_i}$.
\end{rem}

\begin{notat}
	Given $A,B,C$ objects in a category $\boldsymbol{C}$ and maps $\alpha\!:\!A\!\to C,\ \beta\!:\!B\!\to C$, their \textit{pullback} is a triple $(P,\pphi,\psi)$, where $P$ is an object and $\pphi,\psi$ are maps with $\pphi:P\to A$ and $\psi:P\to B$, such that $\alpha\circ\pphi=\beta\circ\psi$ and such that for every object $K$ and for every couple of maps $\pphi',\psi'$ such that $\alpha\circ\pphi'=\beta\circ\psi'$ there exists a unique map $h:K\to P$ that makes the following diagram commute  
	$$\begin{tikzcd}K\arrow[bend left]{rrd}{\pphi'}\arrow[dotted]{rd}{h}\arrow[bend right]{rdd}{\psi'}&&\\&P\arrow{r}{\pphi}\arrow{d}{\psi}&A\arrow{d}{\alpha}\\&B\arrow{r}{\beta}&C\end{tikzcd}$$
	When it exists, the pullback is unique (up to isomorphisms).
	
	In the category $\boldsymbol{C}$ of abelian groups, the pullback always exists: in the previous notations we can take $P=\{(a,b)\inn A\times B,\ \alpha(a)=\beta(b)\}$, $\pphi=\pi_A|_P,\psi=\pi_B|_P$ (with $\pi_A,\pi_B$ canonical projections). Let $\chi=\alpha\circ\pphi=\beta\circ\psi$. Then $\Ker(\chi)=\Ker(\alpha)\times\Ker(\beta)$ and $\Imm(\chi)=\Imm(\alpha)\cap\Imm(\beta)$.
	
	In particular, if $A$ is an extension of $C$ by $K_A$ and $B$ is an extension of $C$ by $K_B$, their pullback $P$ is an extension of $C$ by $K_A\times K_B$.
\end{notat}

\begin{lem}
	As before, let $\C C$ be a smooth curve over an algebraically closed field $\n K$, let $\M_1,\M_2$ be disjoint moduli over $\C C$. Then  the generalized Jacobian relative to $M=\{\M_1,\M_2\}$ defined in Remark \ref{66rem:genjac3} is, at least as an abelian group, the pullback of the generalized Jacobians relative to $\M_1,\M_2$.
\end{lem}
\begin{proof}	
	Let $\C J_1,\C J_2$ be the generalized Jacobians relative to the moduli $\M_1,\M_2$. Let $\Phi_i:\C J_i\to\C J$ be the canonical homomorphism, that is, let $\Phi_i([A]_{\M_i})=[A]$. Then $\C J_i$ is an extension of the usual Jacobian $\C J$ by a linear group $\Ker\Phi_i\simeq\n G_m^{m_i}\times\n G_a^{a_i}$. 
	
	We can consider the pullback $\C J'$ of $\Phi_1,\Phi_2$; as an abelian group then  \begin{tikzcd} \C J'\arrow{r}\arrow{d}&\C J_1\arrow{d}{\Phi_1}\\\C J_2\arrow{r}{\Phi_2}&\C J\end{tikzcd}
	$\C J'=\left\{([A_1]_{\M_1},[A_2]_{\M_2})\inn\C J_1\times\C J_2,\ [A_1]=[A_2]\right\}$, which will be an extension of $\C J$ by $\n G_m^{m_1+m_2}\times\n G_a^{a_1+a_2}$.\par\medskip
	
	Now, $[A_1]=[A_2]$, if and only if $A_1=A+\div(f_1),\ A_2=A+\div(f_2)$ for some divisor $A$, which we can assume to be prime to $S_1,S_2$ and for some rational functions $f_1,f_2\inn\n K(\C C)$. We can always construct a rational function $h$ such that $ h/f_1\equiv k_1\pmod{\M_1}$ and $h/f_2\equiv k_2\pmod{\M_2}$ for some $k_1,k_2\inn\n K^*$. Then $[A_1]_{\M_1}=[A+\div(h)]_{\M_1}$ and $[A_2]_{\M_2}=[A+\div(h)]_{\M_2}$, so $$\C J'=\{([A]_{\M_1},[A]_{\M_2}),\ A\inn\Div^0_{S_1\cup S_2}(\C C)\}.$$
	
	Now, $([A]_{\M_1},[A]_{\M_2})=0$ if and only if $A=\div (f)$ and $f$ is a rational function such that $f\equiv k_i\mod{\M_i}$, with $k_i\inn\n K^*$, for $i=1,2$, if and only if $A\sim_M 0$.
	Thus $$\C J'\simeq\Pic^0_M(\C C).$$
\end{proof}

\begin{rem}
	Iterating this reasoning we will have that the generalized Jacobian relative to a family of moduli $M=\{\M_1,\dots,\M_n\}$ is the pullback of the generalized Jacobians $\C J_{\M_1},\dots,\C J_{\M_n}$ and thus it is an extension of $\C J$ by $\prod_i \n G_m^{m_i}\times\n G_a^{a_i}$.\par\medskip
	
	In particular, choosing $\M_i=P_i+P_i'$ as in Notation \ref{66lem:genJac2}, so that $\Ker\Phi_{\M_i}\simeq\n G_m$ for every $i$, we will find again that $\C J_S$ is an extension of $\C J$ by $\n G_m^n$.
\end{rem}

\begin{rem}
	Let $D\inn\n K[T]$ be a squarefree polynomial and let $\C H=\C H_D: U^2=D(T)$; let $t_1,\dots,t_n\inn\n K$ be distinct constants, which are not zeros of $D$. 
	
	Let $\M_t$ be the modulus $(P_t)+(P_t')$, with $P_t=\left(t,\sqrt{D(t)}\right)$ and let $\C J_t$ be the corresponding generalized Jacobian. More generally, we will denote by $\C J_{t_1,\dots,t_n}$ the generalized Jacobian relative to the family of moduli $\{\M_{t_1},\dots,\M_{t_n}\}$.\par\medskip
	
	We have seen that for every $m$ there exists an effective divisor $A_m$ of degree at most $d-1+n$ such that $m\delta=A_m-\deg(A_m)(\infty_+)-\div(\pphi_m)$, with $\pphi_m\equiv k_{m,i}\inn\n K\pmod{\M_i}$ for every $m,i$. Moreover, if $\deg A_m$ is minimal, the divisor $A_m$ is unique and in this case, by Remark \ref{6rem:multdeltanonsf}, $\pphi_m=p-q(T-t_1)\cdots (T-t_n)\sqrt D$ with $p/q$ a convergent of $(T-t_1)\cdots(T-t_i)\sqrt D$. However, it is possible that $k_{m,i}=0$ for some $m,i$, so the previous equality cannot be read as an equality in the generalized Jacobian $\C J_{t_1,\dots,t_n}$.\par\medskip
	
	Let $\lambda\inn\n K$ be a fixed constant; let $\cv n$ be the convergents of $\sqrt D$ and let ${u_m}/{v_m}$ be the convergents of $(T-\lambda)\sqrt D$. As we have seen by algebraic methods in Proposition \ref{1prop:prod}, there exists $m$ such that $u_m(\lambda)=0$ if and only if there exists $n$ such that $q_n(\lambda)\neq0$ and $\deg a_{n+1}>1$, where the $a_i$ are the partial quotients of $\sqrt D$.\par\medskip

	More precisely, let $n\delta=(A_1)+\cdots+(A_{d-l})+r(\infty_-)-(d-l+r)(\infty_+)-\div(\pphi)$ with $d-l+r\leq d-1$ minimal and with the $A_i$ affine points of $\C H$. We have seen that this is equivalent to the existence of a convergent $p/q$ of $\sqrt D$ such that $n-r=\deg p$, with $l$ degree of the following partial quotient and $\pphi=p-q\sqrt D$. Now, by Proposition \ref{1prop:prod}, the following are the only possible cases:
	\begin{enumerate}
		\item if $q(\lambda)\neq0$ and $l>1$, then $\frac{(T-\lambda)p}q$ is a convergent of $(T-\lambda)\sqrt D$, $$(n+1)\delta=(A_1)+\cdots(A_{d-l})+(P_\lambda)+(P_\lambda')+r(\infty_-)-(d-l+r+2)(\infty_+)-\div((T-\lambda)\pphi)$$ and, for $d-l+r+2\leq d$ this is minimal. As $(T-\lambda)\pphi\equiv0\pmod{\M_\lambda}$, the previous equality cannot be read on the generalized Jacobian $\C J_\lambda$;
		\item if $q(\lambda)=0$ then $\frac p{q/(T-\lambda)}$ is a convergent of $(T-\lambda)\sqrt D$ and $$n\delta=(A_1)+\cdots+(A_{d-l})+r(\infty_-)-(d-l+r)(\infty_+)-\div(\pphi),$$ with $\pphi\equiv p(\lambda)\neq0\pmod{\M_\lambda}$;
		\item if $q(\lambda)\neq0$, then, denoting by $\cv{n+1}$ the following convergent of $\sqrt D$, we have that $\frac uv$, with $u=p(T) q_{n+1}(\lambda)- p_{n+1}(T)q(\lambda)$ and $v=\frac{q(T) q_{n+1}(\lambda)- q_{n+1}(T)q(\lambda)}{T-\lambda}$, is a convergent of $(T-\lambda)\sqrt D$. Moreover $$(\deg p+l)\delta=(R_1)+\cdots+(R_d)-(d)(\infty_+)-\div(u-v(T-\lambda)\sqrt D),$$ where the $R_i$ are affine points of $\C H$ and $u-v(T-\lambda)\sqrt D\equiv\pm1\pmod{\M_\lambda}$.
	\end{enumerate}
	
	As the first one is the case that most interests us, we must be cautious when applying the previous results on generalized Jacobians.\par\medskip
	
	Now, let $D$ be a squarefree polynomial of degree $2d$; if there exist $d$ different constants $\lambda_1,\dots,\lambda_d\inn\n K$ such that $\lambda_i$ is not a root of $D$ nor of any denominator of the convergents of $(T-\lambda_1)\cdots\widehat{(T-\lambda_i)}\cdots(T-\lambda_d)\sqrt D$, then all the partial quotients of $(T-\lambda_1)\cdots(T-\lambda_d)\sqrt D$ (except the first one) have degree 1 and all the multiples of $\delta$ (except the first $g'-1$) can be written minimally as the sum of exactly $g'=2d-1$ points in the image of the affine part of $\C H$ on the generalized Jacobian $\C J_{\lambda_1,\dots,\lambda_d}$.
\end{rem}

\begin{ex}
	Let $D(T)=T^6+1\inn\n C[T]$; it is easy to see that $D$ is a squarefree Pellian polynomial with $\sqrt D=\left[T^3,\ov{2T^3}\right]$.
	
	We will always denote by $\cv n$ the convergents of the quadratic irrationality $\alpha$ under consideration and we will write $\pphi_n=p_n-q_n\alpha$. 
	
	As in Lemma \ref{6lem:multdelta}, from the continued fraction expansion of $\sqrt D$ follows that
	$$\begin{array}{ll} 
		p_{-1}=1,\ q_{-1}=0&\ \delta=(\infty_-)-(\infty_+)\\
		&2\delta=2(\infty_-)-2(\infty_+)\\
		p_0=T^3,\ q_0=1&3\delta=-\div(\pphi_0)\\
		&4\delta=(\infty_-)-(\infty_+)-\div(\pphi_0)\\
		&5\delta=2(\infty_-)-2(\infty_+)-\div(\pphi_0)\\
		p_1=2T^6+1,\ q_1=2T^3&6\delta=-\div(\pphi_1)\\
		&7\delta=(\infty_-)-(\infty_+)-\div(\pphi_1)\\
		&8\delta=2(\infty_-)-2(\infty_+)-\div(\pphi_1)\\
		\cdots&\cdots
	\end{array}$$
	Obviously, all the previous equalities imply the corresponding equalities on the usual Jacobian $\C J$ of the curve $\C H_D: U^2=D(T)$.\par\medskip
	 
	Let us now consider $(T+1)\sqrt {D(T)}$. It can be seen that
	$$\begin{array}{ll} 
		p_{-1}=1,\ q_{-1}=0&\ \delta=(\infty_-)-(\infty_+)\\
		&2\delta=2(\infty_-)-2(\infty_+)\\
		&3\delta=3(\infty_-)-3(\infty_+)\\
		p_0=T^3(T+1),\ q_0=1&4\delta=2(P)-2(\infty_+)-\div(\pphi_0)\\
		&5\delta=2(P)+(\infty_-)-3(\infty_+)-\div(\pphi_0)\\
		p_1=2T^6+2T^3+1,\ q_1=2T^2-2T+2&6\delta=(P_1)+(P_2)+(P_3)-3(\infty_+)-\div(\pphi_1)\\
		p_2=(T+1)(-T^6-\frac12),\ q_2=-T^3&7\delta=2(P)-2(\infty_+)-\div(\pphi_2)\\
		&8\delta=2(P)+(\infty_-)-3(\infty_+)-\div(\pphi_2)\\
		\cdots&\cdots
	\end{array}$$
	where $P,P_1,P_2,P_3$ are suitable affine points of $\C H_D$.
	
	Now, $\pphi_1\equiv1\pmod{\M_{-1}}$ but $\pphi_0,\pphi_2\equiv0\pmod{\M_{-1}}$, so only the sixth equality (and the first three) can be read in the generalized Jacobian $\C J_{-1}$.\par\medskip
	
	Let us consider $(T-1)(T+1)\sqrt {D(T)}$. Then	
	$$\begin{array}{ll} 
		p_{-1}=1,\ q_{-1}=0&\ \delta\!=\!(\infty_-)-(\infty_+)\\
		&2\delta\!=\!2(\infty_-)-2(\infty_+)\\
		&3\delta\!=\!3(\infty_-)-3(\infty_+)\\
		&4\delta\!=\!4(\infty_-)-4(\infty_+)\\
		p_0\!=\!(T-1)(T+1)T^3,\ q_0=1&5\delta\!=\!2(P_1)+2(P_2)-4(\infty_+)-\div(\pphi_0)\\
		p_1\!=\!2T^6-2T^4+1,\ q_1=2T&6\delta\!=\!2(P_3)+2(P_4)-4(\infty_+)-\div(\pphi_1)\\
		p_2\!=\!(T^6-T^2-\frac12)T,\ q_2=T^2+1&7\delta\!=\!(Q_1)+\cdots+(Q_4)\!-\!4(\infty_+)\!-\!\div(\pphi_2)\\
		p_3\!=\!(T\!-\!1)(T\!+\!1)(-2T^6\!-\!1),q_3\!=\!-2T^3&8\delta\!=\!2(P_1)+2(P_2)-4(\infty_+)-\div(\pphi_3)\\
		\cdots&\cdots
	\end{array}$$
	where the $P_i,Q_i$ are affine points of $\C H_D$.
	
	It is easy to see that $\pphi_1\equiv\begin{cases}1&\pmod{\M_1}\\1&\pmod{\M_{-1}}\end{cases}$ and $\pphi_2\equiv\begin{cases}-\frac12&\pmod{\M_1}\\\frac12&\pmod{\M_{-1}}\end{cases}$, while $\pphi_0,\pphi_3\equiv0$ modulo $\M_{-1}$ and modulo $\M_1$. Thus the fifth and the eighth equalities cannot again be read as equalities in the generalized Jacobian $\C J_{1,-1}$.\par\medskip 
	 		
	Let us consider $(T-2)(T-1)(T+1)\sqrt{D(T)}$. Then we have
	\begin{multline*}p_0=T^6-2T^5-T^4+2T^3+\frac12,\ q_0=1, \text{ so}\\ 6\delta=(P_1)+(P_2)+(P_3)+(P_4)+(P_5)-5(\infty_+)-\div(\pphi_0)\end{multline*}
	\vspace{-1cm}\begin{multline*}p_1=-T^7+\frac52T^6-\frac52T^4+T^3-\frac12T+\frac54,\ q_1=-T+\frac12, \text{ so}\\ 7\delta=(Q_1)+(Q_2)+(Q_3)+(Q_4)+(Q_5)-5(\infty_+)-\div(\pphi_1)
	\end{multline*}
	\vspace{-1cm}\begin{multline*}p_2=\frac45T^8-\frac{32}{25}T^7-\frac45T^6+\frac{32}{25}T^3+\frac25T^2-\frac{16}{25}T-\frac25, q_2=\frac45 T^2+\frac8{25}T+\frac{16}{25}, \text{ so}\\ 8\delta=(R_1)+(R_1)+(R_3)+(R_4)+(R_5)-5(\infty_+)-\div(\pphi_2)\end{multline*}
	\indent $\cdots$
	
	where, as before, the $P_i,Q_i,R_i$ are suitable affine points of $\C H$.
	
	Now, $\pphi_0\equiv\begin{cases}\frac12 &\pmod{\M_{-1}}\\\frac12 &\pmod{\M_1}\\\frac12 &\pmod{\M_2}\end{cases}$, $\pphi_1\equiv\begin{cases}\frac74 &\pmod{\M_{-1}}\\\frac34 &\pmod{\M_1}\\\frac14 &\pmod{\M_2}\end{cases}$, $\pphi_2\equiv\begin{cases} \frac{16}{25}&\pmod{\M_{-1}}\\-\frac{16}{25} &\pmod{\M_1}\\-\frac2{25} &\pmod{\M_2}\end{cases}$, so the three previous equalities can be seen as equalities on the generalized Jacobian $\C J_{-1,1,2}$.	
\end{ex}

\begin{rem}
	We have assumed since the beginning of this section that the base field is algebraically closed. Actually, one could see that, by descent of the base field, generalized Jacobians can be defined over any field.
	
	However, as we have already remarked, the continued fraction of a formal Laurent series, as long as it is defined, does not depend on the choice of the base field. So, even reasoning on the algebraic closure $\ch K$ of $\n K$, we will have that the multiples of $\delta$ that give the convergents of $b\sqrt D$ are linked to rational functions $\pphi_n$ defined over $\n K$.
\end{rem}

\subsection{McMullen's strong Conjecture over number fields}
Using this connection between continued fractions of multiples of $\sqrt D$ and generalized Jacobians, Zannier in \cite{zannier2016hyperelliptic} proved that the sufficient condition for McMullen's Conjecture given in lemma \ref{5lem:MMrootsqn2} holds over any number field. More precisely, he proved the following:

\begin{theo}[Zannier, Theorem 1.7 in \cite{zannier2016hyperelliptic}]\label{66theo:Zrootsqn}
	Let $\n K$ be a number field and let $D\inn\S K$ be a non-Pellian polynomial; let $(\cv n)_n$ be the convergents of $\sqrt D$. Then, for every $l$, there are only finitely many elements of $\ch K$ of degree at most $l$ over $\n K$ which are common zeros of infinitely many $q_n(T)$.  
\end{theo}

\begin{theo}\label{66theo:MMQ}
	Conjecture \ref{5conj:MMfin} holds over any number field.
	
	Actually, as we can choose the constants $\lambda_i$ in infinitely many different ways, also Conjecture \ref{5conj:mult}, and thus McMullen's Conjecture \ref{5conj:MMinf}, hold over any number field.
\end{theo}
\begin{proof}
	If $D$ is Pellian then, applying Theorem \ref{5theo:Mercat}, we can construct infinitely many purely periodic, pairwise non-equivalent, normal elements of $\n K(T,\sqrt D)$.\par\medskip
	
	Let then $D$ be a non-Pellian polynomial.
	
	Taking $l=1$ in the previous Theorem we will have that there exists $\lambda\inn\n K$ such that $q_n(\lambda)=0$ only for finitely many $n$. 
	
	Then, as in Lemma \ref{5lem:MMrootsqn2}, for every polynomial $D\inn\S K$ of degree $2d$ there exist $\lambda_1,\dots,\lambda_{d-1}\inn\n K$ such that $$\ov K((T-\lambda_1)\cdots(T-\lambda_{d-1})\sqrt D)=1,$$  that is, all the partial quotients of $\alpha=(T-\lambda_1)\cdots(T-\lambda_{d-1})\sqrt D$, except possibly for finitely many of them, have degree 1. Then, denoting by $\alpha_n$ the complete quotients of $\alpha$, for $n$ large enough all the partial quotients of $\alpha_n$ will have degree 1, that is, there exists $N$ such that $K(\alpha_n)=1$ for every $n\geq N$ (and $\alpha_n\inn\n K(T,\sqrt D)\setminus\n K(T)$ for every $n$).
\end{proof}

As it was already noticed in Remark \ref{5rem:QimpchQ}, this immediately implies McMullen's Conjecture over $\ch Q$:
\begin{cor}
	Conjecture \ref{5conj:mult} and Conjecture \ref{5conj:MMinf} hold over $\ch Q$.
\end{cor}

Zannier proved the previous result applying to the generalized Jacobians a generalization of Skolem-Mahler-Lech's Theorem for algebraic groups:

\begin{lem}
	Let $\Gamma$ be an algebraic group over $\n C$ and let $\gamma\inn\Gamma$. Then:
	\begin{enumerate}
		\item for every integer $b$, the Zariski closure $Z(b)\sub\Gamma$ of the set $\{\gamma^{nb},\ n\inn\n N\}$ is a commutative algebraic subgroup of $\Gamma$;
		\item there exists an integer $b_0$ such that $Z(b_0)=Z_0$ is the connected component of the identity of $Z(1)$;
		\item for every $b\neq0$, $Z(b)$ is a finite union of cosets of $Z_0$.
	\end{enumerate}
\end{lem}

\begin{theo}[Skolem-Mahler-Lech for algebraic groups,  Theorem 3.2 in \cite{zannier2016hyperelliptic}]
	Let $\Gamma$ be an algebraic group over $\n C$, let $\gamma\inn\Gamma$ and let $(a_n)_n$ be a sequence of integers. Then the Zarisk closure of $\{\gamma^{a_n},\ n\inn\n K\}$ is a finite union of points and cosets of the connected component of the identity of the Zariski closure of $\{\gamma^n,\ n\inn\n Z\}$.
\end{theo}

\begin{cor}\label{66cor:SML}
	Let $U\sub\Gamma$ be a constructible set and let $K=\{k\inn\n Z,\ \gamma^k\inn U\}$. Then $K$ is a finite union of arithmetic progressions (modulo the integer $b_0$ that appears in $2.$ of the previous Lemma), plus or minus finite sets.
\end{cor}

Let $\n K$ be a number field and let $D\inn\S K$, that is, let $D\inn\n K[T]$ be a non-square polynomial of even degree whose leading coefficient is a square in $\n K$; let us assume that $D$ is non-Pellian. Let $\widetilde D$ be the squarefree part of $D$ and let $\C H=\C H_{\widetilde D}$ be the hyperelliptic curve $\C H:U^2=\widetilde D(T)$; let $\C J$ be the Jacobian of $\C H$. $D/\widetilde D$ is the square of a squarefree polynomial. 
As before, let $\infty_+,\infty_-$ be the points at infinity of $\C H$ and let $\delta=[(\infty_-)-(\infty_+)]\inn\C J$. As we have seen before, the convergents of $\sqrt D$ correspond to multiples of $\delta$ written as minimal sum of points in the image of the affine part of $\C H$ on an appropriate generalized Jacobian $\C J_M$; Zannier thus applied Corollary \ref{66cor:SML} in the case $\Gamma=\C J_\M$, $\gamma=\delta$. 

These tools also allowed Zannier to prove the results already quoted in Theorem \ref{2theo:Z2} and Remark \ref{2rem:Z1}.
\clearpage{\pagestyle{empty}\cleardoublepage}
\appendix
\chapter{Hyperbolic plane and continued fractions}\label{ch:hyper}
The continued fraction expansions of real numbers can be related to geodesics of the hyperbolic plane $\n H$ (and of the modular surface $\n M=\SL_2(\n Z)\backslash\n H$). Indeed, the Gauss map $T$, which encodes the continued fraction algorithm, that is the map from $(0, 1)$ to itself defined by $T:x\mapsto\{1/x\}$, where $\{\cdot\}$ denotes the fractional part, can be seen as a first-return map of the geodesic flow on the modular surface. This was already known to Artin \cite{artin1924mechanisches}, who used properties of the continued fractions to deduce the existence of a dense geodesic on $\n M$. Some variations of his method, more suitable for the study of continued fractions, can be found, for instance, in \cite{series1985modular} (whose approach we will follow). In particular, McMullen used this methods to discuss quadratic continued fractions with bounded partial quotients (see Theorem \ref{3theo:MM}), leading to Conjecture \ref{3conj:MM}.\par\medskip 

Similarly, A. Broise-Alamichel and F. Paulin  \cite{broise2007dynamique}, \cite{paulin2002groupe} studied the relation between continued fraction expansions of Laurent series and geodesic flow in Bruhat-Tits trees, extending Artin's work to function fields. They showed that the polynomial analogue of the Gauss map, $T:f\mapsto\{1/f\}$ for $f\inn\n L \setminus \{0\}, \deg f<0$ (where $\{\cdot\}$ is the polynomial analogue of the fractional part) is the first return map of the geodesic flow on the modular surface, which is a quotient of Bruhat-Tits tree by a lattice subgroup.

\section{Elements of hyperbolic geometry}
Let $\n H$ be the upper half complex plane, $$\n H=\{z\inn\n C,\ \Im(z)>0\}.$$

$\n H$ with the Poincar\'e metric $dt^2=\frac{dx^2+dy^2}{y^2}$ is a model for the hyperbolic plane.\par\medskip

Certainly, the boundary of $\n H$ is the projective real line, $\partial\n H=\n R\cup\{\infty\}$. Any point of $\n H$ is at infinite distance from any point of the boundary (while the distance between points on $\partial\n H$ is not defined); let $\ov{\n H}=\n H\cup\partial\n H$.

Let us denote by $s$ the imaginary semi-axis.

\begin{lem}
	$\SL_2(\n R)$ acts on the left on $\n H$ by Möbius transformations, that is, by $$\ds\mm abcd z=\frac{az+b}{cz+d}.$$
\end{lem}
As $-\id$ acts trivially on $\n H$, we have in fact an action of $\PSL R$. To simplify the notation, we will still write $\mm abcd$ for its class in $\PSL R$.

\begin{lem}
	The action of $\PSL R$ on $\n H$ is isometric, that is, for every $z_0,z_1\inn\n H$ and for every $A\inn\PSL R$, $d(Az_0,Az_1)=d(z_0,z_1)$.
\end{lem}

This action extends naturally to $\partial\n H$, setting $\mm abcd(\infty)\!=\!\frac ac$ and $\mm abcd \left(-\frac dc\right)\!=\!\infty$; in particular, for every $A\inn\PSL R$ we have $A(\partial\n H)=\partial\n H$.

\begin{lem}
	The geodesics of $\n H$ are exactly the semicircles centred on $\n R$ and the vertical lines; moreover, any geodesic is the image of the imaginary axis $s$ under an isometry (that is, under a Möbius transformation or under the composition of a Möbius transformation with $z\mapsto-\ov z$).
	
	There exists a unique geodesic connecting any couple of distinct points of $\ch H$. However, for every geodesic $r$ and for every point $z$ outside of it, there exist infinitely many geodesics through $z$ that do not intersect $r$ in $\n H$.
\end{lem}

If $r$ is an oriented geodesic, we will denote by $r^\pm$ its endpoints on $\partial\n H$.

\begin{lem}\label{Alem:unittg} 
	The \textit{tangent bundle} of $\n H$ is $T\n H=\n H\times\n C$ and the \textit{unit tangent bundle} of $\n H$ is $T^1\n H=\left\{(z,v)\inn T\n H,\ |v|_z=1\right\}$.

	The map $\mm abcd (z,v)\mapsto\left(\frac{az+b}{cz+d},\frac v{(cz+d)^2}\right)$ gives a free and transitive action of $\PSL R$ on $T^1\n H$ so, fixing a base point $(z,v)\inn T^1\n H$, we can identify $\PSL R$ and $T^1\n H$. 
\end{lem} 

	The \textit{tangent flow} on $T^1\n H$ is $g:\n R\times T^1\n H\to T^1\n H$, with $g(t,(z,v))=(z',v')$, where $(z',v')$ is the unit tangent vector to the geodesic identified by $(z,v)$ in the point $z'$ at distance $t$ from $z$ in the direction of $v$.

	Choosing as base point $(i,i)$, with the previous identification, we have $$g(t,A)=AG_t, \text{ with } G_t=\mm{e^{t/2}}00{e^{-t/2}},$$ thus $g(t,(z,v))=AG_tA^{-1}(z,v)$, where $A\inn\PSL R$ is the unique matrix such that $(z,v)=A(i,i)$.

	More generally, choosing a base point $(z_0,v_0)$, we have $g(t,(z,v))\!\!=\!B\widetilde{G_t}B^{-1}(z,v)$ with $\widetilde{G_t}=CG_tC^{-1}$, where $C$ is the unique matrix such that $(z_0,v_0)=C(i,i)$ and $B$ is the unique matrix such that $(z,v)=B(z_0,v_0)$.

\section{Cutting sequences and continued fractions}
Let $\Gamma=\PSL Z$ be the modular group; $\Gamma$ is a discrete subgroup of $\PSL R$ (with the topology induced by the euclidean topology on $\n R^4$).

The set $$D=\{z\inn\n H,\ 0\leq\Re(z)\leq1,\ |z|\geq1,\ |z-1|\geq1\}$$ is a fundamental domain for $\Gamma$, that is, $\n H=\bigcup_{A\in\PSL R}AD$ and, denoting by $\mathring D$ the interior of $D$, $A\mathring D\cap B\mathring D\neq\emptyset$ if and only if $A=B$.\par\medskip

We will call ideal triangle a triangle all of whose vertices lie on $\partial\n H$.

Let $\Delta=\left\{z\inn\n H,\ 0\leq\Re(z)\leq1,\ |z-\frac12|\geq\frac12\right\}$ that is, let $\Delta$ be the ideal triangle of vertices $0,1,\infty$. Then $\Delta=D\cup AD\cup A^2D$, where $A=\mm 1{-1}10$. $\Delta$ is then a fundamental domain for some index three subgroup of $\Gamma$; in particular, it is a fundamental domain for the Hecke theta group, that is, the group generated by $\mm 0{-1}10$ and $\mm1201$.

\begin{center}
	\begin{tikzpicture}
		\tkzDefPoint(-1.5,0){-inf}
		\tkzDefPoint(2.5,0){inf}
		\tkzDrawSegment(-inf,inf)
		\tkzDefPoint(0,1){i}
		\tkzDefPoint(0,0){O}
		\tkzDrawPoints(i,O)
		\tkzDefPoint(1,0){1}
		\tkzLabelPoints(O)
		\tkzLabelPoints[above left](i)
		\tkzDefPoint(1/2,1.8){A}
		\begin{scope}
			\tkzClipCircle(A,i)
			\tkzDrawCircle(1,O)
			\tkzDrawCircle(O,1)
			\tkzDrawLine[add=0 and 1](O,i)
			\tkzDefPoint(1,1){1+i}
			\tkzDrawLine[add=0 and 1](1,1+i)
			\tkzDrawPoints(1)
		\end{scope}
		\tkzDefPoint(-1.5,1){B}
		\tkzDefPoint(1.5,1){C}
		\tkzDefPoint(0.2,2){D}
		\tkzLabelPoints(D)
		\tkzClipPolygon(-inf,inf,C,B)
		\tkzDrawLine[add=0 and 1, style=dashed](O,i)
		\tkzDrawLine[add=0 and 1, style=dashed](1,1+i)
		\tkzDrawCircle[style=dashed](1,O)
		\tkzDrawCircle[style=dashed](O,1)
	\end{tikzpicture}
	\begin{tikzpicture}
		\tkzDefPoint(-1,0){-inf}
		\tkzDefPoint(2,0){inf}
		\tkzDefPoint(-1.5,2){B}
		\tkzDefPoint(1.5,2){C}
		\tkzLabelPoints[below](O)
		\tkzClipPolygon(-inf,inf,C,B)
		\tkzDrawSegment(-inf,inf)
		\tkzDefPoint(0,1){i}
		\tkzDefPoint(0,0){O}
		\tkzDefPoint(1,0){1}
		\tkzDefPoint(1,1){1+i}
		\tkzDefPoint(1/2,1.8){A}
		\begin{scope}
			\tkzClipCircle(A,i)
			\tkzDrawCircle(1,O)
			\tkzDrawCircle(O,1)
		\end{scope}
		\tkzDrawLine[add=0 and 1](O,i)
		\tkzDrawLine[add=0 and 1](1,1+i)
		\tkzDefPoint(1/2,0){1/2}
		\tkzDrawCircle(1/2,1)
		\tkzInterCC(1,O)(O,1)
		\tkzGetPoints{P}{Q}
		\tkzDefPoint(1/2,1/2){R}
		\tkzDrawSegment(Q,R)
		\tkzDefPoint(0.2,2){D}
		\tkzLabelPoints(D)
		\tkzDefPoint(0.4,0.9){AD}
		\tkzLabelPoint(AD){\tiny{$A\!D$}}
		\tkzDefPoint(-0.15,0.95){A2D}
		\tkzLabelPoint(A2D){\tiny{$A^2\!\!D$}}
		\tkzDefPoint(1.2,1.2){DE}
		\tkzLabelPoint(DE){$\Delta$}
	\end{tikzpicture}
\end{center}

The tessellation $\C F$ of $\n H$ by ideal triangles given by the images of $\Delta$ is called the \textit{Farey tessellation}. The Farey lines, that is, the sides of the triangles in $\C F$, are exactly the images of the imaginary axis $s$ under $\Gamma$ and the vertex set of $\C F$ is precisely $\n Q\cup\{\infty\}$. More precisely, a geodesic $r$ is a Farey line if and only if its endpoints are rationals $p/q,\ p'/q'$ such that $pq'-p'q=\pm1$. 

\begin{center}
  \begin{tikzpicture}[scale=3]
		\tkzDefPoint(-0.7,0){-inf}
		\tkzDefPoint(1.7,0){inf}
		\tkzDrawSegment(-inf,inf)
		\tkzDefPoint(-0.7,0.8){-iinf}
		\tkzDefPoint(1.7,0.8){iinf}
		\tkzDefPoint(-1,0.5){F}
		\tkzLabelPoint(F){$\C F$}
		\tkzDefPoint(0,0){0}
		\tkzDefPoint(0,1){i}
		\tkzDefPoint(1,1){i+1}
		\tkzDefPoint(1,0){1}
		\tkzLabelPoints[below](0,1)
		\tkzDrawLine[add=0 and -0.2](0,i)
		\tkzDrawLine[add=0 and -0.2](1,i+1)
		\tkzDefPoint(1/2,0){1/2}
		\tkzLabelPoint[below](1/2){$\frac12$}
		\tkzDefPoint(1/4,0){1/4}
		\tkzLabelPoint[below](1/4){$\frac14$}
		\tkzDefPoint(1/3,0){1/3}
		\tkzLabelPoint[below](1/3){$\frac13$}
		\tkzDefPoint(2/3,0){2/3}
		\tkzLabelPoint[below](2/3){$\frac23$}
		\tkzDefPoint(3/4,0){3/4}
		\tkzLabelPoint[below](3/4){$\frac34$}
		\tkzDrawArc(1/2,1)(0)
		\tkzDrawArc(1/4,1/2)(0)
		\tkzDrawArc(3/4,1)(1/2)
		\tkzDefPoint(1/6,0){1/6}
		\tkzDrawArc(1/6,1/3)(0)
		\tkzDefPoint(5/12,0){5/12}
		\tkzDrawArc(5/12,1/2)(1/3)
		\tkzDefPoint(7/12,0){7/12}
		\tkzDrawArc(7/12,2/3)(1/2)
		\tkzDefPoint(5/6,0){5/6}
		\tkzDrawArc(5/6,1)(2/3)
		\tkzDefPoint(1/8,0){1/8}
		\tkzDrawArc(1/8,1/4)(0)
		\tkzDefPoint(7/24,0){7/24}
		\tkzDrawArc(7/24,1/3)(1/4)
		\tkzDefPoint(17/24,0){17/24}
		\tkzDrawArc(17/24,3/4)(2/3)
		\tkzDefPoint(7/8,0){7/8}
		\tkzDrawArc(7/8,1)(3/4)
		
		\tkzClipPolygon(-inf,inf,iinf,-iinf)
		\tkzDefPoint(3/2,0){3/2}
		\tkzDefPoint(2,0){2}
		\tkzDrawArc(3/2,2)(1)
		\tkzDefPoint(5/4,0){5/4}
		\tkzDrawArc(5/4,3/2)(1)
		\tkzDefPoint(7/4,0){7/4}
		\tkzDrawArc(7/4,2)(3/2)
		\tkzDefPoint(7/6,0){7/6}
		\tkzDefPoint(4/3,0){4/3}
		\tkzDrawArc(7/6,4/3)(1)
		\tkzDefPoint(17/12,0){17/12}
		\tkzDrawArc(17/12,3/2)(4/3)
		\tkzDefPoint(19/12,0){19/12}
		\tkzDefPoint(5/3,0){5/3}
		\tkzDrawArc(19/12,5/3)(3/2)
		\tkzDefPoint(11/6,0){11/6}
		\tkzDrawArc(11/6,2)(5/3)
		\tkzDefPoint(9/8,0){9/8}
		\tkzDrawArc(9/8,5/4)(1)
		\tkzDefPoint(31/24,0){31/24}
		\tkzDrawArc(31/24,4/3)(5/4)
		\tkzDefPoint(41/24,0){41/24}
		\tkzDrawArc(41/24,7/4)(5/3)
		\tkzDefPoint(15/8,0){15/8}
		\tkzDrawArc(15/8,2)(7/4)
		
		\tkzDefPoint(-1/2,0){-1/2}
		\tkzDefPoint(-1,0){-1}
		\tkzDrawArc(-1/2,0)(-1)
		\tkzDefPoint(-3/4,0){-3/4}
		\tkzDrawArc(-3/4,-1/2)(-1)
		\tkzDefPoint(-1/4,0){-1/4}
		\tkzDrawArc(-1/4,0)(-1/2)
		\tkzDefPoint(-5/6,0){-5/6}
		\tkzDefPoint(-2/3,0){-2/3}
		\tkzDrawArc(-5/6,-2/3)(-1)
		\tkzDefPoint(-7/12,0){-7/12}
		\tkzDrawArc(-7/12,-1/2)(-2/3)
		\tkzDefPoint(-5/12,0){-5/12}
		\tkzDefPoint(-1/3,0){-1/3}
		\tkzDrawArc(-5/12,-1/3)(-1/2)
		\tkzDefPoint(-1/6,0){-1/6}
		\tkzDrawArc(-1/6,0)(-1/3)
		\tkzDefPoint(-7/8,0){-7/8}
		\tkzDrawArc(-7/8,-3/4)(-1)
		\tkzDefPoint(-17/24,0){-17/24}
		\tkzDrawArc(-17/24,-2/3)(-3/4)
		\tkzDefPoint(-7/24,0){-7/24}
		\tkzDrawArc(-7/24,-1/4)(-1/3)
		\tkzDefPoint(-1/8,0){-1/8}
		\tkzDrawArc(-1/8,0)(-1/4)
	\end{tikzpicture}
\end{center}

Let $r$ be an oriented geodesic that is not a side of $\C F$. Then $r$ is cut into segments by the triangles of $\C F$. If $r$ intersects a triangle $\Delta'$ of $\C F$, either it cuts two sides of $\Delta'$ (that meet each other at a vertex $v$) or it cuts only one of its sides (and one of the endpoints of $r$ is a vertex $v$ of $\Delta'$). In the first case, we will say that the segment cut by $\Delta'$ on $r$ is \textit{relative to} $v$ and we will label it with $R$ (respectively, $L$) if $v$ is on the right (respectively, on the left) of $r$ with respect to its orientation. In the second case, we will say that the segment is relative to $v$ and label it, indifferently, with $R$ or $L$; we will also say that it is a terminal segment.

\begin{center}
	\begin{tikzpicture}[scale=5]
		\tkzDefPoint(0,0){0}
		\tkzLabelPoint[below](0){$v$}
		\tkzDefPoint(1/2,0){1/2}
		\tkzDefPoint(1/3,0){1/3}
		\tkzDefPoint(1/4,0){1/4}
		\tkzDefPoint(0.7,0){t}
		\tkzDefPoint(0,0.4){v1}
		\tkzDefPoint(0.7,0.4){v2}
		\tkzClipPolygon(0,t,v2,v1)
		\tkzDrawArc(1/4,1/2)(0)
		\tkzDefPoint(5/12,0){5/12}
		\tkzDrawArc(5/12,1/2)(1/3)
		\tkzDefPoint(1/6,0){1/6}
		\tkzDrawArc(1/6,1/3)(0)
		\tkzDrawCircle(t,1/6)
		\tkzDefPoint(0.2,0.25){L}
		\tkzLabelPoint(L){\small{$L$}}
		\tkzDefPoint(0.21,0){P1}
		\tkzDefPoint(0.21,0.1){P2}
		\tkzInterLC(P1,P2)(t,1/6)
		\tkzGetPoints{I1}{I2}
		\tkzDrawArc[arrows=<-](t,I2)(1/6)
	\end{tikzpicture}
	\begin{tikzpicture}[scale=5]
		\tkzDefPoint(0,0){0}
		\tkzDefPoint(1/2,0){1/2}
		\tkzDefPoint(1/3,0){1/3}
		\tkzDefPoint(1/4,0){1/4}
		\tkzDefPoint(0.7,0){t}
		\tkzDrawArc(1/4,1/2)(0)
		\tkzDefPoint(5/12,0){5/12}
		\tkzDrawArc(5/12,1/2)(1/3)
		\tkzDefPoint(1/6,0){1/6}
		\tkzDrawArc(1/6,1/3)(0)
		\tkzLabelPoint[below](1/3){$v$}
		\tkzDefPoint(0,0.4){v1}
		\tkzDefPoint(0.7,0.4){v2}
		\tkzClipPolygon(0,t,v2,v1)
		\tkzDrawCircle(t,1/3)
		\tkzDefPoint(0.2,0.25){L}
		\tkzLabelPoint(L){\footnotesize{$L/R$}}
		\tkzDefPoint(0.36,0){P1}
		\tkzDefPoint(0.36,0.1){P2}
		\tkzInterLC(P1,P2)(t,1/3)
		\tkzGetPoints{I1}{I2}
		\tkzDrawArc[arrows=<-](t,I2)(1/3)
	\end{tikzpicture}
	\begin{tikzpicture}[scale=5]
		\tkzDefPoint(0,0){0}
		\tkzDefPoint(1/2,0){1/2}
		\tkzDefPoint(1/3,0){1/3}
		\tkzDefPoint(1/4,0){1/4}
		\tkzDrawArc(1/4,1/2)(0)
		\tkzDefPoint(5/12,0){5/12}
		\tkzDrawArc(5/12,1/2)(1/3)
		\tkzDefPoint(1/6,0){1/6}
		\tkzDrawArc(1/6,1/3)(0)
		\tkzLabelPoint[below](1/2){$v$}
		\tkzDefPoint(0,0.4){v1}
		\tkzDefPoint(0.7,0.4){v2}
		\tkzClipPolygon(0,t,v2,v1)
		\tkzDrawCircle(t,5/12)
		\tkzDefPoint(0.32,0.18){L}
		\tkzLabelPoint(L){\small{$R$}}
		\tkzDefPoint(0.44,0){P1}
		\tkzDefPoint(0.44,0.1){P2}
		\tkzInterLC(P1,P2)(t,5/12)
		\tkzGetPoints{I1}{I2}
		\tkzDrawArc[arrows=<-](t,I2)(5/12)
	\end{tikzpicture}
\end{center}

The sequence $\cdots L^{a_0}R^{a_1}L^{a_2}\cdots$ associated to $r$, with $a_n\inn\n N\setminus\{0\}$ for every $n$, is called the \textit{cutting sequence} of $r$. 

\begin{center}
	\begin{tikzpicture}[scale=3]
	\tkzDefPoint(-0.7,0){-inf}
	\tkzDefPoint(1.7,0){inf}
	\tkzDrawSegment(-inf,inf)
	\tkzDefPoint(-0.7,1){-iinf}
	\tkzDefPoint(1.7,1){iinf}
	\tkzDefPoint(0,0){0}
	\tkzDefPoint(0,1){i}
	\tkzDefPoint(1,1){i+1}
	\tkzDefPoint(1,0){1}
	\tkzDrawSegment(0,i)
	\tkzDrawSegment(1,i+1)
	\tkzDefPoint(1/2,0){1/2}
	\tkzDefPoint(1/4,0){1/4}
	\tkzDefPoint(1/3,0){1/3}
	\tkzDefPoint(2/3,0){2/3}
	\tkzDefPoint(3/4,0){3/4}
	\tkzDrawArc(1/2,1)(0)
	\tkzDrawArc(1/4,1/2)(0)
	\tkzDrawArc(3/4,1)(1/2)
	\tkzDefPoint(1/6,0){1/6}
	\tkzDrawArc(1/6,1/3)(0)
	\tkzDefPoint(5/12,0){5/12}
	\tkzDrawArc(5/12,1/2)(1/3)
	\tkzDefPoint(7/12,0){7/12}
	\tkzDrawArc(7/12,2/3)(1/2)
	\tkzDefPoint(5/6,0){5/6}
	\tkzDrawArc(5/6,1)(2/3)
	\tkzDefPoint(1/8,0){1/8}
	\tkzDrawArc(1/8,1/4)(0)
	\tkzDefPoint(7/24,0){7/24}
	\tkzDrawArc(7/24,1/3)(1/4)
	\tkzDefPoint(17/24,0){17/24}
	\tkzDrawArc(17/24,3/4)(2/3)
	\tkzDefPoint(7/8,0){7/8}
	\tkzDrawArc(7/8,1)(3/4)
	
	\tkzClipPolygon(-inf,inf,iinf,-iinf)
	\tkzDefPoint(3/2,0){3/2}
	\tkzDefPoint(2,0){2}
	\tkzDrawArc(3/2,2)(1)
	\tkzDefPoint(5/4,0){5/4}
	\tkzDrawArc(5/4,3/2)(1)
	\tkzDefPoint(7/4,0){7/4}
	\tkzDrawArc(7/4,2)(3/2)
	\tkzDefPoint(7/6,0){7/6}
	\tkzDefPoint(4/3,0){4/3}
	\tkzDrawArc(7/6,4/3)(1)
	\tkzDefPoint(17/12,0){17/12}
	\tkzDrawArc(17/12,3/2)(4/3)
	\tkzDefPoint(19/12,0){19/12}
	\tkzDefPoint(5/3,0){5/3}
	\tkzDrawArc(19/12,5/3)(3/2)
	\tkzDefPoint(11/6,0){11/6}
	\tkzDrawArc(11/6,2)(5/3)
	\tkzDefPoint(9/8,0){9/8}
	\tkzDrawArc(9/8,5/4)(1)
	\tkzDefPoint(31/24,0){31/24}
	\tkzDrawArc(31/24,4/3)(5/4)
	\tkzDefPoint(41/24,0){41/24}
	\tkzDrawArc(41/24,7/4)(5/3)
	\tkzDefPoint(15/8,0){15/8}
	\tkzDrawArc(15/8,2)(7/4)
	
	\tkzDefPoint(-1/2,0){-1/2}
	\tkzDefPoint(-1,0){-1}
	\tkzDrawArc(-1/2,0)(-1)
	\tkzDefPoint(-3/4,0){-3/4}
	\tkzDrawArc(-3/4,-1/2)(-1)
	\tkzDefPoint(-1/4,0){-1/4}
	\tkzDrawArc(-1/4,0)(-1/2)
	\tkzDefPoint(-5/6,0){-5/6}
	\tkzDefPoint(-2/3,0){-2/3}
	\tkzDrawArc(-5/6,-2/3)(-1)
	\tkzDefPoint(-7/12,0){-7/12}
	\tkzDrawArc(-7/12,-1/2)(-2/3)
	\tkzDefPoint(-5/12,0){-5/12}
	\tkzDefPoint(-1/3,0){-1/3}
	\tkzDrawArc(-5/12,-1/3)(-1/2)
	\tkzDefPoint(-1/6,0){-1/6}
	\tkzDrawArc(-1/6,0)(-1/3)
	\tkzDefPoint(-7/8,0){-7/8}
	\tkzDrawArc(-7/8,-3/4)(-1)
	\tkzDefPoint(-17/24,0){-17/24}
	\tkzDrawArc(-17/24,-2/3)(-3/4)
	\tkzDefPoint(-7/24,0){-7/24}
	\tkzDrawArc(-7/24,-1/4)(-1/3)
	\tkzDefPoint(-1/8,0){-1/8}
	\tkzDrawArc(-1/8,0)(-1/4)
	
	\tkzDrawCircle(2/3,-7/24)
	\tkzDefPoint(2/3,23/24){h}
	\tkzDrawArc[arrows=<-](2/3,h)(-7/24)
	
	\tkzDefPoint(-0.31,0.1){L1}
	\tkzLabelPoint[right](L1){\tiny{L}}
	
	\tkzDefPoint(-0.26,0.18){R1}
	\tkzLabelPoint[left](R1){\tiny{R}}
	
	\tkzDefPoint(-0.2,0.34){R2}
	\tkzLabelPoint[left](R2){\scriptsize{R}}
	
	\tkzDefPoint(-0.15,0.55){R3}
	\tkzLabelPoint[left](R3){\scriptsize{R}}
	
	\tkzDefPoint(1/2,0.9){L2}
	\tkzLabelPoint[right](L2){\footnotesize{L}}
	
	\tkzDefPoint(1.2,0.8){R4}
	\tkzLabelPoint[right](R4){\scriptsize{R}}
	
	\tkzDefPoint(1.6,0.3){L3}
	\tkzLabelPoint[left](L3){\scriptsize{L}}
	
	\tkzDefPoint(1.6,0.14){R5}
	\tkzLabelPoint[right](R5){\tiny{R}}
	\end{tikzpicture}
\end{center}

In general, the cutting sequence of $r$ is doubly infinite; it ends (respectively, starts) if and only if $r^+\inn\n Q$ (respectively, $r^-\inn\n Q$). 

The cutting sequence of $r$ is unique unless it is finite; by convention, we will assume that the first and the second (or the last and the second to last) segments of a terminating sequence are always of the same kind. Actually, this ambiguity corresponds to the two possible regular continued fraction expansions for rational numbers (and our choice corresponds to expansions of the form $[a_0,a_1,\dots,a_n]$ with $a_n>1$).

\begin{rem}
	As the orientation of $r$ is fixed, two consecutive segments are of the same kind if and only if they are relative to the same vertex. If $x\inn r$ is on a Farey line and if the two segments having $x$ as an endpoint are of different kinds, we will write $\cdots L^axR^b\cdots$ (or $\cdots R^axL^b\cdots$).

	Changing the orientation of $r$ is equivalent to reverse its cutting sequence and to change the kind of any segment: if $\cdots L^{a_0}R^{a_1}L^{a_2}\cdots$ is the cutting sequence of $r$, the cutting sequence of $-r$ will be $\cdots R^{a_2}L^{a_1}R^{a_0}\cdots$.
\end{rem}

\begin{lem}
	Let $\sigma$ be a segment cut by $\C F$ on a geodesic $r$ and let $A\inn\Gamma$. Then $A\sigma$ is a segment cut by $\C F$ on $Ar$ and $\sigma,\ A\sigma$ are of the same kind. 
\end{lem}
\begin{proof}
	This follows from the fact that the isometries of $\PSL R$ preserve the orientation.
\end{proof}	

\begin{theo}\label{Atheo:geocf}
	Let $\alpha\inn\n R,\ \alpha\geq1$, let $r\sub\n H$ be a geodesic of endpoints $r^+=\alpha$ and $-1<r^-<0$, let $x$ be the intersection of $r, s$. Then the cutting sequence of $r$ is of the form $$\cdots L^{a_{-2}}R^{a_{-1}}xL^{a_0}R^{a_1}\cdots$$ and the continued fraction expansion of $\alpha$ is $$\alpha=[a_0,a_1,\dots],$$ while $-r^-=[0,a_{-1},a_{-2},\dots]$.
\end{theo}

\begin{proof}[Sketch of Proof]
Let us assume $\alpha\notin\n Q$. Then, as $\alpha>1$ and $r^-\inn(-1,0)$, the first segment after $x$ will be labelled with $L$, while the segment immediately before $x$ will be labelled with $R$. Thus, the number $a_0$ of vertical lines (different from $s$) cut by $r$ is exactly $\floor\alpha$, that is, $\alpha=[a_0,\alpha_1]$.

Now, let $\lambda_{a_0}=\mm0{-1}1{-a_0}\inn\PSL Z$; $r_1=\lambda_{a_0}(r)$ is a geodesic of endpoints $r_1^-=-\frac1{r^--a_0},\ r_1^+=-\alpha_1$, so $0<r_1^-<1,\ r_1^+<-1$. The cutting sequence of $r_1$ is $\cdots L^{a_{-2}}R^{a_{-1}}L^{a_0}yR^{a_1}\cdots$, where $y$ is its intersection with $s$. As before, $a_1$ is the number of vertical lines cut by $r_1$, that is, $a_1=\floor{\alpha_1}$ and $\alpha=[a_0,a_1,\alpha_2]$. Applying the transformation $\rho_{a_1}=\mm0{-1}1{a_1}$ we obtain again a geodesic satisfying the hypotheses of the Theorem. 

Going on in this way, we get that $\alpha=[a_0,a_1,a_2,\dots]$; with a similar reasoning we can also obtain the continued fraction expansion of $r^-$. 
\end{proof}

\begin{rem}
	In particular, we have proved that any two geodesics cutting $s$ and with the same positive (respectively, negative) endpoint have the same positive (respectively, negative) cutting sequence.

	More generally, if $r_1,r_2$ are two geodesics with the same positive (or negative) endpoint $\alpha$, then their cutting sequences are eventually equal.
\end{rem}

We can prove in this setting some classical results for real continued fractions, whose polynomial analogues we have discussed in the previous Chapters (see Lemma \ref{1lem:pnqnprime}, Theorem \ref{1theo:serret}, Lemma \ref{2lem:galois}, Theorem \ref{2theo:lagrange}).  

\begin{rem}
	Let $\alpha>1$ be a real number and let $r$ be a geodesic of endpoints $r^+=\alpha,\ -1<r^-<0$; as in the previous Theorem, let $\cdots R^{a_{-1}}xL^{a_0}R^{a_1}\cdots$ be its cutting sequence. It can be proved by induction that the vertex associated to the segments in the $(n+1)$-th term ($R^{a_{n+1}}$ or $L^{a_{n+1}}$) is the $n$-th convergent of $\alpha$, $\cv n$. 

	As the vertices associated to two consecutive segments either coincide or are connected by a Farey line, we find again that $$p_nq_{n-1}-p_{n-1}q_n=\pm1.$$
\end{rem}

\begin{theo}[Serret]\label{Atheo:Serretgeod}
	Let $\alpha,\beta\inn\n R$. Then, there exists $A\inn\GL_2(\n Z)$ such that $A\alpha=\beta$ if and only if the continued fraction expansions of $\alpha,\beta$ eventually coincide.
\end{theo}

\begin{proof}
	Possibly by translating them, we can assume $\alpha,\beta>0$. 

	Let $\alpha=[a_0,a_1,\dots, a_n,c_1,c_2,\dots]$, $\beta=[b_0,b_1,\dots, b_m,c_1,c_2,\dots]$. Now, by Theorem \ref{Atheo:geocf}, $\alpha$ is $\SL_2(\n Z)$-equivalent to $(-1)^{n-1}[c_1,c_2,\dots]$ and $\beta$ is $\SL_2(\n Z)$-equivalent to $(-1)^{m-1}[c_1,c_2,\dots]$, so $\alpha,\beta$ are $\GL_2(\n Z)$-equivalent.\par\medskip 

	Conversely, let $\alpha=A\beta$ with $A\inn\GL_2(\n Z)$. Certainly $\alpha_1=\mm 011{-a_0} \alpha$ so, possibly substituting $\alpha$ with $\alpha_1$, we can assume $A\inn\SL_2(\n Z)$. Let $r_\alpha,r_\beta$ be geodesics cutting $s$ and with positive endpoints respectively $\alpha,\beta$. Then the cutting sequences of $r_\beta,Ar_\alpha$ eventually coincide, so the continued fractions of $\alpha,\beta$ eventually coincide.
\end{proof}

\begin{theo}[Lagrange]\label{Atheo:Lagrangegeod}
	The continued fraction expansion of $\alpha\inn\n R$ is eventually periodic if and only if $\alpha$ is a real quadratic irrationality
\end{theo}

\begin{theo}[Galois]\label{Atheo:Galoisgeod}
	The continued fraction expansion of $\alpha\inn\n R$ is purely periodic if and only if $\alpha$ is a reduced quadratic irrationality, that is, $\alpha>1$ and $-1<\alpha'<0$, where $\alpha'$ is the Galois conjugate of $\alpha$. 

	Moreover, if $\alpha=[\ov{a_1,\dots,a_{2n}}]$, then ${-1}/{\alpha'}=[\ov{a_{2n},\dots,a_1}]$. 
\end{theo}

\begin{proof}[Proof of Theorem \ref{Atheo:Galoisgeod}]
	Let $r$ be the geodesic connecting $\alpha=[\ov{a_1,\dots,a_{2n}}]>1$ and $\beta=-[0,\ov{a_{2n},\dots,a_1}]$. Its cutting sequence is periodic, so there exists $A\inn\PSL Z$ such that $Ar=r$. In particular, $\alpha,\beta$ are the fixed points of $A$, so they are a couple of conjugated quadratic irrationals, which implies that $\alpha$ is reduced.

	Conversely, let $\alpha=\frac{a+\sqrt{D}}b$ be a reduced quadratic irrationality and let $(x,y)$ be a non-trivial solution of the Pell equation for $b^2D$, that is, let $x^2-Db^2y^2=1$ with $x,y\inn\n Z,\ y\neq0$. Let $A=\mm{x+aby}{(D-a^2)y}{b^2y}{x-aby}$. Then $A\inn\SL_2(\n Z)$ and $A\alpha=\alpha$. Now, $A$ fixes also $\alpha'$, so, if $r$ is the geodesic of endpoints $\alpha,\alpha'$, $Ar=r$. Let $\cdots R^{a_{-1}}xL^{a_0}\cdots$ be the cutting sequence of $r$, with $\{x\}=r\cap s$. As $r=Ar$, its cutting sequence must be periodic; as $\alpha$ is reduced, by Theorem \ref{Atheo:geocf} its continued fraction expansion is read immediately on the cutting sequence of $r$, so it is purely periodic. 
\end{proof}

\begin{proof}[Proof of Theorem \ref{Atheo:Lagrangegeod}]
	$\alpha$ is a real quadratic irrationality if and only if $\alpha$ is $\GL_2(\n Z)$-equivalent to a reduced quadratic irrationality, if and only if, by Theorems \ref{Atheo:Galoisgeod}, \ref{Atheo:Serretgeod}, the continued fraction of $\alpha$ is eventually periodic.
\end{proof}

\section{Modular surface}
Let $\n M$ be the modular surface, that is, the quotient of the hyperbolic plane with respect to the left action of $\SL_2(\n Z)$: $$\n M=\SL_2(\n Z)\backslash\n H;$$ we will denote by $\ov P$ the projection of a point $P\inn\n H$ on $\n M$.  Topologically, $\n M$ is a sphere minus one point ($\n M$ has a cusp which would correspond to the projection of $\infty\inn\ch H$) and with singular points at $\ov i,\ov{\zeta_3}=\ov{\frac{1+i\sqrt3}2}$. 

The geodesics of $\n M$ are exactly the projections of the geodesics of $\n H$. In particular, all the Farey lines project on the line $\ov s$, going from the cusp to $\ov i$ and back again. As cutting sequences are invariant for the action of $\SL_2(\n Z)$, the cutting sequence of a geodesic of $\n M$ is well defined, up to shifts, as the cutting sequence of any of its lifts. Moreover, two geodesics on $\n M$ with the same cutting sequence coincide. 

\begin{lem}[Artin \cite{artin1924mechanisches}]
	 There exist dense geodesics on $\n M$.
\end{lem}
\begin{proof}
	Let $r_1,r_2$ be two geodesics of $\n H$ with cutting sequences $\cdots R^{a_{-1}}xL^{a_0}R^{a_1}\cdots$ and $\cdots R^{b_{-1}}yL^{b_0}R^{b_1}\cdots$, where $x,y$ are their intersections with $s$ or, more generally, with a same Farey line. $r_1,r_2$ are ``close" if and only if the central parts of their cutting sequences coincide, that is, if and only if there exist $M>0,N<0$ such that $a_i=b_i$ for every $N\leq i\leq M$. By Theorem \ref{Atheo:geocf}, if $r_1^+,r_2^+>1,\ -1<r_1^-,r_2^-<0$, this is equivalent to the fact that the first terms of the continued fraction expansions of $r_1^+,r_2^+$ and of $-r_1^-,-r_2^-$ coincide.

	Let $\xi>1$ be a real number such that in its continued fraction expansion appears any finite sequence of natural numbers; let $r\sub\n H$ be a geodesic of positive endpoint $\xi$ (and with $-1<r^-<0$). Let $R$ be the set of all geodesics with the same cutting sequence as $r$, that is, $R=\{Ar,\ A\inn\PSL Z\}$. Then any geodesic in $\n H$ can be approximated arbitrarily well by geodesics in $R$, so the image of $r$ in $\n M$ is dense in $\n M$.
\end{proof}

McMullen formulated in this setting his results on quadratic irrational numbers with bounded partial quotients.

\begin{lem}
	A geodesic $r$ of $\n H$ is mapped to a closed geodesic of $\n M$ if and only if there exists an isometry $A\in\PSL Z$ such that $Ar=r$, if and only if there exists an isometry $A\inn\PSL Z$ that fixes the endpoints of $r$. In particular then the endpoints of $r$ are conjugated quadratic irrationalities in $\n Q(\sqrt d)$, where $d=\tr(A)^2-4$.
\end{lem}

\begin{notat}	
	It can be seen that a matrix $A\inn\PSL R$ has exactly two fixed point on $\partial\n H$ (and no fixed points in $\n H$) if and only if $|\tr(A)|>2$; in this case $A$ is said to be \textit{hyperbolic} and the geodesic connecting its fixed points is called the \textit{axis} of $A$. We can then consider a surjective map from the set of hyperbolic isometries in $\PSL Z$ to the set of closed geodesics of $\n M$, given by $A\mapsto\ov r$, where $r$ is the axis of $A$. 
\end{notat}	

If $A,B\inn\PSL Z$ are conjugated in $\PSL Z$, their axes have the same projection in $\n M$. Moreover, two powers of a same hyperbolic matrix obviously have the same axis; we will say that an hyperbolic isometry is primitive if it is represented by a matrix that is not a non-trivial power of some other matrix in $\PSL Z$. Thus, we have a surjective map from the set of conjugacy classes of primitive hyperbolic isometries in $\PSL Z$ to the set of closed geodesics of $\n M$. It can be shown that this is in fact a bijection.
	
Of course, the lifts of a closed geodesic $\ov r$ will have different endpoints, but they will all be $\SL_2(\n Z)$-equivalent, so in particular they will be defined on a same quadratic extension $\n Q(\sqrt d)$ of $\n Q$. In this case, we will say that $\ov r$ is defined over $\n Q(\sqrt d)$.

\begin{lem}
If a closed geodesic $\ov r$ of $\n M$ is associated to a primitive hyperbolic matrix $A\inn\SL_2(\n Z)$, the length of $\ov r$ is the translation length of $A$, $$L(\ov r)=d(A)=\inf_{z\in\n H}d(z,Az).$$
\end{lem}

A closed geodesic $\ov r$ is said to be \textit{fundamental} if there is not a shorter geodesic whose length divides $L(\ov r)$.
It can be proved that if $A$ fixes a closed geodesic $\ov r$ and $(\tr\, A)^2-4$ is squarefree, then $\ov r$ is fundamental.

\begin{rem}
	As in Lemma \ref{Alem:unittg}, we can identify the unit tangent bundle\footnote{As always, we will speak of the tangent bundle of $\n M$, without bothering about the fact that this is not well defined at the singularities $\ov i$ and $\ov{\zeta_3}$.}  $T^1\n M$ of $\n M$ with $\PSL R/\PSL Z$: $[(z_1,v_1)]\!=\![(z_2,v_2)]\!\inn T^1\n M$ if and only if $(z_1,v_1)\!=\!A(z_2,v_2)\!\inn T^1\n H$ with $A\inn\PSL Z$.
\end{rem}

\begin{defn}
	We will say that a closed geodesic $\ov r\sub\n M$ is \textit{low-lying} if it stays away from the cusp of $\n M$. More precisely, we will say that $\ov r$ is low-lying of height $c$ if its pre-image $r$ in the fundamental domain $D$ has imaginary part bounded by $c$.
	
	A real number $\alpha$ is said to be \textit{absolutely Diophantine} of height $m$ if $\alpha=[a_0,a_1,\dots]$ with $a_n\leq m$ for every $n$.
	
	If $r$ is a geodesic of endpoints $\alpha,\beta$, then $r$ is low-lying of height $c$ if and only if $\alpha,\beta$ are absolutely Diophantine of height $m=m(c)$.
\end{defn}

Thus, to study continued fractions with bounded partial quotients is equivalent to studying geodesics of $\n M$ contained in a compact set.

\begin{theo}[McMullen, Theorem 1.2 in \cite{mcmullen2009uniformly}]
	For any fundamental closed geodesic $\ov r\sub\n M$ there exists a compact subset $Z$ of $\n M$ that contains infinitely many closed geodesics whose lengths are integral multiples of $L(\ov r)$.
\end{theo}

\begin{proof}[Sketch of Proof]
	Let $\ov r\inn\n M$ be a closed fundamental geodesic, let $A\inn\PSL Z$ be a primitive hyperbolic isometry associated to $\ov r$. Possibly changing the sign of $A$, we can assume that the largest eigenvalue of $A$ is a quadratic unity $\epsilon>1$ of norm $1$. Now, it can be proved that $L(\ov r)=d(A)$ depends only on the trace of $A$. Thus, we can assume $A=\mm 0{-1}1t$, which corresponds to substitute $\ov r$ with a geodesic of equal length.

	Let $H\sub\PSL R$ be the centralizer of $A$. Then $H$ is conjugated to the subgroup of diagonal matrices of $\PSL R$.

	Let $r$ be a lift of $\ov r$ on  $\n H$ and let $(z,v)\inn T^1\n H$ be a unit tangent vector to $r$. We can identify $T^1\n H$ with $\PSL R$ with base point $(z_0,v_0)$, so $H$ represents the geodesic flow. Identifying in the same way $T^1\n M$ with $\PSL R/\SL_2(\n Z)$, the compact orbit $H[\id]\simeq H/\langle A\rangle$ projects to $\ov r$. Indeed, we consider the identification $\PSL R\ni B\longleftrightarrow (z,v)\inn T^1(\n H)$ given by $(z,v)=B(z_0,v_0)$. The geodesic flow is then $g(t,(z,v))=B\widetilde{G_t}B^{-1}(z,v)$, with $\widetilde{G_t}=CG_tC^{-1}$, where $(z_0,v_0)=C(i,i)$. Now, $C$ transforms $s$ in $r$, so in fact $H=\left\{\widetilde{G_t},\ t\inn\n R\right\}$. We will write $\widetilde{G_t}*(z,v)=g(t(z,v))$. 

	Certainly, $H*[\id]=\left\{\left[\widetilde{G_t}(z_0,v_0)\right]\inn T^1(\n M),\ t\inn\n R\right\}$ projects on $\ov r$. 
	
	Moreover, $\left[\widetilde{G_t}(z_0,v_0)\right]=[(z_0,v_0)]$ if and only if $\widetilde{G_t}\inn H\cap\SL_2(\n Z)$, if and only if $\widetilde{G_t}\inn\langle A\rangle$ (because $A$ is primitive).

	It is possible to construct a sequence of primitive matrices $(L_m)\sub\PSL R$ such that the corresponding unit tangent vectors $[(z_m,v_m)]\inn T^1\n M$ are in a compact, $A$-invariant subset $Z$ of $T^1\n M$. As $H/\langle A\rangle$ is compact, we can assume that $Z$ is also $H$-invariant. Moreover, we can choose the $L_m$ so that the orbit of $[(z_m,v_m)]$ under $A$ has finite length $k(m)$ for every $m$, but $k(m)\to\infty$ for $m\to\infty$.

	As $A$ is fundamental, the stabilizer of $[(z_m,v_m)]$ in $H$ must then be $\langle A^{k(m)}\rangle$, so $H*[(z_m,v_m)]$ projects on a closed geodesic $\ov r_m$ on $\n M$ of length $k(m)L(\ov r)$.

	Moreover, all these geodesics lie in the projection of $Z$ in $\n M$, compact.
\end{proof}

We can then prove McMullen's Theorem \ref{3theo:MM}:
\begin{theo}[McMullen, Theorem 1.1 in \cite{mcmullen2009uniformly}]
	$\!$For any real quadratic field $\n K\!\!=\!\!\n Q(\sqrt d)$ there exists a compact subset $Z_d$ of $\n M$ that contains infinitely many closed geodesics defined over $\n K$. 

	Equivalently, for every real quadratic field $\n K=\n Q(\sqrt d)$ there exists a constant $m_d\inn\n N$ such that $\n K$ contains infinitely many reduced irrationals which are absolutely Diophantine of height $m_d$. 
\end{theo}

\begin{proof}
	Let $\epsilon>1$ be a unit of $\n K$ of norm $1$ and integral trace $t$; let $A=\mm 0{-1}1t$. By the previous Theorem, we can construct an infinite sequence $(\ov r_n)_n$ of closed geodesics in a compact subset $Z_d$ of $\n M$ such that for every $n$ there exists $L_n\inn M_2(\n Z), k(n)\inn\n Z$ with $L_nA^{k(n)}L_n^{-1}\ov r_n=\ov r_n$. Possibly changing the orientation of $\ov r_n$, we can always choose a lift $r_n$ of $\ov r_n$ such that its endpoints are $r_n^+>1,-1<r_n^-<0$ . Now, $r_n^+,r_n^-$ are fixed by a conjugate (in $M_2(\n Z)$) of a power of $A$, so $r_n^+,r_n^-$ are a pair of Galois conjugate points in $\n K$.

	As the $r_n^+$ are reduced, their continued fractions are purely periodic, and as the geodesics $r_n$ are in a compact set, their partial quotients are uniformly bounded.   
\end{proof}

This led McMullen to ask if the constants $m_d$ can be replaced by some absolute constant $m$:
\begin{conj}[McMullen's conjecture]
	There exists a compact subset $Z$ of $\n M$ such that, for every real quadratic field $\n K=\n Q(\sqrt d)$, $Z$ contains infinitely many closed geodesics defined over $\n K$.
	
	Equivalently, there exists an absolute constant $m$ such that any real quadratic field $\n K=\n Q(\sqrt d)$ contains infinitely many reduced quadratic irrationalities which are absolutely Diophantine of height $m$.
\end{conj}

We can also ask if, given a constant $c$ and a real quadratic field $\n K$, there exist longer and longer closed geodesics defined over $\n K$ which are low-lying of height $c$. Equivalently, we can ask if, given a constant $m$ and a real quadratic field $\n K$, there exist reduced quadratic irrationals in $\n K$ with longer and longer periods which are absolutely Diophantine of height $m$. McMullen thus proposed the following stronger conjecture:
\begin{conj}[McMullen's Arithmetic Chaos Conjecture \cite{mcmullendynamics}]
	There exists a compact subset $Y$ of $T^1(\n M)$ such that for every real quadratic field $\n K=\n Q(\sqrt d)$ the set of closed geodesics contained in $Y$ and defined over $\n K$ has positive entropy.
	
	Equivalently, there exists an absolute constant $m$ such that for every real quadratic field $\n K=\n Q(\sqrt d)$ the set $\{[\ov{a_0,\dots,a_l}]\inn\n K, a_i\leq m \text{ for every } i\}$ has exponential growth for $l\to\infty$.
\end{conj}

However, currently it is not even known if for every quadratic field $\n K=\n Q(\sqrt d)$ there exists a constant $m_d$ such that the previous set grows exponentially for $l\to\infty$.

\clearpage{\pagestyle{empty}\cleardoublepage}

\bibliographystyle{amsplain}
\bibliography{i3-bib}

\end{document}